\documentclass[11pt,letterpaper]{amsart}
\usepackage[english]{babel}
\usepackage{graphicx}
\usepackage{amsmath,amsfonts,amssymb,amsthm,relsize,setspace,nicefrac,yhmath,amscd,eucal}
\usepackage{amsrefs, mathrsfs}
\usepackage{mathtools,tkz-euclide,tikz-3dplot,tkz-tab}
\usepackage{xcolor}
\usepackage{tikz}
\usepackage{enumerate,float} 
\usetikzlibrary{calc}
\usepackage{comment}
\usepackage{anyfontsize}
\usepackage{lmodern}
\usepackage[T1]{fontenc}
\usepackage{subcaption}
\usepackage[colorlinks, linkcolor=red, citecolor=blue, urlcolor=blue, hypertexnames=true]{hyperref} 

\usepackage{hyperref} 

\usepackage[active]{srcltx} 
\usepackage{color,soul} 
\newtheorem{theorem}{Theorem}
\newtheorem{corollary}[theorem]{Corollary}
\newtheorem{lemma}[theorem]{Lemma}
\newtheorem{proposition}[theorem]{Proposition}

\usepackage[toc,page]{appendix}

\newtheorem{remark}{Remark}
\newtheorem{example}{Example}

\theoremstyle{definition}
\newtheorem{definition}[theorem]{Definition}

\newcommand{\be}{\begin{equation}}
\newcommand{\bel}[1]{\begin{equation}\label{#1}}
\newcommand{\ee}{\end{equation}}

\newcommand{\barr}{\begin{eqnarray}}
\newcommand{\earr}{\end{eqnarray}}
\newcommand{\bars}{\begin{eqnarray*}}
\newcommand{\ears}{\end{eqnarray*}}

\newtheorem{subn}{\name}

\newcommand{\bsn}[1]{\def\name{#1}\begin{subn}}
\newcommand{\esn}{\end{subn}}

\newtheorem{sub}{\name}[section]

\newcommand{\bs}{\begin{sub}}
\newcommand{\es}{\end{sub}}

\newcommand{\bth}[1]{\def\name{Theorem}
\begin{sub}\label{t:#1}}
\newcommand{\blemma}[1]{\def\name{Lemma}
\begin{sub}\label{l:#1}}
\newcommand{\bcor}[1]{\def\name{Corollary}
\begin{sub}\label{c:#1}}
\newcommand{\bdef}[1]{\def\name{Definition}
\begin{sub}\label{d:#1}}
\newcommand{\bprop}[1]{\def\name{Proposition}
\begin{sub}\label{p:#1}}

\newcommand{\BA}{\begin{array}}
\newcommand{\EA}{\end{array}}
\newcommand{\BAN}{\renewcommand{\arraystretch}{1.2}
\setlength{\arraycolsep}{2pt}\begin{array}}
\newcommand{\BAV}[2]{\renewcommand{\arraystretch}{#1}
\setlength{\arraycolsep}{#2}\begin{array}}
\newcommand{\BSA}{\begin{subarray}}
\newcommand{\ESA}{\end{subarray}}

\newcommand{\BAL}{\begin{aligned}}
\newcommand{\EAL}{\end{aligned}}
\newcommand{\BALG}{\begin{alignat}}
\newcommand{\EALG}{\end{alignat}}
\newcommand{\BALGN}{\begin{alignat*}}
\newcommand{\EALGN}{\end{alignat*}}

\newcommand{\abs}[1]{\left |#1\right |}

\def\angb<#1>{\langle #1 \rangle}

\definecolor{inkdark}{RGB}{14,54,140}
\definecolor{inklight}{RGB}{198,220,255}

\let\.=\cdot
\let\0=\emptyset

\def\diam{\text{\rm diam}}

\numberwithin{equation}{section}

\theoremstyle{definition}

\let\.=\cdot
\let\0=\emptyset

\def\diam{\text{\rm diam}}

\newenvironment{formula}[1]{\begin{equation}\label{eq:#1}}
{\end{equation}\noindent}

\def\Fi#1{\begin{formula}{#1}}
\def\Ff{\end{formula}\noindent}

\setstcolor{red}

\title[]{Harnack Theory and Rigidity for Singular and Degenerate Fully Nonlinear Elliptic Equations with Hamiltonians}
\makeatletter
\def\l@section{\@tocline{1}{0pt}{0pt}{2.5em}{}}
\def\l@subsection{\@tocline{2}{0pt}{2.5em}{3.5em}{}}
\def\l@subsubsection{\@tocline{3}{0pt}{5em}{4em}{}}
\makeatother

\makeatletter
\renewcommand{\@secnumfont}{\bfseries}
\makeatother

\date{\today}
\subjclass[2020]{35B65, 35B45, 35J60, 35J70, 35D40.}

\keywords{Singular/degenerate fully nonlinear equations $\cdot$ Global regularity $\cdot$ Comparison principle $\cdot$ Harnack inequality $\cdot$ Principal eigenvalues $\cdot$ Liouville–type results.}

\author{Trung-Hieu Huynh}
\address{Department of Mathematics, Ho Chi Minh University of Education,
	280 An Duong Vuong St., Cho Quan Ward, Ho Chi Minh City, Vietnam}
\email{huynhhieu2004@gmail.com}

\author{Tan-Dat Khuu}
\address{Faculty of Mathematics and Statistics, Ton Duc Thang University,
	19 Nguyen Huu Tho St., Tan Hung Ward, Ho Chi Minh City, Vietnam}
\email{tandatkhuu2k3@gmail.com}

\author{Hoang-Hung Vo}
\address{Department of Mathematics, International University,
	Quarter 6, Linh Trung Ward, Thu Duc City, Ho Chi Minh City, Vietnam}
\address{Vietnam National University, Ho Chi Minh City, Vietnam}
\email{vhhung@hcmiu.edu.vn}

\begin{document}
%
%


\begin{abstract}
In this paper, we study three fundamental problems for viscosity solutions to a general class of singular or degenerate fully nonlinear elliptic equations of the form
\begin{align*}
	\Phi(x,|\nabla u|)F(D^2u)-H(x,\nabla u)+c(x)|u|^{i(\Phi)}u =h(x) \text{ in } \Omega,
\end{align*}
where $\Omega$ is a bounded domain, $\Phi$ captures the gradient degeneracy or singularity, and $H$ is a Hamiltonian. Firstly, we obtain global $C^{1,\gamma}$-regularity for viscosity solutions to the Dirichlet problem in the absence of lower-order terms. The proof relies on a step-by-step regularity upgrade, sequentially establishing global boundedness via the Alexandroff--Bakelman--Pucci estimate, a global Lipschitz estimate via boundary barriers, and an iterative approximation lemma driven by a compactness argument. Secondly, we establish an additive Harnack-type inequality for nonnegative viscosity solutions by sliding from below a cusp function of the form $\varphi(x) = -|x|^{1/2}$. This approach is expected to be optimal for our setting. Under an additional homogeneity assumption on the operator, we recover the standard Harnack inequality, which in turn yields Liouville-type theorems in $\mathbb{R}^n$. Under additional homogeneity and comparison assumptions, we develop a generalized Dirichlet principal-eigenvalue theory for the nonlinear operator
\begin{align*}
	\Phi(x,|\nabla u|)F(D^2u)-H(x,\nabla u)+c(x)|u|^{i(\Phi)}u
\end{align*}
in smooth bounded domains. In particular, we establish the existence of principal eigenfunctions and relate the corresponding eigenvalues to the validity of maximum and minimum principles, thereby extending the theory of \cite{BD2006} to operators with a substantially broader class of gradient-dependent ellipticity factors and Hamiltonian terms. Our results provide a unified treatment of global regularity and Harnack estimates for singular and degenerate fully nonlinear equations, encompassing and extending the frameworks of \cite{BBLL2024,BBLL2024_1,BD2016,IS2016}, and yield new Liouville-type rigidity theorems for entire solutions.
\end{abstract}
\maketitle
\tableofcontents




\section{\bf Introduction}
\label{section:1}
In this paper, we establish global regularity results for viscosity solutions to a broad class of degenerate/singular fully nonlinear elliptic equations of the form
\begin{align}\label{1.2}
\left\{\begin{array}{cll}
	\Phi(x,|\nabla u|)F(D^2u)-H(x,\nabla u) &=h(x) &\text{in } \Omega,\\
	u(x) &=g(x) &\text{on }\partial\Omega,
\end{array}\right.
\end{align}
and prove Harnack inequalities for
\begin{align}\label{eq:main}
\Phi(x,|\nabla u|)F(D^2u)-H(x,\nabla u)+c(x)|u|^{i(\Phi)}u =h(x) \text{ in } \Omega,
\end{align}
where $\Omega\subset\mathbb{R}^n$ is a bounded domain, $F:\mathbb S^{n\times n}\to\mathbb{R}$ is a uniformly $(\lambda,\Lambda)$-elliptic operator in the sense of \hyperref[A1]{(A1)}, $\Phi:\Omega\times[0,\infty)\to[0,\infty)$ describes the gradient degeneracy or singularity and satisfies \hyperref[A2]{(A2)}, $H:\overline{\Omega}\times\mathbb{R}^n\to\mathbb{R}$ is a Hamiltonian term satisfying \hyperref[A3]{(A3)}, $c(\cdot), h(\cdot)$ and $g(\cdot)$ are suitable regular functions as in \hyperref[A4]{(A4)}. We also develop a generalized principal eigenvalue theory for the operator associated with \eqref{eq:main}. We recall that, as a consequence of Krylov-Safonov theory \cite{KS1979,KS1981}, viscosity solutions to the homogeneous equation
\begin{align*}
F(D^2u) = 0 \mbox{ in }B_1, \mbox{ where } F \mbox{ is uniformly} \left(\lambda,\Lambda\right)-\text{elliptic},
\end{align*}
belong to $C^{1,\alpha_0}_{\mathrm{loc}}(B_1)$ for a universal constant $\alpha_0 \equiv \alpha_0(n,\lambda,\Lambda) \in (0,1)$.

Over the past few decades, this class of equations has received considerable attention with several specific forms of $\Phi$ and $H$. Models whose ellipticity depends on the gradient are motivated in part by the classical $p$-Laplacian, which is closely connected with the study of local minimizers of the $p$-Dirichlet integral and numerous real-life models. From the perspective of variational analysis, fully nonlinear equations can be seen as the non-divergence counterparts of models arising in variational problems; see \cite{HO2022} for a detailed discussion. The study of degenerate/singular fully nonlinear equations was initiated by the pioneering works of Birindelli and Demengel in \cite{BD2004,BD2007}. They established Liouville-type results and comparison principles in the singular setting in \cite{BD2004} and later proved regularity and uniqueness results for generalized principal eigenfunctions in \cite{BD2007}. 
An important special case of \eqref{1.2}, obtained by setting $H\equiv 0$ and $\Phi(x,t)=t^\alpha$ with $\alpha>-1$, is given by
\begin{align}\label{IS}
|\nabla u|^{\alpha}F(D^2u)=h\text{ in }\Omega.
\end{align}
Imbert and Silvestre in \cite{IS2013} and Araújo, Ricarte, and Teixeira in \cite{ART2015} proved local $C^{1,\gamma}$ regularity for viscosity solutions to \eqref{IS} in the degenerate case.  Subsequent works by Filippis \cite{F2021}, Silva and Ricarte \cite{SR2020}, and Yang, Rădulescu, and Zhang \cite{FDZ2021} extended sharp interior $C^{1,\gamma}$ estimates to degenerate versions of \eqref{1.2} involving the function $\Phi$ of double-phase growth and variable exponents, i.e., $\Phi(x,t)=t^p+a(x)t^q$ with $0\le p\le q$ and $\Phi(x,t)=t^{p(x)}+a(x)t^{q(x)}$ with $0\le p(\cdot)\le q(\cdot)$. Notably, Andrade, Pellegrino, Pimentel, and Teixeira in \cite{PDEE2022} studied the local $C^1$-differentiability properties for solutions to degenerate diffusion equations of the form
\begin{align*}
\Phi(|\nabla u|)F(D^2u)=f\text{ in }B_1,
\end{align*}
assuming that $\Phi$ is a modulus of continuity and its inverse $\Phi^{-1}$ satisfies the Dini condition. In general, we note that this equation does not belong to the class of fully nonlinear equations considered here. 
Recently, Baasandorj, Byun, Lee, and Lee in \cite{BBLL2024_1} established local $C^{1,\gamma}$ regularity for a unified class of degenerate/singular fully nonlinear equations covering the previously mentioned equations with the function $\Phi$ satisfying the assumption \hyperref[A2]{(A2)}. The corresponding theory was subsequently extended to the boundary. In particular, by using suitable compactness arguments, \cite{AB2023} obtained global $C^{1,\gamma}$ estimates for the case $\Phi(x,t)=t^p$ with $p\geq 0$, \cite{JSRR2023} treated the case $\Phi(x,t)=t^{p(x)}+a(x)t^{q(x)}$ with $0\le p(\cdot)\le q(\cdot)$, and \cite{BBLL2024} considered a general class of functions $\Phi$ in the sense of assumption \hyperref[A2]{(A2)}. In addition to these elliptic equations, several extensions to Monge--Ampère equations, degenerate/singular fully nonlinear parabolic equations, degenerate/singular elliptic equations involving the normalized $p$-Laplace operators, and nonlocal equations have been obtained. We refer to Le and Savin \cite{LS2017}, Cao, Hirsch, and Inauen \cite{CHI2025}, Lee, Lee, and Yun \cite{LLY2024}, Lee and Yun \cite{LY2025}, Alcantara, Santos and Urbano \cite{ASU2026}, and Fang, Rădulescu, and Zhang \cite{FDZ2025} for results in these directions.

During the last decade, there has also been growing interest in fully nonlinear equations with gradient terms. The first regularity estimates for this class of equations were obtained by Birindelli and Demengel in \cite{BD2010} for the degenerate fully nonlinear equation
\begin{align*}
\left\{\begin{array}{cll}
	F(D^2u)+b(x)|\nabla u|^{\beta} &=h(x) &\text{in } \Omega,\\
	u(x) &=g(x) &\text{on }\partial\Omega,
\end{array}\right.
\end{align*}
where $\beta\in (0,1)$, proving that its viscosity solutions are globally of class $C^{1,\gamma}$. Subsequently, the authors of \cite{DLP2010} proved global H\"older continuity for the degenerate fully nonlinear elliptic equations of the form
\begin{align*}
-\mathrm{tr}(A(x)D^2u)+\lambda u+|\nabla u|^p=f(x),\quad x\in\Omega,
\end{align*}
where $\Omega$ is an open subset of $\mathbb{R}^n$ ($n\geq 2$) and $p>1$. The next well-known result was proposed by Birindelli and Demengel in \cite{BD2014}. In particular, these authors established boundary $C^{1,\gamma}$ estimates for
\begin{align*}
\left\{\begin{array}{cll}
	|\nabla u|^{\alpha}F(D^2u)+h(x)\cdot\nabla u|\nabla u|^{\alpha} &=f(x) &\text{in } \Omega,\\
	u(x) &=g(x) &\text{on }\partial\Omega.
\end{array}\right.
\end{align*}
They subsequently obtained analogous results in \cite{BD2016, BDL2019} for equations of the form 
\begin{align*}
\left\{\begin{array}{cll}
	|\nabla u|^{\alpha}F(D^2u)+b(x)|\nabla u|^{\beta} &=h(x) &\text{in } \Omega,\\
	u(x) &=g(x) &\text{on }\partial\Omega,
\end{array}\right.
\end{align*}
where $\alpha>-1$ and $\beta\in (0,2+\alpha]$. Additionally, Fang, Kinnunen, and Zhang in \cite{FKZ2025} also proved local $C^{1,\gamma}$-regularity for the nonlocal equations involving Hamiltonian terms. In the present work, we prove global $C^{1,\gamma}$-regularity for viscosity solution of \eqref{1.2} which includes a Hamiltonian term but no zeroth-order term. This extends the regularity theories developed in \cite{BBLL2024_1,BBLL2024,BD2016}. Moreover, this result can be compared with the profound work of Le-Savin \cite{LS2017} on the optimal regularity for eigenfunctions of Monge--Ampère operators. More precisely, the authors in \cite{LS2017} obtained $C^{2,\gamma}$ estimates up to the boundary for solutions to a class of degenerate Monge--Ampère equations. In particular, we first prove the global Lipschitz estimate (Theorem \ref{Lipschitz_estimate}) by the doubling-variable method with a suitable coupling function. Next, we adapt the compactness arguments in \cite{BBLL2024,BD2016} to prove an approximation lemma (Lemma \ref{approx1}), which is the key ingredient in the proof of global $C^{1,\gamma}$ regularity (Theorem \ref{Theo:Global_regularity}). The first difficulty is the simultaneous presence of gradient degeneracy or singularity through $\Phi$ and the Hamiltonian term $H$. To handle both regimes in both degenerate and singular cases, we focus on establishing global $C^{1,\gamma}$ estimates in the degenerate case $(i(\Phi)\geq 0)$, and then transfer it to the singular case $(-1<i(\Phi)<0)$. Additionally, the second difficulty arises in our general setting of Hamiltonian terms $H$, creating a considerable challenge in building the flatness lemma. We combine approximation methods in \cite{BBLL2024,BD2016} with a modified scaling argument leading to an improved approximation lemma. 

The Harnack inequality was first formulated and proved by Harnack for positive harmonic functions in the two-dimensional case; see \cite[paragraph 19, p. 62]{H1887}. In the 1950s, the parabolic Harnack estimate for non-negative caloric functions, that is, non-negative solutions $u$ to the heat equation
\begin{align*}
u_t - \Delta u = 0,
\end{align*} 
was first established by Hadamard in \cite{Ha1954} and independently by Pini in \cite{P1954}. The modern theory of elliptic Harnack inequalities began with Moser's work \cite{Mo1961,Mo1964}, who proved such estimates for non-negative weak solutions of second-order linear equations in divergence form with bounded measurable, uniformly elliptic coefficients. Moser's method also yielded H\"older continuity and was later extended to the parabolic setting. Independently, De Giorgi \cite{G1957} and Nash \cite{NA1958} obtained Hölder regularity by different methods that did not rely directly on Harnack inequalities. These foundational contributions are now collectively known as the De Giorgi--Nash--Moser theory. The theory was subsequently extended to elliptic and parabolic equations in divergence form, as discussed in the seminal papers \cite{LSU1967, S1964,PM2021}.

The De Giorgi--Nash--Moser method does not directly apply to equations in non-divergence form with bounded measurable and uniformly elliptic coefficients. The breakthrough was due to Krylov and Safonov who first proved Harnack inequalities for this class of equations in \cite{KS1979,KS1981}. Their method became central to the regularity theory of linear, nonlinear, and fully nonlinear second-order elliptic and parabolic equations. Caffarelli \cite{C1989} established Krylov and Safonov estimates for fully nonlinear elliptic equations of the form $F(x, D^2u) = 0$ (see also \cite{GT1983,T1988}). A fundamental tool is the Alexandroff--Bakelman--Pucci (ABP) estimate; see \cite{CC1995}. For further results and developments, we refer the interested reader to \cite{K1998,CLV2005,AB2008,KS2004,CFP2026} and the references therein.

Additional difficulties arise when ellipticity depends on the gradient, since the equation may become degenerate or singular at critical points. The classical uniformly elliptic theory is then no longer directly applicable. An early Krylov--Safonov-type result for degenerate equations in non-divergence form was obtained by Delarue \cite{D2010}. Dávila, Felmer, and Quaas \cite{DFQ2010} proved a Harnack inequality for positive solutions of a class of singular elliptic equations of the form
\begin{align*}
F(\nabla u,D^2u) + b(x) \cdot \nabla u|\nabla u|^\alpha + c(x)u|u|^\alpha = f(x) \mbox{ in }\Omega,
\end{align*}
where $-1<\alpha<0$ and $\Omega$ is a domain in $\mathbb R^n$.  This result was extended by Birindelli and Demengel \cite{BD2010-1} to the broader range $\alpha>-1$ in the two-dimensional case. In 2011, Imbert \cite{I2011} established both an Alexandroff--Bakelman--Pucci (ABP) estimate and a Harnack inequality for a class of fully nonlinear elliptic equations in non-divergence form,
\begin{align*}
F(x,u,\nabla u,D^2u) = 0 \text{ in } \Omega,
\end{align*}
which can be either degenerate or singular. Five years later, Imbert and Silvestre \cite{IS2016} advanced this theory by establishing a Harnack inequality under the weaker condition that the equation holds only at points where the gradient is sufficiently large. Their proof uses a sliding-cusp construction to obtain the key $L^\varepsilon$ estimate. $L^\varepsilon$ estimate's proof builds on ideas appearing in the work of Cabré \cite{C1997} and Savin \cite{S2007}. Savin slid paraboloids from below to estimate the measure of the corresponding contact set without invoking the standard ABP argument. Armstrong and Smart \cite{A2014} used translations of singular radial functions in a related non-uniformly elliptic setting. Mooney \cite{M2015} proved large-gradient Harnack estimates for equations with unbounded drift by sliding paraboloids at multiple scales. Le \cite{L2018} later obtained Harnack estimates for linear degenerate or singular equations with unbounded lower-order terms, including linearized Monge--Ampère equations, using generalized sliding paraboloids.

The classical sliding-paraboloid argument is not directly suited to our setting. Ellipticity is controlled by the factor $\Phi(x,|\nabla u|)$, which may vanish or become singular as $|\nabla u|\to0$. Since the gradient of a paraboloid vanishes at its vertex and may be small near it, its contact points need not remain separated from the zero-gradient regime. We therefore adapt the sliding-cusp method of \cite{IS2016}, using translations of the cusp $\varphi(x)=-|x|^{1/2}$. The cusp is singular only at its vertex. The contact construction rules out contact at that point, hence, it is smooth at every contact point. Moreover, $|\nabla\varphi(y-x)|=\frac{1}{2}|y-x|^{-1/2}$, and the contact pairs $(x,y)$ remain in fixed bounded sets. The gradient at each contact point is therefore bounded below by a positive constant. This permits a quantitative application of the viscosity inequality on the contact set without entering the zero-gradient regime.

Our aim is to establish Harnack inequalities for equation \eqref{eq:main}. Although the proof uses contact points with nonzero gradient, it does not require an a priori decomposition of the domain into large- and small-gradient regions. Furthermore, since viscosity solutions need not be differentiable, the contact arguments are formulated entirely in the viscosity framework. We first prove a measure estimate for non-negative supersolutions. We then combine it with a doubling property, obtained from an explicit radial barrier, and the growing ink-spots lemma to derive an $L^\varepsilon$ estimate. Finally, the $L^\varepsilon$ estimate and a local Lipschitz estimate yield the Harnack inequalities in Section \ref{section:4}. The first estimate contains an additive constant $(+1)$, which reflects the lack of invariance of the general equation under amplitude rescaling; see Theorem \ref{Theo:Harnack}. Under the additional homogeneity assumption {\upshape\hyperref[A6]{(A6)}}, this constant can be removed and the standard Harnack inequality is recovered; see Theorem \ref{Theo:HarnackII}.

As direct consequences of the scaled Harnack inequality established under the structural assumption {\upshape\hyperref[A6]{(A6)}}, we derive Liouville-type results in Section \ref{section: 6}. Such results are available for general classes of fully nonlinear elliptic equations \cite{CL2000,BD2004,BD2010-1}. Results for equations containing both first- and zeroth-order terms have largely been obtained in more specific settings, typically as Pucci-type equations with gradient and reaction terms \cite{CF2013}, Hamilton--Jacobi--Bellman equations with drift and potential terms \cite{BC2016}, and infinity-Laplacian equations involving KPP type or Hardy--Hénon nonlinearities \cite{AH2023,KHH2026}. 

We also develop generalized principal eigentheory for the degenerate/\\singular fully nonlinear operator
\begin{align*}
\Phi(x,|\nabla u|)F(D^2u)-H(x,\nabla u)+c(x)|u|^{i(\Phi)}u,
\end{align*}
with Dirichlet boundary condition in bounded domains and use it to characterize the validity of the maximum principle. 

The notion of generalized principal eigenvalue comes from the well-known work of Berestycki, Nirenberg, and Varadhan \cite{BNV1994}, where these authors explored a deep connection between the maximum principle and principal eigenvalues of linear uniformly elliptic operators. Since then, this line of research has attracted considerable attention due to its importance in the study of Liouville-type results, maximum principle, and the nonexistence or existence of positive solutions to different nonlinear equations. 
It also plays a crucial role in probability theory, especially in the theory of large deviations. Based on the ideas in \cite{BNV1994}, Birindelli-Demengel in \cite{BD2007,BD2006} developed an eigenvalue theory to the Dirichlet problem for degenerate/singular fully nonlinear operators with smooth domains.  A Dirichlet generalized principal eigenvalue theory for the same class of operators in non-smooth and unbounded domains has been investigated in \cite{BD2009} and \cite{BD2010-1}, respectively. Moreover, these authors also established a sharp global regularity result and the uniqueness of the first eigenfunction for a class of singular fully nonlinear operators in \cite{BD2010}. Additionally, we mention \cite{P2008} as a counterpart of \cite{BD2007,BD2006} in the case of Neumann boundary condition.

In addition, Quaas-Sirakov in \cite{AB2006,AB2008} studied principal eigenvalues of a general fully nonlinear operator $F$ which is uniformly elliptic, positively homogeneous, and convex. They proved that the operator $F$ has two generalized principal eigenvalues $\lambda_1^+(F,\Omega)$ and $\lambda_1^-(F,\Omega)$, where $\Omega$ is a bounded $C^1$ domain. The former and the latter correspond to a positive eigenfunction $\phi_1^+>0$ and a negative eigenfunction $\phi_1^-<0$, respectively. Moreover, they also asserted that the corresponding Dirichlet problem admits a unique viscosity solution, provided that these principal eigenvalues are positive. Subsequently, Armstrong in \cite{A2009} extended these results to a class of fully nonlinear operators which are not necessarily convex. Notably, Berestycki-Dolcetta-Porretta-Rossi \cite{HIAL2015} proposed the novel notion of generalized principal eigenvalue and used it to characterize the validity of the maximum principle in bounded domains for fully nonlinear degenerate elliptic operators, which covers many known results in \cite{AB2008,BD2007,J2007}. For more on generalized principal eigentheory of nonlinear elliptic operators, we refer to \cite{L2018-1,LS2017} for Monge--Ampère operators and \cite{AH2022,P2011,AH2023} for infinity-Laplace operators. 

The existence of two principal eigenvalues for nonlinear operators was first observed by Pucci \cite{P1966} and Berestycki \cite{B1977}. In particular, these authors proved this property for a Pucci-type operator and Sturm-Liouville equations and showed that maximum and minimum principles can be characterized by the positivity of these two principal eigenvalues. It is worth remarking that when the operator is not odd with respect to the Hessian, the two principal eigenvalues can be distinct. 


Our result generalizes the eigentheory studied in \cite{BD2006} to a broader class of Hamiltonian terms. Our strategy is to adapt the methods in \cite{BD2006} to the present setting. To the best of our knowledge, in the study of generalized principal eigenvalues, two typical structural properties of the
operators are homogeneity and ellipticity. Therefore, we need to impose assumption {\upshape\hyperref[A6]{(A6)}}, enabling us to obtain the positive homogeneity of our operator. Furthermore, one of the main difficulties in our problem is a lack of suitable comparison principles, which arises from the generality of the function $\Phi$ and the Hamiltonian term $H$. To overcome this challenge, we introduce assumption {\upshape\hyperref[A5]{(A5)}}, which allows us to establish a comparison principle and apply Perron's method to prove the existence of principal eigenfunctions.  

The paper is organized as follows. In Section \ref{section:2}, we introduce the notation and structural assumptions and collect the preliminary results used throughout the paper. Section \ref{section:3} focuses on the global regularity theory and the comparison principle for the Dirichlet problem. More precisely, we first establish an ABP estimate and boundary Lipschitz regularity, then prove global $C^{1,\gamma}$-regularity, and finally derive the local estimates and comparison principle needed later. Section \ref{section:4} is devoted to the Harnack inequality after establishing the required technical tools and auxiliary estimates. Section \ref{section: 5} treats the Dirichlet principal eigenvalue problem. Section \ref{section: 6} applies the Harnack inequality to obtain Liouville-type results.
\section{\bf Preliminaries and structural assumptions} \label{section:2}
In this section, for the reader's convenience, we first review some fundamental notation, preliminary definitions, and several technical tools that will be frequently utilized throughout the paper.

Let $\Omega$ be a domain in $\mathbb{R}^n$ ($n \geq 2$) and denote its boundary by $\partial\Omega$. For $x_0 \in \mathbb{R}^n$ and $r > 0$, we denote by $B_r(x_0) = \left\{y \in \mathbb{R}^n : |y - x_0| < r \right\}$ the open ball centered at $x_0$ with radius $r$; when $x_0=0$, this ball is simply denoted by $B_r$. Let $\mathbb{S}^{n\times n}$ signify the set of $n\times n$ real symmetric matrices. 

To avoid cluttering the exposition, we shall use generic symbols such as $C, D,$ etc., to represent various positive constants that depend only on certain prescribed parameters. To specify their exact origin, these dependencies will be explicitly indicated within parentheses; for instance, writing $C \equiv C(n, i(\Phi), \nu_0)$ means that the constant $C$ depends exclusively on $n, i(\Phi),$ and $\nu_0$. Given two positive functions $f_1$ and $f_2$, we will write $f_1 \lesssim f_2$ if there exists a positive universal constant $C$ such that $f_1 \leq Cf_2$.

In addition, the notation $u \prec_z \varphi$ signifies that $\varphi$ touches $u$ from above strictly at the point $z$; that is, there exists an open ball $B_r(z)$ around $z$ such that $u(x) < \varphi(x)$ for all $x \in B_r(z) \backslash \{z\}$ and $u(z) = \varphi(z)$. Similarly, $u \succ_z \varphi$ denotes that $\varphi$ touches $u$ from below strictly at $z$. 

To define the notions of viscosity subsolutions and supersolutions, we introduce the following standard function spaces:
\begin{align*}
	&\mathrm{USC}(\Omega) = \left\{u:\Omega \to \mathbb{R};~u \text{ is upper semicontinuous}\right\},\\
	&\mathrm{LSC}(\Omega) = \left\{u:\Omega \to \mathbb{R};~ u \text{ is lower semicontinuous}\right\}.
\end{align*}
For a measurable map $f : \Omega \to \mathbb{R}^n$ and a given parameter $\gamma \in (0, 1]$, the H\"older semi-norm of $f$ is defined by
\begin{align*}
	[f]_{C^{0,\gamma}(\Omega)} := \sup_{\substack{x, y \in \Omega \\ x \neq y}} \frac{|f(x) - f(y)|}{|x - y|^\gamma}.
\end{align*}
Finally, given a boundary datum $g$, we denote by $\mathrm{Lip}_g(\partial \Omega)$ the optimal Lipschitz constant of $g$ restricted to the boundary $\partial \Omega$.

In this article, we shall focus on the fully nonlinear operator defined by
\begin{align}
	\label{mathcal_G}
	\mathcal{G}(x,u,\nabla u,D^2 u) := \Phi(x, |\nabla u|)F(D^2u) - H(x, \nabla u) + c(x)|u|^{i(\Phi)}u.
\end{align} 
We now state the main assumptions that form the core of our analysis throughout this paper.\\
\noindent\phantomsection\label{A1}\textbf{Assumption A1.} The fully nonlinear operator $F:\mathbb{S}^{n\times n}\to\mathbb{R}$ is continuous and uniformly $(\lambda,\Lambda)$-elliptic in the sense that
\begin{align*}
	\lambda \mathrm{tr}(N)\le F(M+N)-F(M)\le \Lambda\mathrm{tr}(N)
\end{align*}
holds for some constants $0 < \lambda \le \Lambda$ with $F(0)=0$, whenever $M,N\in\mathbb{S}^{n\times n}$ with $N\geq 0$.\\	
\noindent\phantomsection\label{A2}\textbf{Assumption A2.} The function $\Phi:\Omega\times[0,\infty)\to[0,\infty)$ is continuous and satisfies the following properties:
\begin{enumerate}
	\item [(i)] There exist constants $s(\Phi)\geq i(\Phi)>-1$ and $L \geq 1$ such that the map $t\mapsto \frac{\Phi(x,t)}{t^{i(\Phi)}}$ is almost non-decreasing in $(0,\infty)$ in the sense that
	\begin{align*}
		\frac{\Phi(x,t)}{t^{i(\Phi)}} \leq L \frac{\Phi(x,\tau)}{\tau^{i(\Phi)}} \mbox{ whenever }0<t \leq \tau<\infty \mbox{ and }x \in \Omega,
	\end{align*}
	and the map $t\mapsto \frac{\Phi(x,t)}{t^{s(\Phi)}}$ is almost non-increasing in the sense that
	\begin{align*}
		L\frac{\Phi(x,t)}{t^{s(\Phi)}} \geq \frac{\Phi(x,\tau)}{\tau^{s(\Phi)}} \mbox{ whenever }0<t \leq \tau<\infty \mbox{ and }x \in \Omega.
	\end{align*}
	\item [(ii)] There exist constants $0<\nu_0\le \nu_1$ such that $\nu_0\le \Phi(x,1)\le \nu_1$ for all $x\in\Omega$.
\end{enumerate}
\noindent\phantomsection\label{A3}\textbf{Assumption A3.} The Hamiltonian term $H: \overline{\Omega}\times \mathbb{R}^n \to \mathbb{R}$ is continuous, and there exists $0 < m \leq i(\Phi) + 1$ such that for all $(x,p) \in \overline{\Omega}\times\mathbb{R}^n$, the following properties hold:
\begin{itemize}
	\item [(i)] $H$ is positively $m$-homogeneous in $p$, that is, $H(x,tp) = t^m H(x,p)$ for any $t > 0$.
	\item [(ii)] $|H(x,p)| \leq \mathcal{C}|p|^m$ for some constant $\mathcal{C} > 0$.
\end{itemize}
\noindent\phantomsection\label{A4}\textbf{Assumption A4.} 
$c,h\in C(\Omega) \cap L^\infty(\Omega)$, and 
$g \in C^{1,\alpha_g}(\partial \Omega)$ for some $\alpha_g \in (0,1)$.

Before we proceed, we provide several remarks on the structural assumptions. To begin with, the Pucci extremal operators $\mathcal{P}^{\pm}_{\lambda,\Lambda}:\mathbb{S}^{n\times n}\to \mathbb{R}$ are defined by
\begin{align}
	\label{Pucci_operator}
	\mathcal{P}^{+}_{\lambda,\Lambda}(M): = \Lambda \displaystyle\sum_{e_k>0}e_k + \lambda \displaystyle\sum_{e_k<0}e_k, \mbox{ and }\mathcal{P}^{-}_{\lambda,\Lambda}(M) : = \lambda \displaystyle\sum_{e_k>0}e_k + \Lambda \displaystyle\sum_{e_k<0}e_k,
\end{align}
where $\{e_k\}_{k=1}^n$ are the eigenvalues of the matrix $M$. In terms of these extremal operators, the uniform $(\lambda,\Lambda)$-ellipticity of $F$ specified in assumption {\upshape\hyperref[A1]{(A1)}} can be equivalently formulated as
\begin{align}
	\label{Pucci_ope.}
	\mathcal{P}^{-}_{\lambda,\Lambda}(N) \leq F(M+N) - F(M) \leq \mathcal{P}^{+}_{\lambda,\Lambda}(N),
\end{align}
for all $M,N \in \mathbb{S}^{n\times n}$.

To illustrate the structural assumptions discussed above, we provide some concrete examples of $\Phi$ and the Hamiltonian $H$, respectively.
\begin{example}
	\label{Example-Phi}
	The following list provides several typical models for $\Phi(x,t)$ satisfying assumption \hyperref[A2]{\upshape(A2)}, along with admissible choices for the corresponding exponents $i(\Phi)$ and $s(\Phi)$:
	\begin{enumerate}[1.]
		\item Constant exponent power-type: $\Phi(x,\xi) = |\xi|^p$ for $p > -1$: $i(\Phi) = s(\Phi) = p$.
		\item Double phase type growth: $\Phi(x, \xi) = |\xi|^p + a(x)|\xi|^q$ for $-1 < p < q < \infty$ and $0 \leq a \in C(\Omega)$: $i(\Phi) = p$ and $s(\Phi) = q$.
		\item Variable exponent double phase type growth: $\Phi(x, \xi) = |\xi|^{p(x)} + a(x)|\xi|^{q(x)} \\\mbox{ for } p, q \in C(\Omega)$, with $-1<p^-:=\displaystyle\inf_\Omega p(x)\le q^+:=\displaystyle\sup_\Omega q(x)<\infty$, and $0 \leq a \in C(\Omega)$: $i(\Phi) = p^-$ and $s(\Phi) = q^+$.
		\item Logarithmically perturbed variable exponent growth: $\\\Phi(x, \xi) = |\xi|^{p(x)}\left(\log(e+|\xi|)\right)^{\alpha}$, where $p\in C(\Omega)$ satisfies $-1<p^-:=\displaystyle\inf_\Omega p(x) \leq p^+:=\displaystyle\sup_\Omega p(x)<\infty$, and $\alpha\in\mathbb{R}$. In this case, one can take
		\begin{align*}
			(i(\Phi),s(\Phi))=\left\{
			\begin{array}{ll}
				(p^-,p^+ +\varepsilon), &\alpha>0,\\[5pt]
				(p^-,p^+), &\alpha=0,\\[5pt]
				(p^- -\varepsilon,p^+), &\alpha<0,
			\end{array}
			\right.
		\end{align*}
		where $\varepsilon>0$ is fixed; in the case $\alpha<0$, we choose it so that $p^- -\varepsilon>-1$.
		\item Borderline variable exponent double phase type growth: $\Phi(x, \xi) = |\xi|^{p(x)} + a(x)|\xi|^{p(x)}\log(e+|\xi|)$, where $p\in C(\Omega)$ satisfies $-1<p^-:=\displaystyle\inf_\Omega p(x)\leq p^+:=\displaystyle\sup_\Omega p(x)<\infty$, $0\leq a\in C(\Omega)$. In this case, one can take
		\begin{align*}
			i(\Phi) = p^-, \mbox{ and }s(\Phi)=
			\begin{cases}
				p^+ + \varepsilon, & \mathfrak a\not\equiv 0\quad\text{ for any fixed } \varepsilon>0,\\
				p^+, & \mathfrak a\equiv 0.
			\end{cases}
		\end{align*}
		\item Multi-phase growth: $\Phi(x,\xi)=\displaystyle\sum_{k=1}^{M}a_k(x)\,|\xi|^{p_k(x)}$ for $M\in\mathbb{N}$, $a_k\in C(\Omega)$, with $a_k\geq 0$, and $p_k\in C(\overline{\Omega})$ with $-1<\inf\limits_{\Omega}p_k(x)\le \sup\limits_{\Omega}p_k(x)<\infty$. Assume further that there exist constants $0<A_0\le A_1<\infty$ such that
		\begin{align*}
			A_0 \le \sum_{k=1}^{M}a_k(x)\le A_1, \text{ for all }x\in\Omega.
		\end{align*}
		Then
		\begin{align*}
			i(\Phi)=\min_{1\le k\le M}\inf_{\Omega}p_k(x), \mbox{ and } s(\Phi)=\max_{1\le k\le M}\sup_{\Omega}p_k(x).
		\end{align*}
		\item Orlicz double phase type growth: $\Phi(x,\xi)=\psi(|\xi|)+a(x)\chi(|\xi|)$, where $0\le a\in C(\Omega)$ and $\psi,\chi:[0,\infty)\to[0,\infty)$ are continuous functions with $\psi(0)=\chi(0)=0$, $\psi(t),\chi(t)>0$ for $t>0$, and there exist exponents $-1<p_\psi\le q_\psi<\infty$, $-1<p_\chi\le q_\chi<\infty$ such that $t\mapsto \psi(t)/t^{p_\psi}$ and $t\mapsto \chi(t)/t^{p_\chi}$ are almost non-decreasing, while $t\mapsto \psi(t)/t^{q_\psi}$ and $t\mapsto \chi(t)/t^{q_\chi}$ are almost non-increasing:
		\begin{align*}
			i(\Phi)=\min\{p_\psi,p_\chi\}, \mbox{ and } s(\Phi)=\max\{q_\psi,q_\chi\}.
		\end{align*}
	\end{enumerate}
\end{example}
\begin{example}
	\label{Example-Hamiltonian}
	Some typical examples of Hamiltonians $H:\overline{\Omega} \times \mathbb{R}^n \to \mathbb{R}$ that satisfy assumption {\upshape\hyperref[A3]{(A3)}} are the following functions:
	\begin{enumerate}[1.]
		\item $H(x,p)=|p|^m$, where $0<m\le i(\Phi)+1$.
		\item $H(x,p)=|\langle b(x),p\rangle|^m$, where  $b\in C(\overline{\Omega})\cap L^\infty(\overline{\Omega})$ and $0<m\le i(\Phi)+1$.
		\item $H(x,p)=\langle b(x),p\rangle\,|p|^{m-1}$, where $b\in C(\overline{\Omega})\cap L^\infty(\overline{\Omega})$ and $0<m\le i(\Phi)+1$.
	\end{enumerate}
\end{example}
To establish our global regularity results, we assume that $\Omega \subset \mathbb{R}^n$ is a bounded $C^2$-domain, following the 
frameworks in \cite{BD2014,BBLL2024}. More precisely, up to a local coordinate system, we can suppose that $0 \in \partial\Omega$ and there exist a ball $B = B_R(0)$ and $\phi\in C^2(\mathbb{R}^{n-1})$ satisfying $\phi(0)=0,\nabla\phi(0)=0$, and
\begin{align*}
	\Omega\cap B\subset \{y\in B:y_n>\phi(y')\},\mbox{ and } \partial\Omega\cap B= \{y\in B:y_n=\phi(y')\}.
\end{align*}
We now recall the standard notion of viscosity solutions adapted from the framework of Birindelli and Demengel \cite{BD2004,BD2007} as follows: 
\begin{definition}[\bf Viscosity solution]\label{Def:Viscosity_solution}
	An upper semicontinuous function $v$ is called a viscosity subsolution of \eqref{eq:main} if for any $x_0 \in \Omega$
	\begin{itemize}
		\item either there exists $\delta >0$ such that $v$ is constant in $B_\delta(x_0)$ and $ h(x) \le c(x)|v(x_0)|^{i(\Phi)}v(x_0)$, for all $x \in B_\delta(x_0)$.
		\item or for each $\varphi \in C^2(\Omega)$ such that $v \prec_{x_0} \varphi$ and $\nabla\varphi(x_0) \neq 0$, one has
		\begin{align*}
			\Phi(x_0,|\nabla\varphi(x_0)|)F(D^2\varphi(x_0)) \hspace{-0.1cm}-\hspace{-0.1cm}H(x_0,\nabla \varphi(x_0)) \hspace{-0.1cm}+ c(x_0)|v(x_0)|^{i(\Phi)}v(x_0) \hspace{-0.1cm}\geq h(x_0).
		\end{align*}
	\end{itemize}
	Analogously, a lower semicontinuous function $w$ is called a viscosity supersolution of \eqref{eq:main} if for all $x_0 \in \Omega$
	\begin{itemize}
		\item either there exists $\delta >0$ such that $w$ is constant in $B_\delta(x_0)$ and $h(x) \geq c(x)|w(x_0)|^{i(\Phi)}w(x_0)$, for every $x \in B_\delta(x_0)$.
		\item or for each $\varphi \in C^2(\Omega)$ such that $w \succ_{x_0} \varphi$ and $\nabla\varphi(x_0) \neq 0$, one has
		\begin{align*}
			\Phi(x_0,|\nabla\varphi(x_0)|)F(D^2\varphi(x_0))\hspace{-0.1cm}-\hspace{-0.1cm}H(x_0,\nabla \varphi(x_0)) \hspace{-0.11cm}+ \hspace{-0.1cm}c(x_0)|w(x_0)|^{i(\Phi)}w(x_0) \hspace{-0.1cm}\leq h(x_0).
		\end{align*}
	\end{itemize}
	We say that $u \in C(\Omega)$ is called a viscosity solution if it is both a viscosity supersolution and a subsolution to \eqref{eq:main}.
\end{definition}
\begin{remark}
	It is noteworthy that the alternative conditions involving constant test states in Definition \ref{Def:Viscosity_solution} are crucial primarily in the singular regime $-1 < i(\Phi) < 0$, where the governing operator may suffer from degeneracy or a lack of definition when the gradient vanishes {\upshape(}$\nabla u = 0${\upshape)}. In the degenerate regime $i(\Phi) \geq 0$, this framework coincides with the standard classical definition of viscosity solutions {\upshape(}cf. \cite{CC1995,CIL1992}{\upshape)}. Furthermore, for any fixed modulation vector $\xi \in \mathbb{R}^n$, any viscosity solution $v$ of the perturbed equation
	\begin{align*}
		\Phi(x, |\xi + \nabla v|)F(D^2v) - H(x,\xi + \nabla v) = h(x) \text{ in }\Omega,
	\end{align*}
	directly induces a viscosity solution of \eqref{eq:main} with $c=0$ via the relation $u(x) = v(x) + \langle\xi,x\rangle$.
\end{remark}
	%
Following the arguments in \cite[Proposition 1.2]{BD2015} and \cite[Proposition 2.1]{BBLL2024}, we remark that the singular regime $-1 < i(\Phi) < 0$ can be transformed into a degenerate one as follows.
\begin{lemma}\label{transition}
	Assume that $-1<i(\Phi)<0$ and assumptions {\upshape\hyperref[A1]{(A1)}--\hyperref[A4]{(A4)}} hold. Let $u$ be a viscosity solution of \eqref{eq:main} in the sense of Definition \ref{Def:Viscosity_solution}. Then $u$ is a classical viscosity solution of 
	\begin{align*}
		&|\nabla u|^{-i(\Phi)}\Phi(x,|\nabla u|)F(D^2u)\hspace{-0.1cm}-\hspace{-0.1cm}|\nabla u|^{-i(\Phi)}H(x,\nabla u)\hspace{-0.1cm}+\hspace{-0.1cm}|\nabla u|^{-i(\Phi)}c(x)|u|^{i(\Phi)}u\\
		=&|\nabla u|^{-i(\Phi)}h(x) \text{ in }\Omega.
	\end{align*}
\end{lemma}
Note that the map $\Psi: \Omega \times [0,\infty) \to [0,\infty)$ given by $\Psi(x,t)\hspace{-0.11cm}:=t^{-i(\Phi)}\Phi(x,t)$, one can readily verify that $\Psi$ satisfies assumption {\upshape\hyperref[A2]{(A2)}} with $i(\Psi)=0$ and $s(\Psi)=s(\Phi)-i(\Phi)$. Moreover, after multiplying the equation by $|\nabla u|^{-i(\Phi)}$, the Hamiltonian term remains controlled; more precisely, by {\upshape\hyperref[A3]{(A3)}}, we have
\begin{align*}
	|p|^{-i(\Phi)}|H(x,p)| \leq \mathcal{C}|p|^{m-i(\Phi)},
\end{align*}
where $0 < m - i(\Phi) \leq 1$. 

We also recall the standard notions of second-order superjets and subjets introduced in \cite[Section 2]{CIL1992}.
\begin{definition}[\bf Superjet and subjet]
	Let $v: \Omega \to \mathbb{R}$ be an upper semicontinuous function and $w: \Omega \to \mathbb{R}$ be a lower semicontinuous function. For any $x_0 \in \Omega$, the \textit{second-order superjet} of $v$ at $x_0$ is defined by
	\begin{align*}
		J_{\Omega}^{2,+}v(x_0) := \bigg\{ &(p,M) \in \mathbb{R}^n \times \mathbb{S}^{n\times n} : v(x) \leq v(x_0) + \langle p, x - x_0 \rangle \\ 
		&+ \frac{1}{2}\langle M(x - x_0), x - x_0 \rangle + o(|x - x_0|^2) \text{ as } x \to x_0 \bigg\},
	\end{align*}
	and the \textit{second-order subjet} of $w$ at $x_0$ is defined by
	\begin{align*}
		J_\Omega^{2,-}w(x_0) := \bigg\{ &(p,M) \in \mathbb{R}^n \times \mathbb{S}^{n\times n} : w(x) \geq w(x_0) + \langle p, x - x_0 \rangle \\ 
		&+\frac{1}{2} \langle M(x - x_0), x - x_0 \rangle + o(|x - x_0|^2) \text{ as } x \to x_0 \bigg\}.
	\end{align*}
	The closure of the superjet (the set of limiting superjets) is given by
	\begin{align*}
		\overline{J}_{\Omega}^{2,+}v(x_0)=\{(p,X)\in\mathbb{R}^n\times \mathbb{S}^{n\times n}:\exists (p_n,X_n)\in J_{\Omega}^{2,+}v(x_n)\text{ such that }\\(x_n,v(x_n),p_n,X_n)\to (x_0,v(x_0),p,X)\}.
	\end{align*}
	Similarly, we define the closure of the subjet (the set of limiting subjets), denoted by $\overline{J}^{2,-}_\Omega w(x_0)$.
\end{definition}

We now construct two barrier functions to be utilized throughout this paper.
\begin{lemma}[\bf Construction of barrier functions]
	\label{Lem:barrier-construction}
	Suppose assumptions {\upshape\hyperref[A1]{(A1)}--\hyperref[A4]{(A4)}} are in force and $c \le  0$ in $\overline{\Omega}$. Consider the equation
	\begin{align}\label{eq:barrier-construction-main}
		\Phi(x,|\nabla u|)F(D^2u)-H(x,\nabla u)+(c(x)+\beta)|u|^{i(\Phi)}u =h(x),
	\end{align} 
	where the domain is specified in each case below. Then the following barrier functions can be constructed.
	\begin{enumerate}[\upshape(B1)]
		\item \phantomsection\label{barrier_harnack}{\upshape (Interior barrier)}. When $\beta=0$, there exist constants $p>0$ and $K_\ast>0$ such that for every $K_1 \geq K_\ast$ the function
		\begin{align}
			\label{eq:barrier_harnack}
			\chi_{\mathrm{int}}(x):=K_1\left(|x|^{-p}-2^{-p}\right)
		\end{align}
		is a viscosity subsolution to \eqref{eq:barrier-construction-main} in $B_2\setminus \overline{B_{1/4}}$.
		\item \phantomsection\label{barrier_eigenvalue}{\upshape (Exterior-sphere barrier)}. Let $r_0>0$ be the radius of the uniform exterior sphere condition for $\Omega$. Fix $z\in\partial\Omega$ and $r_1\in\left(0,r_0\wedge 1\right)$. Then there exists $y_z\in\mathbb{R}^n\setminus\overline{\Omega}$ such that $\overline{B_{r_1}(y_{z})}\cap\overline{\Omega}=\{z\}$. Moreover, there exist constants $\alpha<0$ and $K_2\geq 1$ such that the function
		\begin{align}
			\label{eq:barrier_eigenvalue}
			\chi^z_{\mathrm{ext}}(x):=K_2\left(r_1^\alpha-|x-y_{z}|^\alpha\right)
		\end{align}
		is a viscosity supersolution to \eqref{eq:barrier-construction-main} with $\beta = 0$ in $\Omega$. Furthermore, for each fixed $k>1$, choosing $K_2$ larger if necessary, there exists $\delta>0$ such that the same type of barrier satisfies the sharper estimate corresponding to $\beta=-k\|c\|_{L^\infty(\Omega)}$ in the boundary layer \begin{align*}
			\left\{x\in\Omega:r_1<|x-y_{z}|<r_1+\delta\right\}.   
		\end{align*}.
	\end{enumerate}
\end{lemma}
\begin{figure}[h]
	\centering
	\begin{subfigure}{0.39\textwidth}
		\centering
		\tikzset{every picture/.style={line width=0.75pt}}
		\resizebox{0.92\linewidth}{!}{\begin{tikzpicture}[x=0.75pt,y=0.75pt,yscale=-1,xscale=1]
			\draw  [color={rgb, 255:red, 32; green, 130; blue, 242 }  ,draw opacity=1 ][fill={rgb, 255:red, 32; green, 130; blue, 242 }  ,fill opacity=0.29 ] (242,139.5) .. controls (242,72.95) and (295.95,19) .. (362.5,19) .. controls (429.05,19) and (483,72.95) .. (483,139.5) .. controls (483,206.05) and (429.05,260) .. (362.5,260) .. controls (295.95,260) and (242,206.05) .. (242,139.5) -- cycle ;
			\draw  [color={rgb, 255:red, 32; green, 130; blue, 242 }  ,draw opacity=1 ][dash pattern={on 4.5pt off 4.5pt}] (274.37,139.5) .. controls (274.37,90.83) and (313.83,51.37) .. (362.5,51.37) .. controls (411.17,51.37) and (450.63,90.83) .. (450.63,139.5) .. controls (450.63,188.17) and (411.17,227.63) .. (362.5,227.63) .. controls (313.83,227.63) and (274.37,188.17) .. (274.37,139.5) -- cycle; 
			\draw  [color={rgb, 255:red, 32; green, 130; blue, 242 }  ,draw opacity=1 ][dash pattern={on 4.5pt off 4.5pt}] (304.23,139.5) .. controls (304.23,107.32) and (330.32,81.23) .. (362.5,81.23) .. controls (394.68,81.23) and (420.77,107.32) .. (420.77,139.5) .. controls (420.77,171.68) and (394.68,197.77) .. (362.5,197.77) .. controls (330.32,197.77) and (304.23,171.68) .. (304.23,139.5) -- cycle ;
			\draw  [fill={rgb, 255:red, 255; green, 255; blue, 255 }  ,fill opacity=1 ] (330.04,139.5) .. controls (330.04,121.57) and (344.57,107.04) .. (362.5,107.04) .. controls (380.43,107.04) and (394.96,121.57) .. (394.96,139.5) .. controls (394.96,157.43) and (380.43,171.96) .. (362.5,171.96) .. controls (344.57,171.96) and (330.04,157.43) .. (330.04,139.5) -- cycle ;
			
			\draw (445,29.6) node [anchor=north west][inner sep=0.75pt]  [font=\scriptsize] [align=left] {$\displaystyle B_{2}$};
			\draw (356.06,132.03) node [anchor=north west][inner sep=0.75pt]  [font=\scriptsize] [align=left] {$\displaystyle B_{1/4}$};
			\draw (354.77,63.78) node [anchor=north west][inner sep=0.75pt]  [font=\scriptsize] [align=left] {$\displaystyle \chi _{int}$};
		\end{tikzpicture}}
		\caption{Interior barrier}
		\label{fig:sub1}
	\end{subfigure}
	\hfill
	\begin{subfigure}{0.6\textwidth}
		\centering
		\tikzset{every picture/.style={line width=0.75pt}}
		\resizebox{0.92\linewidth}{!}{\hspace*{3cm}\begin{tikzpicture}[x=0.75pt,y=0.75pt,yscale=-1,xscale=1]
			\draw [fill={rgb, 255:red, 245; green, 166; blue, 35 }  ,fill opacity=0.07 ][line width=0.75] [line join = round][line cap = round]   (280.09,22.04) .. controls (294.54,4.7) and (297.6,10.21) .. (319.66,19.81) .. controls (323.55,21.51) and (328,24.4) .. (345.5,34.5) .. controls (355.5,44) and (361.84,74.76) .. (363.5,87.5) .. controls (364.5,97.5) and (393.5,229) .. (324.55,247.92) .. controls (311.59,254.4) and (288.16,253.54) .. (275.64,253.7) .. controls (257.79,253.92) and (244.66,244.13) .. (233.4,231.91) .. controls (229.43,227.6) and (225.07,224.85) .. (222.73,219.46) .. controls (212.58,196.03) and (216.63,161.04) .. (225.84,137.64) .. controls (232.93,119.67) and (249.57,106.83) .. (256.97,89.18) .. controls (262.25,76.59) and (263.87,75.28) .. (263.19,64.72) .. controls (262.8,58.61) and (260.15,48.47) .. (263.64,42.05) .. controls (266.88,36.08) and (275.71,27.22) .. (280.09,22.04) -- cycle ;
			\draw  [color={rgb, 255:red, 32; green, 130; blue, 242 }  ,draw opacity=1 ][fill={rgb, 255:red, 32; green, 130; blue, 242 }  ,fill opacity=0.32 ] (418.75,102.43) .. controls (445.67,102.43) and (467.5,124.25) .. (467.5,151.18) .. controls (467.5,178.1) and (445.67,199.93) .. (418.75,199.93) .. controls (391.83,199.93) and (370,178.1) .. (370,151.18) .. controls (370,124.25) and (391.83,102.43) .. (418.75,102.43) -- cycle ;
			\draw [color={rgb, 255:red, 32; green, 130; blue, 242 }  ,draw opacity=1 ] [dash pattern={on 4.5pt off 4.5pt}]  (370,151.18) -- (418.75,151.18) ;
			\draw [shift={(418.75,151.18)}, rotate = 0] [color={rgb, 255:red, 32; green, 130; blue, 242 }  ,draw opacity=1 ][fill={rgb, 255:red, 32; green, 130; blue, 242 }  ,fill opacity=1 ][line width=0.75]      (0, 0) circle [x radius= 3.35, y radius= 3.35]   ;
			\draw [shift={(370,151.18)}, rotate = 0] [color={rgb, 255:red, 32; green, 130; blue, 242 }  ,draw opacity=1 ][fill={rgb, 255:red, 32; green, 130; blue, 242 }  ,fill opacity=1 ][line width=0.75]      (0, 0) circle [x radius= 3.35, y radius= 3.35]   ;
			\draw  [draw opacity=0][dash pattern={on 3.75pt off 3pt on 7.5pt off 3pt}] (366.57,186.77) .. controls (351.62,184.74) and (340,169.58) .. (340,151.18) .. controls (340,133.16) and (351.14,118.24) .. (365.65,115.73) -- (370,151.18) -- cycle ; \draw  [color={rgb, 255:red, 32; green, 130; blue, 242 }  ,draw opacity=1 ][dash pattern={on 3.75pt off 3pt on 7.5pt off 3pt}] (366.57,186.77) .. controls (351.62,184.74) and (340,169.58) .. (340,151.18) .. controls (340,133.16) and (351.14,118.24) .. (365.65,115.73) ;  
			\draw  [draw opacity=0][dash pattern={on 4.5pt off 4.5pt}] (361.95,196.31) .. controls (340.25,192.57) and (323.75,173.93) .. (323.75,151.5) .. controls (323.75,127.84) and (342.1,108.4) .. (365.57,106.21) -- (370,151.5) -- cycle ; \draw  [color={rgb, 255:red, 32; green, 130; blue, 242 }  ,draw opacity=1 ][dash pattern={on 4.5pt off 4.5pt}] (361.95,196.31) .. controls (340.25,192.57) and (323.75,173.93) .. (323.75,151.5) .. controls (323.75,127.84) and (342.1,108.4) .. (365.57,106.21) ;  
			\draw  [draw opacity=0][dash pattern={on 3.75pt off 3pt on 7.5pt off 3pt}] (361.86,205.36) .. controls (331.22,201.32) and (307.75,178.52) .. (307.75,151) .. controls (307.75,123.69) and (330.86,101.04) .. (361.16,96.74) -- (371.63,151) -- cycle ; \draw  [color={rgb, 255:red, 32; green, 130; blue, 242 }  ,draw opacity=1 ][dash pattern={on 3.75pt off 3pt on 7.5pt off 3pt}] (361.86,205.36) .. controls (331.22,201.32) and (307.75,178.52) .. (307.75,151) .. controls (307.75,123.69) and (330.86,101.04) .. (361.16,96.74) ;  
			\draw  [color={rgb, 255:red, 32; green, 130; blue, 242 }  ,draw opacity=1 ] (419.5,146) .. controls (419.5,141.33) and (417.17,139) .. (412.5,139) -- (405.5,139) .. controls (398.83,139) and (395.5,136.67) .. (395.5,132) .. controls (395.5,136.67) and (392.17,139) .. (385.5,139)(388.5,139) -- (378,139) .. controls (373.33,139) and (371,141.33) .. (371,146) ;
			
			\draw (293,52) node [anchor=north west][inner sep=0.75pt]  [font=\large] [align=left] {$\displaystyle \Omega $};
			\draw (405,153) node [anchor=north west][inner sep=0.75pt]  [font=\scriptsize] [align=left] {$\displaystyle y_{z}$};
			\draw (436,93) node [anchor=north west][inner sep=0.75pt]  [font=\scriptsize] [align=left] {$\displaystyle B_{r_{1}}( y_{z})$};
			\draw (351,140) node [anchor=north west][inner sep=0.75pt]  [font=\scriptsize] [align=left] {$\displaystyle z$};
			\draw (292.23,167.18) node [anchor=north west][inner sep=0.75pt]  [font=\scriptsize] [align=left] {$\displaystyle \chi _{ext}^{z}$};
			\draw (391,116) node [anchor=north west][inner sep=0.75pt]  [font=\scriptsize] [align=left] {$\displaystyle r_{1}$};		
		\end{tikzpicture}}
		\caption{Exterior-sphere barrier}
		\label{fig:sub2}
	\end{subfigure}
	\caption{Geometric illustration of the interior and exterior-sphere barrier constructions}
	\label{fig:main}
\end{figure}
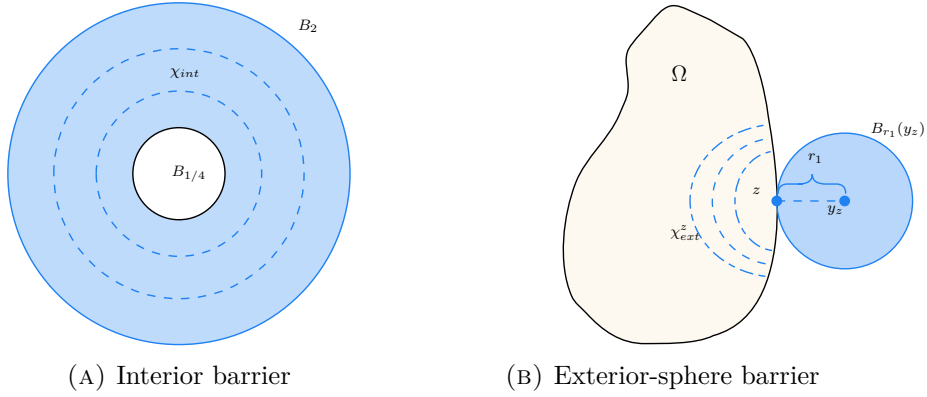
\begin{proof}
	\textbf{}
	
	\textit{Step 1: Proof of \upshape\hyperref[barrier_harnack]{(B1)}.} Suppose $\beta =0$. Let us consider the barrier function \eqref{eq:barrier_harnack}. For $x \in B_2\setminus \overline{B_{1/4}}$, a direct calculation shows that
	\begin{align*}
		&\nabla \chi_{\mathrm{int}}(x) = -K_1p|x|^{-p-2}x,\\
		&D^2\chi_{\mathrm{int}}(x) = -K_1p|x|^{-p-2}I + K_1p(p+2)|x|^{-p-4}(x \otimes x).
	\end{align*}
	The eigenvalues of $D^2\chi_{\mathrm{int}}(x)$ are $K_1p(p+1)|x|^{-p-2}$ of multiplicity $1$ and $-K_1p|x|^{-p-2}$ of multiplicity $n-1$. Then 
	\begin{align*}
		&\mathrm{tr}((D^2(\chi_{\mathrm{int}}(x)))^+) = K_1p(p+1)|x|^{-p-2},\\
		&\mathrm{tr}((D^2(\chi_{\mathrm{int}}(x)))^-) = K_1(n-1)p|x|^{-p-2}.
	\end{align*}
	By the structural assumptions on $\Phi$ and $F$, we have
	\begin{align*}
		&\Phi(x,|\nabla \chi_{\mathrm{int}}(x)|)F(D^2\chi_{\mathrm{int}}(x))\\
		&\geq \frac{\nu_0}{L}
		\min\left\{
		\left(K_1p|x|^{-p-1}\right)^{i(\Phi)},
		\left(K_1p|x|^{-p-1}\right)^{s(\Phi)}
		\right\}\\
		&\qquad\times K_1p|x|^{-p-1}|x|^{-1}
		\left(\lambda(p+1)-\Lambda(n-1)\right).
	\end{align*}
	Here, if we let $A = K_1p|x|^{-p-1}$, then
	\begin{align*}
		A\min\left\{A^{i(\Phi)},A^{s(\Phi)}\right\} &= A\left[A^{-1}\min\left\{A^{i(\Phi)+1},A^{s(\Phi)+1}\right\}\right]\\
		&\geq\min\left\{\left(\frac{K_1p}{|x|^{p+1}}\right)^{i(\Phi)+1},\left(\frac{K_1p}{|x|^{p+1}}\right)^{s(\Phi)+1}\right\}.
	\end{align*}
	We need to select $K_1$ such that $K_1p\geq 2^{p+1}$. Since $0 < i(\Phi) +1 \leq s(\Phi)+1 $, we conclude that
	\begin{align*}
		A\min\left\{A^{i(\Phi)},A^{s(\Phi)}\right\} \geq \left(\frac{p}{|x|^{p+1}}\right)^{i(\Phi)+1}K_1^{i(\Phi)+1}.
	\end{align*}  
	Hence, we deduce that
	\begin{align*}
		&\Phi(x,|\nabla \chi_{\mathrm{int}}(x)|)F(D^2\chi_{\mathrm{int}}(x)) \\
		&\geq 
		\frac{\nu_0}{L}K_1^{i(\Phi)+1}p^{i(\Phi)+1}|x|^{-(p+1)(i(\Phi)+1) - 1}\left[\lambda(p+1) - \Lambda (n-1)\right],
	\end{align*}
	for all $x$ such that $1/4\leq |x|\leq 2$. Choose $p>0$ large enough such that $\lambda(p+1) - \Lambda (n-1) \geq \lambda p/2$. Then there exists $C_1 = C_1(\lambda,\nu_0,L) > 0$ so that
	\begin{align*}
		\Phi(x,|\nabla \chi_{\mathrm{int}}(x)|)F(D^2\chi_{\mathrm{int}}(x)) \geq C_1K_1^{i(\Phi)+1}p^{i(\Phi)+2}|x|^{-(p+1)\left(i(\Phi) + 1\right) -1}.
	\end{align*}
	Moreover, based on assumption {\upshape\hyperref[A3]{(A3)}} and using $c^-= \max\{-c,0\}$, we obtain 
	\begin{align*}
		H(x,\nabla \chi_{\mathrm{int}}(x)) - c(x)(\chi_{\mathrm{int}}(x))^{i(\Phi)+1}&\leq \mathcal{C}\left(K_1p|x|^{-p-1}\right)^{m} \\&\hspace{-0.4cm}+K_1^{i(\Phi)+1}\|c^-\|_{L^\infty(B_{2})}\hspace{-0.1cm}\left(|x|^{-p}-2^{-p}\right)^{i(\Phi)+1}.
	\end{align*}
	Since $K_1p\geq 2^{p+1}$, we have $K_1p|x|^{-p-1}\geq1$ in the annulus. Hence, as $m \leq i(\Phi)+1$, we deduce
	\begin{align*}
		&H(x,\nabla \chi_{\mathrm{int}}(x)) - c(x)(\chi_{\mathrm{int}}(x))^{i(\Phi)+1}\\
		&\leq K_1^{i(\Phi)+1}|x|^{-(p+1)(i(\Phi)+1)}\left[\mathcal{C}p^{i(\Phi)+1} + \|c^-\|_{L^\infty(B_{2})}|x|^{i(\Phi)+1}\right].
	\end{align*}
	Consequently, on the annulus $\{1/4\le|x|\le 2\}$, we derive the following estimate 
	\begin{align*}
		&\Phi(x,|\nabla \chi_{\mathrm{int}}(x)|)F(D^2\chi_{\mathrm{int}}(x)) -H(x,\nabla \chi_{\mathrm{int}}(x))
		+c(x)(\chi_{\mathrm{int}}(x))^{i(\Phi)+1}\\ 
		&\geq K_1^{i(\Phi)+1}|x|^{-(p+1)(i(\Phi)+1)}\left[C_1p^{i(\Phi)+2} - \mathcal{C}p^{i(\Phi)+1} - 2^{i(\Phi)+1}\|c^-\|_{L^\infty(B_{2})}\right].
	\end{align*}
	Therefore, we choose $p>0$ sufficiently large so that $C_1p^{i(\Phi)+2} - \mathcal{C}p^{i(\Phi)+1} - 2^{i(\Phi)+1}\|c^-\|_{L^\infty(B_{2})} >0$. We then choose $K_\ast>0$ sufficiently large, depending also on $\|h\|_{L^\infty(B_2)}$ so that, for every $K_1\geq K_\ast$, we obtain the following inequality in $B_2\setminus \overline{B_{1/4}}$
	\begin{align*}
		&\Phi(x,|\nabla \chi_{\mathrm{int}}(x)|)F(D^2\chi_{\mathrm{int}}(x)) -H(x,\nabla \chi_{\mathrm{int}}(x))
		+c(x)(\chi_{\mathrm{int}}(x))^{i(\Phi)+1}\\
		&\geq \|h\|_{L^\infty(B_{2})}+1,
	\end{align*}
	i.e., $\chi_{\mathrm{int}}$ is a viscosity subsolution to \eqref{eq:barrier-construction-main} in $B_2\setminus \overline{B_{1/4}}$.
	
	\textit{Step 2: Proof of \upshape\hyperref[barrier_eigenvalue]{(B2)}.} As the proof closely follows that of \hyperref[barrier_harnack]{\upshape(B1)}, we shall focus primarily on the necessary modifications.
	\begin{enumerate}[1.]
		\item Set $R:=r_0+\mathrm{diam}(\Omega)$. Let $\rho:=x-y_z$. For $\alpha<0$ and $K_2>0$ to be chosen later, consider the barrier function \eqref{eq:barrier_eigenvalue}. Since $\overline{B_{r_1}(y_z)}\cap\overline{\Omega}=\{z\}$, we have
		$\chi^z_{\mathrm{ext}}(z)=0$ and $\chi^z_{\mathrm{ext}}>0$ in
		$\Omega\cap\{|\rho|>r_1\}$. A direct calculation gives
		\begin{align*}
			&\nabla\chi^z_{\mathrm{ext}}(x)\hspace{-0.1cm}=\hspace{-0.1cm}-K_2\alpha|\rho|^{\alpha-2}\rho,\\
			&D^2\chi^z_{\mathrm{ext}}(x)=-K_2\alpha|\rho|^{\alpha-2}\left[(\alpha-1)\dfrac{\rho\otimes \rho}{|\rho|^2}+\left(I-\dfrac{\rho\otimes \rho}{|\rho|^2}\right)\right].
		\end{align*}
		It follows that $\mathcal{P}^+_{\lambda,\Lambda}(D^2\chi^z_{\mathrm{ext}}(x))= -K_2\alpha|\rho|^{\alpha-2}[\Lambda(n-1)+\lambda(\alpha-1)]$. Choose $\alpha<0$ such that $\Lambda(n-1)+\lambda(\alpha-1)<0$. Then $F(D^2\chi^z_{\mathrm{ext}}) \leq \mathcal{P}_{\lambda,\Lambda}^+(D^2\chi^z_{\mathrm{ext}}) < 0$.
		\item As in \textit{Step 1}, using the structural assumptions on $\Phi$. If we let $A=\dfrac{-K_2\alpha}{|\rho|^{1-\alpha}}$, we have
		\begin{align*}
			\min\{A^{i(\Phi)},A^{s(\Phi)}\}&=A^{-1}\min\{A^{i(\Phi)+1},A^{s(\Phi)+1}\}\\
			&\geq \left(\dfrac{-K_2\alpha}{|\rho|^{1-\alpha}}\right)^{-1}\min\left\{\left(\dfrac{-K_2\alpha}{|\rho|^{1-\alpha}}\right)^{i(\Phi)+1},\left(\dfrac{-K_2\alpha}{|\rho|^{1-\alpha}}\right)^{s(\Phi)+1}\right\}.
		\end{align*} 
		Note that we need to select $K_2$ so that $K_2(-\alpha)\geq R^{1-\alpha}$. Since $0 < i(\Phi) + 1 \leq s(\Phi) + 1$, we obtain
		\begin{align}
			\label{l4}
			\begin{split}
				&\hspace{-0.2cm}\Phi(x,|\nabla \chi^z_{\mathrm{ext}}(x)|)F(D^2\chi^z_{\mathrm{ext}}(x))-H(x,\nabla \chi^z_{\mathrm{ext}}(x))+(c(x)+\beta)(\chi^z_{\mathrm{ext}}(x))^{i(\Phi)+1}\\
				\le &~K_2^{i(\Phi)+1}(-\alpha)^{i(\Phi)+1}|\rho|^{(\alpha-1)(i(\Phi)+1)}\left[\dfrac{\nu_0}{L|\rho|}[\Lambda(n-1)+\lambda(\alpha-1)]+\mathcal{C}\right]\\ &\quad+(c(x)+\beta)K_2^{i(\Phi)+1}(r_1^{\alpha}-|\rho|^{\alpha})^{i(\Phi)+1}.
			\end{split}
		\end{align}
		\textit{Substep 2.1 (Global estimate).} Let $\beta = 0$. Since $r_1<|\rho|\le R$ for every $x\in\Omega$, the preceding estimate is uniform in $\Omega$. Then, choosing first $\alpha<0$ sufficiently negative and subsequently $K_2 \geq 1$ sufficiently large
		\begin{align*}
			\Phi(x,|\nabla\chi^z_{\mathrm{ext}}|)F(D^2\chi^z_{\mathrm{ext}})
			-H(x,\nabla\chi^z_{\mathrm{ext}}) +c(x)(\chi^z_{\mathrm{ext}})^{i(\Phi)+1} \le -\|h\|_{L^\infty(\Omega)}.
		\end{align*}
		Hence $\chi^z_{\mathrm{ext}}$ is a viscosity supersolution to \eqref{eq:barrier-construction-main} in $\Omega$.\\
		\textit{Substep 2.2 (Boundary-layer estimate).} Let $\beta = -k\|c\|_{L^{\infty}(\Omega)}$. We now restrict the above estimate to the boundary layer
		\begin{align*}
			\mathcal{A}_{z,\delta}:=\left\{x\in\Omega:r_1<|x-y_z|<r_1+\delta\right\}.
		\end{align*}
		Fix $r_1\in\left(0,r_0\wedge 1\right)$. Combining the fact that $(\alpha-1)i(\Phi)+\alpha-2=(\alpha-1)(i(\Phi)+1)-1<0$ and the estimate \eqref{l4}, we can choose $K_2$ sufficiently large, with $K_2(-\alpha)\geq R^{1-\alpha}$, and $\delta>0$ sufficiently small such that for all $x\in \mathcal{A}_{z,\delta}$ we have
		\begin{align*}
				&\Phi(x,|\nabla \chi^z_{\mathrm{ext}}(x)|)F(D^2\chi^z_{\mathrm{ext}}(x))
				-H(x,\nabla \chi^z_{\mathrm{ext}}(x))\\
				&\quad
				+\left(c(x)-k\|c\|_{L^{\infty}(\Omega)}\right)
				\chi^z_{\mathrm{ext}}(x)^{i(\Phi)+1}
				\leq -\|h\|_{L^\infty(\Omega)}
				-k^{i(\Phi)+2}\|c\|_{L^{\infty}(\Omega)}.
			\end{align*}
		Thus, $\chi^z_{\mathrm{ext}}$ is a viscosity supersolution to
		\eqref{eq:barrier-construction-main} in
		$\mathcal{A}_{z,\delta}$.
	\end{enumerate}
	This completes the proof.
\end{proof}
\section{\bf Global regularity and Comparison principle for the Dirichlet problem}
\label{section:3}
In this section, we pursue one of the primary objectives of this paper: global regularity. Our approach here is specifically adapted from the techniques developed in \cite{BBLL2024} for fully nonlinear elliptic equations. Furthermore, we also establish the foundational pillars of the underlying viscosity solution framework such as local Lipschitz estimates, the comparison principle, and the existence and uniqueness of viscosity solutions.
\subsection{\bf ABP estimate and boundary Lipschitz regularity}
As a prior step to obtaining the global $C^{0,1}$ and $C^{1,\gamma}$ regularity results, we first focus on the Alexandroff--Bakelman--Pucci (ABP) estimate (Theorem \ref{theo:ABP_estimate}). Essentially, this estimate provides a controlled upper bound for the supremum of a solution $u$ over the domain $\Omega$ in terms of its boundary data on $\partial\Omega$ and the $L^n$-norm of the inhomogeneous term $h$. In particular, we will establish an appropriate version of the ABP estimate for a viscosity subsolution of \eqref{1.2}. For comprehensive discussions on related ABP estimates in alternative structural settings, we strongly recommend the reader to \cite{BBLL2024,DFQ2009,JSRR2023,JSRR2023_1,I2011} and the references therein.
To prove the ABP estimate, we need to define the notion of upper contact set of a function $u$.
\begin{definition}
	For $u: \Omega \to \mathbb{R}$ and $R > 0$, the upper contact set is defined by
	\begin{align*}
		&\Gamma^+(u,\Omega)\hspace{-0.1cm}=\hspace{-0.1cm}\{x\in\Omega:\hspace{-0.1cm}\exists p\in\mathbb{R}^n~\text{such that}~u(y)\hspace{-0.1cm}\le \hspace{-0.1cm}u(x)\hspace{-0.1cm}+\hspace{-0.1cm}\langle p,y-x\rangle\text{ for all }y\in\Omega\},\\
		&\Gamma_R^+(u,\Omega)\hspace{-0.1cm}=\hspace{-0.1cm}\{x\in\Omega:\hspace{-0.1cm}\exists p\in\overline{B_R}\text{ such that }u(y)\hspace{-0.1cm}\le u(x)\hspace{-0.1cm}+\hspace{-0.1cm}\langle p,y-x\rangle\text{ for all }y\in\Omega\}.
	\end{align*}
\end{definition}
\begin{figure}[H]
	\centering
	\tikzset{every picture/.style={line width=0.75pt}}
	\begin{tikzpicture}[x=0.75pt,y=0.75pt,yscale=-1,xscale=1]
		\draw    (251.26,147.26) -- (427.45,147.26) ;
		\draw [shift={(427.45,147.26)}, rotate = 180] [color={rgb, 255:red, 0; green, 0; blue, 0 }  ][line width=0.75]    (0,5.59) -- (0,-5.59)   ;
		\draw [shift={(339.35,147.26)}, rotate = 180] [color={rgb, 255:red, 0; green, 0; blue, 0 }  ][line width=0.75]    (0,5.59) -- (0,-5.59)   ;
		\draw [shift={(251.26,147.26)}, rotate = 180] [color={rgb, 255:red, 0; green, 0; blue, 0 }  ][line width=0.75]    (0,5.59) -- (0,-5.59)   ;
		\draw [color={rgb, 255:red, 208; green, 2; blue, 27 }  ,draw opacity=1 ] [dash pattern={on 4.5pt off 4.5pt}]  (252.05,84.28) -- (251.26,147.26) ;
		\draw [color={rgb, 255:red, 208; green, 2; blue, 27 }  ,draw opacity=1 ] [dash pattern={on 4.5pt off 4.5pt}]  (427.45,61.16) -- (427.45,147.26) ;
		\draw [line width=0.75] [line join = round][line cap = round]   (251.26,139.29) .. controls (261.62,132.11) and (268.79,110.5) .. (272.32,95.97) .. controls (275.52,82.79) and (274.01,87) .. (277.57,73.91) .. controls (278.4,70.83) and (280.75,51.59) .. (285.42,47.73) .. controls (294.31,51.59) and (290.32,101.82) .. (296.7,104.21) .. controls (302.28,106.6) and (303.87,91.45) .. (307.06,93.05) .. controls (311.05,94.64) and (310.25,108.99) .. (316.63,100.22) .. controls (322.21,95.44) and (322.55,107.47) .. (325.79,102.59) .. controls (331.05,94.69) and (335.77,79.49) .. (341.05,78.55) .. controls (346.13,81.88) and (341.13,126.35) .. (350.12,133.31) .. controls (354.9,135.7) and (360.34,113.57) .. (364.32,107.99) .. controls (365.07,106.48) and (367.01,104.94) .. (368.64,105.37) .. controls (374.29,106.86) and (377.22,112.18) .. (380.41,120.15) .. controls (385.99,132.11) and (385.2,112.18) .. (389.98,108.19) .. controls (393.96,106.6) and (399.55,113.77) .. (405.13,116.17) .. controls (412.3,120.15) and (411.21,87.54) .. (416.29,85.87) .. controls (420.27,85.07) and (418.68,102.61) .. (427.45,101.82) ;
		\draw [color={rgb, 255:red, 208; green, 2; blue, 27 }  ,draw opacity=1 ][line width=1.5]    (252.05,84.28) -- (285.42,47.73) ;
		\draw [shift={(285.42,47.73)}, rotate = 312.39] [color={rgb, 255:red, 208; green, 2; blue, 27 }  ,draw opacity=1 ][fill={rgb, 255:red, 208; green, 2; blue, 27 }  ,fill opacity=1 ][line width=1.5]      (0, 0) circle [x radius= 1.74, y radius= 1.74];
		\draw [color={rgb, 255:red, 208; green, 2; blue, 27 }  ,draw opacity=1 ][line width=1.5]    (285.42,47.73) -- (341.05,78.55) ;
		\draw [shift={(341.05,78.55)}, rotate = 28.99] [color={rgb, 255:red, 208; green, 2; blue, 27 }  ,draw opacity=1 ][fill={rgb, 255:red, 208; green, 2; blue, 27 }  ,fill opacity=1 ][line width=1.5]      (0, 0) circle [x radius= 1.74, y radius= 1.74];
		\draw [color={rgb, 255:red, 208; green, 2; blue, 27 }  ,draw opacity=1 ][line width=1.5] (341.05,78.55) -- (368.64,105.37) ;
		\draw [shift={(368.64,105.37)}, rotate = 44.19] [color={rgb, 255:red, 208; green, 2; blue, 27 }  ,draw opacity=1 ][fill={rgb, 255:red, 208; green, 2; blue, 27 }  ,fill opacity=1 ][line width=1.5] (0, 0) circle [x radius= 1.74, y radius= 1.74];
		\draw [color={rgb, 255:red, 208; green, 2; blue, 27 }  ,draw opacity=1 ][line width=1.5]    (368.64,105.37) -- (427.45,61.16) ;
		\draw    (427.45,147.26) -- (440.6,147.61) ;
		\draw [shift={(442.6,147.66)}, rotate = 181.51] [color={rgb, 255:red, 0; green, 0; blue, 0 }  ][line width=0.75]    (10.93,-3.29) .. controls (6.95,-1.4) and (3.31,-0.3) .. (0,0) .. controls (3.31,0.3) and (6.95,1.4) .. (10.93,3.29)   ;
		\draw    (251.26,147.26) -- (238.5,147.26) ;
		
		\draw (241.36,155.91) node [anchor=north west][inner sep=0.75pt]   [align=left] {\mbox{-}1};
		\draw (422.56,155.91) node [anchor=north west][inner sep=0.75pt]   [align=left] {1};
		\draw (288.55,168.16) node [anchor=north west][inner sep=0.75pt]  [font=\small] [align=left] {$\displaystyle \Omega \ =\ ( -1,1)$};
		\draw (268.88,106.47) node [anchor=north west][inner sep=0.75pt]   [align=left] {u};
		\draw (281.07,33.98) node [anchor=north west][inner sep=0.75pt]  [font=\small,color={rgb, 255:red, 208; green, 2; blue, 27 }  ,opacity=1 ] [align=left] {$\displaystyle x_{1}$};
		\draw (339.27,65.48) node [anchor=north west][inner sep=0.75pt]  [font=\small,color={rgb, 255:red, 208; green, 2; blue, 27 }  ,opacity=1 ] [align=left] {$\displaystyle x_{2}$};
		\draw (362.39,84.19) node [anchor=north west][inner sep=0.75pt]  [font=\small,color={rgb, 255:red, 208; green, 2; blue, 27 }  ,opacity=1 ] [align=left] {$\displaystyle x_{3}$};
		\draw (334.45,155.91) node [anchor=north west][inner sep=0.75pt]   [align=left] {0};
		\draw (446.83,143.4) node [anchor=north west][inner sep=0.75pt]  [font=\small,color={rgb, 255:red, 0; green, 0; blue, 0 }  ,opacity=1 ] [align=left] {$\displaystyle x$};	
	\end{tikzpicture}
	\caption{Upper contact set}
\end{figure}
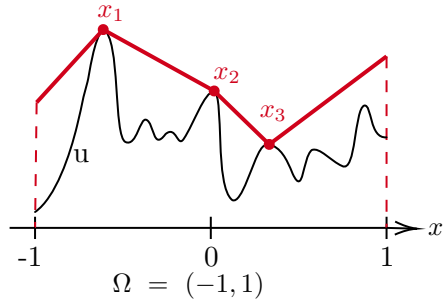 

We are now ready to prove the ABP estimate.
\begin{theorem}[\bf Alexandroff--Bakelman--Pucci estimate]
	\label{theo:ABP_estimate}
	Let $\Omega$ be a \\bounded domain. Suppose that $u\in C(\overline{\Omega})$ is a viscosity subsolution (resp. supersolution) of the following equation 
	\begin{align*}
		\left\{\begin{array}{cll}
			\Phi(x,|\nabla u|)F(D^2u)- H(x,\nabla u) + c(x)|u|^{i(\Phi)}u &= h(x) &\text{in } \Omega,\\
			u(x) &=g(x) &\text{on }\partial\Omega,
		\end{array}\right. 
	\end{align*}
	in $\{x\in \Omega:u(x)>0\}$ (resp. $\{x\in \Omega:u(x)<0\}$) under the assumptions {\upshape\hyperref[A1]{(A1)}--\hyperref[A4]{(A4)}}. Suppose that $h \in L^n(\Omega) \cap C(\Omega)$ and $c \leq 0$. Then there exists a constant $d\equiv d\left(m,n,\lambda,i(\Phi),s(\Phi),L,\nu_0,\mathcal{C},|\Omega|\right) >0$ such that
	\begin{align}
		\label{ABP_main_1}
		\sup_{\Omega}u^+ \leq \sup_{\partial\Omega}g^+ + d\, \diam(\Omega)\max\left\{1,\|h^-\|_{L^n(\Gamma^+(u^+))}^{\frac{1}{i(\Phi)+1}},\|h^-\|_{L^n(\Gamma^+(u^+))}^{\frac{1}{s(\Phi)+1}}\right\},
	\end{align}
	$\left(\text{resp. }\right.$
	\begin{align}
		\label{ABP_main_2}
		\left.\sup_{\Omega}u^- \leq \sup_{\partial\Omega}g^- + d\, \diam(\Omega)\max\left\{1,\|h^+\|_{L^n(\Gamma^+(u^-))}^{\frac{1}{i(\Phi)+1}},\|h^+\|_{L^n(\Gamma^+(u^-))}^{\frac{1}{s(\Phi)+1}}\right\}.\right)
	\end{align}
	Specifically, we have
	\begin{align}
		\label{ABP_main}
		\|u\|_{L^{\infty}(\Omega)}\le \|g\|_{L^{\infty}(\partial\Omega)}+d\,\mathrm{diam}(\Omega)\max\left\{1,\|h\|_{L^n(\Omega)}^{\frac{1}{i(\Phi)+1}},\|h\|_{L^n(\Omega)}^{\frac{1}{s(\Phi)+1}}\right\}. 
	\end{align}
\end{theorem}
\begin{proof}
	In order to organize the proof, we highlight the main steps and specific cases in bold.\\
	The proof consists of two steps. In the first step, we prove the ABP estimate for the viscosity subsolution $u\in C^2(\Omega)\cap C(\overline{\Omega})$. In the second step, we consider $u\in C(\overline{\Omega})$ via approximation based on the sup convolutions.\\
	\textbf{Step 1.} Suppose that the viscosity subsolution $u\in C^2(\Omega)\cap C(\overline{\Omega})$. Let us define 
	\begin{align}
		\label{ABP: eq_0}
		R_0\equiv R_0(u):=\dfrac{1}{\mathrm{diam}(\Omega)}\left(\sup_{x\in \Omega}u^+-\sup_{x\in\partial\Omega}u^+\right).
	\end{align}
	The purpose is to obtain a certain estimate on $R_0$ in terms of $\|h^{-}\|_{L^n(\Gamma^+(u^+))}$ and {data}, from which the estimate \eqref{ABP_main_1} follows. Applying \cite[Theorem 3.1]{M2010}, for all $R<R_0$, we find
	\begin{align}
		\label{ABP: eq_1}
		\int_{B_R}{g(z)dz}\le \int_{\Gamma_R^+(u^+)}{g(\nabla u)|\mathrm{det}(D^2u)|dx},\text{ for all }g\in C(\mathbb{R}^n),~g\geq 0,
	\end{align}
	and
	\begin{align}
		\label{ABP: eq_2}
		D^2u\le 0\text{ on } \Gamma_R^+(u^+)\subset \{x\in\Omega:u(x)>0\}.
	\end{align}
	Let us now discuss the behavior of $\nabla u$ in the set $\Gamma^+_R(u^+)$. Let $x_0\in \Gamma^+_R(u^+)$ be any point. If $\nabla u(x_0)\neq 0$, then we are able to take $u$ as a test function in the definition of viscosity subsolution. In turn, we have
	\begin{align*}
		\Phi(x_0,|\nabla u(x_0)|)F(D^2u(x_0))-H(x_0,\nabla u(x_0))\geq -c(x_0)u^{i(\Phi)+1}(x_0) + h(x_0).
	\end{align*}
	Since $c(x) \leq 0$, we obtain
	\begin{align*}
		\Phi(x_0,|\nabla u(x_0)|)F(D^2u(x_0))-H(x_0,\nabla u(x_0))\geq h(x_0).
	\end{align*}
	By assumptions {\upshape\hyperref[A2]{(A2)}} and {\upshape\hyperref[A3]{(A3)}}, we deduce
	\begin{align*}
		-h^-(x_0)
		\leq \Phi(x_0,|\nabla u(x_0)|)\mathcal{P}^+_{\lambda,\Lambda}(D^2u(x_0))+\mathcal{C}|\nabla u(x_0)|^{m}.
	\end{align*}
	Recalling that $D^2u(x_0)\le 0$ by \eqref{ABP: eq_2}, we find $\mathcal{P}^+_{\lambda,\Lambda}(D^2u(x_0))=\lambda\mathrm{tr}(D^2u(x_0))$ and
	\begin{align}
		\label{ABP: eq_3}
		\left(\dfrac{-\mathrm{tr}(D^2u(x_0))}{n}\right)^n\leq \left(\dfrac{h^-(x_0)+\mathcal{C}|\nabla u(x_0)|^{m}}{n\lambda\Phi(x_0,|\nabla u(x_0)|)}\right)^n.
	\end{align}
	If $\nabla u(x_0)=0$, we distinguish two cases. If  $\det D^2u(x_0)=0$, then this point gives no contribution to the Jacobian term in the area formula. If $\det D^2u(x_0) \neq 0$ and since $D^2u(x_0)\leq 0$ on $\Gamma_R^+(u^+)$, we have  $D^2u(x_0)<0$. Hence $x_0$ is a non-degenerate critical point of $u$. The set of non-degenerate critical points is countable. Consequently, up to a set which is negligible for the determinant term, we may restrict the argument to the set where $\nabla u\neq0$. Let us recall also the following classical inequality
	%
	\begin{align*}
		\det(A)\det(B)\leq \left(\dfrac{\mathrm{tr}(AB)}{n}\right)^n\text{ for all }A,B\in\mathbb S^{n\times n}\text{ with }A,B\geq 0.
	\end{align*}
	As a result, we derive the following estimate
	\begin{align}
		\label{ABP: eq_4}
		|\det D^2u(x)|\leq \left(\dfrac{h^-(x)+\mathcal{C}|\nabla u(x)|^{m}}{n\lambda\Phi(x,|\nabla u(x)|)}\right)^n,
	\end{align}
	for all $x\in \Gamma_R^+(u^+)\setminus \mathcal{U}$, where $\mathcal{U}=\{x\in \Gamma_R^+(u^+):\nabla u(x)=0\}$. 
	We now consider two cases depending on the sign of $i(\Phi)$.\\
	\textbf{Case 1.} $i(\Phi)\geq 0$. In this case, we choose \begin{align*}
		g(z)=\left[k^{\frac{n}{n-1}}+|z|^{\frac{\left(i(\Phi)+1\right) n}{n-1}}\right]^{1-n}\min\{|z|^{i(\Phi)n},|z|^{s(\Phi)n}\}, \text{with}~ k\geq 1. 
	\end{align*} 
	In turn, recalling \eqref{ABP: eq_1}, we obtain
	\begin{align}
		\label{ABP: eq_5}
		\begin{split}
			I_1&:=\int_{B_R}{\dfrac{\min\{|z|^{i(\Phi)n},|z|^{s(\Phi)n}\}}{\left[k^{\frac{n}{n-1}}+|z|^{\frac{\left(i(\Phi)+1\right)n}{n-1}}\right]^{n-1}}dz}\\
			&\leq \int_{\Gamma_R^+(u^+)\setminus\mathcal{U}}{\dfrac{\min\{|\nabla u|^{i(\Phi)n},|\nabla u|^{s(\Phi)n}\}}{\left[k^{\frac{n}{n-1}}+|\nabla u|^{\frac{\left(i(\Phi)+1\right) n}{n-1}}\right]^{n-1}}|\det(D^2u)|dx}.
		\end{split}
	\end{align}
	Combining \eqref{ABP: eq_4} and \eqref{ABP: eq_5}, we deduce
	\begin{align}
		\label{ABP: eq_6}
		\begin{split}
			I_1&\leq \int_{\Gamma_R^+(u^+)\setminus\mathcal{U}}{\dfrac{\min\{|\nabla u|^{i(\Phi)n},|\nabla u|^{s(\Phi)n}\}}{\left[k^{\frac{n}{n-1}}+|\nabla u|^{\frac{\left(i(\Phi)+1\right)n}{n-1}}\right]^{n-1}}\left[\dfrac{h^-+\mathcal{C}|\nabla u|^{m}}{n\lambda\Phi(x,|\nabla u|)}\right]^ndx}\\
			&\leq \dfrac{2^{n-1}}{n^n\lambda^n}\int_{\Gamma_R^+(u^+)\setminus\mathcal{U}}{\dfrac{\min\{|\nabla u|^{i(\Phi)n},|\nabla u|^{s(\Phi)n}\}}{(\Phi(x,|\nabla u|))^n}\hspace{-0.1cm}\cdot\hspace{-0.1cm}\dfrac{(h^-)^n+\mathcal{C}^n|\nabla u|^{mn}}{\left[k^{\frac{n}{n-1}}+|\nabla u|^{\frac{\left(i(\Phi)+1\right)n}{n-1}}\right]^{n-1}}dx}\\
			&\leq \dfrac{2^{n-1}L^nD}{n^n\lambda^n\nu_0^n}\int_{\Gamma^+(u^+)}{\left[ (h^-)^nk^{-n}+1\right]dx},
		\end{split}
	\end{align}
	where we use {\upshape\hyperref[A2]{(A2)}}, $D \equiv D(m,n,i(\Phi),\mathcal{C}) > 0$ and $0<m\leq i(\Phi) + 1$. On the other hand, by co-area formula, we have
	\begin{align}
		\label{ABP: eq_7}
		\begin{split}
			I_1 &= \displaystyle\int_{0}^{R} \dfrac{\min\{t^{i(\Phi)n},t^{s(\Phi)n}\}}{\left[k^{\frac{n}{n-1}}+t^{\frac{\left(i(\Phi)+1\right)n}{n-1}}\right]^{n-1}}\int_{\partial B_t}dSdt\\
			&= n\omega_n\displaystyle\int_{0}^{R} \dfrac{\min\{t^{i(\Phi)n},t^{s(\Phi)n}\}}{\left[k^{\frac{n}{n-1}}+t^{\frac{\left(i(\Phi)+1\right)n}{n-1}}\right]^{n-1}}t^{n-1}dt.
		\end{split}
	\end{align}
	If $R< 1$, it follows from equation \eqref{ABP: eq_7} that
	\begin{align}
		\label{ABP: eq_8}
		\begin{split}
			I_1 &= n\omega_n\displaystyle\int_{0}^{R} \dfrac{t^{(s(\Phi)+1)n-1}}{\left[k^{\frac{n}{n-1}}+t^{\frac{\left(i(\Phi)+1\right)n}{n-1}}\right]^{n-1}}dt\\
			&\geq 2^{1-(n-1)\frac{s(\Phi)+1}{i(\Phi)+1}}n\omega_n\displaystyle\int_{0}^{R} \dfrac{t^{(s(\Phi)+1)n-1}}{\left[k^{\frac{n}{n-1}}+t^{\frac{\left(i(\Phi)+1\right)n}{n-1}}\right]^{(n-1)\frac{s(\Phi)+1}{i(\Phi)+1}}}dt\\
			&\geq 2^{1-(n-1)\frac{s(\Phi)+1}{i(\Phi)+1}}n\omega_n\int_{0}^R{\dfrac{t^{s(\Phi)n+n-1}}{k^{\frac{n}{i(\Phi)+1}(s(\Phi)+1)}+t^{n(s(\Phi)+1)}}dt}\\
			&\geq C_1\ln\left(1+\dfrac{R^{n(s(\Phi)+1)}}{k^{\frac{n}{i(\Phi)+1}(s(\Phi)+1)}}\right),
		\end{split}
	\end{align}
	where a positive constant $C_1 \equiv C_1(n,s(\Phi),i(\Phi))$. Otherwise, $R \geq 1$, then we obtain
	\begin{align}
		\label{ABP: eq_9}
		\begin{split}
			I_1 &= n\omega_n\left[\displaystyle\int_{0}^{1} \dfrac{t^{(s(\Phi)+1)n-1}}{\left[k^{\frac{n}{n-1}}+t^{\frac{\left(i(\Phi)+1\right)n}{n-1}}\right]^{n-1}}dt + \displaystyle\int_{1}^{R} \dfrac{t^{(i(\Phi)+1)n-1}}{\left[k^{\frac{n}{n-1}}+t^{\frac{\left(i(\Phi)+1\right)n}{n-1}}\right]^{n-1}}dt\right]\\
			&\geq n\omega_n \left[2^{1-(n-1)\frac{s(\Phi)+1}{i(\Phi)+1}}\hspace{-0.1cm}\int_{0}^{1}{\dfrac{t^{s(\Phi)n+n-1}}{k^{\frac{n(s(\Phi)+1)}{i(\Phi)+1}}+t^{n(s(\Phi)+1)}}dt}\right. \\
			&\hspace{1.5cm}\left.+ 2^{1-(n-1)}\int_{1}^R{\dfrac{t^{i(\Phi)n+n-1}}{k^{n}+t^{n(i(\Phi)+1)}}dt}\right]\\
			&\geq C_2 \left[\ln\left(1+k^{-\frac{n}{i(\Phi)+1}(s(\Phi)+1)}\right) + \ln\left(\dfrac{k^{n} + R^{n(i(\Phi)+1)}}{k^{n}+1}\right)\right]\\
			&\geq C_2\left[\ln\left(1 + \dfrac{R^{n(i(\Phi)+1)}}{k^n}\right)-\ln{2}\right],
		\end{split}
	\end{align}
	where $C_2>0$ is a constant, depending only $n,s(\Phi)$, and $i(\Phi)$. Consequently, up to changing the constant on the right-hand side, we obtain
	\begin{itemize}
		\item if $R<1$, then
		\begin{align*}
			C_1\ln\left(1+\dfrac{R^{n(s(\Phi)+1)}}{k^{\frac{n}{i(\Phi)+1}(s(\Phi)+1)}}\right)\leq \dfrac{2^{n-1}L^nD}{n^n\lambda^n\nu_0^n}\int_{\Gamma^+(u^+)}{\left[(h^-)^nk^{- n}+1\right]dx},
		\end{align*}
		\item if $R\geq 1$, then
		\begin{align*}
			C_2\ln\left(1 + \dfrac{R^{n(i(\Phi)+1)}}{k^{\frac{n}{i(\Phi)+1}(i(\Phi)+1)}}\right)\hspace{-0.1cm}\leq\hspace{-0.1cm} \dfrac{2^{n-1}L^nD}{n^n\lambda^n\nu_0^n}\int_{\Gamma^+(u^+)}{\left[(h^-)^nk^{- n}+1\right]dx}+C_2\ln{2}.
		\end{align*}
	\end{itemize}
	Assume $R_0>0$ (otherwise there is nothing to prove). Taking
	\begin{align*}
		k:=\max\left\{1,\|h^-\|_{L^n(\Gamma^+(u^+))}^{\frac{i(\Phi)+1}{s(\Phi)+1}},\|h^-\|_{L^n(\Gamma^+(u^+))}\right\},	
	\end{align*} 
	it follows
	\begin{align*}
		\left\{
		\begin{array}{c}
			\hspace{-0.3cm}\dfrac{R^{n(s(\Phi)+1)}}{k^{\frac{n}{i(\Phi)+1}(s(\Phi)+1)}}\hspace{-0.1cm}\leq\hspace{-0.1cm}\mathrm{exp}\left(\dfrac{2^{n-1}L^nD}{C_1n^n\lambda^n\nu_0^n}\int_{\Gamma^+(u^+)}{\left[(h^-)^n k^{-n}+1\right]dx}\right) -1, \text{ if }R<1,\\[15pt]
			\hspace{-0.3cm}\dfrac{R^{n(i(\Phi)+1)}}{k^{\frac{n}{i(\Phi)+1}(i(\Phi)+1)}}\hspace{-0.1cm}\leq\hspace{-0.1cm}2\mathrm{exp}\left(\dfrac{2^{n-1}L^nD}{C_2n^n\lambda^n\nu_0^n}\int_{\Gamma^+(u^+)}{\left[(h^-)^n k^{-n}+1\right]dx}\right) \hspace{-0.1cm}-\hspace{-0.1cm}1, \text{ if }R \geq 1,
		\end{array}
		\right.
	\end{align*}
	Moreover, this choice of $k$ also implies that 
	\begin{align*}
		\int_{\Gamma^+(u^+)} \left[(h^-)^n k^{-n}+1\right]dx \leq 1 + |\Gamma^+(u^+)| \leq 1 + |\Omega|.
	\end{align*}
	Consequently, the exponential factors in the preceding estimates can be bounded by a positive constant $d$, depending only on $m,n,\lambda,i(\Phi),s(\Phi),L,\nu_0,\\\mathcal{C}$ and $|\Omega|$, hence
	\begin{align*}
		R\leq d \max\left\{1,\|h^-\|_{L^n(\Gamma^+(u^+))}^{\frac{1}{i(\Phi)+1}},\|h^-\|_{L^n(\Gamma^+(u^+))}^{\frac{1}{s(\Phi)+1}}\right\}.
	\end{align*} 
	Since this estimate holds for every $R<R_0$, letting $R\uparrow R_0$ yields
	\begin{align}
		\label{ABP: eq_10}
		R_0 \leq d\max\left\{1,\|h^-\|_{L^n(\Gamma^+(u^+))}^{\frac{1}{i(\Phi)+1}},\|h^-\|_{L^n(\Gamma^+(u^+))}^{\frac{1}{s(\Phi)+1}}\right\}.
	\end{align}
	Substituting \eqref{ABP: eq_0} into \eqref{ABP: eq_10}, we get \eqref{ABP_main_1}.\\
	\textbf{Case 2: $-1<i(\Phi)<0$.} We select
	\begin{align*}
		g(z)=\left(\frac{|z|}{|z|+\varepsilon}\right)^{-i(\Phi)n}\dfrac{\min\{|z|^{i(\Phi)n},|z|^{s(\Phi)n}\}}{\left[k^{\frac{n}{n-1}}+|z|^{\frac{\left(i(\Phi)+1\right) n}{n-1}}\right]^{n-1}},
	\end{align*}
	for an arbitrary number $\varepsilon >0$ and $k\geq 1$ an arbitrary constant. Arguing similarly to \eqref{ABP: eq_6}, we also deduce
	\begin{align*}
		&I_2(\varepsilon):=\int_{B_R}{\left(\frac{|z|}{|z|+\varepsilon}\right)^{-i(\Phi)n}\dfrac{\min\{|z|^{i(\Phi)n},|z|^{s(\Phi)n}\}}{\left[k^{\frac{n}{n-1}}+|z|^{\frac{\left(i(\Phi)+1\right) n}{n-1}}\right]^{n-1}}dz}\\
		&\leq \hspace{-0.1cm}\int_{\Gamma_R^+(u^+)\setminus\mathcal{U}}{\hspace{-0.1cm}\left(\frac{|\nabla u|}{|\nabla u|+\varepsilon}\right)^{-i(\Phi)n}\hspace{-0.1cm}\dfrac{\min\{|\nabla u|^{i(\Phi)n},|\nabla u|^{s(\Phi)n}\}}{\left[k^{\frac{n}{n-1}}+|\nabla u|^{\frac{\left(i(\Phi)+1\right) n}{n-1}}\right]^{n-1}}\hspace{-0.1cm}\left[\dfrac{h^-+\mathcal{C}|\nabla u|^{m}}{n\lambda\Phi(x,|\nabla u|)}\right]^ndx}\\
		&\leq \dfrac{2^{n-1}L^nD}{n^n\lambda^n\nu_0^n}\int_{\Gamma^+(u^+)}{\left(\frac{|\nabla u|}{|\nabla u|+\varepsilon}\right)^{-i(\Phi)n}\left[(h^-)^nk^{- n}+1\right]dx},
	\end{align*}
	where we use again {\upshape\hyperref[A2]{(A2)}} and the fact that $i(\Phi) < 0$. Letting $\varepsilon \to 0^+$, we recover the same lower bound as in \textbf{Case 1}, i.e.
	\begin{align*}
		\lim\limits_{\varepsilon \to 0^+}I_2(\varepsilon) \leq \dfrac{2^{n-1}L^nD}{n^n\lambda^n\nu_0^n}\int_{\Gamma^+(u^+)}{\left[(h^-)^nk^{- n}+1\right]dx}.
	\end{align*}
	On the other hand, using co-area formula and recalling $-1 < i(\Phi) <0$, we get 
	\begin{align*}
		I_2(\varepsilon) = n\omega_n\displaystyle\int_{0}^{R}\left(\frac{t}{t+\varepsilon}\right)^{-i(\Phi)n}\dfrac{\min\{t^{i(\Phi)n},t^{s(\Phi)n}\}}{\left[k^{\frac{n}{n-1}}+t^{\frac{\left(i(\Phi)+1\right)n}{n-1}}\right]^{n-1}}t^{n-1}dt.
	\end{align*}
	Since $\left(\frac{t}{t+\varepsilon}\right)^{-i(\Phi)n} \leq 1$ and $-1 < i(\Phi) <0$ ensures integrability near $t=0$, we may apply Lebesgue's dominated convergence theorem, we obtain
	\begin{align*}
		\lim\limits_{\varepsilon \to 0^+}I_2(\varepsilon) = n\omega_n\int_{0}^{R} \dfrac{\min\{t^{i(\Phi)n},t^{s(\Phi)n}\}}{\left[k^{\frac{n}{n-1}}+t^{\frac{\left(i(\Phi)+1\right)n}{n-1}}\right]^{n-1}}t^{n-1}dt.
	\end{align*}
	Repeating the arguments used for $I_1$, we conclude
	\begin{align*}
		\lim\limits_{\varepsilon \to 0^+}I_2(\varepsilon) \geq \left\{
		\begin{array}{ll}
			C_1(n,s(\Phi),i(\Phi))\ln\left(1+\dfrac{R^{n(s(\Phi)+1)}}{k^{\frac{n}{i(\Phi)+1}(s(\Phi)+1)}}\right), &\text{ if }R < 1,\\[20pt]
			C_2(n,s(\Phi),i(\Phi))\ln\left(1 + \dfrac{R^{n(i(\Phi)+1)}}{k^n}\right),&\text{ if }R \geq 1.
		\end{array}
		\right.
	\end{align*}
	Therefore, combining the last two displays, we also get \eqref{ABP_main_1}.\\
	\textbf{Step 2.} 
	Let $u\in C(\overline{\Omega})$. The remaining part follows as in \cite[Theorem 1.1]{M2010} and \cite[Theorem 1]{DFQ2009} by using sup-convolutions. The only additional point is the stability of the lower-order term $c(x)|u|^{i(\Phi)}u$. This follows from the uniform convergence of the sup-convolutions, the continuity of $c$, and the continuity of $|s|^{i(\Phi)}s$ with respect to $s\in\mathbb{R}$. The estimate for supersolutions is obtained by applying the preceding argument to $-u$, with the transformed operator  $\widetilde F(M):=-F(-M)$ and the transformed Hamiltonian $\widetilde H(x,p):=-H(x,-p)$. Combining the two estimates gives \eqref{ABP_main}.
\end{proof}
Next, for the sake of convenience, for the ball $B_r(x_0)$, we denote
\begin{align*}
B'_r(x_0)=B_r(x_0)\cap\{y_n>\phi(y')\},\mbox{ and }B''_r(x_0)=B_r(x_0)\cap\{y_n=\phi(y')\},
\end{align*}
where $\phi$ is introduced in Section \ref{section:2}. In particular, if $x_0=0$, we will write $B'_r(x_0)=B'_r$ and $B''_r(x_0)=B''_r$.
By the ABP estimate (Theorem \ref{theo:ABP_estimate}), any viscosity solution of \eqref{1.2} is bounded in $L^{\infty}(\Omega)$ under the assumptions {\upshape\hyperref[A1]{(A1)}--\hyperref[A4]{(A4)}}. In this section, in order to establish further boundary Lipschitz estimates in the spirit of \cite{JSRR2023,BD2014,BBLL2024}, we focus on a bounded viscosity solution of
\begin{align}\label{eq-C^2}
	\left\{\begin{array}{rccl}
		&\hspace{-0.5cm}\Phi(y,|\nabla u|)F(D^2u)-H(y,\nabla u)+c(y)|u|^{i(\Phi)}u&\hspace{-0.2cm}=h(y) &\hspace{-0.2cm}\text{in } B',\\
		&\hspace{-0.5cm}u(y)&\hspace{-0.2cm}=g(y) &\hspace{-0.2cm}\text{on }B''.
	\end{array}\right.
\end{align}
\begin{remark}[\textbf{Smallness regime}] \label{Remark: smallness-regime}Here we verify that, for a bounded viscosity solution $u$ of
	\begin{align}\label{eq-C^2-smallness}
		\left\{\begin{array}{rccl}
			&\Phi(y,|\xi+\nabla u|)F(D^2u)-H(y,\xi+\nabla u)+c(y)|u|^{i(\Phi)}u&=h(y) &\text{in } B',\\
			&u(y)&=g(y) &\text{on }B'',
		\end{array}\right.
	\end{align}
	we are able to assume
	\begin{align}\label{6.10}
		\begin{split}
			&\|u\|_{L^{\infty}(B_1')}\le 1, \quad \|g\|_{C^{1,\alpha_g}(B_1'')}\le 1, \quad \max\left\{\|h\|_{L^{\infty}(B_1')}, \|c\|_{L^{\infty}(B_1')}, \mathcal{C}\right\} \leq \varepsilon_0,
		\end{split}
	\end{align}
	some positive constant $\varepsilon_0\in (0,1)$, and $\nu_0=\nu_1=1$ in the assumption {\upshape\hyperref[A2]{(A2)}}. In order to consider the problem in a smallness regime as in \eqref{6.10}, for a fixed ball $B_r(x_0) \subset B$ with $x_0=(x_{01},x_{02},\dots,x_{0n})\in\mathbb{R}^n$, we define $\hat{u}:~B_1\cap \{y_n>\hat{\phi}(y')\} \to \mathbb{R}$ by
	\begin{align}
		\label{small_regime_scaling}
		\hat{u}(y) = \frac{u(x_0+ry)}{K},
	\end{align}
	for a function $\hat{\phi}\in C^2(\mathbb{R}^{n-1})$ and positive constants $K \geq 1 \geq r$ which will be determined in a few lines. It can be seen that $\hat{u}$ is a viscosity solution of
	\begin{align}\label{regime-form}
		\left\{\begin{array}{rccl}
			&\hat{\Phi}(y,|\hat{\xi}+\nabla \hat{u}|)\hat{F}(D^2\hat{u})-\hat{H}(y,\hat{\xi}+\nabla \hat{u})+\hat{c}(y)|\hat{u}|^{i(\Phi)}\hat{u}&=\hat{h}(y) &\text{in } \hat{B'_1},\\
			&\hat{u}(y)&=\hat{g}(y) &\text{on }\hat{B''_1},
		\end{array}\right.
	\end{align}
	where
	\begin{align*}
		\begin{array}{ll}
			&\hat{\Phi}(y,t) :=\displaystyle\frac{\Phi\left(x_0+ry,\frac{K}{r}t\right)}{\Phi\left(x_0+ry,\frac{K}{r}\right)},~\hat{F}(M) := \displaystyle \frac{r^2}{K}F\left(\frac{K}{r^2}M\right),
			\\&\hat{H}(y,p) := \displaystyle\frac{r^2}{K\Phi\left(x_0+ry,\frac{K}{r}\right)}H\left(x_0+ry,\frac{K}{r}p\right),\\[15pt]
			& \hat{c}(y) := \displaystyle\frac{r^2K^{i(\Phi)}}{\Phi\left(x_0+ry,\frac{K}{r}\right)}c(x_0+ry),
			~\hat{h}(y) := \displaystyle\frac{r^2}{K\Phi\left(x_0+ry,\frac{K}{r}\right)}h(x_0+ry),\\[15pt]
			& \hat{\phi}(y):=\dfrac{\phi(ry'+x_0')-x_{0n}}{r},~\hat{g}(y) := \dfrac{g(x_0+ry)}{K}, \hat{\xi} := \dfrac{r}{K}\xi,\\[15pt]
			&\hat{B'_1}=B_1\cap\{y_n>\hat{\phi}(y')\}, \mbox{ and }\hat{B''_1}=B_1\cap\{y_n=\hat{\phi}(y')\}.
		\end{array}
	\end{align*}
	Note that $\hat{F}$ remains a uniformly $(\lambda,\Lambda)$-elliptic operator, the map $t \mapsto \displaystyle\frac{\hat{\Phi}(y,t)}{t^{i(\Phi)}}$ is almost non-decreasing, and the map $t \mapsto \displaystyle\frac{\hat{\Phi}(y,t)}{t^{s(\Phi)}}$ is almost non-increasing with the same constant $L \geq 1$ and $s(\Phi) \geq i(\Phi) > -1$ as in assumption {\upshape\hyperref[A2]{(A2)}}, and $\hat{\Phi}(y,1) = 1$ for all $y \in B_1$. 
	It follows directly from the choice of $r$ that 
	\begin{align*}
		\|D^2\hat{\phi}\|_{\infty} \leq \|D^2\phi\|_{\infty} \mbox{ and }\|\hat{g}\|_{C^{1,\alpha_g}(\hat{B''_1})} \leq \frac{1}{K}\|g\|_{C^{1,\alpha_g}(\partial \Omega)}.
	\end{align*}
	Furthermore, the assumptions {\upshape\hyperref[A2]{(A2)}} and {\upshape\hyperref[A3]{(A3)}} imply
	\begin{align*}
		\left\{\begin{array}{l}
			\displaystyle|\hat{H}(y,p)| \leq \dfrac{Lr^{2}}{\nu_0K}\max\left\{\left(\frac{r}{K}\right)^{i(\Phi)},\left(\frac{r}{K}\right)^{s(\Phi)}\right\}\mathcal{C}\left(\dfrac{K}{r}\right)^{m}|p|^m,\\[11.5pt]
			\displaystyle\|\hat{c}\|_{L^\infty(\hat{B'_1})} \leq \frac{L}{\nu_0}K^{i(\Phi)}r^2\max\left\{\left(\frac{r}{K}\right)^{i(\Phi)},\left(\frac{r}{K}\right)^{s(\Phi)}\right\}\|c\|_{L^\infty(\Omega)},\\[11.5pt]
			\displaystyle\|\hat{h}\|_{L^\infty(\hat{B'_1})} \leq \frac{L}{\nu_0}\frac{r^2}{K}\max\left\{\left(\frac{r}{K}\right)^{i(\Phi)},\left(\frac{r}{K}\right)^{s(\Phi)}\right\}\|h\|_{L^\infty(\Omega)}.
		\end{array}
		\right.
	\end{align*}
	Since $K \geq 1 \geq r$ and $-1<i(\Phi) \leq s(\Phi)$, we obtain
	\begin{align*}
		\left\{\begin{array}{l}
			|\hat{H}(y,p)| \leq \dfrac{\mathcal{C}L}{\nu_0}r^{2+i(\Phi)-m}|p|^m,\\[10pt]
			\displaystyle\|\hat{c}\|_{L^\infty(\hat{B'_1})} \leq \frac{L}{\nu_0}r^{2+i(\Phi)} \|c\|_{L^\infty(\Omega)},\\[10pt]
			\displaystyle\|\hat{h}\|_{L^\infty(\hat{B'_1})} \leq \frac{Lr^{2+i(\Phi)}}{\nu_0K^{1+i(\Phi)}}\|h\|_{L^\infty(\Omega)}.
		\end{array}
		\right.
	\end{align*} 
	Now, for given constant $\varepsilon_0 \in (0, 1)$, 
	recalling that $-1 < i(\Phi)$ and $0 < m \leq i(\Phi) + 1$, we select
	\begin{align*}
		K: = 2\left[1 + \|u\|_{L^\infty(\Omega)} + \|g\|_{C^{1,\alpha_g}(\partial \Omega)} +  \left(\frac{L}{\nu_0}\|h\|_{L^\infty(\Omega)}\right)^{\frac{1}{1+i(\Phi)}}\right],
	\end{align*}
	and
	\begin{align*}
		r:=\min\left\{1,\left(\frac{\varepsilon_0}{1 + \frac{L}{\nu_0}\|c\|_{L^\infty(\Omega)}}\right)^{\frac{1}{2+i(\Phi)}},\left(\dfrac{\nu_0\varepsilon_0}{L\mathcal{C}}\right)^{\frac{1}{2+i(\Phi)-m}}\right\}.
	\end{align*}
	It follows that $\hat{u}$ solves equation \eqref{regime-form} within the smallness regime defined by \eqref{6.10}.
\end{remark}
By imposing special conditions on a viscosity supersolution (or subsolution), we can apply the comparison principle without an extra structure condition like {\upshape\hyperref[A5]{(A5)}} in Subsection \ref{Subsection: local and comparison}. Further discussion regarding the comparison principle II is provided in Subsection \ref{Subsection: local and comparison}.
\begin{lemma}[\bf Comparison principle I]
	\label{Comp.Prin1}
	Let $\Omega\subset \mathbb{R}^n$ be a bounded domain and $c,h_1,h_2 \in C(\overline{\Omega})$. Suppose that $u\in C(\overline{\Omega})$ is a viscosity subsolution to
	\begin{align}
		\label{gene_main_com_I}
		\Phi(x,|\nabla u|) F(D^2 u) - H(x,\nabla u) + c(x)|u|^{i(\Phi)}u = h_1(x) \mbox{ in }\Omega,
	\end{align} 
	and $v\in C(\overline{\Omega})\cap C^2(\Omega)$ is a viscosity supersolution to \eqref{gene_main_com_I} with $h_1$ replaced by $h_2$. In the singular case $-1<i(\Phi)<0$, assume additionally that $|\nabla v|>0$ in $\Omega$. Assume further that $v \geq u$ on $\partial\Omega$ and that one of the following is satisfied:
	\begin{enumerate}[\upshape\quad(a)]
		\item\phantomsection\label{CompaI_1} $c<0$ in $\overline{\Omega}$ and $h_1 \geq h_2$ in $\overline{\Omega}$;
		\item\phantomsection\label{CompaI_2} $c \leq 0$ in $\overline{\Omega}$ and $h_1 > h_2$ in $\overline{\Omega}$.
	\end{enumerate}
	Then $v\geq u$ in $\overline{\Omega}$.
\end{lemma}
\begin{proof}
	On the contrary, we assume that $M=\max\limits_{\overline{\Omega}}{(u-v)}>0$. Let $\overline{x}$ be a point for $u-v$ where the maximum is attained, implying $v+M$ touches $u$ from above at $\overline{x}$. In the singular case, the additional assumption ensures that $\nabla v(\overline{x})\neq0$, so that $v+M$ is an admissible test function for $u$ at 
	$\overline{x}$. Note that $v\in C^2(\Omega)$, thanks to the definition of viscosity subsolutions, we obtain that
	\begin{align}\label{com1}
		\Phi(\overline{x},|\nabla v(\overline{x})|)F(D^2v(\overline{x}))\hspace{-0.1cm}-\hspace{-0.1cm}H(\overline{x},\nabla v(\overline{x}))\hspace{-0.1cm}+\hspace{-0.1cm}c(\overline{x})|v(\overline{x})\hspace{-0.1cm}+\hspace{-0.1cm}M|^{i(\Phi)}(v(\overline{x})\hspace{-0.1cm}+\hspace{-0.1cm}M)\hspace{-0.1cm}\geq\hspace{-0.1cm} h_1(\overline{x}).
	\end{align}
	Moreover, since $v\in C^2(\Omega)$ is a viscosity supersolution then
	\begin{align}\label{com2}
		\Phi(\overline{x},|\nabla v(\overline{x})|)F(D^2v(\overline{x}))-H(\overline{x},\nabla v(\overline{x}))+c(\overline{x})|v(\overline{x})|^{i(\Phi)}v(\overline{x})\le h_2(\overline{x}).
	\end{align}
	Combine \eqref{com1} and \eqref{com2}, we deduce that
	\begin{align}\label{6.1.1}
		h_1(\overline{x})-h_2(\overline{x})\leq c(\overline{x})\left[|v(\overline{x})+M|^{i(\Phi)}(v(\overline{x})+M)-|v(\overline{x})|^{i(\Phi)}v(\overline{x})\right].
	\end{align}
	Since $i(\Phi)>-1$, the map $t\mapsto |t|^{i(\Phi)}t$ is strictly increasing. As $M>0$, the quantity in brackets is strictly positive. In case \hyperref[CompaI_1]{\upshape(a)}, the left-hand side is non-negative, whereas the right-hand side is strictly negative.  In case \hyperref[CompaI_2]{\upshape(b)}, the left-hand side is strictly positive, whereas the right-hand side is non-positive. Both cases yield a contradiction.
\end{proof}
\begin{remark}\label{23}
	The same conclusion remains valid if the roles are reversed; namely, if 
	$u\in C^2(\Omega)\cap C(\overline{\Omega})$ is a viscosity subsolution and $v\in C(\overline{\Omega})$ is a viscosity supersolution. In the singular case $-1<i(\Phi)<0$, assume additionally that $|\nabla u|>0$ in $\Omega$.
\end{remark}
The following lemma characterizes the boundary behavior of a viscosity solution $u$ via the distance function $d$. 
Our idea is inspired by the work in \cite{BD2014}.
\begin{lemma}\label{boundary-behavior}
	Assume that $-1 < i(\Phi)$ and $0 < m \leq i(\Phi)+1$. Let $g\in C^{1,\alpha_g}(\partial\Omega)$ and $\xi \in \mathbb{R}^n$.
	Let $d$ be the distance to the hypersurface $\{y_n = \phi(y')\}$. Then for every $0 < r < 1$ and $\zeta \in (0,1)$, if $u$ is a viscosity solution of the problem
	\begin{align}\label{6.11}
		\left\{\begin{array}{rccl}
			&\hspace{-0.4cm}\Phi(y,|\xi+b\nabla u|)F(D^2u)-H(y,\xi+b\nabla u)+c(y)|u|^{i(\Phi)}u&=h(y )
			&\text{in } B'_1,\\
			&\hspace{-0.4cm}u(y)&=g(y) &\text{on } B''_1,
		\end{array}\right.
	\end{align}
	under assumptions {\upshape\hyperref[A1]{(A1)}--\hyperref[A4]{(A4)}} with 
	\begin{align*}
		&\|u\|_{L^{\infty}(B'_1)} \leq 1, \mbox{ and }\max\left\{\|c\|_{L^{\infty}(B'_1)}, \|h\|_{L^{\infty}(B'_1)}, \mathcal{C}\right\} \leq \varepsilon_0.
	\end{align*}
	Then we have the following results:
	\begin{enumerate}[\upshape (i)]
		\item\phantomsection\label{boundary_behav_1} If $\xi = 0$, $c\le 0$ and $b=1$, there exists $\theta_0 > 0$, depending on $n,\lambda,\Lambda,\Omega, \zeta,i(\Phi),\varepsilon_0, r,L, m, \nu_0$, $\mathrm{Lip}_g(\partial \Omega)$  such that for every $0 < \theta<\theta_0$, then
		\begin{align}\label{6.13'}
			|u(y',y_n)-g(y')|\le \dfrac{6}{\theta}\dfrac{d(y)}{1+d(y)^{\zeta}}\text{ in }B'_r.
		\end{align}
		\item\phantomsection\label{boundary_behav_2} If $|\xi| = 1$ and $c = 0$, there exists $\theta > 0$, depending on $n,\lambda,\Lambda,\Omega, \zeta,i(\Phi)$, $s(\Phi),\varepsilon_0$, $r, L, m, \nu_0$, $\mathrm{Lip}_g(\partial \Omega)$ such that for every $0\leq b<\theta/6$, we also have the inequality \eqref{6.13'}. 
	\end{enumerate}
\end{lemma}
\begin{proof}
	For the sake of convenience, the $L^\infty$ norm over $B'_1$ will be denoted simply by $\|\cdot\|_\infty$ for any involved functions, including $u, c,$ and $h$.\\
	The proof is divided into two separate cases:\\
	\textbf{Case 1:} $g\equiv 0$. In this case, we have
	\begin{align*}
		|u(y',y_n)-g(y')|\le \|u\|_{\infty}\le 1,
	\end{align*}
	so we only need to consider the smaller set $\Omega_{\theta}=\{y\in\Omega:d(y)<\theta\}$. Moreover, we choose $\theta_1>0$ such that if $d(y)<\theta_1$,
	then $d$ belongs to $C^2$, hence there exists some constant $C_d>0$ such that $|D^2d|\le C_d$.\\
	The proof is based on the construction of both upper and lower barrier functions. To achieve this, we introduce a function $w \in C^2(\Omega_\theta)$ specified as follows:
	\begin{align}
		\label{eq:boundary-behavior_1}
		w(y)=\left\{
		\begin{array}{ll}
			\dfrac{2}{\theta}\dfrac{d(y)}{1+d(y)^{\zeta}}&\text{for }|y|<r,\\[10pt]
			\dfrac{2}{\theta}\dfrac{d(y)}{1+d(y)^{\zeta}}+\dfrac{1}{(1-r)^3}(|y|-r)^3&\text{for }|y|\geq r.
		\end{array}
		\right.
	\end{align}
	By following the argument in \cite[Lemma 2.2]{BD2014}, we have $w\geq u$ on $\partial (B'_1\cap \Omega_{\theta})$. We need to check that $w$ is a supersolution. A simple calculation reveals that
	\begin{align}
		\label{eq:boundary-behavior_2}
		\nabla w(y)=\left\{
		\begin{array}{ll}
			\hspace{-0.3cm}\dfrac{2}{\theta}\left[\dfrac{1+(1-\zeta)d(y)^{\zeta}}{(1+d(y)^{\zeta})^2}\right]\nabla d(y)&\hspace{-0.3cm}\text{when }|y|<r,\\[15pt]
			\hspace{-0.3cm}\dfrac{2}{\theta}\left[\dfrac{1+(1-\zeta)d(y)^{\zeta}}{(1+d(y)^{\zeta})^2}\right]\nabla d(y)+\dfrac{y}{|y|}\dfrac{3}{(1-r)^3}(|y|-r)^2&\hspace{-0.2cm}\text{for }|y|\geq r.
		\end{array}
		\right.
	\end{align}
	It is easy to check that $|\nabla w|\geq \dfrac{1}{4\theta}$ when $\theta\le \dfrac{1-r}{12}$, $|\nabla w|\le \dfrac{4}{\theta}$, and $\|w\|_{\infty}\le 3$. As a result, if we choose $\theta<\dfrac{1}{4}$, it follows that $|\nabla w|\geq 1$. A straightforward calculation shows that
	\begin{align*}
		D^2w\hspace{-0.1cm}=\hspace{-0.1cm}-\hspace{-0.1cm}\left(\dfrac{2\zeta d^{\zeta-1}}{\theta}\right)\hspace{-0.1cm}\dfrac{(1+\zeta)\hspace{-0.1cm}+\hspace{-0.1cm}(1-\zeta)d^{\zeta}}{(1+d^{\zeta})^3}\hspace{-0.1cm}\left(\nabla d\otimes\nabla d\right) \hspace{-0.1cm}+\hspace{-0.1cm}\dfrac{2}{\theta}\dfrac{1\hspace{-0.1cm}+\hspace{-0.1cm}(1-\zeta)d^{\zeta}}{(1+d^{\zeta})^2}D^2d+\mathcal{R}(y),
	\end{align*}
	where $\|\mathcal{R}(y)\|\le \dfrac{6}{(1-r)^2}+\dfrac{3(n-1)}{r(1-r)}$. Therefore, we obtain that
	\begin{align*}
		\mathcal{P}_{\lambda,\Lambda}^{+}(D^2w)&\le -2\zeta\theta ^{\zeta-2}\lambda\dfrac{1+\zeta}{(1+\theta^{\zeta})^3}+\dfrac{2}{\theta}nC_d\Lambda+\dfrac{6n\Lambda}{(1-r)^2} + C(n,r) \\
		&\le C(-\theta^{\zeta-2}+\theta^{-1}),
	\end{align*}
	for some $C>0$. As $\zeta-2<-1<0$, it is possible to choose $\theta\in (0,1)$ sufficiently small such that $\mathcal{P}_{\lambda,\Lambda}^{+}(D^2w)<0$. At this stage, we separate the remaining part of the proof into several cases:
	
	\textit{Step 1: Proof of \upshape\hyperref[boundary_behav_1]{(i)}.} Suppose that $\xi = 0$ and $b = 1$. Thanks to the assumptions {\upshape\hyperref[A2]{(A2)}} and {\upshape\hyperref[A3]{(A3)}}, we derive that
	\begin{align*}
		&\Phi(y,|\nabla w|)\mathcal{P}_{\lambda,\Lambda}^{+}(D^2w)\hspace{-0.1cm}\lesssim\hspace{-0.1cm} -\frac{\nu_0}{L}|\nabla w|^{i(\Phi)}(\theta^{\zeta-2}-\theta^{-1})\hspace{-0.1cm}\leq\hspace{-0.1cm} -\frac{\nu_0}{L}(\theta^{\zeta-i(\Phi)-2}\hspace{-0.1cm}-\hspace{-0.1cm}\theta^{-i(\Phi)-1}\hspace{-0.1cm}),\\
		&-H(y,\nabla w)\le \mathcal{C}|\nabla w|^{i(\Phi)+1}\le \mathcal{C}4^{i(\Phi)+1}\theta^{-i(\Phi)-1},\\
		&c(y)|w|^{i(\Phi)}w \leq \|c\|_\infty 3^{i(\Phi) + 1}.
	\end{align*}
	As a result, we obtain that
	\begin{align*}
		&\Phi(y,|\nabla w|)\mathcal{P}_{\lambda,\Lambda}^{+}(D^2w)-H(y,\nabla w) + c(y)|w|^{i(\Phi)}w \\
		\leq~& -\frac{\nu_0}{L}(\theta^{\zeta-i(\Phi)-2}-\theta^{-i(\Phi)-1})+\mathcal{C}4^{i(\Phi)+1}\theta^{-i(\Phi)-1}+\|c\|_{\infty}3^{i(\Phi)+1}.
	\end{align*}
	Since $\zeta-i(\Phi)-2<-i(\Phi)-1<0$, we finally choose $\theta\in (0,1)$ small enough such that
	\begin{align*}
		\Phi(y,|\nabla w|)\mathcal{P}_{\lambda,\Lambda}^{+}(D^2w)-H(y,\nabla w) + c(y)|w|^{i(\Phi)}w\le -\|h\|_{\infty}-1.
	\end{align*}
	Furthermore, we also have $F(D^2w) \leq \mathcal{P}_{\lambda,\Lambda}^{+}(D^2w)$. 
	Hence, we derive
	\begin{align*}
		\Phi(y,|\nabla w|)F(D^2w)-H(y,\nabla w) + c(y)|w|^{i(\Phi)}w\le h(y)-1.
	\end{align*}
	By comparison principle I (Lemma \ref{Comp.Prin1}), we deduce that
	\begin{align*}
		u\le w=\dfrac{2}{\theta}\dfrac{d(y)}{1+d(y)^{\zeta}} \text{ in }B'_r.
	\end{align*}
	The lower bound for $u$ can be obtained by applying the same argument to $-w$; the Hamiltonian term is controlled by its growth assumption, and the reaction term is harmless since $c\leq0$.
	
	\textit{Step 2: Proof of \upshape\hyperref[boundary_behav_2]{(ii)}.} Suppose $c = 0$. If $b=0$, there is no degeneracy with respect to $\nabla u$, then the result holds. Now, we focus on the case $b>0$. we know that if we choose $\theta<\dfrac{1-r}{3}$ then $|\nabla w|\le \dfrac{3}{\theta}$ and $\|w\|_\infty \leq 2$. Consequently, it follows that
	\begin{align*}
		\dfrac{1}{2}\le |\xi+b\nabla w|\le \dfrac{3}{2}\text{ for }0<b<\dfrac{\theta}{6}.
	\end{align*}
	Then, by recalling {\upshape\hyperref[A2]{(A2)}} and {\upshape\hyperref[A3]{(A3)}}, we arrive at
	\begin{align*}
		&\Phi(y,|\xi+b\nabla w|)\mathcal{P}_{\lambda,\Lambda}^{+}(D^2w)\le  C(-\theta^{\zeta-2}+\theta^{-1}),\\
		&-H(y,\xi+b\nabla w)\le \mathcal{C}\left(\frac{3}{2}\right)^{i(\Phi)+1},
	\end{align*}
	for some $C>0$. This estimate allows us to choose $\theta\in (0,1)$ small enough so that
	\begin{align*}
		\Phi(y,|\xi+b\nabla w|)\mathcal{P}_{\lambda,\Lambda}^+(D^2w)-H(y,\xi+b\nabla w)\le -\|h\|_{\infty}-1.
	\end{align*}
	We may now proceed as in the proof of \hyperref[boundary_behav_1]{\upshape(i)} to obtain the desired estimate.
	
	\textbf{Case 2:} $g$ is not identically zero. We use the upper and lower barriers constructed in \cite[Lemma 2.2]{BD2014}. For \hyperref[boundary_behav_1]{\upshape(i)}, the additional Hamiltonian and reaction terms are controlled by the growth assumption on $H$, the bound $\|c\|_{\infty} \leq \varepsilon_{0}$, and the fact that the barrier gradients are comparable to $\theta^{-1}$. For \hyperref[boundary_behav_2]{\upshape(ii)}, since $|\xi|=1$, $c=0$, and $0 \leq b < \theta/6$, the quantity $|\xi + b \nabla w|$ remains bounded away from zero and infinity. Hence, after decreasing $\theta$, if necessary, the barriers remain strict super- and subsolutions in both cases, and the desired estimates follow from the comparison principle.
\end{proof}
The primary result of this section is the boundary Lipschitz estimate established below. Its proof is based on Lemma \ref{boundary-behavior} and the Ishii–Jensen Lemma \cite[Theorem 3.2]{CIL1992}.
	%
For the modified equation \eqref{1.2}, the estimate is obtained in the following two normalized regimes:
\begin{enumerate}[1.]
	\item $\xi = 0$; see Lemma \ref{boundary-behavior}\hyperref[boundary_behav_1]{\upshape(i)} and \hyperref[Result1]{\upshape(R1)} in Theorem \ref{Lipschitz_estimate}.
	\item $|\xi| = 1$; see Lemma \ref{boundary-behavior}\hyperref[boundary_behav_2]{\upshape(ii)} and \hyperref[Result2]{\upshape(R2)} in Theorem \ref{Lipschitz_estimate}.
\end{enumerate}
\begin{theorem}[\bf Boundary Lipschitz estimates]\label{Lipschitz_estimate}
	Assume that $-1 < i(\Phi),$ $0 < m \leq i(\Phi)+1$. Let $g$ be a Lipschitz continuous function on $\partial \Omega$, $\xi \in \mathbb{R}^n$ be an arbitrary vector. Suppose $u$ is a viscosity solution of 
	\begin{align}
		\label{Lipschitz_estimate_main}
		\begin{array}{clc}
			\left\{\begin{array}{rccl}
				&\hspace{-0.6cm}\Phi(y,|\xi+b\nabla u|)F(D^2u)-H(y,\xi+b\nabla u)+c(y)|u|^{i(\Phi)}u&=h(y )
				&\text{in } B'_1,\\
				&\hspace{-0.6cm}u(y)&=g(y) &\text{on } B''_1,
			\end{array}\right.
		\end{array}
	\end{align}
	under the assumptions {\upshape\hyperref[A1]{(A1)}--\hyperref[A4]{(A4)}} with
	\begin{align*}
		&\|u\|_{L^\infty(B'_1)} \leq 1, \mbox{ and }\max\left\{\|c\|_{L^\infty(B'_1)}, \|h\|_{L^\infty(B'_1)},\mathcal{C}\right\} \leq \varepsilon_0.
	\end{align*}
	Then we have the following results 
	\begin{enumerate}[\upshape (R1)]
		\item \phantomsection\label{Result1} If $b=1$, $c \leq 0$ and $\xi = 0$, then for every $r\in (0,1)$, $u$ is Lipschitz continuous in $B'_r$ with the estimate
		\begin{align}
			\label{R1:holder_continuity_1}
			[u]_{C^{0,1}(B'_r)} \leq C_{\mathrm{zl}}\left(n,\lambda,\Lambda,i(\Phi),s(\Phi),L,\mathrm{Lip}_g(\partial \Omega),r,m,\varepsilon_0,\mathcal{C}\right).
		\end{align}
		\item \phantomsection\label{Result2} If $|\xi|=1$ and $c = 0$, then for every $r\in (0,1)$, there exists $b_0>0$, depending on $\lambda,\Lambda,\Omega,i(\Phi),s(\Phi),m,\varepsilon_0$, $\nu_0,L,r,\mathcal{C}$, $\mathrm{Lip}_g(\partial \Omega)$ such that for every $0\leq b<b_0$, $u$ is Lipschitz continuous in $B'_r$ with the estimate
		\begin{align}
			\label{R1:holder_continuity_2}
			[u]_{C^{0,1}(B'_r)} \leq C_{\mathrm{dl}}\left(n,\lambda,\Lambda,i(\Phi),s(\Phi),L,\mathrm{Lip}_g(\partial \Omega),r,m,\varepsilon_0,\mathcal{C}\right).
		\end{align}
	\end{enumerate}
\end{theorem}
\begin{proof}
	\textbf{}
	
	\textit{Step 1: Proof of \hyperref[Result1]{\upshape(R1)}.} Suppose $\xi = 0$ and $b=1$. Let $\rho \in (r,1)$ be fixed. For $x_0 \in B'_r$, we consider 
	\begin{align}
		\label{eq:Holder_1}
		\psi(x,y):=u(x)-u(y)-Q\omega(|x-y|)-P(|x-x_0|^2+|y-x_0|^2),
	\end{align}
	where $\omega$ is defined by
	\begin{align}
		\label{eq:Holder_2}
		\omega(t)=\left\{
		\begin{array}{ll}
			t-\omega_0t^{\frac{3}{2}},&\text{if } t\le t_0: = \left(\frac{2}{3\omega_0}\right)^2,\\
			\omega(t_0),&\text{if }t\geq t_0.
		\end{array}
		\right.
	\end{align}
	We claim that $Q,P \gg 1$ sufficiently large,
	\begin{align}
		\label{eq:Holder_main}
		\psi(x,y) \leq 0 \mbox{ for all } (x,y) \in \left(B_\rho \cap \overline{\Omega}\right) \times\left(B_\rho \cap \overline{\Omega}\right).
	\end{align} 
	Observe that this inequality yields the desired Lipschitz estimate.\\
	Firstly, we assume that $y \in B''_\rho$. Then by Lemma \ref{boundary-behavior}\hyperref[boundary_behav_1]{\upshape(i)}, there exists a constant $M_0 >0$ such that
	\begin{align*}
		|u(s',s_n) - g(s')| \leq M_0d(s,\partial \Omega) \mbox{ for }s \in B'_\rho.
	\end{align*}
	It yields that
	\begin{align*}
		|u(x) - u(y)| &\leq |u(x',x_n) - u(x',\phi(x'))| + |u(x',\phi(x'))-u(y',\phi(y'))|\\
		&\leq M_0d(x,\partial \Omega) + \mathrm{Lip}_g(\partial\Omega)|x'-y'| \leq (M_0 + \mathrm{Lip}_g(\partial\Omega)) |x-y|.
	\end{align*}
	Hence, when we choose $Q \geq 3\left(M_0 + \mathrm{Lip}_g(\partial\Omega)\right)$, then
	\begin{align*}
		\psi(x,y) \leq Q\left(\frac{|x-y|}{3} - \omega(|x-y|)\right) -P(|x-x_0|^2+|y-x_0|^2) \leq 0.
	\end{align*}
	Similarly, the case where $x\in B''_\rho$ is treated in the same way. Hence, $\psi(x,y)\leq0$ whenever either $x \in B''_\rho$ or $y \in B''_\rho$. In order to get \eqref{eq:Holder_main}, we will argue by contradiction by assuming that there exist $x,y \in B_\rho \cap \overline{\Omega}$ such that $\psi(x,y)>0$. Now we define an auxiliary function $\eta:~\left(B_\rho \cap \overline{\Omega}\right) \times\left(B_\rho \cap \overline{\Omega}\right) \to \mathbb{R}$ given by
	\begin{align}
		\label{eq:Holder_3}
		\eta(x,y) := Q \omega(|x-y|) + P(|x-x_0|^2 + |y-x_0|^2).
	\end{align}
	Let $(\overline{x},\overline{y}) \in \left(B_\rho \cap \overline{\Omega}\right) \times\left(B_\rho \cap \overline{\Omega}\right)$ be a maximum point for $\psi$, i.e.,
	\begin{align*}
		\psi(\overline{x},\overline{y})=\displaystyle\sup_{(x,y) \in \left(B_\rho \cap \overline{\Omega}\right) \times\left(B_\rho \cap \overline{\Omega}\right)}\psi(x,y)>0.
	\end{align*}
	This together with $\|u\|_\infty \leq 1$ yields
	\begin{align*}
		Q \omega(|\overline{x} - \overline{y}|) \hspace{-0.1cm}+\hspace{-0.1cm} P(|\overline{x}-x_0|^2\hspace{-0.1cm}+\hspace{-0.1cm}|\overline{y}-x_0|^2) \hspace{-0.1cm}\leq\hspace{-0.1cm} u(\overline{x})\hspace{-0.1cm}-\hspace{-0.1cm}u(\overline{y})\hspace{-0.1cm}\leq \hspace{-0.1cm}2\|u\|_{L^\infty(B'_1)} \hspace{-0.1cm}\leq \hspace{-0.1cm}2.
	\end{align*}
	By choosing $P > \max\left\{8/(\rho - r)^2,1/(2(r + \rho))\right\}$, then we ensure
	\begin{align}
		\label{eq:Holder_*}
		\begin{split}
			|\overline{x} - x_0|, |\overline{y} - x_0| \leq \frac{\rho - r}{2}, |\overline{x} - \overline{y}| \leq\sqrt{2\left(|\overline{x}-x_0|^2 + |\overline{y} - x_0|^2\right)} \leq \frac{2}{\sqrt{P}} \leq 1.
		\end{split}
	\end{align}
	This implies $\overline{x}, \overline{y} \in B'_{(\rho+r)/2}$. Moreover, since $\psi(\overline{x},\overline{y}) > 0$ while $\psi(z,z) = -2P|z - x_{0}|^2 \leq 0$ for every $z$, we necessarily have $\overline{x} \neq \overline{y}$.\\
	We now invoke the Crandall-Ishii-Lions lemma; see \cite[Theorem 3.2]{CIL1992}), to assure the existence of a limiting superjet  $(\varkappa_{\overline{x}}, X_\delta)$ of $u$ at $\overline{x}$ and a limiting subjet $(\varkappa_{\overline{y}}, -Y_\delta)$ of $u$ at $\overline{y}$, such that the matrices $X_\delta,Y_\delta\in\mathbb S^{n\times n}$ satisfy the matrix inequality 
	\begin{align}
		\label{eq:Holder_4}
		\left(\begin{matrix}
			X_\delta&0\\
			0&Y_\delta
		\end{matrix}\right)\le 	Q\left(\begin{matrix}
			Z&-Z\\
			-Z&Z
		\end{matrix}\right)+(2P+\delta)\left(\begin{matrix}
			I&0\\
			0&I
		\end{matrix}\right),
	\end{align}
	where $\delta \in (0,1)$, that only depends on the norm of $Z$, which can be selected to be sufficiently small. Here,
	\begin{align}
		\label{eq:Holder_5}
		\begin{split}
			&\varkappa_{\overline{x}} := D_x\eta(\overline{x},\overline{y})=  Q\omega'(|\overline{x}-\overline{y}|)\frac{\overline{x}-\overline{y}}{|\overline{x}-\overline{y}|}+2P(\overline{x}-x_0),\\
			&\varkappa_{\overline{y}} := -D_y\eta(\overline{x},\overline{y})=Q\omega'(|\overline{x}-\overline{y}|)\frac{\overline{x} - \overline{y}}{|\overline{x} - \overline{y}|}-2P(\overline{y}-x_0),\\
			&Z:=D^2(\omega(|\cdot|))(\overline{x} - \overline{y}) \\&\hspace{0.5cm}= \dfrac{\omega'(|\overline{x}-\overline{y}|)}{|\overline{x}-\overline{y}|}I + \left(\omega''(|\overline{x}-\overline{y}|)-\dfrac{\omega'(|\overline{x}-\overline{y}|)}{|\overline{x}-\overline{y}|}\right)\frac{\left(\overline{x} - \overline{y}\right) \otimes \left(\overline{x} - \overline{y}\right)}{|\overline{x} - \overline{y}|^2}.
		\end{split}		
	\end{align}
	We first choose $\omega_0>0$ sufficiently small so that $t_0 \geq 2 > \rho + r$. Note that $s \mapsto \omega'(s)$ is decreasing on $s \in [0,t_0]$. If we choose $Q> 0$ large enough so that $2P(\rho + r) \leq Q\omega'(\rho + r)/2$, then we collect $2P|\overline{x} - x_0|,2P|\overline{y} - x_0| \leq Q\omega'(|\overline{x} - \overline{y}|)/2$. Applying triangle inequality to \eqref{eq:Holder_5} and by the choice of $P$, we obtain
	\begin{align}
		\label{eq:Holder_7}
		1 \leq \dfrac{Q}{2}\omega'(|\overline{x}-\overline{y}|)\le |\varkappa_{\overline{x}}|, |\varkappa_{\overline{y}}|\le 2Q\omega'(|\overline{x}-\overline{y}|) \leq 2Q.
	\end{align}
	On the other hand, for $X_\delta,Y_\delta \in \mathbb{S}^{n\times n}$, we will use the matrix inequality \eqref{eq:Holder_4} to vectors of the form $(\xi,\xi)$ for any $\xi\in \mathbb{R}^{n}$, we obtain
	\begin{align*}
		\langle (X_\delta+Y_\delta)\xi,\xi\rangle \leq (4P +2\delta)|\xi|^2,
	\end{align*}
	It follows that all the eigenvalues of $X_\delta +Y_\delta$ are less than or equal to $4P +2\delta$. Nevertheless, re-applying \eqref{eq:Holder_4} to the vector $\left(\frac{\overline{x} - \overline{y}}{|\overline{x} - \overline{y}|},\frac{\overline{y} - \overline{x}}{|\overline{x} - \overline{y}|}\right)$, we have
	\begin{align*}
		\left\langle (X_\delta+Y_\delta)\frac{\overline{x} -\overline{y}}{|\overline{x} - \overline{y}|},\frac{\overline{x} - \overline{y}}{|\overline{x} - \overline{y}|}\right\rangle \leq (4P + 2\delta + 4Q\omega''(|\overline{x} - \overline{y}|))\left|\frac{\overline{x} - \overline{y}}{|\overline{x} - \overline{y}|}\right|^2.
	\end{align*}
	This implies that the matrix $X_\delta + Y_\delta$ possesses at least one eigenvalue less than $4P + 2\delta + 4Q\omega''(|\overline{x} - \overline{y}|)$. By definition of $\omega$ in \eqref{eq:Holder_2} together with $|\overline{x} - \overline{y}| < 1$ in \eqref{eq:Holder_*}, we can choose $Q > \frac{4P +2}{3\omega_0}$ then
	\begin{align*}
		4P + 2\delta + 4Q\omega''(|\overline{x} - \overline{y}|) = 4P+ 2\delta -3Q\omega_0|\overline{x} - \overline{y}|^{-1/2}\hspace{-0.1cm}< 4P \hspace{-0.1cm}+\hspace{-0.1cm} 2\hspace{-0.1cm}-3Q\omega_0 <0.
	\end{align*}
	This yields that at least one eigenvalue of $X_\delta + Y_\delta$ is negative. By the definition of the extremal Pucci operator in \eqref{Pucci_operator}, we have
	\begin{align*}
		\mathcal{P}^+_{\lambda,\Lambda}(X_\delta + Y_\delta) \leq\left(\lambda + \Lambda(n-1)\right)(4P + 2\delta) -3\lambda Q\omega_0|\overline{x} - \overline{y}|^{-1/2}.
	\end{align*}
	We now employ the definition of limiting superjet and limiting subjet together with $\|u\|_\infty \leq 1$:
	\begin{align*}
		&\Phi(\overline{x},|\varkappa_{\overline{x}}|)F(X_\delta) -H(\overline{x},\varkappa_{\overline{x}}) \geq h(\overline{x}) - c(\overline{x})|u(\overline{x})|^{i(\Phi)}u(\overline{x}) \geq -\|h\|_\infty - \|c\|_\infty,\\
		&\Phi(\overline{y},|\varkappa_{\overline{y}}|)F(-Y_\delta) -H(\overline{y},\varkappa_{\overline{y}}) \leq h(\overline{y}) - c(\overline{y})|u(\overline{y})|^{i(\Phi)}u(\overline{y}) \leq \|h\|_\infty + \|c\|_\infty,
	\end{align*}
	and from \eqref{Pucci_ope.}, we obtain
	\begin{align*}
		F(X_\delta) \leq \mathcal{P}^+_{\lambda,\Lambda}(X_\delta + Y_\delta)+ F(-Y_\delta).
	\end{align*}
	Combining the last three displays, we arrive at
	\begin{align}
		\label{eq:Holder_6}
		\begin{split}
			\frac{-\hspace{-0.1cm}\|h\|_\infty \hspace{-0.1cm}-\hspace{-0.1cm} \|c\|_\infty}{\Phi(\overline{x},|\varkappa_{\overline{x}}|)} \hspace{-0.1cm}+\hspace{-0.1cm} \frac{H(\overline{x},\varkappa_{\overline{x}})}{\Phi(\overline{x},|\varkappa_{\overline{x}}|)} \hspace{-0.1cm}\leq \hspace{-0.1cm}C_1 \hspace{-0.1cm}-\hspace{-0.1cm}3\lambda Q\omega_0|\overline{x} - \overline{y}|^{-1/2} \hspace{-0.1cm}+\hspace{-0.1cm} \frac{\|h\|_\infty \hspace{-0.1cm}+\hspace{-0.1cm} \|c\|_\infty}{\Phi(\overline{y},|\varkappa_{\overline{y}}|)} \hspace{-0.1cm}+\hspace{-0.1cm} \frac{H(\overline{y},\varkappa_{\overline{y}})}{\Phi(\overline{y},\hspace{-0.1cm}|\varkappa_{\overline{y}}|)},
		\end{split}
	\end{align}
	where $C_1 := \left(\lambda + \Lambda(n-1)\right)(4P + 2\delta)$. Combining the assumptions {\upshape\hyperref[A2]{(A2)}}, {\upshape\hyperref[A3]{(A3)}} and \eqref{eq:Holder_7}, we deduce that
	\begin{align*}
		&\frac{H(\overline{y},\varkappa_{\overline{y}})}{\Phi(\overline{y},| \varkappa_{\overline{y}}|)} - \frac{H(\overline{x}, \varkappa_{\overline{x}})}{\Phi(\overline{x},| \varkappa_{\overline{x}}|)} \leq \frac{L\mathcal{C}}{\nu_0}\left(| \varkappa_{\overline{y}}|^{m-i(\Phi)}+| \varkappa_{\overline{x}}|^{m-i(\Phi)}\right),\\
		&\frac{\|h\|_\infty + \|c\|_\infty}{\Phi(\overline{y},|\varkappa_{\overline{y}}|)} + \frac{\|h\|_\infty + \|c\|_\infty}{\Phi(\overline{x},|\varkappa_{\overline{x}}|)} \leq \frac{L\left(\|h\|_{\infty} +\|c\|_{\infty}\right)}{\nu_0}\left(|\varkappa_{\overline{x}}|^{-i(\Phi)} +|\varkappa_{\overline{y}}|^{-i(\Phi)}\right).
	\end{align*}
	Substituting the aforementioned inequalities into \eqref{eq:Holder_6}, we get
	\begin{align*}
		-L\left(\|h\|_{\infty} \hspace{-0.1cm}+\hspace{-0.1cm}\|c\|_{\infty}\right)\left(|\varkappa_{\overline{x}}|^{-i(\Phi)}\hspace{-0.1cm} +\hspace{-0.1cm}|\varkappa_{\overline{y}}|^{-i(\Phi)}\right)\hspace{-0.1cm}\leq &\nu_0C_1 \hspace{-0.1cm}+\hspace{-0.1cm}L\mathcal{C}\left(| \varkappa_{\overline{y}}|^{m-i(\Phi)}\hspace{-0.1cm}+\hspace{-0.1cm}| \varkappa_{\overline{x}}|^{m-i(\Phi)}\right)\\
		&-3Q\nu_0\lambda\omega_0|\overline{x} - \overline{y}|^{-1/2},
	\end{align*}
	We now distinguish between two cases based on the sign of $i(\Phi)$.
	\begin{enumerate}[\text{Case }1.]
		\item ($i(\Phi) \geq 0$) From the estimate \eqref{eq:Holder_7}, we know that $|\varkappa_{\overline{x}}|,|\varkappa_{\overline{y}}| \geq 1$ and $|\overline{x} - \overline{y}| \leq 1$ so we conclude that
		\begin{align*}
			-2L\left(\|h\|_{\infty} +\|c\|_{\infty}\right)\leq \nu_0C_1 +Q|\overline{x} - \overline{y}|^{-1/2}\left(4L\mathcal{C}|\overline{x} - \overline{y}|^{1/2}-3\nu_0\lambda\omega_0\right).
		\end{align*}
		We now select $P \geq  \left(\frac{2^{3/2}L\mathcal{C}}{\nu_0\lambda\omega_0}\right)^{4}$ and apply $|\overline{x} - \overline{y}| < \frac{2}{\sqrt{P}}$ to derive
		\begin{align*}
			4L\mathcal{C}|\overline{x} - \overline{y}|^{1/2} \leq 4L\mathcal{C}\left(\frac{2}{\sqrt{P}}\right)^{1/2} \leq 2\nu_0\lambda\omega_0.
		\end{align*}
		It follows that
		\begin{align*}
			\nu_0Q\lambda\omega_0 \leq \nu_0C_1 + 2L\left(\|h\|_{\infty} +\|c\|_{\infty}\right).
		\end{align*}
		As a result, if we choose $Q > \frac{\nu_0C_1 + 2L\left(\|h\|_{\infty} +\|c\|_{\infty}\right)}{\nu_0\lambda\omega_0}$, we reach a contradiction.
		\item ($-1 < i(\Phi) < 0$) Recalling that $|\varkappa_{\overline{x}}|,|\varkappa_{\overline{y}}| \leq 2Q$, we deduce
		\begin{align*}
			-2L\left(\|h\|_{\infty} \hspace{-0.1cm}+\hspace{-0.1cm}\|c\|_{\infty}\right)(2Q)^{-i(\Phi)}\hspace{-0.1cm}\leq \hspace{-0.1cm}\nu_0C_1 \hspace{-0.1cm}+\hspace{-0.1cm}Q|\overline{x} - \overline{y}|^{-1/2}\hspace{-0.1cm}\left(4L\mathcal{C}|\overline{x} - \overline{y}|^{1/2}\hspace{-0.1cm}-\hspace{-0.1cm}3\nu_0\lambda\omega_0\right).
		\end{align*}
		Choosing $P \geq 4 \left(\frac{2^{3/2}L\mathcal{C}}{\nu_0\lambda\omega_0}\right)^{4}$, using \eqref{eq:Holder_*}, we obtain
		\begin{align}
			\label{eq:Holder_pre_9}
			4L\mathcal{C}|\overline{x} - \overline{y}|^{1/2} \leq4L\mathcal{C}\left(\frac{2}{\sqrt{P}}\right)^{1/2} \leq 2\nu_0\lambda\omega_0.
		\end{align}
		By means of \eqref{eq:Holder_pre_9} and $|\overline{x} - \overline{y}|^{-1/2} \geq 1$, we further arrive at
		\begin{align*}
			-4L\left(\|h\|_{\infty} + \|c\|_{\infty}\right)Q^{-i(\Phi)} \leq \nu_0C_1 - Q |\overline{x} - \overline{y}|^{-1/2}\nu_0\lambda\omega_0 \leq \nu_0C_1 - Q\nu_0\lambda\omega_0.
		\end{align*}
		Since $-i(\Phi) < 1$, this inequality does not hold for sufficiently large $Q > 0$.
	\end{enumerate}
	
	\textit{Step 2: Proof of \hyperref[Result2]{\upshape(R2)}.} Suppose $|\xi| = 1$ and $c = 0$.
	Since the proof of this case is quite similar to \hyperref[Result1]{\upshape(R1)}, here we focus on figuring out the differences.  
	\begin{enumerate}[1.] 
		\item  Since $|\xi| = 1$, if we choose $0 \leq b < \theta/6$, we can apply Lemma \ref{boundary-behavior}\hyperref[boundary_behav_2]{\upshape(ii)} to $u$. Hence, we deduce $\psi(x, y) \leq 0$ whenever either $x$ or $y$ lies on the boundary portion $B_{\rho}''$.
		\item We next follow the contradiction argument of the proof of \hyperref[Result1]{\upshape(R1)} and the difference occurs when we employ the definition of limiting superjet $(\varkappa_{\overline{x}},X)$ of $u$ at $\overline{x}$ and subjet $(\varkappa_{\overline{y}},-Y)$ of $u$ at $\overline{y}$, combining $\|u\|_\infty \leq 1$:
		\begin{align*}
			&\Phi(\overline{x},|\xi+b\varkappa_{\overline{x}}|)F(X)-H(\overline{x},\xi+b\varkappa_{\overline{x}}) \geq -\|h\|_\infty,\\
			&\Phi(\overline{y},|\xi+b\varkappa_{\overline{y}}|)F(-Y)-H(\overline{y},\xi+b\varkappa_{\overline{y}}) \leq \|h\|_\infty.
		\end{align*}
		Next, we choose $Q>0$ sufficiently large such that
		\begin{align}\label{9.1}
			3\lambda\omega_0 Q >\left(\lambda + \Lambda(n-1)\right)(4P + 2\delta)+ 2\left[\dfrac{\|h\|_{\infty}+L\mathcal{C}(5/4)^m}{c_\Phi}\right],
		\end{align}
		and $b_0=\min\left\{\dfrac{1}{8Q},\dfrac{\theta}{6}\right\}$. Note that $|\varkappa_{\overline{x}}|,|\varkappa_{\overline{y}}|\le 2Q$, if $0 \leq b < b_0$, then it follows that
		\begin{align}\label{lip1}
			|b\varkappa_{\overline{x}}|,|b\varkappa_{\overline{y}}|\le \dfrac{1}{4}	\text{ and } \dfrac{3}{4}\le |\xi+b\varkappa_{\overline{x}}|,|\xi+b\varkappa_{\overline{y}}|\le \dfrac{5}{4}.
		\end{align}
		As in the proof of \hyperref[Result1]{\upshape(R1)}, we obtain 
		\begin{align*}
			\dfrac{-\|h\|_{\infty} + H(\overline{x},\xi+b\varkappa_{\overline{x}})}{\Phi(\overline{x},|\xi+b\varkappa_{\overline{x}}|)}\le C_1 -3\lambda Q\omega_0 + \dfrac{\|h\|_{\infty}+H(\overline{y},\xi+b\varkappa_{\overline{y}})}{\Phi(\overline{y},|\xi+b\varkappa_{\overline{y}}|)},
		\end{align*}
		where $C_1:=\left(\lambda + \Lambda(n-1)\right)(4P + 2\delta)$. Let
		\begin{align*}
			c_\Phi:=\displaystyle\frac{\nu_0}{L}\min_{t\in[3/4,5/4]}\min\{t^{i(\Phi)},t^{s(\Phi)}\}>0.	
		\end{align*} 
		Combining this estimate, the assumptions {\upshape\hyperref[A2]{(A2)}}, {\upshape\hyperref[A3]{(A3)}}, and $0 < m \leq i(\Phi)+1$ yields that
		\begin{align*}
			\dfrac{-\|h\|_{\infty}-L\mathcal{C}|\xi+b\varkappa_{\overline{x}}|^m}{c_\Phi}\le C_1 -3\lambda Q\omega_0 +\dfrac{\|h\|_{\infty}+L\mathcal{C}|\xi+b\varkappa_{\overline{y}}|^m}{c_\Phi}.
		\end{align*}
		Combining this inequality with \eqref{lip1}, we imply that
		\begin{align*}
			3\lambda\omega_0 Q \leq C_1 + 2\left[\dfrac{\|h\|_{\infty}+L\mathcal{C}(5/4)^m}{c_\Phi}\right],
		\end{align*}
		which is a contradiction to \eqref{9.1}. Our proof is complete.\qedhere
	\end{enumerate}
\end{proof}
\subsection{\bf Global $C^{1,\gamma}$-regularity}
In this subsection, we suppose that $\Omega$ is a bounded $C^2$-domain. Additionally, in the present work, the global $C^{1,\gamma}$-regularity theory is established for equations without absorption terms. This restriction is due to the fact that the affine normalization used in the iteration scheme does not preserve the structure of the zero-order term $c(x)|u|^{i(\Phi)}u$.\\
The subsequent lemma is a direct consequence of \hyperref[Result1]{\upshape(R1)} and \hyperref[Result2]{\upshape(R2)} of Theorem \ref{Lipschitz_estimate} combined with the "cutting lemma". Its proof can be established by adopting the arguments presented in \cite[Theorem 4.1]{BBLL2024} and \cite[Lemma 3.2]{FDZ2021}.
\begin{lemma}[\textbf{Stability}]\label{Stability}
	Let $\{g_k\}_{k}$ be a sequence of Lipschitz continuous functions such that $g_k\to g_{\infty}$. Suppose that $\{u_k\}_{k}$ is a sequence of uniformly bounded viscosity solutions of
	\begin{align*}
		\left\{\begin{array}{cll}
			\Phi_k(y,|\xi_k+b_k\nabla u_k|)F_k(D^2u_k)-H_k(y,\xi_k+b_k\nabla u_k)&=f_k(y) &\text{in } B'_1,\\
			u_k(y)&=g_k(y) &\text{on }B''_1,
		\end{array}\right.
	\end{align*}
	where $\{\xi_k\}_{k}\subset  \mathbb{R}^n,\{f_k\}_{k}\subset C(B_1')$, and $\{F_k\}_k \subset C(\mathbb{S}^{n\times n},\mathbb{R})$ is uniformly $(\lambda,\Lambda)$-elliptic, $\{H_k\}_{k}$ is a sequence of functions continuous on $\overline{B_1'}\times\mathbb{R}^n$ and satisfy there exists a constant $m\in (0,i(\Phi)+1]$ such that
	\begin{align*}
		|H_k(y,p)|\le \mathcal{C}_k|p|^m\text{ for all }k\in\mathbb{N},
	\end{align*}  
	with the sequence $\{\mathcal{C}_k\}_{k}$ converging to 0.
	\begin{enumerate}[\upshape (i)]
		\item \phantomsection\label{Stability_1} Suppose further that $\xi_k=0$ and  $b_k=1$ for every $k\in\mathbb{N}$, $f_k\to 0$ (uniformly),  
		and $F_k\to F_{\infty}$. Then one can extract a subsequence from $\{u_k\}_{k}$ which converges to $u_{\infty}$ on $\overline{B_1'}$. Moreover, such a limit $u_{\infty}$ satisfies
		\begin{align*}
			\left\{\begin{array}{cll}
				F_\infty(D^2u_{\infty})&=0 &\text{in } B'_1,\\
				u_{\infty}(y)&=g_{\infty}(y) &\text{on } B''_1.
			\end{array}\right.
		\end{align*}
		\item \phantomsection\label{Stability_2} Suppose further that $|\xi_k|=1$ and  for every $k\in\mathbb{N}$, $b_k\to 0, f_k\to 0$ (uniformly), 
		and $F_k\to F_{\infty}$. Then one can extract a subsequence from $\{u_k\}_k$ which converges to $u_{\infty}$ on $\overline{B_1'}$. Moreover, such a limit $u_{\infty}$ satisfies 
		\begin{align*}
			\left\{\begin{array}{cll}
				F_\infty(D^2u_{\infty})&=0 &\text{in } B'_1,\\
				u_{\infty}(y)&=g_{\infty}(y) &\text{on } B''_1.
			\end{array}\right.
		\end{align*}
	\end{enumerate}
\end{lemma}

We recall a well-known regularity result proved by Milakis-Silvestre in \cite[Proposition 2.2]{MS2006}, namely if $v$ is a viscosity solution of a uniformly $(\lambda,\Lambda)$-elliptic equation
\begin{align*}
	\left\{\begin{array}{cll}
		\mathcal{F}(D^2v)&=0 &\text{in } B_1\cap \{y_n>0\},\\
		v(y)&=g(y) &\text{on } B_1\cap \{y_n=0\},
	\end{array}\right.
\end{align*}
where $g\in C^{1,\alpha_g}(B_1\cap\{y_n=0\})$, then there exists $\alpha_0=\alpha_0(n,\lambda,\Lambda,\alpha_g)\in (0,1)$ such that we have $v\in C^{1,\alpha_0}(\overline{B_{1/2}\cap\{y_n>0\}})$. Now, we are ready to represent the approximate lemma in the degenerate case.
\begin{lemma}[\textbf{Approximation lemma; degenerate case}]\label{approx1}
	Suppose the assumptions {\upshape\hyperref[A1]{(A1)}--\hyperref[A4]{(A4)}} hold true with $i(\Phi)\geq 0$ and $\nu_0=\nu_1=1$. Let $\xi\in\mathbb{R}^n$ and $u\in C(B'_1(x))$ is a viscosity solution of
	\begin{align}\label{eq-C^2.2}
		\left\{\begin{array}{cll}
			\Phi(y,|\xi+\nabla u|)F(D^2u)-H(y,\xi+\nabla u)&=h(y) &\text{in } B'_1(x),\\
			u(y)&=g(y) &\text{on } B''_1(x),
		\end{array}\right.
	\end{align}
	with $\|u\|_{L^{\infty}(B'_1(x))}\le 1$ and $\|g\|_{C^{1,\alpha_g}(B''_1(x))}\le 1$. Then for any $\varpi>0$, there exists a constant $\varepsilon_0=\varepsilon_0(n,\lambda,\Lambda,i(\Phi),s(\Phi),m,L,\varpi)>0$ such that if
	\begin{align*}
		\max\{\|h\|_{L^{\infty}(B'_1(x))},\mathcal{C},\mathcal{C}|\xi|^{(m-i(\Phi))_{+}}\}\le \varepsilon_0,
	\end{align*}
	then one can find $v\in C^{1,\alpha_{0}}(B'_{3/4}(x))$, for some $\alpha_{0}\in(0,1)$, which is the viscosity solution of a uniformly $(\lambda,\Lambda)$-elliptic equation
	\begin{align}\label{eq-C^2.2'}
		\left\{\begin{array}{cll}
			\mathcal{F}(D^2v)&=0 &\text{in } B'_{3/4}(x),\\
			v(y)&=g(y) &\text{on } B''_{3/4}(x),
		\end{array}\right.
	\end{align}
	such that
	\begin{equation}\label{app}
		\|u-v\|_{L^{\infty}(B'_{1/2}(x))}\leq \varpi.
	\end{equation}
\end{lemma} 
\begin{proof}
	We argue by contradiction and suppose that there exist $\varpi_0>0$ and sequences of functions $\{F_k\}_{k=1}^{\infty}$, $\{\Phi_k\}_{k=1}^{\infty},\{H_k\}_{k=1}^{\infty},\{h_k\}_{k=1}^{\infty},\hspace{-0.1cm}\{g_k\}_{k=1}^{\infty},\hspace{-0.1cm}\{u_k\}_{k=1}^{\infty}$ and a sequence of vectors $\{\xi_k\}_{k=1}^{\infty}$ such that
	\begin{enumerate}[(C1)]
		\item \phantomsection\label{(C1)}$F_k \in C(\mathbb S^{n\times n},\mathbb{R})$ is uniformly $(\lambda,\Lambda)$-elliptic;
		\item \phantomsection\label{(C2)} for $\Phi_k\in C(B_1(x)\times [0,\infty),[0,\infty))$, the map $t\mapsto\Phi_k(x,t)/t^{i(\Phi)}$ is almost non-decreasing and the map $t\mapsto\Phi_k(x,t)/t^{s(\Phi)}$ is almost non-increasing with constant $L \geq 1$, and $\Phi_k(y,1)=1$ for all $y\in B'_1(x)$;
		\item \phantomsection\label{(C3)}$H_k\in C(B_1(x)\times\mathbb{R}^n,\mathbb{R})$ satisfies there exists $\mathcal{C}_k>0$ such that 
		\begin{align}\label{11.4}
			\begin{split}
			&|H_k(y,p)|\le \mathcal{C}_k|p|^m\text{ for all }y\in B_1(x),~p\in\mathbb{R}^n,\\ 
			&\max\left\{\mathcal{C}_k,\mathcal{C}_k|\xi_k|^{(m-i(\Phi))_{+}}\right\}\le\dfrac{1}{k};
			\end{split}
		\end{align}
		\item \phantomsection\label{(C4)}$h_k\in C(B'_1(x))$ such that $\|h_k\|_{L^{\infty}(B'_1(x))}\le \dfrac{1}{k}$;
		\item \phantomsection\label{(C5)}$u_k\in C(B'_1(x))$ with $\|u_k\|_{L^{\infty}(B'_1(x))}\le 1$ is the viscosity solution of
		\begin{align}\label{eq-C^2.2.1}
			\left\{\begin{array}{cll}
				\Phi_k(y,|\xi_k+\nabla u_k|)F_k(D^2u_k)-H_k(y,\xi_k+\nabla u_k)&=h_k(y) &\text{in } B'_1(x),\\
				u_k(y)&=g_k(y) &\text{on } B''_1(x),
			\end{array}\right.
		\end{align}
		with $\|g_k\|_{C^{1,\alpha_g}(B''_1(x))}\le 1$, but 
		\begin{align}\label{10.1}
			\|u_k-v\|_{L^{\infty}(B'_{1/2}(x))}> \varpi_0, \mbox{ for any }k \in \mathbb{N},
		\end{align}
		for any $v$ satisfying \eqref{eq-C^2.2'}.
	\end{enumerate}
	Observe furthermore that $F_k$ converges locally uniformly to some $(\lambda,\Lambda)$-elliptic operator $F_{\infty}$ by condition \hyperref[(C1)]{(C1)}, and $g_k$ uniformly converges to $g_{\infty}$ by condition \hyperref[(C5)]{(C5)}. For a further discussion, we consider two cases:
	
	\textbf{Case 1:} $\{\xi_k\}_{k=1}^\infty$ is bounded. Up to a subsequence, $\xi_k$ converges to some vector $\xi_{\infty}$. 
	Let $\overline{u}_k(y):=u_k(y) +\langle y,\xi_k\rangle$ and $\overline{g}_k(y):=g_k(y)+\langle y,\xi_k\rangle$, we have $\overline{u}_k$ is the viscosity solution of
	\begin{align*}
		\left\{\begin{array}{cll}
			\Phi_k(y,|\nabla \overline{u}_k|)F_k(D^2\overline{u}_k)-H_k(y,\nabla\overline{u}_k)&=h_k(y)&\text{in } B'_{1}(x),\\
			\overline{u}_k(y)&=\overline{g}_k(y) &\text{on } B''_{1}(x).
		\end{array}\right.
	\end{align*}
	Applying \hyperref[Result1]{\upshape(R1)} of Theorem \ref{Lipschitz_estimate} for $\overline{u}_k$ and Arzela-Ascoli theorem, we conclude that $u_k\to u_{\infty}$ uniformly in $B'_r(x)$ for any $0<r<1$. Hence, by Lemma \ref{Stability}\hyperref[Stability_1]{\upshape(i)}, $u_{\infty}$ solves 
		\begin{align}\label{eq-C^2.2.2}
			\left\{\begin{array}{cll}
				F_\infty(D^2u_\infty)&=0&\text{in } B'_{3/4}(x),\\
				u_\infty(y)&=g_{\infty}(y) &\text{on } B''_{3/4}(x),
			\end{array}\right.
		\end{align}
		which contradicts \eqref{10.1} by choosing $v={u}_{\infty},~g=g_{\infty},~\mathcal{F}=F_{\infty}$.
		
		\textbf{Case 2:} $\{\xi_k\}_{k=1}^\infty$ is unbounded. Dividing \eqref{eq-C^2.2.1} by $\Phi_k(y,|\xi_k|)$ it becomes
		\begin{align}\label{eq-C^2.2.3}
			\left\{\begin{array}{cll}
				\hat{\Phi}_k(y,|e_k+b_k\nabla u_k|)F_k(D^2u_k)-\hat{H}_k(y,e_k+b_k\nabla u_k)&=\hat{h}_k(y) &\text{in } B'_{1}(x),\\
				u_k(y)&=g_k(y) &\text{on } B''_{1}(x),
			\end{array}\right.
		\end{align}
		where $e_k=\dfrac{\xi_k}{|\xi_k|},b_k=\dfrac{1}{|\xi_k|}$ and $\hat{\Phi}_k,\hat{H}_k,\hat{h}_k$ are defined by
		\begin{align*}
			\hat{\Phi}_k(y,t):=\dfrac{\Phi_k(y,|\xi_k|t)}{\Phi_k(y,|\xi_k|)},~\hat{H}_k(y,p)=\dfrac{H_k(y,|\xi_k|p)}{\Phi_k(y,|\xi_k|)},\mbox{ and }\hat{h}_k(y)=\dfrac{h_k(y)}{\Phi_k(y,|\xi_k|)}.
		\end{align*}
		Note that we have the following estimate for large enough $k$
		\begin{align*}
			&|\hat{H}_k(y,p)|=\dfrac{|H_k(y,|\xi_k|p)|}{\Phi_k(y,|\xi_k|)}\le \dfrac{\mathcal{C}_k|\xi_k|^m|p|^m}{1/L|\xi_k|^{i(\Phi)}}\le L\mathcal{C}_k|\xi_k|^{(m-i(\Phi))_{+}}|p|^m,\\
			&|\hat{h}_k(y)|=\dfrac{|h_k(y)|}{\Phi_k(y,|\xi_k|)}\le L|h_k(y)||\xi_k|^{-i(\Phi)}\le \dfrac{L}{k}|\xi_k|^{-i(\Phi)}.
		\end{align*}
		Combining the above estimates with \eqref{11.4}, it follows that $\mathcal{C}_k|\xi_k|^{(m-i(\Phi))_{+}}$ and $\hat{h}_k$ uniformly converge to 0. By \hyperref[Result2]{\upshape(R2)} of Theorem \ref{Lipschitz_estimate}, $u_k$ is uniformly Lipschitz. Therefore, applying Arzela-Ascoli theorem and Lemma \ref{Stability}\hyperref[Stability_2]{\upshape(ii)}, up to a subsequence, on $B'_r(x)$ for any $0 < r < 1$, $u_k$  converges to $u_{\infty}$ satisfying
		\begin{align}\label{eq-C^2.2.4}
			\left\{\begin{array}{cll}
				F_\infty(D^2u_\infty)&=0 &\text{in }B'_{3/4}(x),\\
				u_\infty(y)&=g_{\infty}(y) &\text{on }B''_{3/4}(x),
			\end{array}\right.
		\end{align}
		which enables us to obtain the same contradiction as in \textbf{Case 1}. 
	\end{proof}
	\begin{lemma}[\bf Pointwise boundary $C^{1,\gamma'}$-estimate; degenerate case]\label{approx2}
		Suppose the assumptions {\upshape\hyperref[A1]{(A1)}--\hyperref[A4]{(A4)}} are in force with $i(\Phi)\geq 0$. Let $\gamma'$ be chosen to satisfy 
		\begin{align}\label{gamma'}
			\gamma'\in (0,\alpha_g)\cap \left(0,\dfrac{1}{1+s(\Phi)}\right].
		\end{align}
		Assume that $u\in C(B'_1(x))$ is a viscosity solution of
		\begin{align}\label{eq-C^2.2.12}
			\left\{\begin{array}{cll}
				\Phi(y,|\nabla u|)F(D^2u)-H(y,\nabla u)&=h(y) &\text{in } B'_1(x),\\
				u(y)&=g(y) &\text{on }B''_1(x),
			\end{array}\right.
		\end{align}
		then there exist $\sigma_0\in (0,1),~\rho\in (0,1/2)$ and $\overline{c}>0$ depending on $\gamma',n,\lambda,\Lambda,\\\|D^2\phi\|_{L^{\infty}(\Omega)},\|g\|_{C^{1,\alpha_g}(\partial\Omega)},i(\Phi)$, $s(\Phi),m$ and $\mathcal{C}$ such that if 
		\begin{align*}
			\|u\|_{L^{\infty}(B'_1(x))}\le 1,\mbox{ and } \max\{\|h\|_{L^{\infty}(B'_1(x))},\mathcal{C}\}\le \sigma_0,
		\end{align*}
		then there exists an affine function $l(y)=a+\langle b, y-x\rangle$ with
		$|a|+|b|\le \overline{c}$ and such that for each $0<r\leq \rho$, we have
		\begin{equation}\label{flagness-2}
			\|u-l\|_{L^{\infty}(B'_{r}(x))}\leq Cr^{1+\gamma'},
		\end{equation}
		for some universal positive constant $C$. 
	\end{lemma}
	\begin{proof} By the smallness regime in Remark \ref{Remark: smallness-regime}, there is no loss of generality to assume that
		\begin{align*}
			\begin{split}
				&\|u\|_{L^{\infty}(B'_1(x))}\le 1,\quad\|g\|_{C^{1,\alpha_g}(B''_1(x))}\le 1, \quad \max\left\{\|h\|_{L^{\infty}(B'_1(x))}, \mathcal{C}\right\} \leq \sigma_0,
			\end{split}
		\end{align*}
		for $\sigma_0 \in (0,1)$ and $\nu_0=\nu_1=1$. As in \cite{BBLL2024,BBLL2024_1,JSRR2023,BPRT2020}, the proof is based on the induction argument: we claim that there exist universal constants $0<\rho\ll1$ and $\overline{c}>1$ and a sequence of affine functions
		\begin{align*}
			l_k(y):=a_k+\langle b_k,y-x\rangle,
		\end{align*}
		where $\{a_k\}_{k=1}^\infty \subset \mathbb{R}$ and $\{b_k\}_{k=1}^\infty \subset \mathbb{R}^n$ satisfy the following properties for every $k\in\mathbb{N}$
		\begin{enumerate}[(E1)]
			\item \phantomsection\label{(E1)}$\|u-l_k\|_{L^{\infty}(B'_{\rho^k}(x))}\leq \rho^{k(1+\gamma')}$;
			\item \phantomsection\label{(E2)}$|a_k-a_{k-1}|\le \overline{c}\rho^{(k-1)(1+\gamma')},\mbox{ and }|b_k-b_{k-1}|\le \overline{c}\rho^{(k-1)\gamma'}$.
		\end{enumerate}
		\textbf{Step 1: Initial step.} Without loss of generality, we may assume $x=0$. Let $v$ be the approximation function defined in Lemma \ref{approx1} for a constant $\varpi >0$ to be determined later. Now, thanks to the pointwise boundary estimate for uniformly elliptic fully nonlinear operators obtained in \cite[Theorem 2.1]{AB2023}, there exists an affine function $l_1$ and a universal constant $\overline{c} >0$ such that 
		\begin{equation}\label{zhengzexing}
			\sup\limits_{y\in B'_\rho}{|v(y)-l_1(y)|}\leq \overline{c}\rho^{1+\alpha_g} \mbox{ for all } \rho\in\left(0,\frac{1}{2}\right],
		\end{equation}
		and
		\begin{equation*}
			|l_1(0)|+|\nabla l_1(0)|\leq \overline{c}.
		\end{equation*}
		Combining \eqref{app} and \eqref{zhengzexing} with the triangle inequality, we obtain
		\begin{align}\label{12.1}
			\sup\limits_{y\in B'_\rho}{|u(y)-l_1(y)|}\leq \varpi+\overline{c}\rho^{1+\alpha_g}.
		\end{align}	
		Now, thanks to \eqref{gamma'}, it is possible to choose $0<\rho\ll 1$ sufficiently small such that
		\begin{align}\label{12.2}
			\overline{c}\rho^{\alpha_g-\gamma'}\leq \frac{1}{2}\text{ and }\rho^{1-\gamma'(1+s(\Phi))}\le 1.
		\end{align}
		Finally, we choose $\varpi:=\dfrac{1}{2}\rho^{1+\gamma'}$ in \eqref{12.1}, and set $a_0=b_0=0,a_1=l_1(0)$, and $b_1=\nabla l_1(0)$, which completes our proof for the initial step.
		
		\textbf{Step 2: Iterative process.} Suppose by induction that, for a value $k\geq 1$, statements \hyperref[(E1)]{(E1)} and \hyperref[(E2)]{(E2)} hold for some $l_k$. We are going to construct $l_{k+1}$ satisfying \hyperref[(E1)]{(E1)} and \hyperref[(E2)]{(E2)}. Let $r_k=\rho^k$ and we define a rescaled function
		\begin{align*}
			u_k(y):=\dfrac{u(r_ky)-l_k(r_ky)}{r_k^{1+\gamma'}}.
		\end{align*}
		Then $u_k$ solves the following problem in the viscosity sense
		\begin{align}\label{eq-C^2.2.2'}
			\left\{\begin{array}{cll}
				\Phi_k(y,|\xi_k+\nabla u_k|)F_k(D^2u_k)-H_k(y,\xi_k+\nabla u_k)&=h_k(y) &\text{in } \mathcal{B}'_k,\\
				u_k(y)&=g_k(y) &\text{on }\mathcal{B}''_k,
			\end{array}\right.
		\end{align}
		where 
		\begin{align*}
			&\Phi_k(y,t):=\dfrac{\Phi(r_ky,r_k^{\gamma'}t)}{\Phi(r_ky,r_k^{\gamma'})},~F_k(M)=r_k^{1-\gamma'}F(r_k^{\gamma'-1}M),\\[5pt]
			&H_k(y,p)=r_k^{1-\gamma'}\dfrac{H(r_ky,r_k^{\gamma'}p)}{\Phi(r_ky,r_k^{\gamma'})},h_k(y)=\dfrac{r_k^{1-\gamma'}h(r_ky)}{\Phi(r_ky,r_k^{\gamma'})},\\[5pt]
			&\mathcal{B}'_k=B_1\cap\{y_n>\phi_k(y')\},\mathcal{B}''_k=B_1\cap\{y_n=\phi_k(y')\},\\[5pt]
			&\phi_k(y')=r_k^{-1}\phi(r_ky'),~g_k(y)=\dfrac{g(r_k y)-l_k(r_ky)}{r_k^{1+\gamma'}},\mbox{ and }\xi_k=r_k^{-\gamma'}b_k.
		\end{align*}
		It can be easily checked that
		\begin{enumerate}[(a)]
			\item $F_k$ satisfies {\upshape\hyperref[A1]{(A1)}} with the same constants $(\lambda,\Lambda)$;
			\item $\Phi_k$ satisfies {\upshape\hyperref[A2]{(A2)}} with the same constants $(i(\Phi),s(\Phi))$ and $\Phi_k(y,1) \equiv 1$.
			\item $|H_k(y,p)|\le \mathcal{C}_k|p|^m$ with $\mathcal{C}_k=Lr_k^{1-\gamma' (1+s(\Phi)-m)}\mathcal{C}\le L\sigma_0$;
			\item $\|h_k\|_{L^{\infty}(\mathcal{B}'_k)}\le Lr_k^{1-\gamma'(1+s(\Phi))}\|h\|_{L^{\infty}(B'_1)}\le L\sigma_0$;
			\item $\|u_k\|_{L^{\infty}(\mathcal{B}'_k)}\le 1$ thanks to the induction hypothesis. Indeed, we have
			\begin{align*}
				\|u_k\|_{L^{\infty}(\mathcal{B}'_k)}=\dfrac{\|u-l_k\|_{L^{\infty}(B'_{r_k})}}{r_k^{1+\gamma'}}\le 1;
			\end{align*}
			\item $\|D^2\phi_k\|_{L^{\infty}(\mathcal{B}'_k)}\le r_k\|D^2\phi\|_{L^{\infty}(B_1')}\le \|D^2\phi\|_{L^{\infty}(B'_1)}$.
		\end{enumerate}
		Additionally, we can estimate $\mathcal{C}_k|\xi_k|^{(m-i(\Phi))_{+}}$. Indeed, we have
		\begin{align}\label{H_k}
			\mathcal{C}_k|\xi_k|^{(m-i(\Phi))_{+}}\le Lr_k^{1-\gamma' (1+s(\Phi)-m)-\gamma'(m-i(\Phi))_{+}}|b_k|^{(m-i(\Phi))_{+}}\mathcal{C}.
		\end{align}
		If $0<m\leq i(\Phi)$,
		using \eqref{gamma'} and \eqref{H_k}, we derive that
		\begin{align*}
			\mathcal{C}_k|\xi_k|^{(m-i(\Phi))_{+}}\le L\sigma_0.
		\end{align*}
		On the other hand, if $i(\Phi)<m\leq 1+i(\Phi)$, then it follows from \eqref{gamma'} and \eqref{H_k} that
		\begin{align}\label{11.1}
			\mathcal{C}_k|\xi_k|^{(m-i(\Phi))_{+}}\le L\sigma_0|b_k|^{(m-i(\Phi))_{+}}.
		\end{align}
		Choose $\rho > 0$ sufficiently small so that $\rho^{\gamma'} \leq \frac{1}{2}$. Thus we can deduce from the induction hypothesis that		
		\begin{align}\label{11.2}
			|b_{k}|\leq |b_{0}|+\sum_{j=1}^{k}\abs{b_{j}-b_{j-1}}\leq C\sum_{j=1}^{k}\rho^{\gamma'(j-1)}\leq \frac{C}{1-\rho^{\gamma'}}\leq 2C.
		\end{align}
		Combining \eqref{11.1} and \eqref{11.2}, we obtain
		\begin{align*}
			\mathcal{C}_k|\xi_k|^{(m-i(\Phi))_{+}}\le L\sigma_0(2C)^{m-i(\Phi)}.
		\end{align*}
		In summary, for $0<m\leq 1+i(\Phi)$, we have
		\begin{align*}
			\mathcal{C}_k|\xi_k|^{(m-i(\Phi))_{+}}\leq L\sigma_0 (2C)^{(m-i(\Phi))_{+}}.
		\end{align*}
		Finally, we estimate the upper bound in $C^{1,\alpha_g}$ for $g_k$. A straightforward calculation reveals that
		\begin{align*}
			|\nabla g_k(y)-\nabla g_k(z)|&=r_k^{-\gamma'}|\nabla g(r_ky)-\nabla g(r_kz)|\\
			&\le r_k^{-\gamma'}\|g\|_{C^{1,\alpha_g}(\partial\Omega)}|r_k(y-z)|^{\alpha_g}\\
			&\le \|g\|_{C^{1,\alpha_g}(\partial\Omega)}|y-z|^{\alpha_g}.
		\end{align*}
		By recalling the fact that $u = g$ on $\partial\Omega$, together with the induction hypothesis, we infer that
		\begin{align*}
			\|g_k\|_{C^{1,\alpha_g}(\partial\Omega)}\le \|g\|_{C^{1,\alpha_g}(\partial\Omega)}\le 1.
		\end{align*}
		Hence, we now choose $\sigma_0>0$ small enough to apply Lemma \ref{approx1} for $u_k$ and then follow the similar argument as in the initial
		step to confirm the existence of an affine function $\overline{l}(y)=\overline{a}+\langle \overline{b},y\rangle$ satisfying
		\begin{align*}
			\|u_k-\overline{l}\|_{L^{\infty}(B_{\rho}\cap \{y_n>\phi_k(y')\})}\le \rho^{1+\gamma'}\text{ and }|\overline{a}|+|\overline{b}|\le \overline{c}.
		\end{align*}
		By rescaling argument, we have
		\begin{align*}
			\|u-l_{k+1}\|_{L^{\infty}(B'_{\rho^{k+1}})}\le \rho^{(k+1)(1+\gamma')},
		\end{align*}
		where $l_{k+1}(y):=l_k(y)+\rho^{k(1+\gamma')}\overline{l}(\rho^{-k}y)$. It is straightforward to check that
		\begin{align*}
			|a_{k+1}-a_{k}|\le \overline{c}\rho^{k(1+\gamma')}, \mbox{ and }|b_{k+1}-b_{k}|\le \overline{c}\rho^{k\gamma'}.
		\end{align*}
		Hence, \hyperref[(E1)]{(E1)} and \hyperref[(E2)]{(E2)} hold for $k + 1$. Finally, it is easy to verify that $\{a_k\}_{k\in\mathbb{N}}\subset\mathbb{R}$ and $\{b_k\}_{k\in\mathbb{N}}\subset\mathbb{R}^n$ are Cauchy sequences, thus they converge. Let us denote
		\begin{align*}
			a=\lim_{k\to\infty} a_k,\mbox{ and }b=\lim_{k\to\infty} b_k.
		\end{align*}
		By enlarging $\overline{c}$, if necessary, we may assume that $|a|+|b|\le \overline{c}$. Additionally, for the affine function $l(y)=a+\langle b, y\rangle,y\in\mathbb{R}^n$, following analogous techniques in \cite[Corollary 4.1]{JSRR2023} enables us to obtain \eqref{flagness-2}. Therefore, the proof is complete.
	\end{proof}
	Finally, we establish the global $C^{1,\gamma}$-regularity of the viscosity solution of \eqref{1.2} without zero-order terms. This result extends the global regularity result in \cite[Theorem 1.2]{BBLL2024} for degenerate/singular fully nonlinear equations without Hamiltonian terms.
	\begin{theorem}[\bf Global $C^{1,\gamma}$-regularity]
		\label{Theo:Global_regularity}
		Suppose the assumptions {\upshape\hyperref[A1]{(A1)}--\hyperref[A4]{(A4)}} hold true.  Let $\gamma$ be chosen to satisfy
		\begin{equation}\label{exponent}
			\gamma\in \left\{
			\begin{array}{lcl}
				(0,\alpha_{0})\cap \left(0,\frac{1}{1+s(\Phi)}\right] \cap(0,\alpha_g)& \text{if} & i(\Phi)\geq 0, \\
				(0,\alpha_{0})\cap \left(0,\frac{1}{1+s(\Phi)-i(\Phi)}\right] \cap (0,\alpha_g)&\text{if}&  -1<i(\Phi)<0.
			\end{array}
			\right.
		\end{equation}
		If $u\in C(\overline{\Omega})$ is a viscosity solution of 
		\begin{align*}
			\left\{\begin{array}{cll}
				\Phi(x,|\nabla u|)F(D^2u)-H(x,\nabla u)&=h(x)&\text{in } \Omega,\\
				u&=g &\text{on }\partial\Omega,
			\end{array}\right.
		\end{align*}
		then $u\in C^{1,\gamma}(\overline{\Omega})$. In particular, there holds
		\begin{enumerate}[\upshape (i)]
			\item If $0<m<i(\Phi)+1$, there exists a constant $C=C(n,\lambda,\Lambda,\hspace{-0.1cm}L,i(\Phi),\hspace{-0.1cm}\gamma,\hspace{-0.1cm}m,\hspace{-0.1cm}\mathcal{C})$ such that
			\begin{align*}
				\|u\|_{C^{1,\gamma}(\overline{\Omega})}\hspace{-0.1cm}\le \hspace{-0.1cm}C\hspace{-0.1cm}\left(1\hspace{-0.1cm}+\|u\|_{L^{\infty}(\Omega)}\hspace{-0.1cm}+\|g\|_{C^{1,\alpha_g}(\partial\Omega)}\hspace{-0.1cm}+\|h/\nu_0\|_{L^{\infty}(\Omega)}^{\frac{1}{i(\Phi)+1}}\hspace{-0.1cm}+\hspace{-0.1cm}\left(\mathcal{C}/\nu_0\right)^{\frac{1}{1+i(\Phi)-m}}\right).
			\end{align*}
			\item If $m=i(\Phi)+1$, there exists a constant $C=C(n,\lambda,\Lambda,L,i(\Phi),\gamma,\mathcal{C})$ such that
			\begin{align*}
				\|u\|_{C^{1,\gamma}(\overline{\Omega})}\le C\left(1+\|u\|_{L^{\infty}(\Omega)}+\|g\|_{C^{1,\alpha_g}(\partial\Omega)}+\|h/\nu_0\|_{L^{\infty}(\Omega)}^{\frac{1}{i(\Phi)+1}}\right).
			\end{align*}
		\end{enumerate}	 
	\end{theorem}
	\begin{proof} 
		The proof is divided into two cases.\\
		\textbf{Case 1:} $i(\Phi)\geq 0$. We note that, by applying Lemma \ref{approx2}, a viscosity solution $u$ can be approximated by an affine function with an error of order $r^{1+\gamma'}$ at boundary points. By following the argument in the proof of \cite[Theorem 1.1]{AB2023}, we can derive the desired global $C^{1,\gamma}$-estimate with $\gamma$ satisfying \eqref{exponent}.\\
		\textbf{Case 2:} $-1<i(\Phi)<0$. The idea of this case is inspired from \cite[Theorem 1.2]{BBLL2024}. Indeed, we claim that Lemma \ref{approx2} still holds for the singular case. Applying \hyperref[Result1]{\upshape(R1)} of Theorem \ref{Lipschitz_estimate}, it follows that 
		\begin{align*}
			\|u\|_{C^{0,1}(B_{3/4}')}\le c,
		\end{align*}
		for a universal constant $c>0$. Thanks to Lemma \ref{transition}, we see that $u$ is a viscosity solution of the following equation
		\begin{align}\label{eq-sing-dege}
			\left\{\begin{array}{cll}
				\tilde{\Phi}(y,|\nabla u|)F(D^2u)-\tilde{H}(y,\nabla u)&=\tilde{h}(y) &\text{in } B'_{3/4},\\
				u(y)&=g(y) &\text{on }B''_{3/4},
			\end{array}\right.
		\end{align}
		where
		\begin{align*}
			\tilde{\Phi}(y,t):=t^{-i(\Phi)}\Phi(y,t), ~\tilde{H}(y,p):=|p|^{-i(\Phi)}H(y,p), \tilde{h}(y):=|\nabla u|^{-i(\Phi)}h(y).
		\end{align*}
		Note that $\tilde{\Phi}$ satisfies the properties that the map $t\mapsto {\tilde{\Phi}(y,t)}$ is almost non-decreasing and the map $t\mapsto \frac{\tilde{\Phi}(y,t)}{t^{s(\Phi)-i(\Phi)}}$ is almost non-increasing with constant $L\geq 1$. Moreover, we have
		\begin{align*}
			|\tilde{H}(y,p)|\le |p|^{-i(\Phi)}|H(y,p)|\le \mathcal{C}|p|^{m-i(\Phi)},
		\end{align*} 
		where $0<m-i(\Phi)\leq 1$. Additionally, we can estimate
		\begin{align*}
			\|\tilde{h}\|_{L^{\infty}\left(\Omega\right)}\leq c^{-i(\Phi)}\|{h}\|_{L^{\infty}\left(\Omega\right)}. 
		\end{align*}
		Thus, one can repeat the argument in the proof of Lemma \ref{approx2} to obtain global $C^{1,\gamma}$-estimate. The proof is completed.
	\end{proof}
	\subsection{\bf Local estimates and comparison principle}
	\label{Subsection: local and comparison}
	Here, we also provide local Lipschitz estimates for solutions of \eqref{eq:main}. This result plays a pivotal role in the proof of the Harnack inequality in Section \ref{section:4}. The techniques used to prove the following results are an adaptation of methods in \cite{BBLL2024_1}.
	\begin{corollary}[\bf Local Lipschitz estimates]\label{local_estimate_(1.1)}
		Let $u$ be a viscosity solution of 
		\begin{align*}
			\Phi(x,|\nabla u|)F(D^2u) - H(x,\nabla u) + c(x)|u|^{i(\Phi)}u = h(x) \mbox{ in }B_1,
		\end{align*}
		under the assumptions {\upshape\hyperref[A1]{(A1)}--\hyperref[A4]{(A4)}} with $\|u\|_{L^\infty(B_1)} \leq 1$ and $\|h\|_{L^\infty(B_1)} \leq \varepsilon_0 < 1$, for some $\varepsilon_0 \in (0,1)$, and $\nu_0 = \nu_1 = 1$.
		Let $B_r \equiv B_r(z_0) \subset B_1$ be any ball. If $-1 < i(\Phi)$ and $0 < m \leq i(\Phi)+1$, then $u$ is Lipschitz continuous in $B_{r/2}$ with the estimate
		\begin{align*}
			[u]_{C^{0,1}(B_{r/2})} \leq C_{\mathrm{zl}},
		\end{align*}
		for some constant $C_{\mathrm{zl}} = C_{\mathrm{zl}}\left(n,\lambda,\Lambda,L,r,m,i(\Phi),\|c\|_{L^\infty(B_1)},\mathcal{C}\right)$.
	\end{corollary}
	\begin{proof}
		The proof follows by a doubling variables argument similar to that used for the boundary Lipschitz estimate (Theorem \ref{Lipschitz_estimate}), with the following simplifications:
		\begin{itemize}
			\item \textbf{Domain and Boundary:} Since we are considering interior estimates in the ball $B_1$ rather than the half-ball $B_1'$, the boundary behavior analysis and the use of the boundary lemma (Lemma \ref{boundary-behavior}) are no longer required.
			\item \textbf{Exclusion of Boundary Maximums:} By choosing the parameter $P \geq 32/r^2$ in the auxiliary function $\psi(x,y)$, we ensure that any maximum point $(\overline{x}, \overline{y})$ stays within the interior of $B_r$ (specifically, $|\overline{x} - x_0|, |\overline{y} - x_0| \leq r/4$). This eliminates the need for boundary data $g$ or the analysis of maximums occurring on the boundary.
		\end{itemize}
		The remaining steps, involving the application of the Crandall-Ishii-Lions Lemma and the treatment of the Pucci extremal operators, are analogous to the proof of Result 1 in Theorem \ref{Lipschitz_estimate}.
	\end{proof}
	\begin{remark}
		More generally, we also provide regularity estimates for solutions of \eqref{eq:main} in the smallness regime (similar to Remark \ref{Remark: smallness-regime} but for interior estimates without boundary data $g$) for any arbitrary vector $\xi \in \mathbb{R}^n$. Specifically:
		\begin{itemize}
			\item When $\xi =0$, the equation reduces to the form considered in Corollary \ref{local_estimate_(1.1)}, and we recover the Lipschitz regularity result established therein.
			\item When $|\xi|$ is large, a Lipschitz estimate is obtained, provided that $i(\Phi) > 0$ and $0 < m \leq i(\Phi)$.
			\item For small values of $|\xi|$, the solution satisfies a H\"older regularity estimate $u \in C^{0,\beta}(B_{r/2})$ for some $\beta \in (0,1)$, under the conditions $i(\Phi) \geq 0$ and $0 < m \leq i(\Phi) + 1$.
		\end{itemize}
		These results ensure that the regularity of the solution is maintained across different regimes of the shift vector $\xi$. For a detailed treatment and the corresponding proofs of these regularity results, the interested reader may refer to \cite{BBLL2024_1}.
	\end{remark} 
	Next, we proceed to establish a second version of the comparison principle, which will be instrumental in establishing the principal eigentheory in Section \ref{section: 5}.\\
	\noindent Before proceeding with the proof, it is necessary to impose the additional structure condition {\upshape\hyperref[A5]{(A5)}}. This condition is required because the operator $\mathcal{G}$, defined in \eqref{mathcal_G}, depends explicitly on the spatial variable $x$ and exhibits degenerate ellipticity, making the establishment of a comparison principle for $\mathcal{G}$ and Perron’s method for the associated Dirichlet problem infeasible through standard arguments. By invoking {\hyperref[A5]{\upshape(A5)$_1$}}, we can effectively control the oscillation of the $x$-dependent components, thereby enabling the doubling-of-variables technique. Additionally, the introduction of this auxiliary assumption \hyperref[A5]{\upshape(A5)$_2$} is essential to overcome the potential blow-up of the Hamiltonian, which would otherwise invalidate doubling-variable penalization.\\
	\noindent\phantomsection\label{A5}\textbf{Assumption A5.} 
	\begin{enumerate}[1.]
		\item For every $C_0\geq1$, there exists a continuous function
		$\omega_{\Phi F}:[0,\infty)\to[0,\infty)$ with
		$\omega_{\Phi F}(0)=0$ such that
		\begin{align*}
			\Phi(x,\theta|x-y|)F(X)-\Phi(y,\theta|x-y|)F(-Y)\le \omega_{\Phi F}(\theta|x-y|^2+|x-y|),
		\end{align*}
		whenever $\theta>0,x,y\in\Omega$ with $x\neq y$, $X,Y\in \mathbb S^{n\times n}$ and
		\begin{align}
			\label{A2-3}
			-C_0 \theta\left(\begin{array}{cc}
				I &0\\
				0 &I
			\end{array}\right)\le \left(\begin{array}{cc}
				X &0\\
				0 &Y
			\end{array}\right) \leq C_0\theta\left(\begin{array}{cc}
				I &-I\\
				-I &I
			\end{array}\right).
		\end{align}
		\item there exists a continuous function $\omega_{H}: [0,\infty) \to [0,\infty)$ with $\omega_{H}(0) = 0$ such that
		\begin{align*}
			H(y,\theta(x-y)) - H(x,\theta(x-y)) \leq \omega_{H}(\theta|x-y|^2+|x-y|),
		\end{align*}
		for every $\theta>0$ and $x,y \in \overline{\Omega}$.
	\end{enumerate}
	\begin{remark}
		To illustrate the assumption {\hyperref[A5]{\upshape(A5)$_1$}}, we present a specific instance involving the pair $(\Phi, F)$. Indeed, we assume that $F$ is degenerate elliptic and that $\Phi$ is independent of $x$-variable, meaning $\Phi(x, \xi) \equiv \Phi(\xi)$. Under these conditions, for any matrices $X, Y$ satisfying the relation \eqref{A2-3}, the inequality $F(X) \leq F(-Y)$ holds. Consequently, we obtain the estimate
		\begin{align*}
			\Phi(x, \theta|x - y|)F(X) \hspace{-0.1cm}- \hspace{-0.1cm}\Phi(y, \theta|x - y|)F(-Y)\hspace{-0.1cm} = \hspace{-0.1cm}\Phi(\theta|x - y|)(F(X) - F(-Y))\hspace{-0.1cm} \leq \hspace{-0.1cm}0.
		\end{align*}
	\end{remark}
	\begin{remark}
		We can list certain functions satisfying the assumption \hyperref[A5]{\upshape(A5)$_2$} as follows
		\begin{enumerate}[1.]
			\item $H(x,p)=|p|^m$ with $m\in (0,i(\Phi)+1]$. 
			\item $H(x,p)=|\langle b(x),p\rangle|^m$ with
			$b\in C^{0,1}(\overline\Omega)$ and
			$0<m\le \min\{1,i(\Phi)+1\}$.
			\item $H(x,p)=\langle b(x),p\rangle|p|^{m-1}$ with
			$b\in C^{0,\beta}(\overline\Omega)$ and $\beta\in [m,1]$
			and $0<m\le \min\{1,i(\Phi)+1\}$.
		\end{enumerate}
	\end{remark}
	We first record an adaptation of the zero-gradient cutting lemma in \cite{BD2004,BD2007}, which will be used to treat contact points with vanishing gradient in the comparison argument.
	\begin{lemma}
		\label{Lem: cutting_zero_grad}
		Let $\Omega\subset \mathbb{R}^n$ be a bounded domain. Suppose the assumptions {\upshape\hyperref[A1]{(A1)}--\hyperref[A4]{(A4)}} are in force and $c,h_1, h_2\in C(\overline{\Omega})$. Assume that $-1 < i(\Phi)$ and fix $q > \max\left\{2, \frac{i(\Phi)+2}{i(\Phi)+1}\right\}$, let $u\in \mathrm{USC}(\overline{\Omega})$ be a viscosity subsolution (resp. $v\in \mathrm{LSC}(\overline{\Omega})$ be a viscosity supersolution) of
		\begin{align*}
			&\Phi(x,|\nabla u|) F(D^2 u) - H(x,\nabla u) + c(x)|u|^{i(\Phi)}u \geq h_1(x),\\
			&\hspace{-1.2cm}\left(\mbox{resp. }\Phi(x,|\nabla v|) F(D^2 v) - H(x,\nabla v) + c(x)|v|^{i(\Phi)}v \leq h_2(x).\right)
		\end{align*}
		Suppose that $\bar{x}$ is some point in $\Omega$ and $C>0$ such that
		\begin{align*}
			u(x) - C|x - \bar{x}|^{q} \leq u(\bar{x}),~~
			\left(\mbox{resp. }v(x) + C|x - \bar{x}|^{q} \geq v(\bar{x}),\right)
		\end{align*}
		where $\bar{x}$ is a strict local maximum (resp. strict local minimum) of the left hand side and $u$ (resp. $v$) is not locally constant around $\bar{x}$. Then,
		\begin{align*}
			c(\bar{x})|u(\bar{x})|^{i(\Phi)}u(\bar{x}) \geq h_1(\bar{x}),
			~~\left(\mbox{resp. }c(\bar{x})|v(\bar{x})|^{i(\Phi)}v(\bar{x}) \leq h_2(\bar{x}).\right)
		\end{align*}
	\end{lemma}
	\begin{proof}
		The continuous supersolution case was proved in \cite[Lemma 2.11]{BD2004}, while the lower semicontinuous version is given in \cite[Lemma 2]{BD2007}. The same argument applies to the present operator. Indeed, along the nonzero-gradient contact sequence constructed therein, the structural assumptions yield
		\begin{align*}
			\Phi(y_n,|p_n|)|F(X_n)|\leq C|y_n-t_n|^{q(1+i(\Phi))-(2+i(\Phi))} \to 0,
		\end{align*}
		and
		\begin{align*}
			|H(y_n,p_n)|\leq C|y_n-t_n|^{m(q-1)}\to0.
		\end{align*}
		Passing to the limit gives the supersolution assertion. The subsolution assertion follows analogously, with the inequalities reversed.
	\end{proof}
	For analogous assumptions in different contexts, the reader is referred to \cite[(i)(A5)]{BBLL2024}, \cite[Condition 2]{BD2004}, \cite[Condition 3.14]{CIL1992}, \cite[Remark 2.2]{JSRR2023}, and \cite[Assumption H5]{BD2007}. Notably, the condition \hyperref[A5]{\upshape(A5)$_2$} covers the class of Hamiltonian terms surveyed in \cite{BD2007}.
	\begin{lemma}[\bf Comparison principle II]
		\label{Comp.Prin}
		Let $\Omega\subset \mathbb{R}^n$ be a bounded domain. Suppose the assumptions {\upshape\hyperref[A1]{(A1)}--\hyperref[A3]{(A3)}} and {\upshape\hyperref[A5]{(A5)}} are in force and \\$c,h_1,h_2 \in C(\overline{\Omega})$. Suppose that $u\in \mathrm{USC}(\overline{\Omega})$ is a bounded viscosity subsolution to
		\begin{align}
			\label{gene_main_com}
			\Phi(x,|\nabla u|) F(D^2 u) - H(x,\nabla u) + c(x)|u|^{i(\Phi)}u = h_1(x) \mbox{ in }\Omega,
		\end{align} 
		and $v\in \mathrm{LSC}(\overline{\Omega})$ is a viscosity supersolution to \eqref{gene_main_com} with $h_1$ replaced by $h_2$. Furthermore, assume that $v\geq u$ on $\partial\Omega$ and one of the following holds :
		\begin{enumerate}[\upshape \quad(a)]
			\item\phantomsection\label{CompaII_1} $c<0$ in $\overline{\Omega}$ and $h_1 \geq h_2$ in $\overline{\Omega}$;
			\item\phantomsection\label{CompaII_2} $c \le 0$ in $\overline{\Omega}$ and $h_1 > h_2$ in $\overline{\Omega}$.
		\end{enumerate}
		Then $v\geq u$ in $\overline{\Omega}$.
	\end{lemma}
	\begin{proof}
		Let us suppose, for the purpose of contradiction, that there exists $M > 0$ such that $M = \displaystyle\max_{\overline{\Omega}}(u - v)$. Let us consider $\varepsilon\in \mathbb{N}$ and $q > \max\left\{2, \frac{i(\Phi)+2}{i(\Phi)+1}\right\}$ along with the following couple function
		\begin{align*}
			\omega_\varepsilon(x,y) = u(x) - v(y) - \frac{\varepsilon}{q}|x-y|^q \mbox{ for } x,y \in \overline{\Omega}.
		\end{align*}
		For each $\varepsilon > 0$ large, we define the maximum of $\omega_\varepsilon$
		\begin{align*}
			M_\varepsilon = \max_{\overline{\Omega} \times \overline{\Omega}}\left(u(x) - v(y) - \frac{\varepsilon}{q}|x-y|^q\right) < \infty.
		\end{align*}
		Note that $M_\varepsilon \geq M$ for all $\varepsilon$. Let $(x_\varepsilon,y_\varepsilon) \in \overline{\Omega}\times\overline{\Omega}$ be a point for $\omega_\varepsilon$ where the maximum is attained. It is then standard to show that \cite[Lemma 3.1]{CIL1992}
		\begin{align}
			\label{proof_1}
			\lim\limits_{\varepsilon \to \infty}\varepsilon|x_\varepsilon - y_\varepsilon|^q = 0, \mbox{ and } \lim\limits_{\varepsilon \to \infty}M_\varepsilon = M.
		\end{align}
		Up to a subsequence, there exists $z_0 \in \Omega$ such that
		\begin{align}
			\label{proof_2}
			\lim\limits_{\varepsilon \to \infty}x_\varepsilon = \lim\limits_{\varepsilon \to \infty}y_\varepsilon = z_0,
		\end{align}
		where $u(z_0) - v(z_0) = M$. Moreover, one observes that
		\begin{align*}
			M > 0 \geq \max_{\partial \Omega}(u - v),
		\end{align*}
		i.e maximizer can not move towards the boundary, hence, $x_\varepsilon, y_\varepsilon \in \Omega_1$ for some interior domain $\Omega_1 \Subset \Omega$ and $\varepsilon > 0$ sufficiently large. 
		
		\textbf{Claim:} For $\varepsilon$ large enough, there exist $x_\varepsilon$ and $y_\varepsilon$ such that $x_\varepsilon,y_\varepsilon$ is a maximum pair for $\omega_\varepsilon$ and $x_\varepsilon \neq y_\varepsilon$. \\
		Indeed, suppose that $x_\varepsilon = y_\varepsilon$. Then one would have
		\begin{align*}
			&\omega_\varepsilon(x_\varepsilon,x_\varepsilon) = u(x_\varepsilon) - v(x_\varepsilon) \geq u(x_\varepsilon) - v(y) - \frac{\varepsilon}{q}|x_\varepsilon-y|^q;\\
			&\omega_\varepsilon(x_\varepsilon,x_\varepsilon) = u(x_\varepsilon) - v(x_\varepsilon) \geq u(x) - v(x_\varepsilon) - \frac{\varepsilon}{q}|x-x_\varepsilon|^q;
		\end{align*}
		and then $x_\varepsilon$ would be a local minimum for 
		\begin{align*}
			\Theta_1 = v(y) + \frac{\varepsilon}{q}|x_\varepsilon-y|^q,
		\end{align*}
		and similarly a local maximum for
		\begin{align*}
			\Theta_2 = u(x) - \frac{\varepsilon}{q}|x_\varepsilon-x|^q.
		\end{align*}
		Firstly, $x_\varepsilon$ cannot be simultaneously a strict local maximum and a strict local minimum. Indeed, suppose that this were the case. If $v$ is locally constant near $x_{\varepsilon}$, the supersolution inequality follows from the constant branch in the definition; otherwise, it follows from Lemma \ref{Lem: cutting_zero_grad}. The corresponding subsolution inequality for $u$ follows analogously.
		\begin{align*}
			c(x_\varepsilon)|u(x_\varepsilon)|^{i(\Phi)}u(x_\varepsilon) \geq h_1(x_\varepsilon),
		\end{align*}
		and 
		\begin{align*}
			c(x_\varepsilon)|v(x_\varepsilon)|^{i(\Phi)}v(x_\varepsilon) \leq h_2(x_\varepsilon).
		\end{align*}
		It yields that
		\begin{align*}
			h_1(x_\varepsilon) - h_2(x_\varepsilon) \leq c(x_\varepsilon)\left[ |u(x_\varepsilon)|^{i(\Phi)}u(x_\varepsilon) - |v(x_\varepsilon)|^{i(\Phi)}v(x_\varepsilon)\right],
		\end{align*}
		which leads to a contradiction to \hyperref[CompaII_1]{\upshape(a)} and \hyperref[CompaII_2]{\upshape(b)}, since $u(x_\varepsilon) > v(x_\varepsilon)$ and the map $t \mapsto |t|^{i(\Phi)}t$ is strictly increasing. Consequently, $x_\varepsilon$ cannot be both a strict minimum of $\Theta_1$ and a strict maximum of $\Theta_2$. In the former case, we can find $\delta > 0$ and $R > \delta$ such that $\overline{B_R(x_\varepsilon)} \subset \Omega$ and
		\begin{align*}
			v(x_\varepsilon) = \displaystyle \inf_{\delta \leq |x - x_\varepsilon| \leq R}\left\{ v(x) + \dfrac{\varepsilon}{q}|x - x_\varepsilon|^q\right\}.
		\end{align*}
		Letting $y_\varepsilon$ be a point where this infimum is attained, we have
		\begin{align*}
			v(x_\varepsilon) = v(y_\varepsilon) + \dfrac{\varepsilon}{q}|x_\varepsilon - y_\varepsilon|^q,
		\end{align*}
		Moreover, $(x_\varepsilon, y_\varepsilon)$ remains a maximum point for $\omega_\varepsilon$ since the following relation holds for all $(x, y) \in \Omega \times \Omega$:
		\begin{align*}
			u(x_\varepsilon) - v(y_\varepsilon) - \dfrac{\varepsilon}{q}|x_\varepsilon - y_\varepsilon|^q = u(x_\varepsilon) - v(x_\varepsilon) \geq u(x) - v(y) - \dfrac{\varepsilon}{q}|x-y|^q.
		\end{align*}
		This completes the proof of the claim. In the other case, one can choose $y_\varepsilon \neq x_\varepsilon$, sufficiently close to $x_\varepsilon$, such that $(y_\varepsilon, x_\varepsilon)$ is also a maximizing pair for $\omega_\varepsilon$.
		
		Set $\theta_\varepsilon: = \varepsilon|x_\varepsilon - y_\varepsilon|^{q-2}$ and $\eta_\varepsilon = \theta_\varepsilon(x_\varepsilon - y_\varepsilon)$. Since $x_\varepsilon\neq y_\varepsilon$, we have $\eta_\varepsilon\neq0$. At this stage, we are able to apply an extension of Ishii’s acclaimed result \cite[Lemma 1]{BD2007} with $\alpha = i(\Phi)$, there exist symmetric matrices $X_\varepsilon, Y_\varepsilon \in \mathbb{S}^{n \times n}$ such that
		\begin{align*}
			(\eta_\varepsilon, X_\varepsilon) \in \overline{J}^{2,+} u(x_\varepsilon),~(\eta_\varepsilon, -Y_\varepsilon) \in \overline{J}^{2,-} v(y_\varepsilon),
		\end{align*}
		and 
		\begin{align*}
			-4C_q\theta_\varepsilon\left(\begin{array}{lc}
				I &0\\
				0 &I
			\end{array}\right)\le \left(\begin{array}{lc}
				X_\varepsilon &0\\
				0 &Y_\varepsilon
			\end{array}\right) \leq 3C_q\theta_\varepsilon\left(\begin{array}{lc}
				I &-I\\
				-I &I
			\end{array}\right).
		\end{align*}
		where $C_q = 2^{q-3}q(q-1)$. By the definition of viscosity subsolution and supersolution for $u$ and $v$, respectively, we have
		\begin{align}
			h_1(x_\varepsilon) &\leq \Phi(x_{\varepsilon},|\eta_{\varepsilon}|)F(X_\varepsilon)- H(x_{\varepsilon},\eta_\varepsilon) + c(x_\varepsilon)|u(x_\varepsilon)|^{i(\Phi)}u(x_{\varepsilon}),\label{sub_com}\\
			h_2(y_\varepsilon) &\geq \Phi(y_{\varepsilon},|\eta_{\varepsilon}|)F(-Y_\varepsilon)- H(y_{\varepsilon},\eta_\varepsilon) + c(y_\varepsilon)|v(y_\varepsilon)|^{i(\Phi)}v(y_{\varepsilon}).\label{super_com}
		\end{align}
		Subtracting \eqref{super_com} from \eqref{sub_com} yields
		\begin{align*}
			h_1(x_\varepsilon) \leq  &h_2(y_\varepsilon) + \left[c(x_\varepsilon)|u(x_\varepsilon)|^{i(\Phi)}u(x_\varepsilon) - c(y_\varepsilon)|v(y_\varepsilon)|^{i(\Phi)}v(y_\varepsilon)\right]\\
			&+ \hspace{-0.1cm}\left[H(y_{\varepsilon},\eta_\varepsilon) - H(x_{\varepsilon},\eta_\varepsilon)\right]+
			\left[\Phi(x_{\varepsilon},|\eta_{\varepsilon}|)F(X_\varepsilon)-\Phi(y_{\varepsilon},|\eta_{\varepsilon}|)F(-Y_\varepsilon)\right].
		\end{align*}
		By using the assumption {\upshape\hyperref[A5]{(A5)}} with $C_0 = 4C_q$, we collect
		\begin{align}\label{15.3''}
			h_1(x_\varepsilon) \leq h_2(y_\varepsilon) &+ \left[c(x_\varepsilon)|u(x_\varepsilon)|^{i(\Phi)}u(x_\varepsilon) - c(y_\varepsilon)|v(y_\varepsilon)|^{i(\Phi)}v(y_\varepsilon)\right]\notag\\ &\hspace{-0.7cm}+\omega_{H}\left(\theta_\varepsilon|x_{\varepsilon}-y_{\varepsilon}|^2+|x_{\varepsilon}-y_{\varepsilon}|\right)+\omega_{\Phi F}\left(\theta_\varepsilon|x_{\varepsilon}-y_{\varepsilon}|^2+|x_{\varepsilon}-y_{\varepsilon}|\right).
		\end{align}
		Letting $\varepsilon \to \infty$ in \eqref{15.3''}, it follows from \eqref{proof_1} and \eqref{proof_2} that
		\begin{align}\label{15.2'}
			\max_{\overline{\Omega}}\left(h_2 - h_1\right) + c(z_0)\left[|u(z_0)|^{i(\Phi)}u(z_0)-|v(z_0)|^{i(\Phi)}v(z_0)\right] \geq 0.
		\end{align} 
		This yields a contradiction to \hyperref[CompaII_1]{\upshape(a)} and \hyperref[CompaII_2]{\upshape(b)} since $u(z_0)>v(z_0)$. The result then follows.
		\qedhere
	\end{proof}
	We now proceed to demonstrate that the existence of viscosity sub/\\supersolutions to \eqref{exist_main}; we refer to \cite[Lemma 2.2]{JSRR2023}, \cite[Lemma 6.4]{BBLL2024} for a similar result.
	\begin{lemma}[\bf Existence of sub/supersolutions]
		\label{Lem:exist_sub/super}
		Suppose assumptions {\upshape\hyperref[A1]{(A1)}--\hyperref[A5]{(A5)}} are in force. Assume further that $c < 0$ in $\overline{\Omega}$. there exists a viscosity subsolution $\underline{u} \in C(\overline{\Omega})$ and a viscosity supersolution $\overline{u} \in C(\overline{\Omega})$ of 
		\begin{align}
			\label{exist_main}
			\Phi(x,|\nabla u|) F(D^2 u) - H(x,\nabla u) + c(x)|u|^{i(\Phi)}u = h(x) \mbox{ in }\Omega,
		\end{align}
		with $\underline{u}=\overline{u} = g$ on $\partial\Omega$. Moreover, there exists a positive constant $d$, depending only on $n, \lambda, \Lambda, i(\Phi), \nu_0, L, r_0,\\ \diam(\Omega)$, $\mathcal{C}, \|c\|_{L^\infty(\Omega)}, \|h\|_{L^\infty(\Omega)}, \|g\|_{L^\infty(\partial\Omega)}$ such that
		\begin{align*}
			-d \leq \underline{u} \leq \overline{u} \leq d.
		\end{align*}
	\end{lemma}
	\begin{proof}
		We first construct a global supersolution to \eqref{exist_main}. Fix an arbitrary point $z_0\in\partial\Omega$, and let $\chi_{\mathrm{ext}}^{z_0}$ be the exterior barrier \hyperref[barrier_eigenvalue]{\upshape(B2)} given by Lemma \ref{Lem:barrier-construction}. Define a function 
		\begin{align*}
			V^+(x):=G+\chi_{\mathrm{ext}}^{z_0}(x).
		\end{align*}
		where $G:=\|g\|_{L^\infty(\partial\Omega)} + 1$. Since $\chi_{\mathrm{ext}}^{z_0}\geq0$, we have $V^+>g$ on $\partial\Omega$. Moreover, $V^+$ differs from $\chi_{\mathrm{ext}}^{z_0}$ only by a positive constant. Since the map $t\mapsto |t|^{i(\Phi)}t$ is increasing and $c<0$, it follows from Lemma \ref{Lem:barrier-construction} that
		\begin{align*}
			&\Phi(x,|\nabla V^+|)F(D^2V^+)-H(x,\nabla V^+) +c(x)(V^+)^{i(\Phi)+1} \\
			&\leq\Phi(x,|\nabla\chi_{\mathrm{ext}}^{z_0}|) F(D^2\chi_{\mathrm{ext}}^{z_0}) -H(x,\nabla\chi_{\mathrm{ext}}^{z_0}) +c(x)(\chi_{\mathrm{ext}}^{z_0})^{i(\Phi)+1} \leq -\|h\|_{L^\infty(\Omega)}.   
		\end{align*}
		Hence, $V^+$ is a global viscosity supersolution of \eqref{exist_main}.
		For a fixed $\delta \in (0,1)$, we define
		\begin{align*}
			v_{z,\delta}(x) := g(z) + \delta + M_\delta \chi^z_{\mathrm{ext}}(x),
		\end{align*}
		where the constant $M_\delta \geq 1$ can be chosen so that $v_{z,\delta} \geq g$ on $\partial \Omega$. This is possible because $K_2 \geq 1$ and $g$ is continuous on $\partial \Omega$. Indeed, $M_\delta$ depends only on the modulus of continuity
		of $g$, and is independent of $z$. Then repeating the computation in Lemma \ref{Lem:barrier-construction} for $v_{z,\delta}$, we collect the same estimation of \eqref{l4} with $\beta = 0$
		\begin{align}
			\label{proof:barr-1}
			\begin{split}
				&\Phi(x,|\nabla v_{z,\delta}(x)|)F(D^2v_{z,\delta}(x))-H(x,\nabla v_{z,\delta}(x))+c(x)|v_{z,\delta}(x)|^{i(\Phi)}v_{z,\delta}(x)\\
				\le &~\widetilde{K_2}^{i(\Phi)+1}(-\alpha)^{i(\Phi)+1}|\rho|^{(\alpha-1)(i(\Phi)+1)}\left[\dfrac{\nu_0}{L|\rho|}[\Lambda(n-1)+\lambda(\alpha-1)]+\mathcal{C}\right]\\
				&+ \|c\|_{L^\infty(\Omega)}\|g\|_{L^\infty(\partial \Omega)}^{i(\Phi)+1},
			\end{split}
		\end{align} 
		where $\widetilde{K_2} = M_\delta K_2$. Since $r_1<|\rho|\leq R$ for every $x\in\Omega$, the right-hand side of \eqref{proof:barr-1} can be estimated uniformly in $\Omega$. Then, choosing first $\alpha<0$ sufficiently negative and subsequently $K_2 \geq 1$ sufficiently large, we obtain
		\begin{align*}
			&\Phi(x,|\nabla v_{z,\delta}(x)|)F(D^2v_{z,\delta}(x))
			-H(x,\nabla v_{z,\delta}(x)) +c(x)|v_{z,\delta}(x)|^{i(\Phi)}v_{z,\delta}(x) \\
			&\le -\|h\|_{L^\infty(\Omega)},
		\end{align*}
		i.e., $v_{z,\delta}$ is a viscosity supersolution of \eqref{exist_main}. 
		Hence, we now set
		\begin{align}
			\label{proof:barr0}
			\tilde{v}_{z,\delta}:=\min\{v_{z,\delta},V^+\}.
		\end{align}
		Since both $v_{z,\delta}$ and $V^+$ are continuous viscosity supersolutions of \eqref{exist_main}, so is $\tilde{v}_{z,\delta}$.\\ 
		We define $W(x) = \displaystyle \inf_{\substack{z\in\partial\Omega\\\delta\in(0,1)}}
		\tilde{v}_{z,\delta}(x)$ and let
		\begin{align*}
			W_\ast(x) = \lim\limits_{r \to 0} ~\inf\left\{W(y):~y\in\overline{\Omega},|x-y| \leq r\right\},
		\end{align*}
		denote the lower semicontinuous (LSC) envelope of $W$ , ensuring that $W_\ast$ is  the largest LSC function less than or equal to $W$. Since $\tilde{v}_{z,\delta} \geq -\|g\|_{L^\infty(\partial\Omega)}$ uniformly in $z$ and $\delta$, the standard stability of viscosity supersolutions under infima implies that $W_\ast(x)$ is a viscosity supersolution of \eqref{exist_main}. We next verify the family $\{\tilde{v}_{z,\delta}:z\in\partial\Omega, \delta\in(0,1)\}$ is locally equi-Lipschitz in $\Omega$. Fix $x_0\in\Omega$ and $r>0$ such that $\overline{B_{2r}(x_0)}\Subset\Omega$, set $d_{\mathcal{K}}:=\operatorname{dist}(\mathcal{K},\partial\Omega)>0$ with $\mathcal{K}=\overline{B_{2r}(x_0)}$. Since $\overline{B_{r_1}(y_z)}\subset\Omega^c$, for every $x\in\mathcal{K}$ and $z\in\partial\Omega$, we obtain $ |x-y_z|\geq r_1+d_{\mathcal{K}}$. By the explicit formula for $\chi_{\mathrm{ext}}^z$
		\begin{align*}
			&\chi_{\mathrm{ext}}^z(x)\geq K_2\left[r_1^\alpha-(r_1+d_{\mathcal{K}})^\alpha\right]=:\mu_{\mathcal{K}}>0,\\ 
			&|\nabla\chi_{\mathrm{ext}}^z(x)|\leq K_2(-\alpha)(r_1+d_{\mathcal{K}})^{\alpha-1}=:L_{\mathcal{K}}^{\chi},
		\end{align*}
		uniformly in $z\in\partial\Omega$. Also, we put $C_{\mathcal{K}}:=\|V^+\|_{L^\infty(\mathcal{K})}+\|g\|_{L^\infty(\partial\Omega)}$. We distinguish the following two cases.\\
		\textbf{Case 1.} $M_\delta$ large. If $M_\delta\mu_{\mathcal{K}}\geq C_{\mathcal{K}}$, then
		\begin{align*}
			v_{z,\delta}\geq -\|g\|_{L^\infty(\partial\Omega)}+M_\delta\mu_{\mathcal{K}}\geq V^+ \text{ in }\mathcal{K},    
		\end{align*}
		and hence $\tilde{v}_{z,\delta}=V^+$ in $\mathcal{K}$. \\
		\textbf{Case 2.} $M_\delta$ not large. If $M_\delta\mu_{\mathcal{K}}<C_{\mathcal{K}}$, then $\mathrm{Lip}_{\mathcal{K}}(v_{z,\delta})\leq \frac{C_{\mathcal{K}}}{\mu_{\mathcal{K}}}L_{\mathcal{K}}^{\chi}$. Since the minimum of two Lipschitz functions is Lipschitz, we obtain
		\begin{align*}
			\mathrm{Lip}_{\mathcal{K}}(\tilde{v}_{z,\delta})\leq \max\left\{\operatorname{Lip}_{\mathcal{K}}(V^+),\ \frac{C_{\mathcal{K}}}{\mu_{\mathcal{K}}}L_{\mathcal{K}}^{\chi}\right\},
		\end{align*}
		uniformly in $z$ and $\delta$. Hence the family $\{\tilde{v}_{z,\delta}:z\in\partial\Omega, \delta\in(0,1)\}$ is locally equi-Lipschitz in $\Omega$. It follows that $W$ is locally Lipschitz in $\Omega$, and therefore $W=W_*\in C(\Omega)$. We define $\overline{u}:=W$.\\
		We next establish the upper boundary estimate. Fix $z \in \partial\Omega$. By the definition of $W$, we obtain $W(x) \leq \tilde{v}_{z,\delta}(x) \leq v_{z,\delta}(x)$ in $\Omega$. Since $\chi_{\mathrm{ext}}^z$ is continuous and $\chi_{\mathrm{ext}}^z(z)=0$, we have $\chi_{\mathrm{ext}}^z(x) \to 0$ as $x \to z$ and 
		\begin{align*}
			\limsup_{\Omega \ni x \to z} W(x)
			\leq g(z) + \delta.
		\end{align*}
		Letting $\delta \downarrow 0$, we obtain $\displaystyle\limsup_{\Omega \ni x \to z} \overline{u}(x) \leq g(z)$. An analogous construction yields a viscosity subsolution $\underline{u}$ satisfying
		\begin{align*}
			\liminf_{\Omega \ni x \to z} \underline{u}(x) \geq g(z)\text{ for every } z \in \partial \Omega.
		\end{align*}
		Since every member of the lower family lies below every member of the upper family on $\partial \Omega$, Lemma \ref{Comp.Prin}, followed by taking the corresponding supremum and infimum, yields $\underline{u} \leq \overline{u}$ in $\Omega$. Consequently, we obtain
		\begin{align*}
			g(z)\leq\liminf_{\Omega \ni x \to z} \underline{u}(x)\leq\liminf_{\Omega \ni x \to z} \overline{u}(x)\leq\limsup_{\Omega \ni x \to z} \overline{u}(x)\leq g(z).
		\end{align*}
		Hence $\underline{u}, \overline{u} \in C(\overline \Omega)$ and $\underline{u} = \overline{u} = g$ on $\partial\Omega$. Finally, since $\overline{u}\leq V^+$ and the analogous construction gives $V^-\leq\underline{u}$, the required uniform bounds follow.
	\end{proof}
	The last main theorem in this section concerns the existence and uniqueness of solutions to the Dirichlet problem. It follows from Perron's method together with the comparison principle. In addition, this following theorem is essential to establish the sequence of monotone iteration functions in Lemma \ref{normalize}.
	\begin{theorem}[\bf Existence and uniqueness of a viscosity solution]\label{Theo: Exist}
		Suppose assumptions {\upshape\hyperref[A1]{(A1)}--\hyperref[A5]{(A5)}} are in force and $c < 0$ in $\overline{\Omega}$. Then there exists a unique viscosity solution $u \in C(\overline{\Omega})$ of the following equation
		\begin{align*}
			\Phi(x,|\nabla u|) F(D^2 u) \hspace{-0.1cm}- \hspace{-0.1cm}H(x,\nabla u) \hspace{-0.1cm}+\hspace{-0.1cm} c(x)|u|^{i(\Phi)}u = h(x) \mbox{ in }\Omega,\text{ and }u=g\text{ on }\partial\Omega.
		\end{align*}
	\end{theorem}
	\begin{proof}
		By Lemma \ref{Lem:exist_sub/super}, there exists a viscosity subsolution
		$\underline{u}\in C(\overline\Omega)$ and a viscosity supersolution
		$\overline{u}\in C(\overline\Omega)$ of the Dirichlet problem such that
		$\underline{u} = \overline{u} = g$ on $\partial\Omega$ and $\underline{u} \leq \overline{u}$. Then, applying Perron's method \cite[Theorem 4.1]{CIL1992} adapted to Definition \ref{Def:Viscosity_solution}, together with the comparison principle II (Lemma \ref{Comp.Prin}), there exists a viscosity solution $u\in C(\overline\Omega)$ satisfying 
		\begin{align*}
			\underline{u}\leq u\leq \overline{u} \text{ in } \overline\Omega,
		\end{align*}
		with the boundary condition $u = g$ on $\partial\Omega$. Uniqueness follows immediately from Lemma \ref{Comp.Prin}.
	\end{proof}
	\section{\bf Harnack inequality}
	\label{section:4}
	This section is devoted to the derivation of a Harnack inequality for a class of singular or degenerate fully nonlinear elliptic equations, whose proof is based on the sliding cusp method. 
	\subsection{\bf Technical tools and auxiliary results}
	Throughout the analysis, we frequently apply rescaling techniques to subsolutions and supersolutions. Consequently, it is essential to examine how the nonlinearities transform under such scaling, particularly to verify that the rescaled operators continue to satisfy the necessary assumptions. We apply scaling arguments similar to those in Section \ref{section:2}, with $\xi=0$. In addition, the following lemma shows that the class of equations under consideration is stable under the rescaling and normalization procedures used throughout the paper.\\
	Specifically, the conclusion follows from the standard change-of-variables argument for viscosity solutions. Indeed, test functions are transferred through the affine map $T(y)=x_0+t_0y$, and the supersolution and subsolution inequalities are preserved because $V>0$, $t_0>0$, and the equation is multiplied by a positive factor. The constant branch is preserved as well, since $\hat{u}$ is locally constant if and only if $u$ is locally constant in the corresponding transformed neighborhood.
	\begin{lemma}[\bf Rescaling solutions]
		\label{Scaling}
		Let $\Omega\subset\mathbb{R}^n$ be an open set. Let $0<t_0\le1$, $\rho_0>0$, and $x_0\in\Omega$ such that $B_{t_0\rho_0}(x_0)\subset\Omega$. Let $u$ be a non-negative viscosity supersolution (respectively, subsolution) of \eqref{eq:main} in $B_{t_0\rho_0}(x_0)$. Consider the linear map $T: B_{\rho_0} \to B_{t_0\rho_0}(x_0)$ defined by $T(y) = x_0+t_0y$. For positive constant $V \leq 1$ (Note that $V = \frac{1}{K}$ of \eqref{small_regime_scaling}), we define
		\begin{align*}
			\hat{u}(y) = Vu(T(y)) =  Vu(x_0+t_0y), \mbox{ for }y\in B_{\rho_0}.
		\end{align*}
		It can be seen that $\hat{u}$ is a viscosity supersolution (respectively, subsolution) of
		\begin{align}
			\label{scaling-form}
			\hat{\Phi}\left(y,|\nabla \hat{u}|\right)\hat{F}(D^2\hat{u}) -\hat{H}(y,\nabla \hat{u}) +\hat{c}(y)\hat{u}^{i(\Phi)+1} =\hat{h}(y) \mbox{ in }B_{\rho_0},
		\end{align}
		where
		\begin{align*}
			\begin{array}{ll}
				&\hspace*{-0.5cm}\hat{\Phi}(y,t) =\displaystyle\frac{\Phi\left(x_0+t_0y,\frac{t}{Vt_0}\right)}{\Phi\left(x_0+t_0y,\frac{1}{Vt_0}\right)}, \quad \hat{F}(M) = \displaystyle Vt_0^2F\left(\frac{1}{Vt_0^2}M\right),\\[15pt]
				&\hspace*{-0.5cm}\hat{H}(y,p) = \displaystyle\frac{Vt_0^2}{\Phi\left(x_0+t_0y,\frac{1}{Vt_0}\right)}H\left(x_0+t_0y,\frac{p}{Vt_0}\right),\\[15pt]
				&\hspace*{-0.5cm}\hat{c}(y) = \displaystyle\frac{t_0^2V^{-i(\Phi)}}{\Phi\left(x_0+t_0y,\frac{1}{Vt_0}\right)}c(x_0+t_0y), \hspace*{0.3cm}\hat{h}(y) = \displaystyle\frac{Vt_0^2}{\Phi\left(x_0+t_0y,\frac{1}{Vt_0}\right)}h(x_0+t_0y).\\
				\vspace*{-0.001pt}\textbf{}
			\end{array}
		\end{align*}
		Note that $\hat{F}$ remains a uniformly $(\lambda,\Lambda)-\mbox{elliptic}$ operator, the map $t \mapsto \displaystyle\frac{\hat{\Phi}(y,t)}{t^{i(\Phi)}}$ is almost non-decreasing, and the map $t \mapsto \displaystyle\frac{\hat{\Phi}(y,t)}{t^{s(\Phi)}}$ is almost non-increasing with the same constant $L \geq 1$ as in assumption {\upshape\hyperref[A2]{(A2)}}. In addition, $\hat{\Phi}(y,1) = 1$ for all $y \in B_{\rho_0}$. 
		Combining {\upshape\hyperref[A2]{(A2)}} and {\upshape\hyperref[A3]{(A3)}} with the fact that $Vt_0 \le 1$
		and $-1<i(\Phi) \leq s(\Phi)$, we obtain
		\begin{align}\label{18.2}
			\begin{split}
				\left\{\begin{array}{l}
					\displaystyle|\hat{H}(y,p)| \leq \frac{L\mathcal{C}}{\nu_0}t_0^{2-m+i(\Phi)}|p|^m,\\[10pt]
					\displaystyle\|\hat{c}\|_{L^\infty(B_{\rho_0})} \leq \frac{L}{\nu_0}t_0^{2+i(\Phi)} \|c\|_{L^\infty(B_{t_0\rho_0}(x_0))},\\[10pt]
					\displaystyle\|\hat{h}\|_{L^\infty(B_{\rho_0})} \leq \frac{L}{\nu_0}V^{1+i(\Phi)}t_0^{2+i(\Phi)}\|h\|_{L^\infty(B_{t_0\rho_0}(x_0))}.
				\end{array}
				\right.
			\end{split}
		\end{align} 
		For later use, we set
		\begin{align*}
			H_0:=\max\left\{1,\frac{L}{\nu_0}\right\}.
		\end{align*}
		In particular, if $\|h\|_{L^\infty(B_{t_0\rho_0}(x_0))}\leq 1$ then the scaling estimate \eqref{18.2} yields $\|\hat{h}\|_{L^\infty(B_{\rho_0})}\leq H_0$, for all $0<V\leq1$ and $0<t_0\leq1$. Thus, the rescaled equation remains in the same structural class, with quantitative bounds depending only on the original structural data.
		
		In particular, if $u$ is a viscosity solution of \eqref{eq:main}, then $\hat u$ is a viscosity solution of \eqref{scaling-form}.
	\end{lemma} 
	To this end, we establish a collection of auxiliary results and quantitative bounds for an admissible class defined by assumptions {\upshape\hyperref[A1]{(A1)}--\hyperref[A4]{(A4)}}. Since this class encompasses all the original equations and the localized or rescaled equations arising in the argument, the corresponding auxiliary estimates hold uniformly across the entire class; see Figure \ref{fig:Admissible_class}.
	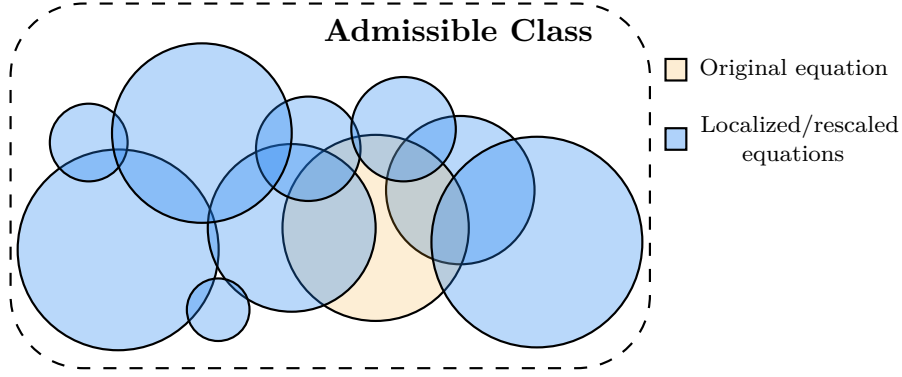
\begin{figure}[H]
		\centering
		\tikzset{every picture/.style={line width=0.75pt}}
		\resizebox{0.92\linewidth}{!}{\begin{tikzpicture}[x=0.75pt,y=0.75pt,yscale=-1,xscale=1]
				\draw  [fill={rgb, 255:red, 255; green, 255; blue, 255 }  ,fill opacity=1 ][dash pattern={on 4.5pt off 4.5pt}] (53.5,49) .. controls (53.5,30.5) and (68.5,15.5) .. (87,15.5) -- (313.34,15.5) .. controls (331.84,15.5) and (346.84,30.5) .. (346.84,49) -- (346.84,149.5) .. controls (346.84,168) and (331.84,183) .. (313.34,183) -- (87,183) .. controls (68.5,183) and (53.5,168) .. (53.5,149.5) -- cycle ;
				\draw  [fill={rgb, 255:red, 32; green, 130; blue, 242 }  ,fill opacity=0.31 ] (225.77,101.2) .. controls (225.77,82.39) and (241.02,67.14) .. (259.84,67.14) .. controls (278.65,67.14) and (293.9,82.39) .. (293.9,101.2) .. controls (293.9,120.02) and (278.65,135.27) .. (259.84,135.27) .. controls (241.02,135.27) and (225.77,120.02) .. (225.77,101.2) -- cycle ;
				\draw  [fill={rgb, 255:red, 245; green, 166; blue, 35 }  ,fill opacity=0.19 ] (178.26,118.56) .. controls (178.26,94.95) and (197.39,75.82) .. (221,75.82) .. controls (244.61,75.82) and (263.74,94.95) .. (263.74,118.56) .. controls (263.74,142.17) and (244.61,161.3) .. (221,161.3) .. controls (197.39,161.3) and (178.26,142.17) .. (178.26,118.56) -- cycle ;
				\draw  [fill={rgb, 255:red, 32; green, 130; blue, 242 }  ,fill opacity=0.31 ] (246.6,125.07) .. controls (246.6,98.35) and (268.26,76.69) .. (294.99,76.69) .. controls (321.71,76.69) and (343.37,98.35) .. (343.37,125.07) .. controls (343.37,151.79) and (321.71,173.45) .. (294.99,173.45) .. controls (268.26,173.45) and (246.6,151.79) .. (246.6,125.07) -- cycle ;
				\draw  [fill={rgb, 255:red, 32; green, 130; blue, 242 }  ,fill opacity=0.31 ] (56.75,128.65) .. controls (56.75,103.19) and (77.4,82.54) .. (102.86,82.54) .. controls (128.32,82.54) and (148.97,103.19) .. (148.97,128.65) .. controls (148.97,154.11) and (128.32,174.76) .. (102.86,174.76) .. controls (77.4,174.76) and (56.75,154.11) .. (56.75,128.65) -- cycle ;
				\draw  [fill={rgb, 255:red, 32; green, 130; blue, 242 }  ,fill opacity=0.31 ] (143.98,118.56) .. controls (143.98,97.29) and (161.22,80.05) .. (182.49,80.05) .. controls (203.76,80.05) and (221,97.29) .. (221,118.56) .. controls (221,139.83) and (203.76,157.07) .. (182.49,157.07) .. controls (161.22,157.07) and (143.98,139.83) .. (143.98,118.56) -- cycle ;
				\draw  [fill={rgb, 255:red, 32; green, 130; blue, 242 }  ,fill opacity=0.31 ] (166.11,82.22) .. controls (166.11,68.98) and (176.84,58.24) .. (190.08,58.24) .. controls (203.32,58.24) and (214.06,68.98) .. (214.06,82.22) .. controls (214.06,95.46) and (203.32,106.19) .. (190.08,106.19) .. controls (176.84,106.19) and (166.11,95.46) .. (166.11,82.22) -- cycle ;
				\draw  [fill={rgb, 255:red, 32; green, 130; blue, 242 }  ,fill opacity=0.31 ] (209.83,73.21) .. controls (209.83,59.97) and (220.56,49.24) .. (233.8,49.24) .. controls (247.04,49.24) and (257.78,59.97) .. (257.78,73.21) .. controls (257.78,86.45) and (247.04,97.19) .. (233.8,97.19) .. controls (220.56,97.19) and (209.83,86.45) .. (209.83,73.21) -- cycle ;
				\draw  [fill={rgb, 255:red, 32; green, 130; blue, 242 }  ,fill opacity=0.31 ] (71.51,79.34) .. controls (71.51,69.49) and (79.5,61.5) .. (89.35,61.5) .. controls (99.21,61.5) and (107.2,69.49) .. (107.2,79.34) .. controls (107.2,89.2) and (99.21,97.19) .. (89.35,97.19) .. controls (79.5,97.19) and (71.51,89.2) .. (71.51,79.34) -- cycle ;
				\draw  [fill={rgb, 255:red, 32; green, 130; blue, 242 }  ,fill opacity=0.31 ] (100.04,75) .. controls (100.04,52.24) and (118.5,33.78) .. (141.26,33.78) .. controls (164.03,33.78) and (182.49,52.24) .. (182.49,75) .. controls (182.49,97.77) and (164.03,116.23) .. (141.26,116.23) .. controls (118.5,116.23) and (100.04,97.77) .. (100.04,75) -- cycle ;
				\draw  [fill={rgb, 255:red, 32; green, 130; blue, 242 }  ,fill opacity=0.31 ] (134.38,156.15) .. controls (134.38,148.21) and (140.81,141.78) .. (148.75,141.78) .. controls (156.69,141.78) and (163.12,148.21) .. (163.12,156.15) .. controls (163.12,164.09) and (156.69,170.52) .. (148.75,170.52) .. controls (140.81,170.52) and (134.38,164.09) .. (134.38,156.15) -- cycle ;
				\draw  [fill={rgb, 255:red, 245; green, 166; blue, 35 }  ,fill opacity=0.2 ] (354,40.94) -- (363.55,40.94) -- (363.55,50.84) -- (354,50.84) -- cycle ; 
				\draw  [fill={rgb, 255:red, 32; green, 130; blue, 242 }  ,fill opacity=0.36 ] (354,72.77) -- (363.55,72.77) -- (363.55,82.67) -- (354,82.67) -- cycle ;
				
				\draw (195.75,20.89) node [anchor=north west][inner sep=0.75pt]   [align=left] {\textbf{Admissible Class}};
				\draw (368.15,39.57) node [anchor=north west][inner sep=0.75pt]  [font=\scriptsize] [align=left] {Original equation};
				\draw (368.55,59.77) node [anchor=north west][inner sep=0.75pt]  [font=\scriptsize] [align=left] {\begin{minipage}[lt]{63.85pt}\setlength\topsep{0pt}
						\begin{center}
							Localized/rescaled \\equations
						\end{center}
				\end{minipage}};
		\end{tikzpicture}}	
		\caption{Schematic illustration of Admissible class}
		\label{fig:Admissible_class}
	\end{figure}
	Next, we show a corollary of Vitali's covering lemma. This measure-theoretic proposition serves as an alternative to the standard Calderón-Zygmund decomposition \cite{CC1995} in deriving the $L^\varepsilon$ estimate presented in Theorem \ref{Theo:L-estimate}. It is often referred to as the growing ink-spots lemma which was first introduced by Krylov–Safonov \cite{KS1979}. As noted in \cite{KS1981}, the suggestive name growing (or crawling) ink-spots was originally attributed to E. M. Landis. For the convenience of the reader, we will only restate here and revisit the complete proofs in \cite[Lemma 2.1]{IS2016} and related references.
	\begin{lemma}[\bf Growing ink-spots lemma]
		\label{lem: Growing-ink-spots}
		Let $E \subset F \subset B_1$ be two open sets. Suppose that there exists $\delta\in (0,1)$ such that the following two assumptions hold
		\begin{itemize}
			\item If any ball $B \subset B_1$ satisfies $|B \cap E| > (1-\delta)|B|$, then $B \subset F$.
			\item $|E| \leq (1-\delta)|B_1|$
		\end{itemize}
		Then $|E| \leq (1-\theta\delta)|F|$ for some constant $\theta$ depending only on the dimension.
	\end{lemma}
	\begin{figure}[h]
		\centering
		\tikzset{every picture/.style={line width=0.75pt}} 
		\resizebox{0.92\linewidth}{!}{\begin{tikzpicture}[x=0.75pt,y=0.75pt,yscale=-1,xscale=1]
				\draw  [color={rgb, 255:red, 0; green, 0; blue, 0 }  ,draw opacity=1 ][fill={rgb, 255:red, 255; green, 255; blue, 255 }  ,fill opacity=1 ][line width=1.5]  (85.97,137.5) .. controls (85.97,78.58) and (134.75,30.82) .. (194.91,30.82) .. controls (255.08,30.82) and (303.85,78.58) .. (303.85,137.5) .. controls (303.85,196.41) and (255.08,244.17) .. (194.91,244.17) .. controls (134.75,244.17) and (85.97,196.41) .. (85.97,137.5) -- cycle ;
				\draw  [color={rgb, 255:red, 0; green, 0; blue, 0 }  ,draw opacity=1 ][fill={rgb, 255:red, 255; green, 255; blue, 255 }  ,fill opacity=1 ][line width=0.75]  (203.24,73.5) .. controls (203.24,61.6) and (213.18,51.96) .. (225.44,51.96) .. controls (237.7,51.96) and (247.64,61.6) .. (247.64,73.5) .. controls (247.64,85.39) and (237.7,95.03) .. (225.44,95.03) .. controls (213.18,95.03) and (203.24,85.39) .. (203.24,73.5) -- cycle ;
				\draw  [color={rgb, 255:red, 0; green, 0; blue, 0 }  ,draw opacity=1 ][fill={rgb, 255:red, 255; green, 255; blue, 255 }  ,fill opacity=1 ][line width=0.75]  (96.24,140.96) .. controls (96.24,128.86) and (106.67,119.05) .. (119.53,119.05) .. controls (132.4,119.05) and (142.83,128.86) .. (142.83,140.96) .. controls (142.83,153.06) and (132.4,162.88) .. (119.53,162.88) .. controls (106.67,162.88) and (96.24,153.06) .. (96.24,140.96) -- cycle ;
				\draw  [color={rgb, 255:red, 0; green, 0; blue, 0 }  ,draw opacity=1 ][fill={rgb, 255:red, 255; green, 255; blue, 255 }  ,fill opacity=1 ][line width=0.75]  (197.74,185.53) .. controls (197.74,166.77) and (213.13,151.55) .. (232.12,151.55) .. controls (251.11,151.55) and (266.5,166.77) .. (266.5,185.53) .. controls (266.5,204.3) and (251.11,219.51) .. (232.12,219.51) .. controls (213.13,219.51) and (197.74,204.3) .. (197.74,185.53) -- cycle ;
				\draw [draw opacity=0][fill={rgb, 255:red, 32; green, 69; blue, 115 }  ,fill opacity=1 ][line width=0.75] [line join = round][line cap = round]   (175.91,146.35) .. controls (177.48,142.82) and (179.3,134.95) .. (185.41,134.49) .. controls (188.75,134.24) and (191.66,136.77) .. (194.91,137.5) .. controls (196.66,137.89) and (199.61,136.28) .. (200.29,137.83) .. controls (202.07,141.91) and (202.78,147.02) .. (200.83,151.03) .. controls (199.63,153.5) and (195.48,153.94) .. (192.58,153.7) .. controls (186.97,153.24) and (181.04,152.06) .. (176.45,149.03) .. controls (174.74,147.9) and (176.57,145.13) .. (176.63,143.18) ;
				\draw [draw opacity=0][fill={rgb, 255:red, 34; green, 76; blue, 125 }  ,fill opacity=1 ][line width=0.75] [line join = round][line cap = round]   (133.72,176.05) .. controls (130.6,173.14) and (127.4,182.61) .. (133.18,181.56) .. controls (145.64,179.3) and (133.45,177.22) .. (134.08,176.05) ;
				\draw [draw opacity=0][fill={rgb, 255:red, 46; green, 84; blue, 127 }  ,fill opacity=1 ][line width=0.75] [line join = round][line cap = round]   (143.91,213.17) .. controls (146.75,211.88) and (159.33,204.11) .. (159.86,214.84) .. controls (160.05,218.5) and (152.15,220.54) .. (149.83,220.02) .. controls (142.58,218.38) and (144.09,218.07) .. (143.91,213.17) -- cycle ;
				\draw [draw opacity=0][fill={rgb, 255:red, 36; green, 76; blue, 122 }  ,fill opacity=1 ][line width=0.75] [line join = round][line cap = round]   (266.34,126.85) .. controls (268.27,119.9) and (269.1,121.24) .. (274.58,123) .. controls (276.82,123.73) and (278.38,118.34) .. (280.5,119.33) .. controls (285.16,121.5) and (286.5,128.95) .. (285.16,133.03) .. controls (284,136.54) and (280.11,134.64) .. (277.81,135.37) .. controls (275.83,135.99) and (274.43,138.95) .. (272.43,138.37) .. controls (269.32,137.48) and (267.14,134.54) .. (265.62,131.86) .. controls (264.66,130.17) and (268.09,127.83) .. (266.34,126.85) -- cycle ;
				\draw [draw opacity=0][fill={rgb, 255:red, 41; green, 79; blue, 122 }  ,fill opacity=1 ][line width=0.75] [line join = round][line cap = round]   (180.95,98.98) .. controls (186.05,91.25) and (199.32,91.11) .. (195.11,101.32) .. controls (193.61,104.93) and (181.34,108.62) .. (180.95,98.98) -- cycle ;
				\draw [draw opacity=0][fill={rgb, 255:red, 43; green, 87; blue, 139 }  ,fill opacity=1 ][line width=0.75] [line join = round][line cap = round]   (113.02,95.88) .. controls (109.63,93.88) and (117.88,90.5) .. (120.02,91.2) .. controls (120.96,91.14) and (121.81,88.97) .. (122.35,89.7) .. controls (125.54,94) and (122.35,98.1) .. (118.94,100.72) .. controls (117.78,101.61) and (115.88,99.76) .. (114.46,100.22) .. controls (113.54,100.52) and (112.81,103.08) .. (112.49,102.23) .. controls (111.57,99.8) and (112.06,101.21) .. (113.02,95.88) -- cycle ;
				\draw [line width=1.5]    (302.5,125) -- (395.5,125) ;
				\draw [shift={(398.5,125)}, rotate = 180] [color={rgb, 255:red, 0; green, 0; blue, 0 }  ][line width=1.5]    (14.21,-4.28) .. controls (9.04,-1.82) and (4.3,-0.39) .. (0,0) .. controls (4.3,0.39) and (9.04,1.82) .. (14.21,4.28);
				\draw  [fill={rgb, 255:red, 44; green, 81; blue, 123 }  ,fill opacity=1 ] (219,258.11) -- (230,258.11) -- (230,269.51) -- (219,269.51) -- cycle ;
				\draw  [fill={rgb, 255:red, 32; green, 130; blue, 242 }  ,fill opacity=0.36 ] (338,258.11) -- (349,258.11) -- (349,269.51) -- (338,269.51) -- cycle ;
				\draw [draw opacity=0][fill={rgb, 255:red, 38; green, 86; blue, 140 }  ,fill opacity=1 ][line width=0.75] [line join = round][line cap = round]   (153.18,50.64) .. controls (152.01,52.54) and (148.5,56.49) .. (156.22,56.49) .. controls (161.55,56.49) and (157.29,50.26) .. (153.18,50.64) -- cycle ;
				\draw [draw opacity=0][fill={rgb, 255:red, 36; green, 76; blue, 122 }  ,fill opacity=1 ][line width=0.75] [line join = round][line cap = round]   (170.86,79.31) .. controls (169.07,74.05) and (184.34,71.47) .. (186.46,72.46) .. controls (191.12,74.63) and (192.47,82.08) .. (191.12,86.16) .. controls (189.96,89.67) and (186.07,87.77) .. (183.77,88.49) .. controls (181.79,89.12) and (180.39,92.08) .. (178.39,91.5) .. controls (175.28,90.61) and (173.1,87.67) .. (171.58,84.99) .. controls (170.63,83.3) and (170.5,76.63) .. (170.86,79.31) -- cycle ;
				\draw [draw opacity=0][fill={rgb, 255:red, 34; green, 76; blue, 125 }  ,fill opacity=1 ][line width=0.75] [line join = round][line cap = round]   (226.72,124.05) .. controls (223.6,121.14) and (220.4,130.61) .. (226.18,129.56) .. controls (238.64,127.3) and (229.27,124.88) .. (226.72,124.05) -- cycle ;
				\draw [draw opacity=0][fill={rgb, 255:red, 32; green, 73; blue, 118 }  ,fill opacity=1 ][line width=0.75] [line join = round][line cap = round]   (105.33,139.84) .. controls (102.08,134.96) and (108.03,131.6) .. (111.38,129.2) .. controls (111.81,128.9) and (111.78,130.45) .. (112.3,130.48) .. controls (113.99,130.59) and (115.13,128) .. (116.15,127.37) .. controls (117.37,126.6) and (118.27,129.39) .. (118.54,129.38) .. controls (125.57,129.17) and (126.16,127.03) .. (132.47,132.14) .. controls (132.86,132.45) and (133.59,132.14) .. (133.94,132.5) .. controls (135.19,133.8) and (135.76,135.66) .. (137.06,136.9) .. controls (137.34,137.17) and (138.22,136.15) .. (138.16,136.54) .. controls (137.34,141.85) and (142,150) .. (131.19,154.14) .. controls (130.68,154.52) and (130.87,152.3) .. (130.27,152.49) .. controls (127.14,153.49) and (124.8,156.19) .. (121.84,157.63) .. controls (117.63,159.66) and (112.7,153.31) .. (110.47,151.57) .. controls (107.57,149.32) and (95.64,151.97) .. (105.88,140.02) ;
				\draw [draw opacity=0][fill={rgb, 255:red, 37; green, 77; blue, 123 }  ,fill opacity=1 ][line width=0.75] [line join = round][line cap = round]   (210.04,74.55) .. controls (209.66,62.09) and (221.16,54.31) .. (233.7,59.52) .. controls (238.81,61.64) and (236.51,66.11) .. (239.2,69.78) .. controls (241.13,72.41) and (245.09,69.95) .. (244.15,74.74) .. controls (240.42,93.82) and (207.41,95.36) .. (210.04,73.64) ;
				\draw [draw opacity=0][fill={rgb, 255:red, 33; green, 71; blue, 116 }  ,fill opacity=1 ][line width=0.75] [line join = round][line cap = round]   (206.32,173.13) .. controls (207.76,163.77) and (209,169.78) .. (216.22,161.03) .. controls (220.93,155.32) and (233.21,163.02) .. (235.4,162.58) .. controls (236.98,162.26) and (237.7,160.26) .. (239.11,159.49) .. controls (244.72,156.4) and (254.95,165.9) .. (257.37,169.7) .. controls (264.83,181.43) and (256.61,187.18) .. (259.53,196.92) .. controls (259.6,197.15) and (259.45,197.54) .. (259.22,197.54) .. controls (256.32,197.54) and (253.46,201.02) .. (250.56,203.11) .. controls (247.19,205.55) and (243.78,208.36) .. (239.73,209.3) .. controls (239.09,209.45) and (238.79,207.81) .. (238.18,208.06) .. controls (236.63,208.7) and (236.15,211.51) .. (234.47,211.47) .. controls (231.8,211.39) and (229.98,208.29) .. (227.35,207.75) .. controls (224.71,207.21) and (221.75,211.88) .. (218.69,208.68) .. controls (214.84,204.67) and (213.59,196.61) .. (210.34,193.21) .. controls (208.51,191.3) and (205.93,190.21) .. (204.15,188.26) .. controls (200.71,184.49) and (211.39,168.13) .. (206.32,173.13) -- cycle ;
				\draw [draw opacity=0][fill={rgb, 255:red, 38; green, 86; blue, 140 }  ,fill opacity=1 ][line width=0.75] [line join = round][line cap = round]   (267.18,89.64) .. controls (266.01,91.54) and (262.5,95.49) .. (270.22,95.49) .. controls (275.55,95.49) and (271.29,89.26) .. (267.18,89.64) -- cycle ;
				\draw [draw opacity=0][fill={rgb, 255:red, 33; green, 72; blue, 116 }  ,fill opacity=1 ][line width=0.75] [line join = round][line cap = round]   (224.76,138.13) .. controls (225.73,132.64) and (225.57,129.71) .. (237.98,132.35) .. controls (240.14,132.81) and (239.88,136.87) .. (239.01,138.9) .. controls (235.88,146.19) and (221.52,146.43) .. (224.76,138.13) -- cycle ;
				\draw [draw opacity=0][fill={rgb, 255:red, 42; green, 86; blue, 137 }  ,fill opacity=1 ][line width=0.75] [line join = round][line cap = round]   (135.19,93.35) .. controls (131.02,93.35) and (129.84,100.18) .. (130.71,103.15) .. controls (130.95,103.96) and (132.55,102.92) .. (133.23,103.43) .. controls (133.72,103.8) and (136.81,107.94) .. (137.99,107.35) .. controls (148.28,102.2) and (139.1,91.39) .. (135.19,93.35) -- cycle ;
				\draw  [color={rgb, 255:red, 0; green, 0; blue, 0 }  ,draw opacity=1 ][fill={rgb, 255:red, 255; green, 255; blue, 255 }  ,fill opacity=1 ][line width=1.5]  (398.97,137.5) .. controls (398.97,78.58) and (447.75,30.82) .. (507.91,30.82) .. controls (568.08,30.82) and (616.85,78.58) .. (616.85,137.5) .. controls (616.85,196.41) and (568.08,244.17) .. (507.91,244.17) .. controls (447.75,244.17) and (398.97,196.41) .. (398.97,137.5) -- cycle ;
				\draw  [color={rgb, 255:red, 0; green, 0; blue, 0 }  ,draw opacity=1 ][fill={rgb, 255:red, 255; green, 255; blue, 255 }  ,fill opacity=1 ][line width=0.75]  (516.24,73.5) .. controls (516.24,61.6) and (526.18,51.96) .. (538.44,51.96) .. controls (550.7,51.96) and (560.64,61.6) .. (560.64,73.5) .. controls (560.64,85.39) and (550.7,95.03) .. (538.44,95.03) .. controls (526.18,95.03) and (516.24,85.39) .. (516.24,73.5) -- cycle ;
				\draw  [color={rgb, 255:red, 0; green, 0; blue, 0 }  ,draw opacity=1 ][fill={rgb, 255:red, 255; green, 255; blue, 255 }  ,fill opacity=1 ][line width=0.75]  (409.24,140.96) .. controls (409.24,128.86) and (419.67,119.05) .. (432.53,119.05) .. controls (445.4,119.05) and (455.83,128.86) .. (455.83,140.96) .. controls (455.83,153.06) and (445.4,162.88) .. (432.53,162.88) .. controls (419.67,162.88) and (409.24,153.06) .. (409.24,140.96) -- cycle ;
				\draw  [color={rgb, 255:red, 0; green, 0; blue, 0 }  ,draw opacity=1 ][fill={rgb, 255:red, 255; green, 255; blue, 255 }  ,fill opacity=1 ][line width=0.75]  (510.74,185.53) .. controls (510.74,166.77) and (526.13,151.55) .. (545.12,151.55) .. controls (564.11,151.55) and (579.5,166.77) .. (579.5,185.53) .. controls (579.5,204.3) and (564.11,219.51) .. (545.12,219.51) .. controls (526.13,219.51) and (510.74,204.3) .. (510.74,185.53) -- cycle ;
				\draw [draw opacity=0][fill={rgb, 255:red, 32; green, 69; blue, 115 }  ,fill opacity=1 ][line width=0.75] [line join = round][line cap = round]   (488.91,146.35) .. controls (490.48,142.82) and (492.3,134.95) .. (498.41,134.49) .. controls (501.75,134.24) and (504.66,136.77) .. (507.91,137.5) .. controls (509.66,137.89) and (512.61,136.28) .. (513.29,137.83) .. controls (515.07,141.91) and (515.78,147.02) .. (513.83,151.03) .. controls (512.63,153.5) and (508.48,153.94) .. (505.58,153.7) .. controls (499.97,153.24) and (494.04,152.06) .. (489.45,149.03) .. controls (487.74,147.9) and (489.57,145.13) .. (489.63,143.18) ;
				\draw [draw opacity=0][fill={rgb, 255:red, 34; green, 76; blue, 125 }  ,fill opacity=1 ][line width=0.75] [line join = round][line cap = round]   (446.72,176.05) .. controls (443.6,173.14) and (440.4,182.61) .. (446.18,181.56) .. controls (458.64,179.3) and (446.45,177.22) .. (447.08,176.05) ;
				\draw [draw opacity=0][fill={rgb, 255:red, 46; green, 84; blue, 127 }  ,fill opacity=1 ][line width=0.75] [line join = round][line cap = round]   (456.91,213.17) .. controls (459.75,211.88) and (472.33,204.11) .. (472.86,214.84) .. controls (473.05,218.5) and (465.15,220.54) .. (462.83,220.02) .. controls (455.58,218.38) and (457.09,218.07) .. (456.91,213.17) -- cycle ;
				\draw [draw opacity=0][fill={rgb, 255:red, 36; green, 76; blue, 122 }  ,fill opacity=1 ][line width=0.75] [line join = round][line cap = round]   (579.34,126.85) .. controls (581.27,119.9) and (582.1,121.24) .. (587.58,123) .. controls (589.82,123.73) and (591.38,118.34) .. (593.5,119.33) .. controls (598.16,121.5) and (599.5,128.95) .. (598.16,133.03) .. controls (597,136.54) and (593.11,134.64) .. (590.81,135.37) .. controls (588.83,135.99) and (587.43,138.95) .. (585.43,138.37) .. controls (582.32,137.48) and (580.14,134.54) .. (578.62,131.86) .. controls (577.66,130.17) and (581.09,127.83) .. (579.34,126.85) -- cycle ;
				\draw [draw opacity=0][fill={rgb, 255:red, 41; green, 79; blue, 122 }  ,fill opacity=1 ][line width=0.75] [line join = round][line cap = round]   (493.95,98.98) .. controls (499.05,91.25) and (512.32,91.11) .. (508.11,101.32) .. controls (506.61,104.93) and (494.34,108.62) .. (493.95,98.98) -- cycle ;
				\draw [draw opacity=0][fill={rgb, 255:red, 43; green, 87; blue, 139 }  ,fill opacity=1 ][line width=0.75] [line join = round][line cap = round]   (426.02,95.88) .. controls (422.63,93.88) and (430.88,90.5) .. (433.02,91.2) .. controls (433.96,91.14) and (434.81,88.97) .. (435.35,89.7) .. controls (438.54,94) and (435.35,98.1) .. (431.94,100.72) .. controls (430.78,101.61) and (428.88,99.76) .. (427.46,100.22) .. controls (426.54,100.52) and (425.81,103.08) .. (425.49,102.23) .. controls (424.57,99.8) and (425.06,101.21) .. (426.02,95.88) -- cycle ;
				\draw [draw opacity=0][fill={rgb, 255:red, 38; green, 86; blue, 140 }  ,fill opacity=1 ][line width=0.75] [line join = round][line cap = round]   (466.18,50.64) .. controls (465.01,52.54) and (461.5,56.49) .. (469.22,56.49) .. controls (474.55,56.49) and (470.29,50.26) .. (466.18,50.64) -- cycle ;
				\draw [draw opacity=0][fill={rgb, 255:red, 36; green, 76; blue, 122 }  ,fill opacity=1 ][line width=0.75] [line join = round][line cap = round]   (483.86,79.31) .. controls (482.07,74.05) and (497.34,71.47) .. (499.46,72.46) .. controls (504.12,74.63) and (505.47,82.08) .. (504.12,86.16) .. controls (502.96,89.67) and (499.07,87.77) .. (496.77,88.49) .. controls (494.79,89.12) and (493.39,92.08) .. (491.39,91.5) .. controls (488.28,90.61) and (486.1,87.67) .. (484.58,84.99) .. controls (483.63,83.3) and (483.5,76.63) .. (483.86,79.31) -- cycle ;
				\draw [draw opacity=0][fill={rgb, 255:red, 34; green, 76; blue, 125 }  ,fill opacity=1 ][line width=0.75] [line join = round][line cap = round]   (539.72,124.05) .. controls (536.6,121.14) and (533.4,130.61) .. (539.18,129.56) .. controls (551.64,127.3) and (542.27,124.88) .. (539.72,124.05) -- cycle ;
				\draw [draw opacity=0][fill={rgb, 255:red, 32; green, 73; blue, 118 }  ,fill opacity=1 ][line width=0.75] [line join = round][line cap = round]   (418.33,139.84) .. controls (415.08,134.96) and (421.03,131.6) .. (424.38,129.2) .. controls (424.81,128.9) and (424.78,130.45) .. (425.3,130.48) .. controls (426.99,130.59) and (428.13,128) .. (429.15,127.37) .. controls (430.37,126.6) and (431.27,129.39) .. (431.54,129.38) .. controls (438.57,129.17) and (439.16,127.03) .. (445.47,132.14) .. controls (445.86,132.45) and (446.59,132.14) .. (446.94,132.5) .. controls (448.19,133.8) and (448.76,135.66) .. (450.06,136.9) .. controls (450.34,137.17) and (451.22,136.15) .. (451.16,136.54) .. controls (450.34,141.85) and (455,150) .. (444.19,154.14) .. controls (443.68,154.52) and (443.87,152.3) .. (443.27,152.49) .. controls (440.14,153.49) and (437.8,156.19) .. (434.84,157.63) .. controls (430.63,159.66) and (425.7,153.31) .. (423.47,151.57) .. controls (420.57,149.32) and (408.64,151.97) .. (418.88,140.02) ;
				\draw [draw opacity=0][fill={rgb, 255:red, 37; green, 77; blue, 123 }  ,fill opacity=1 ][line width=0.75] [line join = round][line cap = round]   (523.04,74.55) .. controls (522.66,62.09) and (534.16,54.31) .. (546.7,59.52) .. controls (551.81,61.64) and (549.51,66.11) .. (552.2,69.78) .. controls (554.13,72.41) and (558.09,69.95) .. (557.15,74.74) .. controls (553.42,93.82) and (520.41,95.36) .. (523.04,73.64) ;
				\draw [draw opacity=0][fill={rgb, 255:red, 33; green, 71; blue, 116 }  ,fill opacity=1 ][line width=0.75] [line join = round][line cap = round]   (519.32,173.13) .. controls (520.76,163.77) and (522,169.78) .. (529.22,161.03) .. controls (533.93,155.32) and (546.21,163.02) .. (548.4,162.58) .. controls (549.98,162.26) and (550.7,160.26) .. (552.11,159.49) .. controls (557.72,156.4) and (567.95,165.9) .. (570.37,169.7) .. controls (577.83,181.43) and (569.61,187.18) .. (572.53,196.92) .. controls (572.6,197.15) and (572.45,197.54) .. (572.22,197.54) .. controls (569.32,197.54) and (566.46,201.02) .. (563.56,203.11) .. controls (560.19,205.55) and (556.78,208.36) .. (552.73,209.3) .. controls (552.09,209.45) and (551.79,207.81) .. (551.18,208.06) .. controls (549.63,208.7) and (549.15,211.51) .. (547.47,211.47) .. controls (544.8,211.39) and (542.98,208.29) .. (540.35,207.75) .. controls (537.71,207.21) and (534.75,211.88) .. (531.69,208.68) .. controls (527.84,204.67) and (526.59,196.61) .. (523.34,193.21) .. controls (521.51,191.3) and (518.93,190.21) .. (517.15,188.26) .. controls (513.71,184.49) and (524.39,168.13) .. (519.32,173.13) -- cycle ;
				\draw [draw opacity=0][fill={rgb, 255:red, 38; green, 86; blue, 140 }  ,fill opacity=1 ][line width=0.75] [line join = round][line cap = round]   (580.18,89.64) .. controls (579.01,91.54) and (575.5,95.49) .. (583.22,95.49) .. controls (588.55,95.49) and (584.29,89.26) .. (580.18,89.64) -- cycle ;
				\draw [draw opacity=0][fill={rgb, 255:red, 33; green, 72; blue, 116 }  ,fill opacity=1 ][line width=0.75] [line join = round][line cap = round]   (537.76,138.13) .. controls (538.73,132.64) and (538.57,129.71) .. (550.98,132.35) .. controls (553.14,132.81) and (552.88,136.87) .. (552.01,138.9) .. controls (548.88,146.19) and (534.52,146.43) .. (537.76,138.13) -- cycle ;
				\draw [draw opacity=0][fill={rgb, 255:red, 42; green, 86; blue, 137 }  ,fill opacity=1 ][line width=0.75] [line join = round][line cap = round]   (448.19,93.35) .. controls (444.02,93.35) and (442.84,100.18) .. (443.71,103.15) .. controls (443.95,103.96) and (445.55,102.92) .. (446.23,103.43) .. controls (446.72,103.8) and (449.81,107.94) .. (450.99,107.35) .. controls (461.28,102.2) and (452.1,91.39) .. (448.19,93.35) -- cycle ;
				\draw [draw opacity=0][fill={rgb, 255:red, 32; green, 130; blue, 242 }  ,fill opacity=0.32 ][line width=0.75] [line join = round][line cap = round]   (433.94,82.22) .. controls (427.11,83.7) and (421.23,90.7) .. (418.07,95.82) .. controls (413.98,102.45) and (410.4,114.51) .. (408.5,126.81) .. controls (406.63,138.96) and (406.42,151.34) .. (409.01,158.96) .. controls (413.89,173.34) and (420.72,190.13) .. (440.09,187.13) .. controls (468.46,182.73) and (467.06,176.28) .. (466.32,143.42) .. controls (466.13,135.2) and (465.57,127.15) .. (464.5,119) .. controls (463.98,115.03) and (464.37,106.54) .. (462.5,103) .. controls (458.91,96.2) and (458.5,93) .. (452.5,89) ;
				\draw [draw opacity=0][fill={rgb, 255:red, 32; green, 130; blue, 242 }  ,fill opacity=0.32 ][line width=0.75] [line join = round][line cap = round]   (535.89,40.82) .. controls (528.67,31.8) and (513.78,39.5) .. (506.1,45.03) .. controls (487.06,58.74) and (470.56,75.1) .. (476.31,100.4) .. controls (476.62,101.75) and (478.89,111.81) .. (482.14,114) .. controls (499.03,125.41) and (514.87,111.12) .. (532.33,105.26) .. controls (542.24,101.93) and (560.55,99.81) .. (565.68,88.74) .. controls (570.15,79.09) and (572.8,54.42) .. (560.5,49.24) .. controls (551.9,45.63) and (543.01,42.77) .. (534.27,39.53) ;
				\draw [draw opacity=0][fill={rgb, 255:red, 32; green, 130; blue, 242 }  ,fill opacity=0.3 ][line width=0.75] [line join = round][line cap = round]   (486.93,168.28) .. controls (481.61,159.78) and (477.4,157.37) .. (478.18,148.2) .. controls (478.59,143.46) and (486.66,135.27) .. (488.22,134.28) .. controls (504.45,124.06) and (548.6,106.93) .. (565.93,106.11) .. controls (577.03,105.59) and (594.79,102.16) .. (598.64,112.59) .. controls (604.53,128.59) and (604.98,145.72) .. (594.1,159.54) .. controls (592.81,161.18) and (589.92,161.26) .. (588.92,163.1) .. controls (583.58,172.92) and (585.53,192.62) .. (581.8,201.95) .. controls (571.78,227.01) and (534.74,232.96) .. (513.15,223) .. controls (500.81,217.3) and (501.82,204.61) .. (497.94,194.18) .. controls (494.54,185.07) and (490.38,176.27) .. (486.6,167.31) ;
				
				\draw (202.62,43.51) node [anchor=north west][inner sep=0.75pt]  [font=\footnotesize] [align=left] {$\displaystyle B$};
				\draw (120.94,103.72) node [anchor=north west][inner sep=0.75pt]  [font=\footnotesize] [align=left] {$\displaystyle B$};
				\draw (241.97,137.2) node [anchor=north west][inner sep=0.75pt]  [font=\footnotesize] [align=left] {$\displaystyle B$};
				\draw (234,256.51) node [anchor=north west][inner sep=0.75pt]   [align=left] {E (ink spots)};
				\draw (353,256.51) node [anchor=north west][inner sep=0.75pt]   [align=left] {F (grown open set)};
				\draw (314,81.51) node [anchor=north west][inner sep=0.5pt]  [font=\normalsize] [align=left] {\begin{minipage}[lt]{50.06pt}\setlength\topsep{0pt}
						\begin{center}
							\textbf{Growing }\\\textbf{ink-spots}
						\end{center}
				\end{minipage}};
				\draw (515.62,43.51) node [anchor=north west][inner sep=0.75pt]  [font=\footnotesize] [align=left] {$\displaystyle B$};
				\draw (433.94,103.72) node [anchor=north west][inner sep=0.75pt]  [font=\footnotesize] [align=left] {$\displaystyle B$};
				\draw (554.97,137.2) node [anchor=north west][inner sep=0.75pt]  [font=\footnotesize] [align=left] {$\displaystyle B$};
		\end{tikzpicture}}
		\caption{Schematic illustration of the growing ink-spots lemma}
	\end{figure}
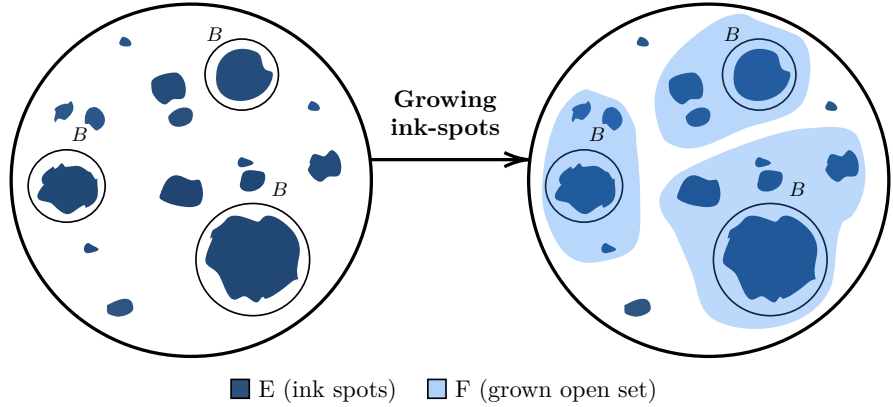
	Next, we turn to Lemma \ref{Measure_est}, which is central to the proof of the $L^\varepsilon$ estimate. To prove this lemma, we first obtain the measure estimate for semiconcave supersolutions under the slightly larger bound $\|h\|_{L^\infty(B_1)}\leq H_0$. We then remove the semiconcavity assumption via inf-convolution.
	\begin{proposition}[\bf A measure estimate for semiconcave supersolutions]
		\label{pro_2:measure-est}
		There exist a small constant $\delta>0$ and a large constant $M_1>0$ such that, if $u\in\mathrm{LSC}(B_1)$ is a non-negative semiconcave viscosity supersolution in $B_1$ of an equation of the form \eqref{eq:main}, whose coefficients satisfy assumptions {\upshape\hyperref[A1]{(A1)}--\hyperref[A4]{(A4)}}, with $c\leq0$ and $\|h\|_{L^\infty(B_1)}\leq H_0$, and
		\begin{align*}
			|\left\{u > M_1\right\} \cap B_1| > (1-\delta)|B_1|,
		\end{align*}
		then we have 
		\begin{align*}
			u > 1 \mbox{ in } B_{1/4}.
		\end{align*}
	\end{proposition}
	\begin{remark}
		Notably, the constant $M_1$ provided by the proposition above is absolute, being independent of $\lambda,\Lambda,\nu_0,L,i(\Phi),s(\Phi),\mathcal{C}$, $\|c\|_{L^\infty(B_1)}$ or dimension $n$. In contrast, the choice of the parameter $\delta$ is sensitive to these factors.
	\end{remark}
	\begin{proof}
		In order to organize the proof, we highlight the main steps in bold.
		
		We assume that there exists a semiconcave function $u$ satisfying
		\begin{align*}
			u \geq 0 \mbox{ in }B_1.
		\end{align*}
		For the sake of contradiction, we suppose further that
		\begin{align}
			\label{pro_2:measure-est-assumption}
			\exists x_0 \in B_{1/4}: u(x_0) \leq 1 \mbox{ and }|\left\{u > M_1\right\} \cap B_1| > (1-\delta)|B_1|,
		\end{align}
		We first consider the semiconcavity assumption, which will be crucial for our subsequent analysis. More precisely, we assume that $u$ satisfies the semiconcavity bound $D^2u \leq C_0I$, where $I$ denotes the identity matrix, in the sense that $u(x) - C_0|x|^2/2$ is a concave function. Consequently, for any $x_0 \in B_1$, there exists a vector $p \in \mathbb{R}^n$ in the superdifferential of $u$ at $x_0$ (with $p = \nabla u(x_0)$ whenever $u$ is differentiable at $x_0$) such that
		\begin{align}
			\label{pro_2:measure-est-eq1}
			u(x) \leq u(x_0) + \langle p, x - x_0\rangle + \frac{C_0}{2}|x - x_0|^2,
		\end{align}
		for all $x \in B_1$. Next, define $w(x):= C_0|x|^2/2 - u(x)$ then $w$ is convex. Therefore, by Aleksandrov's theorem, $w$ is twice differentiable almost everywhere in $B_1$, and hence so is $u$. This means that there exists a set $E \subset B_1$ of measure zero such that at every point $x \in B_1 \setminus E$, the function $u$ is differentiable, we have
		\begin{align*}
			u(y) = {}&u(x) + \langle y - x,\nabla u(x)\rangle\\
			&+ \frac{1}{2}\langle D^2u(x)(y - x), (y - x)\rangle + o(|x - y|^2) \mbox{ as } y \to x.
		\end{align*}
		Moreover, according to \cite{HU1983}, we also have
		\begin{align*}
			\nabla u(y) = \nabla u(x) + D^2u(x)(y - x) + o(|x - y|),
		\end{align*}
		where by $\nabla u(y)$ stands for a vector in the superdifferential of $u$ at $y$.
		
		\textbf{Step 1. Constructing lower barriers for $u$ in the form of cusps.} We consider the open set $U = \{u > M_1\} \cap B_{1/4}$. From assumption \eqref{pro_2:measure-est-assumption}, we have $|U| > |B_{1/4}| - \delta|B_1|$, which is a large measure for $\delta$ small. For instance, we can let $|U| \geq |B_{1/8}|$, which is a constant depending only on $n$. For every $x\in U$, define
		\begin{align*}
			q(x):=\inf_{z\in B_1}\left\{u(z)-\varphi(z-x)\right\},
		\end{align*}
		where $\varphi(z) = -10|z|^{1/2}$. Recalling assumption \eqref{pro_2:measure-est-assumption}, there exists a point $x_0\in B_{1/4}$ such that $u(x_0)\leq 1$. Since $x,x_0\in B_{1/4}$, we have $|x_0-x|\leq 1/2$. Hence, by the definition of $\varphi$, we obtain 
		\begin{align}
			\label{pro_2:measure-est-eq2}
			q(x) \leq u(x_0) - \varphi(x_0-x) \leq 1 - \varphi(x_0-x) \leq  1 + 5\sqrt{2}.
		\end{align}
		Next, we can choose $\rho\in(0,1)$ sufficiently close to $1$ so that $10\sqrt{\rho-\frac14}>1+5\sqrt{2}$. For every $z\in B_1\setminus B_\rho$, using $u\geq0$ in $B_1$ and $x\in B_{1/4}$, we obtain
		\begin{align*}
			u(z)-\varphi(z-x) =u(z)+10|z-x|^{1/2}\geq 10\left(|z|-|x|\right)^{1/2}>1+5\sqrt{2}.
		\end{align*}
		Combining this estimate with \eqref{pro_2:measure-est-eq2}, we see that the infimum defining $q(x)$ cannot be approached in $B_1\setminus B_\rho$. Furthermore, since $u\in\mathrm{LSC}(B_1)$ and $z\mapsto\varphi(z-x)$ is continuous, the function $z\mapsto u(z)-\varphi(z-x)$ is lower semicontinuous on $B_1$. Its restriction to the compact set $\overline{B_\rho}\Subset B_1$ therefore attains its minimum, i.e., $q(x) =\displaystyle\min_{z\in\overline{B_\rho}}\left\{u(z)-\varphi(z-x)\right\}$. Consequently, there exists $y\in\overline{B_\rho}$; since $\overline{B_\rho}\Subset B_1$, one has $y\in B_1$ such that
		\begin{align}
			\label{pro_2:measure-est-eq3}
			u(y) - \varphi(y-x) = \displaystyle\min_{z\in B_1}\left\{u(z) - \varphi(z-x)\right\} = q(x).
		\end{align}
		Equivalently, we obtain 
		\begin{align}
			\label{pro_2:measure-est-eq4}
			u(y) = \varphi(y-x) + q(x), \mbox{ and } u(z) \geq \varphi(z-x) +q(x),\text{ for all } z \in B_1.
		\end{align}
		Moreover, since $\varphi\leq0$, we obtain
		%
		\begin{align*}
			u(y) = \varphi(y-x) + q(x) \leq 1 + 5\sqrt{2}.
		\end{align*}
		Now, we choose $M_1 = 2 + 5\sqrt{2}$, leading to $u(y) \leq M_1 - 1< M_1$. Since $x\in U$ and hence $u(x)>M_1$, it follows that $y \notin U$ and $y \neq x$. Therefore, the function $z\mapsto \varphi(z-x)+q(x)$ is smooth in a neighborhood of $y$ and touches $u$ from below at $y$. In particular, as $z \to y$,
		\begin{align*}
			u(z) \geq u(y) + \langle\nabla\varphi(y-x), z-y\rangle + o(|z-y|).
		\end{align*}
		As $u$ is a semiconcave function, we obtain \eqref{pro_2:measure-est-eq1} at point $x_0=y$. Hence, we have
		\begin{align}\label{21.1}
			\langle \nabla \varphi(y-x), z-y\rangle + o(|z-y|) \leq \langle p,z-y\rangle + \frac{C_0}{2}|z - y|^2.
		\end{align}
		Putting $z = y + th$ and $z = y - th$ in \eqref{21.1}, respectively, and letting $t \to 0$, we deduce
		\begin{align*}
			\nabla u(y) = \nabla \varphi(y-x).
		\end{align*}
		A further analysis of the second derivatives of $u$ at $y$ is postponed until later in the proof.
		
		\textbf{Step 2. Definition of the contact set $\mathcal{F}$}. Let us denote $\mathcal{F}$ the collection of contact points $y\in \overline{B_\rho}\Subset B_1$ corresponding to all $x \in U$. Equivalently, for every $y \in \mathcal{F}$, there exists $x_y \in U$ satisfying \eqref{pro_2:measure-est-eq3}. 
		In view of the aforementioned bound $u(y) \leq M_1-1$ for all $y \in \mathcal{F}$, it follows that
		\begin{align*}
			\mathcal{F} \subset B_1 \cap \{u \leq M_1 -1\}.
		\end{align*}
		The geometric construction in Steps 1 and 2 is illustrated in Figure \ref{fig:sliding_cusp}
		\begin{figure}[H]
			\centering
			\begin{subfigure}{0.39\textwidth}
				\centering
				\tikzset{every picture/.style={line width=0.75pt}}
				\resizebox{0.92\linewidth}{!}{\begin{tikzpicture}[x=0.75pt,y=0.75pt,yscale=-1,xscale=1]
						\draw   (202,145.75) .. controls (202,74.64) and (259.64,17) .. (330.75,17) .. controls (401.86,17) and (459.5,74.64) .. (459.5,145.75) .. controls (459.5,216.86) and (401.86,274.5) .. (330.75,274.5) .. controls (259.64,274.5) and (202,216.86) .. (202,145.75) -- cycle ;
						\draw   (246,115.25) .. controls (246,78.66) and (275.66,49) .. (312.25,49) .. controls (348.84,49) and (378.5,78.66) .. (378.5,115.25) .. controls (378.5,151.84) and (348.84,181.5) .. (312.25,181.5) .. controls (275.66,181.5) and (246,151.84) .. (246,115.25) -- cycle ;
						\draw [color={rgb, 255:red, 32; green, 130; blue, 242 }  ,draw opacity=1 ][fill={rgb, 255:red, 32; green, 130; blue, 242 }  ,fill opacity=0.34 ][line width=0.75] [line join = round][line cap = round]   (265.41,84.02) .. controls (267.5,79) and (271.23,74.15) .. (276.34,70.7) .. controls (283.36,65.95) and (292.59,63.79) .. (300.93,63.18) .. controls (305.66,62.84) and (310.67,69.93) .. (312.89,71.04) .. controls (319.1,74.15) and (326.78,70.54) .. (333.72,70.7) .. controls (338.71,70.81) and (339.46,78.71) .. (340.9,80.95) .. controls (344.22,86.11) and (346.93,79.61) .. (351.83,81.97) .. controls (360.84,86.31) and (368.45,98.66) .. (365.83,108.96) .. controls (362.19,123.28) and (354.65,122.59) .. (344.65,128.09) .. controls (338.2,131.63) and (333.18,140.58) .. (325.87,139.7) .. controls (323.57,139.42) and (320.91,135.83) .. (318.01,137.31) .. controls (317.88,137.38) and (304.52,148.41) .. (298.88,146.53) .. controls (290.51,143.74) and (281.84,141.03) .. (274.63,135.94) .. controls (256.15,122.9) and (253.5,103) .. (265.41,84.02) -- cycle ;
						\draw [color={rgb, 255:red, 0; green, 0; blue, 0 }  ,draw opacity=1 ][fill={rgb, 255:red, 155; green, 155; blue, 155 }  ,fill opacity=0.28 ][line width=0.75] [line join = round][line cap = round]   (263.53,207.33) .. controls (268.03,205.13) and (272.67,203.07) .. (277.53,201.87) .. controls (290.43,198.69) and (303.41,207.85) .. (316.13,204.26) .. controls (323.44,202.2) and (329.62,197.29) .. (336.29,193.67) .. controls (340.4,191.44) and (345.65,194.9) .. (350.29,194.35) .. controls (366.01,192.5) and (387.65,179.07) .. (404.61,176.25) .. controls (414.3,174.63) and (427.49,186.04) .. (426.47,195.72) .. controls (425.73,202.76) and (424.77,209.9) .. (422.37,216.56) .. controls (420.97,220.42) and (412.67,221.46) .. (410.41,223.05) .. controls (403.88,227.67) and (398.45,233.73) .. (391.97,238.42) .. controls (389.55,240.17) and (386.53,238.42) .. (384.11,238.42) .. controls (380.95,238.42) and (379.13,242.33) .. (376.6,244.23) .. controls (370.66,248.68) and (363.46,246.05) .. (356.78,249.69) .. controls (352.87,251.82) and (349.59,255.09) .. (345.51,256.87) .. controls (339.45,259.51) and (332.31,257.77) .. (325.7,257.89) .. controls (322.59,257.95) and (319.55,259.71) .. (316.48,259.26) .. controls (308.85,258.13) and (305.64,251.42) .. (298.71,250.03) .. controls (284.64,247.22) and (274.27,247.4) .. (265.58,233.3) .. controls (263.85,230.5) and (259.58,211.28) .. (263.53,207.33) -- cycle ;
						\draw [color={rgb, 255:red, 32; green, 130; blue, 242 }  ,draw opacity=1 ]   (312.25,115.25) -- (371.88,207.98) ;
						\draw [shift={(373.5,210.5)}, rotate = 237.26] [fill={rgb, 255:red, 32; green, 130; blue, 242 }  ,fill opacity=1 ][line width=0.08]  [draw opacity=0] (8.93,-4.29) -- (0,0) -- (8.93,4.29) -- cycle    ;
						\draw [shift={(312.25,115.25)}, rotate = 57.26] [color={rgb, 255:red, 32; green, 130; blue, 242 }  ,draw opacity=1 ][fill={rgb, 255:red, 32; green, 130; blue, 242 }  ,fill opacity=1 ][line width=0.75]      (0, 0) circle [x radius= 3.35, y radius= 3.35]   ;
						\draw [color={rgb, 255:red, 208; green, 2; blue, 27 }  ,draw opacity=1 ]   (377.5,215.5) ;
						\draw [shift={(377.5,215.5)}, rotate = 0] [color={rgb, 255:red, 208; green, 2; blue, 27 }  ,draw opacity=1 ][fill={rgb, 255:red, 208; green, 2; blue, 27 }  ,fill opacity=1 ][line width=0.75]      (0, 0) circle [x radius= 3.35, y radius= 3.35]   ;
						\draw [shift={(377.5,215.5)}, rotate = 0] [color={rgb, 255:red, 208; green, 2; blue, 27 }  ,draw opacity=1 ][fill={rgb, 255:red, 208; green, 2; blue, 27 }  ,fill opacity=1 ][line width=0.75]      (0, 0) circle [x radius= 3.35, y radius= 3.35]   ;
						
						\draw (265.41,93.02) node [anchor=north west][inner sep=0.75pt]  [font=\fontsize{4.5pt}{1pt}\selectfont,color={rgb, 255:red, 0; green, 0; blue, 0 }  ,opacity=1 ] [align=left] {$\displaystyle U\ =\ \{u\  >\ M_{1}\} \ \cap \ B_{1/4}$};
						\draw (271.53,210.33) node [anchor=north west][inner sep=0.75pt]  [font=\scriptsize] [align=left] {\begin{minipage}[lt]{43.59pt}\setlength\topsep{0pt}
								\begin{center}
									$\displaystyle \mathcal{F}$\\[-3pt]
									{\fontsize{4.5pt}{2pt}\selectfont(Contact set)}
								\end{center}
						\end{minipage}};
						\draw (338,223) node [anchor=north west][inner sep=0.75pt]  [font=\fontsize{4.5pt}{2pt}\selectfont] [align=left] {$\displaystyle u( y) \ \leq \ M_{1} \ -1$};
						\draw (398.2,18.71) node [anchor=north west][inner sep=0.75pt]  [font=\footnotesize] [align=left] {$\displaystyle B_{1}$};
						\draw (352.2,44.71) node [anchor=north west][inner sep=0.75pt]  [font=\footnotesize] [align=left] {$\displaystyle B_{1/4}$};
						\draw (340.11,139.94) node [anchor=north west][inner sep=0.75pt]  [font=\tiny,rotate=-58.66] [align=left] {$\displaystyle y-x$};
						\draw (316.93,105) node [anchor=north west][inner sep=0.75pt]  [font=\small,color={rgb, 255:red, 32; green, 130; blue, 242 }  ,opacity=1 ] [align=left] {$\displaystyle x$};
						\draw (382,205) node [anchor=north west][inner sep=0.75pt]  [font=\small,color={rgb, 255:red, 208; green, 2; blue, 27 }  ,opacity=1 ] [align=left] {$\displaystyle y$};
				\end{tikzpicture}}
				\caption{Geometric picture}
			\end{subfigure}
			\hfill
			\begin{subfigure}{0.6\textwidth}
				\centering
				\tikzset{every picture/.style={line width=0.75pt}}
				\resizebox{0.92\linewidth}{!}{\begin{tikzpicture}[x=0.75pt,y=0.75pt,yscale=-1,xscale=1]
						\draw    (278.5,209.82) -- (558.02,209.82) ;
						\draw [shift={(561.02,209.82)}, rotate = 180] [fill={rgb, 255:red, 0; green, 0; blue, 0 }  ][line width=0.08]  [draw opacity=0] (8.93,-4.29) -- (0,0) -- (8.93,4.29) -- cycle;
						\draw [color={rgb, 255:red, 32; green, 130; blue, 242 }  ,draw opacity=1 ]   (285.07,182.37) .. controls (350.77,169.71) and (397.59,99.87) .. (434.87,90.9) .. controls (472.16,81.93) and (509.78,193.63) .. (547.88,186.6) ;
						\draw    (287.7,122.56) .. controls (320.55,116.23) and (346.89,38.23) .. (370.48,58.53) .. controls (394.08,78.82) and (416.23,95.38) .. (434.87,90.9) .. controls (453.52,86.41) and (503.21,20.53) .. (521.6,61.34) .. controls (540,102.16) and (533.43,83.86) .. (541.31,100.75) ;
						\draw [color={rgb, 255:red, 208; green, 2; blue, 27 }  ,draw opacity=1 ] [dash pattern={on 3.75pt off 3pt on 7.5pt off 3pt}]  (434.87,90.9) -- (434.87,209.82) ;
						\draw [shift={(434.87,90.9)}, rotate = 90] [color={rgb, 255:red, 208; green, 2; blue, 27 }  ,draw opacity=1 ][fill={rgb, 255:red, 208; green, 2; blue, 27 }  ,fill opacity=1 ][line width=0.75]      (0, 0) circle [x radius= 3.35, y radius= 3.35]   ;
						\draw [color={rgb, 255:red, 0; green, 0; blue, 0 }  ,draw opacity=1 ] [dash pattern={on 4.5pt off 4.5pt}]  (373.11,134.52) -- (373.11,211.93) ;
						\draw [color={rgb, 255:red, 0; green, 0; blue, 0 }  ,draw opacity=1 ]   (373.11,211.93) -- (430.02,101.12) ;
						\draw [shift={(430.93,99.34)}, rotate = 117.18] [color={rgb, 255:red, 0; green, 0; blue, 0 }  ,draw opacity=1 ][line width=0.75]    (10.93,-3.29) .. controls (6.95,-1.4) and (3.31,-0.3) .. (0,0) .. controls (3.31,0.3) and (6.95,1.4) .. (10.93,3.29)   ;
						\draw [shift={(373.11,211.93)}, rotate = 297.18] [color={rgb, 255:red, 0; green, 0; blue, 0 }  ,draw opacity=1 ][fill={rgb, 255:red, 0; green, 0; blue, 0 }  ,fill opacity=1 ][line width=0.75]      (0, 0) circle [x radius= 3.35, y radius= 3.35]   ;
						
						\draw (554.21,222.93) node [anchor=north west][inner sep=0.75pt]  [font=\small] [align=left] {$\displaystyle z$};
						\draw (368.93,222.93) node [anchor=north west][inner sep=0.75pt]  [font=\small] [align=left] {$\displaystyle x$};
						\draw (433.32,62.49) node [anchor=north west][inner sep=0.75pt]  [font=\small,color={rgb, 255:red, 208; green, 2; blue, 27 }  ,opacity=1 ] [align=left] {$\displaystyle y$};
						\draw (545.01,93.45) node [anchor=north west][inner sep=0.75pt]  [font=\small,color={rgb, 255:red, 0; green, 0; blue, 0 }  ,opacity=1 ] [align=left] {$\displaystyle u$};
						\draw (526.78,151.54) node [anchor=north west][inner sep=0.75pt]  [font=\scriptsize,color={rgb, 255:red, 32; green, 130; blue, 242 }  ,opacity=1 ] [align=left] {$\displaystyle \varphi ( z-x) \ +\ q( x)$};
				\end{tikzpicture}}
				\caption{Touching picture}
			\end{subfigure}
			\caption{Sliding-Cusp Contact Mechanism in the Measure Estimate}
			\label{fig:sliding_cusp}
		\end{figure}
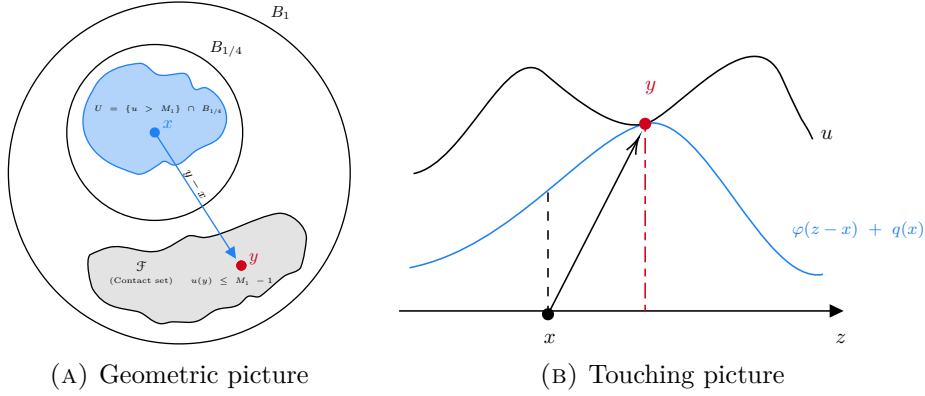
		
		\textbf{Step 3. Lipschitz continuity of $\nabla u$ on $\mathcal{F}$.} 
		Since $u$ is semiconcave, it is locally Lipschitz in $B_1$. Moreover, by the boundary separation obtained in \textbf{Step 1}, all contact points remain in a compact subset of $B_1$. Hence $u$ admits a modulus of continuity on the region under consideration. Since $u(x)>M_1$ for $x\in U$ and $u(y)\leq M_1-1$ for $y\in\mathcal F$, there exists $\varepsilon>0$ depending on the modulus of continuity of $u$ such that $|y - x| > \varepsilon$. By decreasing $\varepsilon$ if necessary, we may further assume that $\varepsilon\leq 1-\rho$. Since $\mathcal F\subset\overline{B_\rho}$, it follows that $\mathrm{dist}(y,\partial B_1)\geq1-\rho\geq\varepsilon$ for every $y\in\mathcal F$. Given that $\varphi$ is singular at the origin, we only evaluate $\varphi(y - x)$ outside a neighborhood of this singularity; this is precisely the significance of the constant $\varepsilon$. In this region, $\varphi$ is of class $C^2$ and its Hessian matrix
		\begin{align*}
			D^2 \varphi(y-x) = -5|y-x|^{-3/2}I + \frac{15}{2}|y-x|^{-7/2}\left[(y-x) \otimes(y-x)\right],
		\end{align*}
		satisfies the bound $|D^2 \varphi| \leq C\varepsilon^{-3/2}$. Let $(x_1,y_1)$ and $(x_2,y_2)$ be two pairs of points for which \eqref{pro_2:measure-est-eq3} holds, and set $r := 2|y_1 - y_2|$.
		We distinguish two cases.\\
		\textbf{Case 1.} If $|y_1-y_2|\geq \varepsilon/4$, then
		\begin{align*}
			|\nabla u(y_i)|=|\nabla\varphi(y_i-x_i)|=5|y_i-x_i|^{-1/2}\le 5\varepsilon^{-1/2}, \mbox{ for }i=1,2.
		\end{align*}
		Hence
		\begin{align*}
			|\nabla u(y_1)-\nabla u(y_2)| \leq |\nabla u(y_1)|+|\nabla u(y_2)| \leq 10\varepsilon^{-1/2} \leq C\varepsilon^{-3/2}r.
		\end{align*}
		Therefore, the desired Lipschitz estimate follows in this case.\\
		\textbf{Case 2.} Assume now that $|y_1-y_2|<\varepsilon/4$. Then $r<\varepsilon/2$. Since $y_1\in\mathcal F\subset\overline{B_\rho}$ and $r<\varepsilon/2\leq(1-\rho)/2$, we have $B_r(y_1)\subset B_1$. Hence, for every $z\in B_r(y_1)$, we have
		\begin{align*}
			|z-x_1|\geq |y_1-x_1|-|z-y_1|>\varepsilon-r>\varepsilon/2.
		\end{align*}
		Hence $\varphi(\cdot-x_1)$ is of class $C^2$ in $B_r(y_1)$, and $\|D^2\varphi(\cdot-x_1)\|_{L^\infty(B_r(y_1))}\le C\varepsilon^{-3/2}$.\\
		Specifically, the second-order Taylor expansion of $\varphi(\cdot - x_1)$ around $y_1$ yields
		\begin{align*}
			\varphi(z-x_1) &= \varphi(y_1-x_1) + \langle \nabla \varphi(y_1-x_1),z-y_1\rangle\\ 
			&\qquad+ \frac{1}{2} \langle D^2\varphi(\xi)(z-y_1), (z-y_1)\rangle\\
			&\geq \varphi(y_1-x_1) + \langle \nabla \varphi(y_1-x_1),z-y_1\rangle - C\varepsilon^{-3/2}r^2,
		\end{align*}
		where $\xi$ lies on the line segment connecting $y_1 - x_1$ and $z - x_1$. Moreover, from \eqref{pro_2:measure-est-eq4}, we obtain
		\begin{align}
			\label{pro_2:measure-est-eq5}
			\begin{split}
				u(z) &\geq \varphi(z-x_1) + q(x_1)\\ &\geq \varphi(y_1-x_1) + q(x_1)+ \langle \nabla \varphi(y_1-x_1),z-y_1\rangle - C\varepsilon^{-3/2}r^2\\
				&= u(y_1) + \langle\nabla u(y_1),z-y_1\rangle - C\varepsilon^{-3/2}r^2.
			\end{split}
		\end{align}
		By choosing $z = y_2$ in the Taylor expansion and applying similar arguments as \eqref{pro_2:measure-est-eq5}
		, we obtain
		\begin{align*}
			u(y_2) \geq u(y_1) + \langle\nabla u(y_1),y_2 - y_1\rangle - C\varepsilon^{-3/2}r^2.
		\end{align*}
		A symmetric reasoning applies to $y_2$, resulting in the following inequality
		\begin{align}
			\label{pro_2:measure-est-eq6}
			u(y_1) \geq u(y_2) + \langle\nabla u(y_2),y_1 - y_2\rangle - C\varepsilon^{-3/2}r^2.
		\end{align}
		Substituting \eqref{pro_2:measure-est-eq6} into \eqref{pro_2:measure-est-eq5}, we deduce
		\begin{align}
			\label{pro_2:measure-est-eq7}
			u(z) &\geq u(y_2) + \langle \nabla u(y_2),y_1 - y_2\rangle + \langle\nabla u(y_1),z - y_1\rangle - C\varepsilon^{-3/2}r^2.
		\end{align}
		Furthermore, from \eqref{pro_2:measure-est-eq1}, we also have
		\begin{align}
			\label{pro_2:measure-est-eq8}
			u(z) \leq u(y_2) + \langle\nabla u(y_2) ,z-y_2\rangle  + Cr^2.
		\end{align}
		From \eqref{pro_2:measure-est-eq7} and \eqref{pro_2:measure-est-eq8}, we collect
		\begin{align*}
			u(y_2)  + \langle \nabla u(y_2),y_1 - y_2\rangle &+ \langle\nabla u(y_1),z - y_1\rangle - C\varepsilon^{-3/2}r^2\\ 
			&\leq u(y_2) + \langle\nabla u(y_2) ,z-y_2\rangle + Cr^2.
		\end{align*}
		As a consequence, we arrive at
		\begin{align}\label{21.2}
			\langle \nabla u(y_1)- \nabla u(y_2),z - y_1\rangle \leq C\left(\varepsilon^{-3/2} + 1\right)r^2.
		\end{align}
		Since the inequality \eqref{21.2} is valid for any $z\in B_r(y_1)$, it yields 
		\begin{align*}
			\left|\nabla u(y_1)- \nabla u(y_2)\right| \leq C\left(\varepsilon^{-3/2} + 1\right)r.
		\end{align*}
		This establishes the Lipschitz continuity of $\nabla u$ on $\mathcal{F}$. Notably, the resulting Lipschitz seminorm $[\nabla u]_{C^{0,1}(\mathcal{F})}$ is determined by $\varepsilon$, and 
		thus the modulus of continuity of $u$. It is important to remark that this constant is not universal.
		
		\textbf{Step 4. The map $m: \mathcal{F} \to U$.} Following the preceding arguments, $u$ is necessarily differentiable at each point $y \in \mathcal{F}$, where its gradient coincides with that of the test function, i.e., $\nabla u(y) = \nabla \varphi(y-x)$. Note that the relationship $\nabla \varphi(y - x) = -5|y - x|^{-3/2}(y - x)$ uniquely determines the vector difference $y - x$. Consequently, for each point $y \in \mathcal{F}$, there exists a unique $x \in U$ satisfying the contact condition \eqref{pro_2:measure-est-eq3}. Let us define $m:~\mathcal{F} \to U$ as the function that maps $y$ into $x$. The implicit relationship 
		\begin{align}
			\label{pro_2:measure-est-eq9}
			\nabla u(y) = \nabla \varphi(y-m(y)),
		\end{align} 
		allows us to derive an explicit expression for the contact map. By utilizing the injectivity of $\nabla \varphi$, we can solve for $m(y)$ as follows:
		\begin{align*}
			m(y) = y - \left(\nabla \varphi\right)^{-1}\nabla u(y),
		\end{align*}
		where $(\nabla \varphi)^{-1}$ denotes the inverse of the gradient mapping $\nabla \varphi: \mathbb{R}^n\setminus \{0\} \to \mathbb{R}^n\setminus \{0\}$. This formulation confirms that $m$ is a well-defined function on the contact set $\mathcal{F}$.\\
		Having established the Lipschitz continuity of $\nabla u$ on $\mathcal{F}$, we observe that the inverse map $(\nabla\varphi)^{-1}$, which maps $\nabla\varphi(y-x)$ to $y-x$. As previously noted, $|y-x|\geq\varepsilon$ for every contact pair, while $x\in B_{1/4}$ and $y\in B_1$ imply $|y-x|\leq 5/4$. Hence the set of vectors $\nabla\varphi(y-x)$ arising from the contact pairs is contained in a compact subset of $\mathbb{R}^n\setminus\{0\}$. Consequently, $(\nabla\varphi)^{-1}$ is smooth and Lipschitz on the relevant range, with a Lipschitz constant depending on $\varepsilon$. Together with the Lipschitz continuity of $\nabla u$ on $\mathcal{F}$, this implies that $m$ is Lipschitz on $\mathcal{F}$.\\
		%
		Moreover, the contact set $\mathcal F$ is analytic, being the projection of the Borel set of contact pairs associated with \eqref{pro_2:measure-est-eq3}; in particular, it is Lebesgue measurable. By Rademacher's theorem, the map $m$ is differentiable almost everywhere in $\mathcal{F}$. By extending $m$ to a Lipschitz map on $\mathbb{R}^n$ if necessary, the area formula gives
		\begin{align}
			\label{pro_2:measure-est-eq10}
			\int_{\mathcal{F}}Jm(y)dy = \int_{\mathbb{R}^n}\mathcal{H}^0\left(\mathcal{F}\cap m^{-1}(\{x\})\right)d\mathcal{H}^n(x),
		\end{align}
		where $Jm(y)=\left|\det Dm(y)\right|$. Since $m(\mathcal{F}) = U$, it follows that 
		\begin{align*}
			\mathcal{H}^0(\mathcal{F} \cap m^{-1}(\{x\})) \geq \mathbf{1}_U(x) \text{ for all } x \in \mathbb{R}^n. 
		\end{align*}
		Moreover, as $m$ maps between subsets of $\mathbb{R}^n$, the $n$-dimensional Hausdorff measure $\mathcal{H}^n$ coincides with the Lebesgue measure. Therefore, we deduce
		\begin{align}
			\label{pro_2:measure-est-eq11}
			\int_{\mathbb{R}^n}\mathcal{H}^0\left(\mathcal{F}\cap m^{-1}\left(\{x\}\right)\right)d\mathcal{H}^n(x) \geq \int_{\mathbb{R}^n}\mathbf{1}_U(x)d\mathcal{H}^n(x) = \mathcal{H}^n(U) = |U|.
		\end{align}
		Combining \eqref{pro_2:measure-est-eq10} and \eqref{pro_2:measure-est-eq11}, we deduce
		\begin{align}
			\label{pro_2:measure-est-eq12}
			|U|\leq\int_{\mathcal{F}}\left|\det Dm(y)\right|dy.
		\end{align}
		
		\textbf{Step 5. A structural estimate for $Dm(y)$.} While the preceding estimate for $\left|Dm(y)\right|$ was contingent upon $\varepsilon$, this was primarily a technical requirement to validate the formulation in \eqref{pro_2:measure-est-eq12}. In this step, we establish a more robust bound for $|Dm(y)|$, depending only on $n,\lambda,\Lambda,\nu_0,L,i(\Phi),s(\Phi),$ $\mathcal{C}$, and $\|c\|_{L^\infty(B_1)}$.\\
		As initially established, $u$ is pointwise twice differentiable almost everywhere, except on a set $E$ of measure zero. Consequently, for any $y \in \mathcal{F} \setminus E$, there exists a symmetric matrix $D^2u(y) \in \mathbb{S}^{n\times n}$ such that $\left(\nabla u(y), D^2u(y)\right) \in \overline{J}^{2,-}u(y)$. Hence, by the definition of viscosity supersolution 
		\begin{align}
			\label{pro_2:measure-est-eq13}
			\Phi(y,|\nabla u(y)|)F(D^2u(y)) \leq H(y,\nabla u(y)) - c(y)(u(y))^{i(\Phi)+1} + h(y).
		\end{align}
		Furthermore, from assumptions {\upshape\hyperref[A1]{(A1)}}--{\upshape\hyperref[A2]{(A2)}}, we have
		\begin{align}
			\label{pro_2:measure-est-eq14}
			\begin{split}
				&\mathcal{P}^-_{\lambda,\Lambda}(D^2u(y)) = \lambda\mathrm{tr}\left((D^2u(y))^+\right) - \Lambda\mathrm{tr}\left((D^2u(y))^-\right) \leq F(D^2u(y)),\\
				&\frac{\nu_0}{L}\min\left\{|\nabla u(y)|^{i(\Phi)},|\nabla u(y)|^{s(\Phi)}\right\} \leq \Phi(y,|\nabla u(y)|).
			\end{split}
		\end{align}
		By coupling \eqref{pro_2:measure-est-eq13} with \eqref{pro_2:measure-est-eq14}, we deduce
		\begin{align}\label{21.3}
			\begin{split}
				\lambda\mathrm{tr}\left((D^2u(y))^+\right) &- \Lambda\mathrm{tr}\left((D^2u(y))^-\right)\\&\leq \frac{L\left[|H(y,\nabla u(y))| - c(y)(u(y))^{i(\Phi)+1} + h^+(y)\right]}{\nu_0\min\left\{|\nabla u(y)|^{i(\Phi)},|\nabla u(y)|^{s(\Phi)}\right\}}.
			\end{split}
		\end{align}
		Moreover, since the cusp is defined as $\varphi = -10|z|^{1/2}$, then a straightforward calculation reveals that $|\nabla u(y)| = |\nabla \varphi(y-x)| = 5|y-x|^{-1/2} \geq 2\sqrt{5} > 1$ for all $x \in B_{1/4}$ and $y \in B_{1}$. 
		Consequently, it deduces from \eqref{21.3} and assumptions {\upshape\hyperref[A1]{(A1)}--\hyperref[A3]{(A3)}} that
		\begin{align*}
			\lambda\mathrm{tr}\left((D^2u(y))^+\right) &- \Lambda\mathrm{tr}\left((D^2u(y))^-\right)\\
			&\leq \frac{L\left[\mathcal{C}|\nabla u(y)|^{m} + c^-(y)(u(y))^{i(\Phi)+1}+h^+(y)\right]}{\nu_0|\nabla u(y)|^{i(\Phi)}}.
		\end{align*}
		Since $m \in (0,i(\Phi)+1]$, we get $|\nabla u(y)|^{m} \leq |\nabla u(y)|^{i(\Phi)+1}$. Plus, within the contact set $\mathcal{F}$, we recall from \textbf{Step 2} that $u(y) \leq M_1 - 1$. Thus, we arrive at
		\begin{align*}
			&\lambda\mathrm{tr}\left((D^2u(y))^+\right) - \Lambda\mathrm{tr}\left((D^2u(y))^-\right)\\
			&\quad\leq \frac{L}{\nu_0}\left[\mathcal{C}|\nabla u(y)| + \frac{\|c^-\|_{L^\infty(B_1)}\left(M_1-1\right)^{i(\Phi)+1}}{|\nabla u(y)|^{i(\Phi)}} + \frac{\|h\|_{L^\infty(B_1)}}{|\nabla u(y)|^{i(\Phi)}}\right].
		\end{align*}
		Since $|\nabla u(y)|=5|y-x|^{-1/2}$, it follows that 
		$|\nabla u|^{-i(\Phi)} \leq C\left(1+|y-x|^{-1/2}\right)$ where $C>0$ depends only on $n,\lambda,\Lambda,\nu_0,L,i(\Phi),s(\Phi),\mathcal{C}$, and $\|c\|_{L^\infty(B_1)}$. Consequently, we can write
		\begin{align}
			\label{pro_2:measure-est-eq16}
			\lambda\mathrm{tr}\left((D^2u(y))^+\right) - \Lambda\mathrm{tr}\left((D^2u(y))^-\right) \leq C\left(1 + |y-x|^{-1/2}\right).
		\end{align}
		Furthermore, \eqref{pro_2:measure-est-eq4} implies $D^2u(y) \geq D^2\varphi(y-x)$. Hence, we obtain
		\begin{align*}
			\mathrm{tr}\left((D^2u(y))^-\right)\leq \mathrm{tr}\left((D^2\varphi(y-x))^-\right).
		\end{align*}
		Combining this with \eqref{pro_2:measure-est-eq16}, we obtain
		\begin{align*}
			|D^2u(y)|&\leq C\left[1 + |y-x|^{-1/2}+\mathrm{tr}\left(\left(D^2\varphi(y-x)\right)^-\right)\right] \\
			&\leq C\left[1 + |y-x|^{-3/2}\right].
		\end{align*}
		Finally, at almost every point $y\in\mathcal{F}\setminus E$, differentiating \eqref{pro_2:measure-est-eq9} yields
		\begin{align*}
			D^2u(y) = D^2\varphi(y-m(y))(I-Dm(y)),
		\end{align*}
		and consequently,
		\begin{align*}
			Dm(y) = D^2\varphi(y-m(y))^{-1}\left(D^2\varphi(y-m(y)) - D^2u(y)\right).
		\end{align*}
		Writing $x := m(y)$, we infer from the previous estimates that
		\begin{align*}
			\left|Dm(y)\right| \leq \|D^2\varphi(y-x)^{-1}\| \cdot \|D^2\varphi(y-x) - D^2u(y)\| \leq C.
		\end{align*}
		Indeed, the last inequality follows from the identities $\| D^2\varphi(y-x)^{-1} \| = C|x-y|^{3/2}$ and $\| D^2\varphi(y-x) - D^2u(y) \| \leq C(1 + |x-y|^{-3/2})$. We observe a crucial cancellation of the $|x-y|$ dependence in the final estimate. This property is not universal; it would be lost with other choices of $\varphi$, such as the radial function $\varphi(x) = -|x|$.\\
		Therefore, we have obtained $\left|Dm(y)\right| \leq C$ almost everywhere in $\mathcal{F}$, for a universal constant $C$. Substituting this estimate into \eqref{pro_2:measure-est-eq12} yields
		\begin{align*}
			|U| \leq \int_{\mathcal{F}} C^ndy = C^n\left|\mathcal{F}\right|.
		\end{align*}
		This establishes a lower bound for the measure of the contact set $\mathcal{F}$. Since $|U| > |B_{1/4}| - \delta |B_1|$, it follows that
		\begin{align*}
			|\mathcal{F}| \geq C^{-n} \left( |B_{1/4}| - \delta |B_1| \right) > \delta |B_1|,
		\end{align*}
		provided $\delta > 0$ is chosen sufficiently small. Combining it with the inclusion $\mathcal{F} \subset \{u \leq M_1-1\}$ and assumption \eqref{pro_2:measure-est-assumption}, we get a contradiction, thereby concluding the proof.
	\end{proof}
	\begin{lemma}[\bf A measure estimate]
		\label{Measure_est}
		There exist a small constant $\delta>0$ and a large constant $M_1>0$, such that if $u\in\mathrm{LSC}(B_1)$ is a non-negative viscosity supersolution in $B_1$ of an equation of the form \eqref{eq:main}, whose coefficients satisfy assumptions {\upshape\hyperref[A1]{(A1)}--\hyperref[A4]{(A4)}}, with $c\leq0$ and $\|h\|_{L^\infty(B_1)}\leq1$,
		and
		\begin{align*}
			|\left\{u > M_1\right\} \cap B_1| > (1-\delta)|B_1|,
		\end{align*}
		then we have 
		\begin{align*}
			u > 1 \mbox{ in } B_{1/4}.
		\end{align*}
	\end{lemma}
	\begin{figure}[h]
		\centering
		\tikzset{every picture/.style={line width=0.75pt}} 
		\resizebox{0.92\linewidth}{!}{\begin{tikzpicture}[x=0.75pt,y=0.75pt,yscale=-1,xscale=1]
				\draw  [color={rgb, 255:red, 0; green, 0; blue, 0 }  ,draw opacity=1 ][fill={rgb, 255:red, 255; green, 255; blue, 255 }  ,fill opacity=1 ] (160.63,158.6) .. controls (160.63,107.05) and (204.16,65.26) .. (257.85,65.26) .. controls (311.54,65.26) and (355.07,107.05) .. (355.07,158.6) .. controls (355.07,210.14) and (311.54,251.93) .. (257.85,251.93) .. controls (204.16,251.93) and (160.63,210.14) .. (160.63,158.6) -- cycle ;
				\draw [line width=1.5]    (355.07,159.6) -- (434.56,159.66) ;
				\draw [shift={(437.56,159.66)}, rotate = 180.05] [color={rgb, 255:red, 0; green, 0; blue, 0 }  ][line width=1.5]    (14.21,-4.28) .. controls (9.04,-1.82) and (4.3,-0.39) .. (0,0) .. controls (4.3,0.39) and (9.04,1.82) .. (14.21,4.28)   ;
				\draw [draw opacity=0][fill={rgb, 255:red, 32; green, 130; blue, 242 }  ,fill opacity=0.53 ][line width=1.5] [line join = round][line cap = round]   (179.7,126.71) .. controls (182.66,120.4) and (187.27,114.93) .. (191.93,109.74) .. controls (197.58,103.45) and (209.45,106.54) .. (216.79,100.67) .. controls (222.75,95.9) and (232.25,82.96) .. (240.86,81.73) .. controls (249.21,80.53) and (257.73,79.23) .. (266.12,80.15) .. controls (272.48,80.84) and (278.12,84.66) .. (284.27,86.46) .. controls (287.91,87.53) and (292.09,87.22) .. (295.32,89.22) .. controls (308.99,97.73) and (309.59,114.36) .. (315.84,127.89) .. controls (319.08,134.93) and (329.72,132.15) .. (333.99,140.52) .. controls (345.09,162.3) and (330.02,180.1) .. (322.94,199.32) .. controls (320.56,205.78) and (321.2,213.24) .. (318.2,219.44) .. controls (313.31,229.58) and (303.52,231.48) .. (292.16,236.41) .. controls (275.52,243.64) and (264.91,239) .. (249.15,234.44) .. controls (246.86,233.77) and (244.25,234.54) .. (242.04,233.65) .. controls (238.22,232.09) and (235.4,228.34) .. (231.39,227.33) .. controls (226.17,226.03) and (204.77,222.5) .. (201.8,219.44) .. controls (195.69,213.18) and (188.67,203.41) .. (184.04,195.77) .. controls (178.59,186.78) and (171.26,166.39) .. (172.59,155.52) .. controls (173.51,148.01) and (182.25,133.1) .. (179.7,126.71) -- cycle ;
				\draw  [color={rgb, 255:red, 0; green, 0; blue, 0 }  ,draw opacity=1 ][dash pattern={on 4.5pt off 4.5pt}] (235.92,162.6) .. controls (235.92,152.1) and (244.39,143.58) .. (254.85,143.58) .. controls (265.3,143.58) and (273.78,152.1) .. (273.78,162.6) .. controls (273.78,173.1) and (265.3,181.61) .. (254.85,181.61) .. controls (244.39,181.61) and (235.92,173.1) .. (235.92,162.6) -- cycle ;
				\draw [draw opacity=0][fill={rgb, 255:red, 255; green, 255; blue, 255 }  ,fill opacity=1 ][line width=0.75] [line join = round][line cap = round]   (269.57,102.7) .. controls (255.16,97.58) and (249.69,115.33) .. (262.48,118.35) .. controls (267.26,119.48) and (272.47,119.14) .. (277.2,117.81) .. controls (279.33,117.21) and (282.83,114.87) .. (281.72,112.95) .. controls (278.82,107.97) and (272.69,105.83) .. (268.17,102.27) ;
				\draw [draw opacity=0][fill={rgb, 255:red, 255; green, 255; blue, 255 }  ,fill opacity=1 ][line width=0.75] [line join = round][line cap = round]   (298.69,148.36) .. controls (300.23,145.95) and (301.05,142.87) .. (303.31,141.13) .. controls (310.2,135.85) and (320.39,137.2) .. (319.11,147.82) .. controls (318.68,151.35) and (315.97,154.27) .. (315.24,157.75) .. controls (314.63,160.67) and (317.54,164.95) .. (315.13,166.71) .. controls (305.65,173.66) and (289.41,169.7) .. (292.78,155.59) .. controls (293.56,152.33) and (296.94,150.34) .. (299.01,147.71) ;
				\draw [draw opacity=0][fill={rgb, 255:red, 255; green, 255; blue, 255 }  ,fill opacity=1 ][line width=0.75] [line join = round][line cap = round]   (263.88,207.6) .. controls (256.72,205.68) and (266.66,197.27) .. (271.18,201.01) .. controls (275.45,204.55) and (271.12,206.49) .. (272.04,210.4) .. controls (272.74,213.36) and (275.6,215.59) .. (276.02,218.61) .. controls (276.21,219.96) and (274.91,222.48) .. (273.66,221.95) .. controls (270.13,220.46) and (267.41,217.26) .. (265.27,214.07) .. controls (263.96,212.1) and (265.44,208.59) .. (263.55,207.17) ;
				\draw [draw opacity=0][fill={rgb, 255:red, 255; green, 255; blue, 255 }  ,fill opacity=1 ][line width=0.75] [line join = round][line cap = round]   (292.16,181.89) .. controls (286.49,178.09) and (281.22,187.23) .. (280.82,191.73) .. controls (280.27,197.91) and (290.91,193.77) .. (291.58,199.26) .. controls (292.06,203.15) and (287.06,210.14) .. (295.43,209.3) .. controls (304.77,208.36) and (311.51,192.53) .. (303.5,186.14) .. controls (299.92,183.27) and (294.92,182.92) .. (290.62,181.31) ;
				\draw [draw opacity=0][fill={rgb, 255:red, 255; green, 255; blue, 255 }  ,fill opacity=1 ][line width=0.75] [line join = round][line cap = round]   (227.72,104.79) .. controls (229.01,96.41) and (216,107.88) .. (215.24,109.95) .. controls (212.97,116.18) and (214.52,123.72) .. (211.13,129.41) .. controls (210.95,129.72) and (186.79,142.04) .. (200.42,148.72) .. controls (203.15,150.06) and (206.54,149.93) .. (209.52,149.31) .. controls (228.78,145.35) and (224.18,123.19) .. (226.11,109.07) .. controls (226.37,107.15) and (228.1,105.32) .. (227.58,103.46) ;
				\draw  [draw opacity=0][fill={rgb, 255:red, 255; green, 255; blue, 255 }  ,fill opacity=1 ] (238.98,201.04) .. controls (259.01,203.31) and (239.47,221.79) .. (226.1,218.36) .. controls (212.72,214.93) and (175.82,167.81) .. (206.52,166.55) .. controls (237.22,165.3) and (218.95,198.76) .. (238.98,201.04) -- cycle ;
				\draw [fill={rgb, 255:red, 74; green, 144; blue, 226 }  ,fill opacity=0.49 ][line width=1.5]    (168.65,255.27) .. controls (181.63,222.29) and (189.08,205.77) .. (213.6,192.34) ;
				\draw [shift={(216.75,190.67)}, rotate = 153.07] [fill={rgb, 255:red, 0; green, 0; blue, 0 }  ][line width=0.08]  [draw opacity=0] (11.61,-5.58) -- (0,0) -- (11.61,5.58) -- cycle    ; 
				\draw [fill={rgb, 255:red, 74; green, 144; blue, 226 }  ,fill opacity=0.49 ][line width=1.5]    (209.5,49.5) .. controls (224.62,59.9) and (242.42,80.11) .. (242.64,102.56) ;
				\draw [shift={(242.5,106.5)}, rotate = 274.76] [fill={rgb, 255:red, 0; green, 0; blue, 0 }  ][line width=0.08]  [draw opacity=0] (11.61,-5.58) -- (0,0) -- (11.61,5.58) -- cycle    ;
				\draw  [color={rgb, 255:red, 0; green, 0; blue, 0 }  ,draw opacity=1 ][fill={rgb, 255:red, 155; green, 155; blue, 155 }  ,fill opacity=0.05 ] (437.56,159.66) .. controls (437.56,108.11) and (481.09,66.33) .. (534.78,66.33) .. controls (588.47,66.33) and (632,108.11) .. (632,159.66) .. controls (632,211.21) and (588.47,253) .. (534.78,253) .. controls (481.09,253) and (437.56,211.21) .. (437.56,159.66) -- cycle ;
				\draw  [color={rgb, 255:red, 0; green, 0; blue, 0 }  ,draw opacity=1 ][fill={rgb, 255:red, 184; green, 233; blue, 134 }  ,fill opacity=1 ] (508.67,159.66) .. controls (508.67,145.82) and (520.36,134.6) .. (534.78,134.6) .. controls (549.2,134.6) and (560.89,145.82) .. (560.89,159.66) .. controls (560.89,173.51) and (549.2,184.73) .. (534.78,184.73) .. controls (520.36,184.73) and (508.67,173.51) .. (508.67,159.66) -- cycle ;
				
				\draw (262.37,132) node [anchor=north west][inner sep=0.75pt]  [font=\footnotesize,color={rgb, 255:red, 0; green, 0; blue, 0 }  ,opacity=1 ] [align=left] {$\displaystyle B_{1/4}$};
				\draw (318.95,72.24) node [anchor=north west][inner sep=0.75pt]  [font=\footnotesize] [align=left] {$\displaystyle B_{1}$};
				\draw (120.22,32.4) node [anchor=north west][inner sep=0.75pt]  [font=\footnotesize,color={rgb, 255:red, 0; green, 0; blue, 0 }  ,opacity=1 ] [align=left] {$\displaystyle E\ =\ \{u\  >\ M_{1}\} \ \cap \ B_{1}$};
				\draw (367,111) node [anchor=north west][inner sep=0.75pt]  [font=\normalsize] [align=left] {\begin{minipage}[lt]{42.52pt}\setlength\topsep{0pt}
						\begin{center}
							{\small \textbf{Measure }}\\{\small \textbf{estimate}}
						\end{center}
				\end{minipage}};
				\draw (135.69,259.26) node [anchor=north west][inner sep=0.75pt]  [font=\footnotesize,color={rgb, 255:red, 0; green, 0; blue, 0 }  ,opacity=1 ] [align=left] {$\displaystyle \{u\ \leq \ M_{1}\}$};
				\draw (557.79,133.62) node [anchor=north west][inner sep=0.75pt]  [font=\footnotesize,color={rgb, 255:red, 0; green, 0; blue, 0 }  ,opacity=1 ] [align=left] {$\displaystyle B_{1/4}$};
				\draw (578.95,60.46) node [anchor=north west][inner sep=0.75pt]  [font=\footnotesize] [align=left] {$\displaystyle B_{1}$};
				\draw (515.42,152.33) node [anchor=north west][inner sep=0.75pt]  [font=\footnotesize,color={rgb, 255:red, 0; green, 0; blue, 0 }  ,opacity=1 ] [align=left] {$\displaystyle u\  >\ 1$};		
		\end{tikzpicture}}			
		\caption{Schematic illustration of the measure estimate}
	\end{figure}
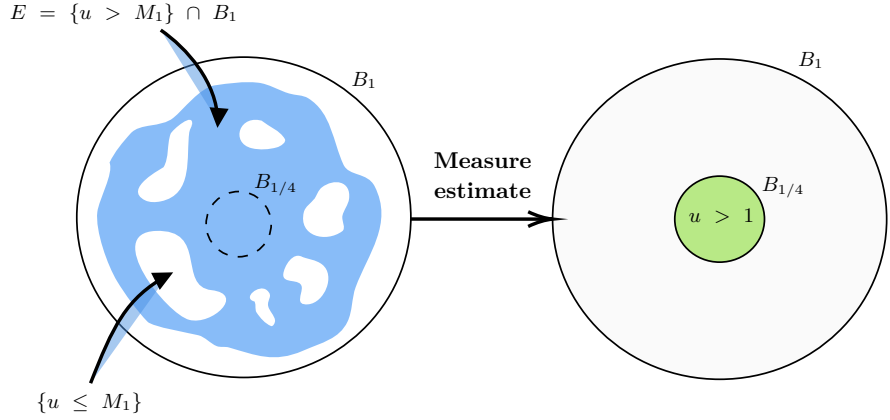
	\begin{proof}
		Let $u$ be a merely lower semicontinuous supersolution in $B_1$. Set $v := \min(u, 2M_1)$ where $M_1$ is given by Proposition \ref{pro_2:measure-est}. By construction, $v$ is lower semicontinuous in $B_1$, satisfies $0\le v\le 2M_1$, and $\{v>M_1\}=\{u>M_1\}$. We consider the inf-convolution of $v$ with parameter $\varepsilon > 0$:
		\begin{align}
			\label{pro_1:measure-est-eq0}
			v_\varepsilon(x) = \displaystyle\inf_{y \in B_1} \left(v(y) + \frac{|y - x|^2}{2\varepsilon}\right).
		\end{align}
		For $x\in B_{1-\eta}$ with $\eta>0$, we can show that the minimizing point $y_x$ actually belongs to $B_1$. Indeed, there is a sequence $\{y_j\}_{j \geq 1}\subset B_1$ such that
		\begin{align*}
			v_{\varepsilon}(x)+\dfrac{1}{j}\geq v(y_j)+\dfrac{|y_j-x|^2}{2\varepsilon}\text{ for all }j\geq 1.
		\end{align*}
		This implies that
		\begin{align*}
			\frac{\left|y_j - x\right|^2}{2\varepsilon} \leq v_\varepsilon(x) - v\left(y_j\right) + \dfrac{1}{j}\leq v(x) + \dfrac{1}{j}.
		\end{align*}
		Since $v(x)\le 2M_1$ it follows that $\left|y_j - x\right| \leq \sqrt{2\varepsilon(2M_1+1)}$. Consequently, for any fixed $\eta>0$, if $\varepsilon>0$ is sufficiently small so that $\sqrt{2\varepsilon(2M_1+1)}<\eta$, then the sequence $\{y_j\}_{j \geq 1}$ lies in a compact subset of $B_1$. Since $v$ is lower semicontinuous, the infimum is attained at some point $y_x\in B_1$, yielding
		\begin{align*}
			v_\varepsilon(x) = v\left(y_x\right) + \frac{\left|y_x - x\right|^2}{2\varepsilon}.
		\end{align*}
		Fix $\sigma\in(0,M_1)$ and define
		\begin{align*}
			\Omega_{\varepsilon,\eta,\sigma}:=\{x\in B_{1-\eta}: v_\varepsilon(x)<2M_1-\sigma\}.
		\end{align*}
		We assert that the semiconcave measure argument can be applied to $v_\varepsilon$ in $\Omega_{\varepsilon,\eta,\sigma}$. More precisely, at every contact point with nonzero gradient, the viscosity supersolution inequality for $u$ can be transmitted to $v_\varepsilon$ via the minimizing point $y_x$. All the coefficients appearing in the resulting relation satisfy the same uniform structural bounds as the original coefficients. Indeed, let $x\in \Omega_{\varepsilon,\eta,\sigma}$ and $y_x$ be a minimizer associated with $v_\varepsilon(x)$. Since
		\begin{align*}
			v(y_x)\leq v_\varepsilon(x)<2M_1-\sigma<2M_1,
		\end{align*}
		we have $v(y_x)=u(y_x)$. 
		
		We first verify the constant branch. Suppose that $v_\varepsilon\equiv a$ in some ball $B_\rho(x_0)\Subset\Omega_{\varepsilon,\eta,\sigma}$. For each $x\in B_\rho(x_0)$. Set
		\begin{align*}
			Q_x(z):=v(y_x)+\frac{|y_x-z|^2}{2\varepsilon}.
		\end{align*}
		Then $v_\varepsilon\leq Q_x$ and $Q_x(x)=v_\varepsilon(x)=a$. Since $v_\varepsilon\equiv a$ in $B_{\rho}(x_0)$, the test function $Q_x$ attains a local minimum at $x$. Hence $0=\nabla Q_x(x)=(x-y_x)/\varepsilon$, or equivalently $y_x=x$. This yields $a=v_\varepsilon(x)=v(x)$. Since $x\in\Omega_{\varepsilon,\eta,\sigma}$, we have $a<2M_1-\sigma<2M_1$, and ensures $v(x)=u(x)$, thereby implying $u\equiv a$ in $B_\rho(x_0)$. By the constant branch in the definition of viscosity supersolutions for $u$, we obtain
		\begin{align*}
			h(x)\geq c(x)|a|^{i(\Phi)}a\text{ for every }x\in B_\rho(x_0).
		\end{align*}
		Because $y_x=x$, the perturbed coefficients coincide with the
		original ones in the constant branch. Therefore, the constant-branch
		condition for $v_\varepsilon$ follows.
		
		Next, we consider the nonzero-gradient branch. Let $\phi\in C^2(B_1)$ be such that $v_\varepsilon\succ_{x_0}\phi$ for some $x_0\in \Omega_{\varepsilon,\eta,\sigma}$ and $\nabla\phi(x_0)\neq0$. Let $y_{x_0}$ be a minimizing point in the definition of $v_\varepsilon(x_0)$. Define
		\begin{align*}
			\psi(y):=\phi(x_0+y-y_{x_0})-\frac{|y_{x_0}-x_0|^2}{2\varepsilon}.
		\end{align*}
		We claim that $v\succ_{y_{x_0}}\psi$. Indeed, for $y$ sufficiently close to $y_{x_0}$, set $\xi:=x_0+y-y_{x_0}$. Since $v_\varepsilon\succ_{x_0}\phi$, we have $\phi(\xi)\leq v_\varepsilon(\xi)$. Combining this with the definition of $v_\varepsilon$ in \eqref{pro_1:measure-est-eq0} leads to
		\begin{align*}
			\phi(x_0+y-y_{x_0})  \leq v(y)+\dfrac{|y_{x_0}-x_0|^2}{2\varepsilon}.
		\end{align*}
		Thus $\psi(y)\leq v(y)$. Moreover, equality holds at $y=y_{x_0}$, and hence $v\succ_{y_{x_0}}\psi$. Furthermore, $\nabla\psi(y_{x_0})=\nabla\phi(x_0)\neq0$, and $D^2\psi(y_{x_0})=D^2\phi(x_0)$. Since
		\begin{align*}
			v(y_{x_0})\leq v_\varepsilon(x_0)<2M_1-\sigma<2M_1,
		\end{align*} 
		we have $v(y_{x_0})=u(y_{x_0})$. Since $v\leq u$ in $B_1$, it follows that $\psi\leq v\leq u$ near $y_{x_0}$ and $\psi(y_{x_0})=v(y_{x_0})=u(y_{x_0})$, which establishes $u\succ_{y_{x_0}}\psi$. As $u$ is a viscosity supersolution, the corresponding supersolution inequality holds at $y_{x_0}$:
		\begin{align*}
			&\Phi\left(y_{x_0},|\nabla\phi(x_0)|\right) F\left(D^2\phi(x_0)\right)-H\left(y_{x_0},\nabla\phi(x_0)\right)+c(y_{x_0})u(y_{x_0})^{i(\Phi)+1}\\
			&\leq h(y_{x_0}).
		\end{align*}
		Taking into account that, $u(y_{x_0})=v(y_{x_0}) \leq v_\varepsilon(x_0)$, and since $c\leq0,~0 \leq i(\Phi)+1$, we have
		\begin{align*}
			c(y_{x_0})v_\varepsilon(x_0)^{i(\Phi)+1} \leq c(y_{x_0})u(y_{x_0})^{i(\Phi)+1}.
		\end{align*}
		Therefore,
		\begin{align*}
			&\Phi\left(y_{x_0},|\nabla\phi(x_0)|\right)F\left(D^2\phi(x_0)\right)-H\left(y_{x_0},\nabla\phi(x_0)\right)+c(y_{x_0})v_\varepsilon(x_0)^{i(\Phi)+1}\\
			&\leq h(y_{x_0}).
		\end{align*}
		Thus, at every nonzero-gradient contact point $x_0$, the required supersolution inequality for $v_\varepsilon$ is available with the coefficients evaluated at a corresponding minimizing point $y_{x_0}$. Although the minimizing point need not be unique, the uniformity of the assumptions $\Phi$, $H$, $c$, and $h$ ensures that the structural estimates used in the proof of Proposition \ref{pro_2:measure-est} remain valid uniformly for $v_\varepsilon$.
		
		Recall that $v_\varepsilon$ is semiconcave and $D^2v_\varepsilon \leq \varepsilon^{-1}I$. Since $v$ is lower semicontinuous and bounded, $v_\varepsilon$ increases pointwise to $v$ as $\varepsilon \to 0$. In particular,
		\begin{align*}
			\{u>M_1\}=\{v>M_1\}=\bigcup_{\varepsilon>0}\{v_\varepsilon>M_1\},
		\end{align*}
		Note that as $\varepsilon \to 0$, the sets $\{v_\varepsilon > M_1\}$ form an increasing nested collection, therefore
		\begin{align}
			\label{pro_1:measure-est-eq1}
			|\{u>M_1\}\cap B_1| = |\{v>M_1\}\cap B_1| = \lim_{\varepsilon\to0}|\{v_\varepsilon>M_1\}\cap B_1|.
		\end{align}
		Let $\delta>0$ be the constant in Proposition \ref{pro_2:measure-est} and $\mu>0$ such that
		\begin{align*}
			|\{u>M_1\}\cap B_1|=(1-\delta)|B_1|+\mu.
		\end{align*}
		We first choose $\eta\in(0,1/2)$ sufficiently small so that  $\delta|B_{1-\eta}| > \delta |B_1| -\mu /2$. Having fixed such an $\eta$, by \eqref{pro_1:measure-est-eq1}, we choose $\varepsilon>0$ sufficiently small so that $	\sqrt{2\varepsilon(2M_1+1)}<\eta$ and
		\begin{align*}
			|\{v_\varepsilon>M_1\}\cap B_1|>(1-\delta)|B_1|+\frac{\mu}{2} \iff |\{v_\varepsilon\leq M_1\}\cap B_1|<\delta |B_1|-\frac{\mu}{2}.
		\end{align*}
		It follows that
		\begin{align*}
			|\{v_\varepsilon\le M_1\}\cap B_{1-\eta}| \leq |\{v_\varepsilon\leq M_1\}\cap B_1| < \delta |B_1|-\displaystyle\frac{\mu}{2} < \delta |B_{1-\eta}|.
		\end{align*}
		Hence, we arrive at
		\begin{align*}
			|\left\{v_\varepsilon > M_1\right\} \cap B_{1-\eta}| > (1-\delta)|B_{1-\eta}|.
		\end{align*}
		
		We now rescale $B_{1-\eta}$ onto $B_1$ by setting $\tilde{v}_\varepsilon(z):=v_\varepsilon((1-\eta)z)$. The measure condition is preserved under this change of variables, namely,
		\begin{align*}
			|\{\tilde{v}_\varepsilon>M_1\}\cap B_1|=(1-\eta)^{-n}|\{v_\varepsilon>M_1\}\cap B_{1-\eta}|>(1-\delta)|B_1|.
		\end{align*}
		Moreover, every contact point $\tilde{z}$ on $\mathcal{F}$ arising in the proof of Proposition \ref{pro_2:measure-est} satisfies $\tilde{v}_\varepsilon(\tilde{z})\leq M_1-1$. Hence, at the corresponding point $z=(1-\eta)\tilde{z}$, one has
		\begin{align*}
			v_\varepsilon(z)\leq M_1-1<2M_1-\sigma,
		\end{align*}
		so that $z\in\Omega_{\varepsilon,\eta,\sigma}$. Therefore, the transferred supersolution inequality established above is available at every contact point used in the rescaled argument.\\
		Since $\eta>0$ is chosen sufficiently small, we may assume that $1-\eta\in [1/2,1)$. Moreover, at every relevant contact point, the right-hand side appearing in the rescaled transferred supersolution inequality is bounded by
		\begin{align*}
			\dfrac{L}{\nu_0}(1-\eta)^{2+i(\Phi)}\|h\|_{L^\infty(B_1)}\leq H_0.
		\end{align*}
		Therefore, the argument of Proposition \ref{pro_2:measure-est}  applies uniformly to $\tilde{v}_\varepsilon$. We conclude that $\tilde{v}_\varepsilon>1$ in $B_{1/4}$, or equivalently, $v_\varepsilon>1$ in $B_{(1-\eta)/4}$. Since $u\geq v\geq v_\varepsilon$ and since $\eta>0$ can be chosen arbitrarily small, we conclude that $u>1$ in $B_{1/4}$. This completes the proof.		
	\end{proof}
	\subsection{\bf The $L^\varepsilon$ estimate and Harnack inequality}
	By employing barrier function \hyperref[barrier_harnack]{\upshape(B1)} in Lemma \ref{Lem:barrier-construction}, we establish the following doubling property for the lower bounds of supersolutions.
	\begin{lemma}[\bf Doubling property for supersolutions]
		\label{lem:Doubling}
		If $u \in C(\overline{B_2})$ is a non-negative viscosity supersolution in $B_2$ of an equation of the form \eqref{eq:main}, whose coefficients satisfy assumptions {\upshape\hyperref[A1]{(A1)}--\hyperref[A4]{(A4)}}, with $c\leq0$ and $\|h\|_{L^\infty(B_2)}\leq H_0$,
		and $u > M_2$ in $B_{1/4}$ for some large constant $M_2$, depending only on $n,\lambda,\Lambda,\nu_0,L,i(\Phi),s(\Phi)$, $\mathcal{C}$, and $\|c\|_{L^\infty(B_2)}$, then $u > 1$ in $B_1$.
	\end{lemma}
	\begin{proof}
		We compare the function $u$ with 
		\begin{align*}
			B(x):=K_B\left(|x|^{-p}-2^{-p}\right) \mbox{ with }K_B := \frac{M_2}{2\cdot 4^p}.
		\end{align*}
		Fix $p$ and $K_\ast$ as in the interior barrier \hyperref[barrier_harnack]{\upshape(B1)} in Lemma \ref{Lem:barrier-construction}. Since the proof of the interior barrier depends only on the structural assumptions and the corresponding quantitative bounds, the constants $p$ and $K_\ast$ can be chosen uniformly for the admissible class under consideration. Then choose $M_2 \geq 1$ sufficiently large that 
		\begin{align*}
			K_B \geq K_\ast, \mbox{ and } B > 1 \mbox{ in } B_1\setminus\overline{B_{1/4}}.
		\end{align*}
		By Lemma \ref{Lem:barrier-construction}, the function $B$ is a viscosity subsolution of the equation under consideration in $B_2\setminus\overline{B_{1/4}}$. Since $u \in C(\overline{B_2})$ and $u > M_2$ in $B_{1/4}$, then $u \geq M_2$ on $\partial B_{1/4}$. Moreover, we have
		\begin{align*}
			B =0 \leq u\mbox{ on }\partial B_2,\mbox{ and }B = \frac{M_2}{2}\left(1-2^{-3p}\right)< M_2 \mbox{ on }\partial B_{1/4}.
		\end{align*}
		Therefore, $u \geq B$ on $\partial\left(B_2\setminus \overline{B_{1/4}}\right)$. By the comparison principle in Remark \ref{23}, we obtain $B \leq u$ in the ring $B_2\setminus \overline{B_{1/4}}$. Therefore, $u\geq B>1$ in $B_1\setminus \overline{B_{1/4}}$. On the other hand, $u>M_2\geq1$ in $B_{1/4}$. Hence $u>1$ in $B_1$, which finishes the proof.
	\end{proof}
	Combining Lemmas \ref{Measure_est} and \ref{lem:Doubling}, we obtain
	\begin{corollary}
		\label{cor:measure-doubling}
		There exist a small constant $\delta>0$ and a large constant $M>0$, such that if $u\in C(\overline{B_2})$ is a non-negative viscosity supersolution in $B_2$ of an equation of the form \eqref{eq:main}, whose coefficients satisfy assumptions {\upshape\hyperref[A1]{(A1)}--\hyperref[A4]{(A4)}}, with $c\leq0$ and $\|h\|_{L^\infty(B_2)} \leq H_0$, and
		\begin{align*}
			|\{u > M\} \cap B_1| &> (1 - \delta)|B_1|,
		\end{align*}
		we deduce $u > 1$ in $B_1$.
	\end{corollary}
	\begin{proof}
		Let $M_1$ and $M_2$ be the constants from Lemma \ref{Measure_est} and Lemma \ref{lem:Doubling}. By enlarging $M_2$ if necessary, we may additionally assume that
		\begin{align*}
			\displaystyle \frac{L}{\nu_0}M_2^{-(1+i(\Phi))}H_0\leq1.
		\end{align*}
		Set $v:=u/M_2$. By the scaling result in Lemma \ref{Scaling}, $v$ is a supersolution of a rescaled equation of the form \eqref{scaling-form} with spatial scale $t_0 = 1$, radius $\rho_0=2$, center $x_0 = 0$, 
		$V=1/M_2$, and the rescaled right-hand side satisfies
		\begin{align*}
			\displaystyle\|h_v\|_{L^\infty(B_2)}\leq\frac{L}{\nu_0}M_2^{-(1+i(\Phi))}\|h\|_{L^\infty(B_2)}\leq1.
		\end{align*}
		Moreover, we also have
		\begin{align*}
			\left|\left\{v > M_1\right\}\cap B_1\right| = \left|\left\{u > M_1M_2\right\}\cap B_1\right| > (1-\delta)|B_1|,
		\end{align*}
		when choosing $M=M_1M_2$. Consequently, $v$ satisfies the assumptions of Lemma \ref{Measure_est}, we conclude that $v > 1$ in $B_{1/4}$, which means that $u > M_2$ in $B_{1/4}$. Next, we can apply Lemma \ref{lem:Doubling} to get $u > 1$ in $B_1$.
	\end{proof}
	The following corollary is just a scaled version of the above result.
	\begin{corollary}
		\label{cor:measure-doubling-scale}
		Let $x_0\in\mathbb{R}^n$. There exist a small constant $\delta > 0$ and a large constant $M > 1$ such that if $u\in C(\overline{B_R(x_0)})$ is a non-negative viscosity supersolution in $B_R(x_0)$ of an equation of the form \eqref{eq:main} for some $0 <R \leq 2$ and $\kappa \geq 1$ , whose coefficients satisfy assumptions {\upshape\hyperref[A1]{(A1)}--\hyperref[A4]{(A4)}}, with $c\leq0$, $\|h\|_{L^\infty(B_R(x_0))} \leq 1$, satisfying
		\begin{align*}
			|\{u > \kappa M\} \cap B_{R/2}(x_0)| > (1 - \delta)|B_{R/2}|,
		\end{align*}
		then $u > \kappa$ in $B_{R/2}(x_0)$.
	\end{corollary}
	\begin{proof}
		We define $u_R(y):=u(x_0+Ry/2)/\kappa$, for all $y\in B_2$. Then $u_R\geq 0$ in $B_2$. By scaling arguments as in Lemma \ref{Scaling}, the function $u_R$ is a viscosity supersolution of the corresponding rescaled equation in $B_2$ with spatial scale $t_0 = R/2$, radius $\rho_0=2$, center $x_0$, $V=1/\kappa$, and the rescaled right-hand side satisfies
		\begin{align*}
			\displaystyle\|h_R\|_{L^\infty(B_2)}\leq\frac{L}{\nu_0}\kappa^{-(1+i(\Phi))}\left(\dfrac{R}{2}\right)^{2+i(\Phi)}\|h\|_{L^\infty(B_R(x_0))}\leq H_0.
		\end{align*}
		Moreover,
		\begin{align*}
			\{u_R>M\}\cap B_1 =\left\{y\in B_1:u(x_0+Ry/2)>\kappa M\right\},
		\end{align*}
		and therefore, by the change of variables $z=x_0+Ry/2$,
		\begin{align*}
			|\{u_R>M\}\cap B_1|
			=\left(\frac{2}{R}\right)^n |\{u>\kappa M\}\cap B_{R/2}(x_0)|
			>(1-\delta)|B_1|.
		\end{align*}
		Therefore, $u_R$ satisfies the assumptions of Corollary \ref{cor:measure-doubling} for the corresponding rescaled equation. We infer that $u_R>1$ in $B_1$. Scaling back, we conclude that $u>\kappa$ in $B_{R/2}(x_0)$.
	\end{proof}
	Combining Corollary \ref{cor:measure-doubling} with Lemma \ref{lem: Growing-ink-spots}, we obtain the $L^\varepsilon$ estimate.
	\begin{theorem}[\bf $L^\varepsilon$ estimate]
		\label{Theo:L-estimate}
		There exist constants $\varepsilon>0$ and $C>0$, such that if $u\in C(\overline{B_2})$ is a non-negative viscosity supersolution in $B_2$ of an equation of the form \eqref{eq:main}, whose coefficients satisfy assumptions {\upshape\hyperref[A1]{(A1)}--\hyperref[A4]{(A4)}}, with $c\leq0$, $\|h\|_{L^\infty(B_2)} \leq 1$,
		and $\inf\limits_{B_1} u < 1$. Then the following estimate holds
		\begin{align}\label{L-ep}
			|\{u > t\} \cap B_1| \leq Ct^{-\varepsilon}\mbox{ for all } t > 0.
		\end{align}
	\end{theorem}
	\begin{remark}
		We refer to \eqref{L-ep} as the $L^\varepsilon$ estimate, since after possibly reducing $\varepsilon>0$, it yields the inequality $\displaystyle\int_{B_1}u^\varepsilon(x)dx\leq C$, where $C$ depends only on the structural data.
	\end{remark}
	\begin{proof}
		In order to establish our desired result, it is sufficient to demonstrate the following inequality
		\begin{align}\label{26.1}
			|\{u > M^k\} \cap B_1| \leq \tilde{C}M^{-\varepsilon k},
		\end{align}
		where $M$ is the constant given in Corollary \ref{cor:measure-doubling-scale} and $\varepsilon > 0$ is a suitably chosen parameter.\\
		Let us define open sets $A_k := \{u > M^k\} \cap B_1$. Since $\inf\limits_{B_1} u < 1$, from Corollary \ref{cor:measure-doubling}, we deduce 
		\begin{align}
			\label{L-estimate-eq1}
			|A_1| \leq (1-\delta)|B_1|.
		\end{align}
		Because $A_k \subset A_1$ for all $k >1$, we also have $|A_k| \leq (1-\delta)|B_1|$ for all $k$. 
		Moreover, let $B_\rho(x_0)\subset B_1$ for all $0 < \rho \leq 1$, then $|x_0| + \rho \leq 1$, implying $|x_0| + 2\rho  \leq 1 + \rho \leq 2$. Therefore, we collect $B_{2\rho}(x_0)\subset B_2$. Assume that $|B_\rho(x_0)\cap A_{k+1}|>(1-\delta)|B_\rho(x_0)|$. On the other hand,
		\begin{align*}
			A_{k+1}=\{u>M^{k+1}\}\cap B_1=\{u>M\cdot M^k\}\cap B_1,
		\end{align*}
		we can apply Corollary \ref{cor:measure-doubling-scale} to the restriction of $u$ on $B_{2\rho}(x_0)$, 
		with $R=2\rho$ and $\kappa=M^k$. Hence, we obtain $u>M^k$ in $B_\rho(x_0)$. Therefore, $B_\rho(x_0)\subset A_k$.
		Using Lemma \ref{lem: Growing-ink-spots} implies that
		\begin{align*}
			|A_{k+1}|\leq (1-\theta\delta)|A_k|.
		\end{align*}
		Hence, by induction, $|A_k|\leq (1-\theta\delta)^{k-1}|A_1|$. Thus from \eqref{L-estimate-eq1} we obtain 
		\begin{align}\label{26.2}
			|A_k| \leq (1-\theta\delta)^{k-1}(1-\delta)|B_1|.
		\end{align}
		We rewrite this geometric decay in the form of a power law. Specifically, we define  $\varepsilon=-\frac{\log(1-\theta\delta)}{\log M}$. Then the inequality \eqref{26.2} can be rewritten as
		\begin{align*}
			|A_k| \leq (1-\theta\delta)^{-1}(1-\delta)|B_1| M^{-\varepsilon k} =: \tilde{C} M^{-\varepsilon k}.
		\end{align*}
		This proves the desired estimate \eqref{26.1}, which is the discrete form of the $L^\varepsilon$ estimate.\\
		Finally, let $t>0$. If $0<t<M$, then the estimate \eqref{L-ep} follows by choosing $C$ sufficiently large, because $|\{u>t\}\cap B_1|\leq |B_1|$. If $t \geq M$, choose $k \geq 1$ such that $M^k \leq t< M^{k+1}$. Then
		\begin{align*}
			\{u>t\}\cap B_1\subset \{u>M^k\}\cap B_1,
		\end{align*}
		and hence
		\begin{align*}
			|\{u>t\}\cap B_1|
			\le \tilde{C} M^{-\varepsilon k}
			\le \tilde{C} M^\varepsilon t^{-\varepsilon}.
		\end{align*}
		After replacing $C$ by a larger constant, the desired estimate follows for all $t>0$.
	\end{proof}
	The following interior version follows by repeating the proof of Theorem \ref{Theo:L-estimate}, with the growing ink-spots iteration localized to $B_{3/4}$. At each propagation step, we use only balls whose fixed enlargement required in Corollary \ref{cor:measure-doubling-scale} remains contained in $B_1$. Since the remaining argument is unchanged, we omit the details.
	\begin{corollary}[\bf Interior $L^\varepsilon$ estimate]
		\label{remark:interior_form}
		There exist constants $\varepsilon>0$ and $C>0$ such that if $u\in C(\overline{B_1})$ is a non-negative viscosity supersolution in $B_1$ of an equation of the form \eqref{eq:main}, whose coefficients satisfy assumptions {\upshape\hyperref[A1]{(A1)}--\hyperref[A4]{(A4)}}, with $c\leq0$, $\|h\|_{L^\infty(B_1)} \leq 1$,
		and $\inf\limits_{B_{1/2}} u < 1$. Then the following estimate holds
		\begin{align*}
			|\{u>t\}\cap B_{3/4}|\le C t^{-\varepsilon} \text{ for all }t>0.
		\end{align*}
	\end{corollary}
	Finally, we will establish and prove the following Harnack inequality.
	\begin{theorem}[\bf Harnack inequality I]
		\label{Theo:Harnack}
		Assume that $c \leq 0$ and $u\in C(\overline{B_1})$ is a non-negative viscosity solution of \eqref{eq:main} in $B_1$, under the assumptions {\upshape\hyperref[A1]{(A1)}--\hyperref[A4]{(A4)}}. Then there exists a constant $C_H\geq1$, depending only on $n,\lambda,\Lambda,\nu_0,L,i(\Phi),s(\Phi),\mathcal{C}$, and $\|c\|_{L^\infty(B_1)}$ such that 
		\begin{align}
			\label{eq:Harnack_2}
			\displaystyle \sup_{B_{1/2}} u \leq C_H\left(\displaystyle \inf_{B_{1/2}}u +\|h\|_{L^\infty(B_1)}^{\frac{1}{i(\Phi)+1}} + 1\right).
		\end{align}
	\end{theorem}
	\begin{proof}
		Define a function $\hat{u}(y):=\displaystyle Vu(y)$. By Lemma \ref{Scaling}, with $x_0=0$, $\rho_0 = 1$, $t_0=1$, the function $\hat{u}$ is a viscosity solution of the rescaled equation associated with \eqref{eq:main}, namely 
		\begin{align*}
			\hat{\Phi}\left(y,|\nabla \hat{u}|\right)\hat{F}(D^2\hat{u}) -\hat{H}(y,\nabla \hat{u}) +\hat{c}(y)\hat{u}^{i(\Phi)+1} =\hat{h}(y) \mbox{ in }B_1,
		\end{align*}
		where
		\begin{align*}
			\begin{array}{ll}
				&\hat{\Phi}(y,t) =\displaystyle\frac{\Phi\left(y,\frac{t}{V}\right)}{\Phi\left(y,\frac{1}{V}\right)},~ \hat{F}(M) = \displaystyle VF\left(\frac{1}{V}M\right),\hat{H}(y,p) = \displaystyle\frac{V}{\Phi\left(y,\frac{1}{V}\right)}H\left(y,\frac{p}{V}\right),\\[15pt]
				&\hat{c}(y) = \displaystyle\frac{V^{-i(\Phi)}}{\Phi\left(y,\frac{1}{V}\right)}c(y),\text{ and } \hat{h}(y) = \displaystyle\frac{V}{\Phi\left(y,\frac{1}{V}\right)}h(y).\\
				&\textbf{}
			\end{array}
		\end{align*}
		Next, if one chooses
		\begin{align*}
			V = \left(\displaystyle\inf_{B_{1/2}}u +\left(\frac{L}{\nu_0}\|h\|_{L^\infty(B_1)} \right)^{\frac{1}{i(\Phi)+1}}+1\right)^{-1},
		\end{align*}
		then $0 < V \leq 1$ and we have
		\begin{align*}
			\inf_{B_{1/2}}\hat{u} = \frac{\displaystyle\inf_{B_{1/2}}u}{\displaystyle\inf_{B_{1/2}}u  +\left(\frac{L}{\nu_0}\|h\|_{L^\infty(B_1)}\right)^{\frac{1}{i(\Phi)+1}}+1}< 1.
		\end{align*}
		Additionally, the scaling estimate \eqref{18.2} from Lemma \ref{Scaling} gives us $\|\hat{h}\|_{L^\infty(B_1)} \leq 1$ and it ensures $\hat{\mathcal{C}}$ and $\|\hat{c}\|_{L^\infty(B_1)}$ are bounded in terms of the original structural data. Let $\beta > 0$ and define the function 
		\begin{align*}
			g_t(y) = t(3/4 - |y|)^{-\beta} \mbox{ in } B_{3/4}.
		\end{align*}
		We consider the minimum value of $t$ such that $g_t \geq \hat{u}$ in $B_{3/4}$. The proof aims to demonstrate that this minimum value cannot be arbitrarily large. Since the case $t \leq 1$ 
		is trivial, we shall focus on the case $t \geq 1$.\\
		Given that $t$ is the minimum value satisfying $g_t \geq \hat{u}$, there exists a point $z_\ast \in B_{3/4}$ where $g_t(z_\ast) = \hat{u}(z_\ast)$. Let $r_\ast = \left(3/4 - \left|z_\ast\right|\right)/2$, meaning that $2r_\ast$ corresponds to the distance from $z_\ast$ to $\partial B_{3/4}$. We then define
		\begin{align*}
			G_0 := g_t(z_\ast) = t(2r_\ast)^{-\beta} \geq 1.
		\end{align*}
		We intend to estimate the measure of the set $\{\hat{u} \geq G_0/2\} \cap B_{r_\ast}(z_\ast)$ in two different ways. A contradiction will arise if the minimum value $t$ is too large.
		We begin by establishing an upper bound for the measure. By 
		Corollary \ref{remark:interior_form}, applied to $\hat u$, since $B_{1/2}\Subset B_{3/4}\Subset B_1$, $\displaystyle\inf_{B_{1/2}}\hat{u}<1$ and $\|\hat{h}\|_{L^\infty(B_1)} \leq 1$, we obtain
		\begin{align}
			\label{Harnack:eq_1}
			\left| \left\{\hat{u} > G_0/2\right\} \cap B_{r_\ast}(z_\ast)\right| \leq \left| \left\{\hat{u} > G_0/2\right\} \cap B_{3/4} \right| \leq CG_0^{-\varepsilon} = Ct^{-\varepsilon}(2r_\ast)^{\beta \varepsilon}.
		\end{align}
		Next, we establish a lower bound. Fix a sufficiently large structural constant $\beta > 0$ and choose $\mu\in (0,1)$ sufficiently small structural constant satisfying
		\begin{align*}
			C\mu^{2+i(\Phi)}\leq\varepsilon_0,\mbox{ } A:=\left(\frac{2-\mu}{2}\right)^{-\beta}\leq 2, \mbox{ and }\beta \geq \frac{n}{\varepsilon},
		\end{align*}
		where $\varepsilon$ and $\varepsilon_0$ are the constants from Theorem \ref{Theo:L-estimate} and Corollary \ref{local_estimate_(1.1)}, respectively. For every $x\in B_1$, we have
		\begin{align*}
			\left|z_\ast+\mu r_\ast x\right|\leq |z_\ast|+\mu r_\ast = \frac{3}{4}-\left(2-\mu\right)r_\ast.
		\end{align*}
		As a result, $B_{\mu r_{*}}(z_{*})\subset B_{3/4}$. Since $\hat{u}\leq g_t$ in $B_{3/4}$, it follows that
		\begin{align*}
			\hat{u}(z_\ast+\mu r_\ast x)\leq t(2r_\ast-\mu r_\ast)^{-\beta}=A G_0 \mbox{ for all }x\in B_1.
		\end{align*}
		In fact, the maximum of $\hat{u}$ in the ball $B_{\mu r_\ast}(z_\ast)$ is bounded above by the maximum of $g_t$. Let us define
		\begin{align*}
			w(x):=\frac{\hat{u}(z_\ast+\mu r_\ast x)}{AG_0}~ \mbox{ in }B_1.
		\end{align*}
		Then $0\leq w\leq 1$ in $B_1$, and $w(0) = 1/A \geq 1/2$. 
		By Lemma \ref{Scaling}, applied with $x_0=z_\ast$, $\rho_0=1$, $t_0=\mu r_\ast$, $V=(AG_0)^{-1}\leq 1$, the function $w$ is a viscosity solution of a rescaled equation in $B_1$. Moreover, its rescaled coefficients satisfy
		\begin{align*}
			\mathcal C_w\le \frac{L}{\nu_0}(\mu r_\ast)^{2-m+i(\Phi)}\widehat{\mathcal{C}}, \mbox{ and } \|c_w\|_{L^\infty(B_1)}\le \frac{L}{\nu_0}(\mu r_\ast)^{2+i(\Phi)}\|\hat{c}\|_{L^\infty(B_1)}.
		\end{align*}
		Since $2-m+i(\Phi)\geq1$, $2+i(\Phi)>1$, and $\mu r_\ast\le1$, both $\mathcal C_w$ and $\|c_w\|_{L^\infty(B_1)}$ are bounded uniformly in terms of the original structural data. Furthermore, the rescaled right-hand side $h_w$ satisfies
		\begin{align*}
			\|h_w\|_{L^\infty(B_1)}\leq \frac{L}{\nu_0}(AG_0)^{-(1+i(\Phi))} (\mu r_\ast)^{2+i(\Phi)}\|\hat h\|_{L^\infty(B_1)}\leq C\mu^{2+i(\Phi)}.
		\end{align*}
		Hence, by the choice of $\mu$, we have $\|h_w\|_{L^\infty(B_1)} \leq \varepsilon_0$ with $0 < \varepsilon_0 <1$ and Corollary \ref{local_estimate_(1.1)} applies to $w$.
		Therefore, there exists $C_{zl}>0$ such that
		\begin{align*}
			[w]_{C^{0,1}(B_{1/2})}\leq C_{zl}.
		\end{align*}
		Choose $\theta:=\displaystyle\frac{1}{2}\min\left\{\frac{1}{2},\frac{1}{4C_{zl}}\right\}$. Then, for every $x\in B_\theta$,
		\begin{align*}
			w(x)\geq w(0)-C_{zl}|x| > \frac{1}{A}-\frac{1}{4} \geq \frac{1}{2A}. 
		\end{align*}
		It yields that
		\begin{align*}
			\hat{u}(z_\ast+\mu r_\ast x) = AG_0w(x)> \frac{G_0}{2} \text{ for all }x\in B_\theta,
		\end{align*}
		which leads to
		\begin{align*}
			\left|\{\hat{u}> G_0/2\}\cap B_{r_\ast}(z_\ast)\right| \geq |B_{\mu\theta r_\ast}(z_\ast)| =(\mu\theta)^n|B_{r_\ast}|.
		\end{align*}
		Combining this estimate with \eqref{Harnack:eq_1}, we have
		\begin{align*}
			(\mu\theta)^n|B_1|
			\leq C2^{\beta\varepsilon}t^{-\varepsilon}r_\ast^{\beta\varepsilon-n}.
		\end{align*}
		Since $\beta\geq n/\varepsilon$ and $0<r_\ast \leq 3/8 < 1$, we obtain $(\mu\theta)^n|B_1|
		\le C2^{\beta\varepsilon}t^{-\varepsilon}$. Hence $t\le C_\ast$, where $C_\ast>0$ depends only on the structural data so
		\begin{align*}
			\displaystyle \sup_{B_{1/2}}\hat{u} \leq t\left(\frac{3}{4} - \frac{1}{2}\right)^{-\beta} \leq 4^\beta C_\ast.
		\end{align*}
		Therefore, we arrive at
		\begin{align*}
			\displaystyle \sup_{B_{1/2}}u = V^{-1}\displaystyle \sup_{B_{1/2}}\hat{u} \leq 4^\beta C_\ast \left[\displaystyle\inf_{B_{1/2}}u +\left(\frac{L}{\nu_0}\|h\|_{L^\infty(B_1)}\right)^{\frac{1}{i(\Phi)+1}}+1\right].
		\end{align*}
		This completes the proof.
	\end{proof}
	\begin{remark}
		It is worth noting that the Harnack inequality established in Theorem \ref{Theo:Harnack} differs slightly from the classical Harnack inequality due to the appearance of an additional constant term on the right-hand side. However, we conclude by pointing out that Theorem \ref{Theo:Harnack}, as stated, is sharp in the sense that this additive constant cannot be removed under the assumptions {\upshape\hyperref[A1]{(A1)}--\hyperref[A4]{(A4)}}. To verify this assertion, let us consider the following explicit choices for $\Phi, F,$ and $H$:
		\begin{align*}
			\Phi(x, t) = 1,~F(M) = \mathrm{tr}(M), \mbox{ and }H(x, p) = |p|^{1/2},
		\end{align*}
		for all $x\in\mathbb{R}^n,~t\geq 0,~M\in\mathbb S^{n\times n}, ~p\in\mathbb{R}^n$. In this case, it is straightforward to check that these functions satisfy  assumptions {\upshape\hyperref[A1]{(A1)}--\hyperref[A3]{(A3)}}. 
		By setting $c = 0$ and $h = 0$, the equation \eqref{eq:main} can be rewritten as
		\begin{align}
			\label{eq:Couterexample_1}
			\Delta u - |\nabla u|^{1/2} = 0.
		\end{align}
		Let $u(x) = \frac{1}{12}(x_1)_+^3$. It is directly verified that $u \in C^2(\mathbb{R}^n)$ is a non-negative viscosity solution of \eqref{eq:Couterexample_1}. 
		Since $\displaystyle\inf_{B_{1/2}} u = 0$ and $\displaystyle\sup_{B_{1/2}} u > 0$, the estimate \eqref{eq:Harnack_2} is violated if there is not the additive constant term.
	\end{remark}
	Motivated by this sharpness, we observe that establishing a standard Harnack inequality, i.e., one without the additive constant term, requires stricter structural conditions. Specifically, by imposing the homogeneity requirements stipulated in assumption {\upshape\hyperref[A6]{(A6)}}, we can indeed eliminate this constant. Indeed, similar homogeneity assumptions have been employed in earlier literature to derive such standard Harnack-type estimates; see \cite{DFQ2010,BD2010-1}. This stronger condition, however, comes at the expense of generality, as the analysis is then restricted to the prototype case $\Phi(x, \xi) = |\xi|^p$; see Example \ref{Example-Phi}, and its natural generalizations of the form $b(x)|\xi|^p$, provided that the weight function $b$ satisfies appropriate regularity assumptions.\\
	\noindent\phantomsection\label{A6}\textbf{Assumption A6.} 
	\begin{enumerate}
		\item $F$ is positively homogeneous, i.e., $F(tM)=tF(M),\text{ for every }t>0,~M\in \mathbb{S}^{n\times n}$.
		\item $\Phi(x,t\xi)=t^{i(\Phi)}\Phi(x,\xi)$, for all $t,\xi>0,~x\in\mathbb{R}^n$.
		\item $H$ is positively $(i(\Phi)+1)$-homogeneous in $p$, that is, $H(x,tp) = t^{i(\Phi)+1} H(x,p)$ for all $t > 0$.
	\end{enumerate}	
	\begin{corollary}[\bf Harnack inequality II]
		\label{Theo:HarnackII}
		Assume that $c \leq 0$ and $u\in C(\overline{B_1})$ is a non-negative viscosity solution of \eqref{eq:main} in $B_1$, under the assumptions {\upshape\hyperref[A1]{(A1)}--\hyperref[A4]{(A4)}} and {\upshape\hyperref[A6]{(A6)}}. Then there exists a constant $C_H\geq1$, depending only on $n,\lambda,\Lambda,\nu_0,L,i(\Phi),s(\Phi),\mathcal{C}$, and $\|c\|_{L^\infty(B_1)}$ such that
		\begin{align}
			\label{eq:Harnack_II_2}
			\displaystyle \sup_{B_{1/2}} u \leq C_H\left(\displaystyle \inf_{B_{1/2}}u +\|h\|_{L^\infty(B_1)}^{\frac{1}{i(\Phi)+1}}\right).
		\end{align}
	\end{corollary}
	\begin{proof}
		Define $v(x):=Ku(x)$ for any $x \in B_1$ and $K > 0$. 
		Thanks to the assumption {\upshape\hyperref[A6]{(A6)}}, we arrive at
		\begin{align}
			\label{eq:Harnack_II_1}
			\begin{split}
				&\Phi(x,|\nabla v|)F(D^2v) - H(x,\nabla v) + c(x)v^{i(\Phi)+1}\\
				&= K^{i(\Phi)+1} \left(\Phi(x,|\nabla u|)F(D^2u) - H(x,\nabla u) + c(x)u^{i(\Phi)+1}\right)\\
				&=\tilde{h}(x),
			\end{split}
		\end{align}
		where $\tilde{h}(x) = K^{i(\Phi)+1}h(x)$. It means that if $u$ is a non-negative viscosity solution of \eqref{eq:main} then $v$ is a non-negative viscosity solution of \eqref{eq:Harnack_II_1} for all $K>0$. Applying Harnack inequality I - Theorem \ref{Theo:Harnack} for $v$, we deduce
		\begin{align*}
			\sup_{B_{1/2}}{v}\le C_H\left(\inf_{B_{1/2}}v+\|\tilde{h}\|_{L^\infty(B_1)}^{\frac{1}{i(\Phi)+1}} + 1\right).
		\end{align*}
		It yields that 
		\begin{align}\label{29.1}
			\displaystyle \sup_{B_{1/2}} u \leq C_H\inf_{B_{1/2}}u +C_H\|h\|_{L^\infty(B_1)}^{\frac{1}{i(\Phi)+1}} + \dfrac{C_H}{K}.
		\end{align}
		Letting $K \to \infty$ in \eqref{29.1}, then we conclude that
		\begin{align*}
			\displaystyle \sup_{B_{1/2}} u \leq C_H\left(\inf_{B_{1/2}}u + \|h\|_{L^\infty(B_1)}^{\frac{1}{i(\Phi)+1}}\right).
		\end{align*}
		This completes the proof.
	\end{proof}
	Finally, we state a scaled version of Corollary \ref{Theo:HarnackII}.
	\begin{lemma}[\bf Scaled Harnack inequality]
		\label{cor:scaled-Harnack}
		Let $c \leq 0$, $x_0\in\mathbb{R}^n$, and $0 < R \leq 1$. Suppose that $u\in C(\overline{B_{R}(x_0)})$ is a non-negative viscosity solution of \eqref{eq:main} in $B_R(x_0)$, under the assumptions {\upshape\hyperref[A1]{(A1)}--\hyperref[A4]{(A4)}} and {\upshape\hyperref[A6]{(A6)}}. There exists a constant $C_H\geq1$, depending only on $n,\lambda,\Lambda,\nu_0,L,i(\Phi),s(\Phi),\mathcal{C}$, and $\|c\|_{L^\infty(B_{R}(x_0))}$ such that
		\begin{align*} 
			\sup_{B_{R/2}(x_0)}u\leq C_H\left(\inf_{B_{R/2}(x_0)}u+ R^{\frac{2+i(\Phi)}{i(\Phi)+1}} \|h\|_{L^\infty(B_{R}(x_0))}^{\frac{1}{i(\Phi)+1}}\right).
		\end{align*}
	\end{lemma}
	\begin{proof}
		Define $v(y):=u(x_0+Ry)$ for any $y \in B_1$. 
		By assumption {\upshape\hyperref[A6]{(A6)}} and an analogous scaling argument in Lemma \ref{Scaling}, we also obtain that $v$ is a non-negative viscosity solution of
		\begin{align*}
			\Phi_s(y,|\nabla v|)F_s(D^2v) - H_s(y,\nabla v) + c_s(y)v^{i(\Phi)+1}= h_s(y)\mbox{ in }B_1,
		\end{align*}
		where
		\begin{align*}
			&\Phi_s(y,t) = \Phi(x_0+Ry,t),~F_s(M) = F(M),~H_s(y,p) = RH(x_0+Ry, p),\\
			&c_s(y) = R^{i(\Phi)+2}c(x_0 + Ry), \mbox{ and }h_s(y) = R^{i(\Phi)+2}h(x_0+Ry).
		\end{align*}
		Therefore, we may apply Corollary \ref{Theo:HarnackII} to $v$ and obtain
		\begin{align*}
			\sup_{B_{1/2}}v\leq C_H\left(\inf_{B_{1/2}}v+\|h_s\|_{L^\infty(B_1)}^{\frac{1}{1+i(\Phi)}}\right).
		\end{align*}
		By scaling back, we conclude that
		\begin{align*}
			\sup_{B_{R/2}(x_0)}u\leq C_H\left(\inf_{B_{R/2}(x_0)}u+R^{\frac{2+i(\Phi)}{i(\Phi)+1}}\|h\|_{L^\infty(B_R(x_0))}^{\frac{1}{1+i(\Phi)}}\right),
		\end{align*}
		as claimed.
	\end{proof}
	\section{\bf Dirichlet principal eigenvalue problem}\label{section: 5}
	This section is devoted to the Dirichlet principal eigenvalue problem for the homogeneous operator, under the assumption that $\Omega$ is a bounded $C^2$ domain.
	
	To the best of our knowledge, the standard qualitative properties of generalized principal eigenvalues are not available at this level of generality without a suitable structural homogeneity assumption; we refer the reader to \cite{BD2007} and the references therein for further details. Consequently, we impose an additional assumption on $\Phi, F,$ and $H$ to ensure the homogeneity of the operator $\mathcal{G}$, defined in \eqref{mathcal_G}; see assumption {\upshape\hyperref[A6]{(A6)}}. 
	\begin{definition}\label{pre}
		We define the generalized principal eigenvalues of the operator $\mathcal{G}$ as follows
		\begin{align*}
			\beta^+=\sup\{\beta\in\mathbb{R}:\exists\psi\in C(\overline{\Omega}),&\mathcal{G}(x,\psi,\nabla \psi,D^2\psi)+\beta\psi^{i(\Phi)+1}\le 0\text{ in }\Omega,\\
			&\text{ and } \psi>0\text{ in }\overline{\Omega}\},
		\end{align*}
		\begin{align*}
			\beta^-=\sup\{\beta\in\mathbb{R}:\exists\psi\in C(\overline{\Omega}),&\mathcal{G}(x,\psi,\nabla \psi,D^2\psi)+\beta|\psi|^{i(\Phi)}\psi\geq 0\text{ in }\Omega,\\
			&\text{ and }\psi<0\text{ in }\overline{\Omega}\}.
		\end{align*}
	\end{definition}
	\begin{remark}
		When both $F$ and $H$ are odd functions, one can follow arguments similar to those in \cite{AB2008} to show that $\beta^+=\beta^-$. In general, however, these two quantities may differ; see \cite[Section 2]{AB2008} for examples in this direction. Additionally, Definition \ref{pre} is similar to the one considered in \cite{BD2006}.
	\end{remark}
	
	\begin{proposition} Suppose the assumptions {\upshape\hyperref[A1]{(A1)}--\hyperref[A6]{(A6)}} hold. Then we have
		\begin{enumerate}[\upshape(i)]
			\item \phantomsection\label{eigen_result_1}  $\beta^+>-\infty$. That means there exists $\beta\in\mathbb{R}$ such that the equation
			\begin{align}\label{17.1}
				\Phi(x,|\nabla u|)F(D^2u)-H(x,\nabla u)+(c(x)+\beta)u^{i(\Phi)+1}=0\text{ in }\Omega
			\end{align}
			admits a positive supersolution.
			\item \phantomsection\label{eigen_result_2}  $\beta^->-\infty$. That means there exists $\beta\in\mathbb{R}$ such that the equation
			\begin{align}\label{17.2}
				\Phi(x,|\nabla u|)F(D^2u)-H(x,\nabla u)+(c(x)+\beta)|u|^{i(\Phi)}u=0\text{ in }\Omega
			\end{align}
			admits a negative subsolution.
		\end{enumerate}
	\end{proposition}
	\begin{proof}
		\textbf{}
		
		\textit{Step 1: Proof of \upshape\hyperref[eigen_result_1]{(i)}.} It is clear that $\beta^+\geq -\|c\|_{L^{\infty}(\Omega)}$, since $\psi\equiv 1$ is a positive supersolution to the equation \eqref{17.1} with $\beta=-\|c\|_{L^{\infty}(\Omega)}$.
		
		\textit{Step 2: Proof of \upshape\hyperref[eigen_result_2]{(ii)}.} Similarly, it is easy to check that $\beta^-\geq  -\|c\|_{L^{\infty}(\Omega)}$, with $\psi\equiv -1$ being the negative subsolution to the equation \eqref{17.2} for $\beta=-\|c\|_{L^{\infty}(\Omega)}$.
	\end{proof}
	\begin{lemma}\label{dist}
		Assume {\upshape\hyperref[A1]{(A1)}--\hyperref[A6]{(A6)}} hold. Let $u\in C(\overline\Omega)$ be a positive viscosity solution of
		\begin{align*}
			\Phi(x,|\nabla u|)F(D^2u)-H(x,\nabla u)+c(x)u^{i(\Phi)+1}=h(x)\text{ in }\Omega,
		\end{align*}
		and $u=0$ on $\partial\Omega$, we have the following estimate for all $x\in\Omega$ 
		\begin{align*}
			u(x)\le C\mathrm{dist}(x,\partial\Omega),
		\end{align*}
		for some constant $C$ depending on $\|u\|_{L^\infty(\Omega)},n,\lambda,\Lambda,L,\nu_0,i(\Phi),\mathcal{C},\|c\|_{L^{\infty}(\Omega)},$ and $\|h\|_{L^{\infty}(\Omega)}$.
	\end{lemma}
	\begin{proof}
		Choose $k>\max\{1,\|u\|_{L^\infty(\Omega)}\}$. Then we readily obtain that
		\begin{align*}
			&\Phi(x,|\nabla u|)F(D^2u)-H(x,\nabla u)+(c(x)-k\|c\|_{L^{\infty}(\Omega)})u^{i(\Phi)+1}\\
			&\geq h-k^{i(\Phi)+2}\|c\|_{L^{\infty}(\Omega)}.
		\end{align*}
		Fix $z\in\partial\Omega$ and let $B_{r_1}(y_z)$ be the ball that touches $\Omega$ from outside at the point $z$. Recalling Lemma \ref{Lem:barrier-construction} for barrier function \hyperref[barrier_eigenvalue]{\upshape(B2)}, we have
		\begin{align*}
			&\Phi(x,|\nabla \chi^z_{\mathrm{ext}}|)F(D^2\chi^z_{\mathrm{ext}})-H(x,\nabla \chi^z_{\mathrm{ext}})+(c(x)-k\|c\|_{L^{\infty}(\Omega)})(\chi^z_{\mathrm{ext}})^{i(\Phi)+1}\\
			&\leq h-k^{i(\Phi)+2}\|c\|_{L^{\infty}(\Omega)}.
		\end{align*}
		Now $K_2$ can be chosen large enough so that $\chi_{\mathrm{ext}}^z(x)\geq k$ on $\{|x-y_z|=r_1+\delta\}$.  
		Then by comparison principle II (Lemma \ref{Comp.Prin}), it follows that $u\le \chi_{\mathrm{ext}}^z$ in $\Omega\cap (B_{r_1+\delta}(y_z)\setminus B_{r_1}(y_z))$. Since $z$ is arbitrary, we conclude the result for $u$. 
	\end{proof}
	\begin{theorem}[\bf Maximum principle]
		\label{max.prin}
		Suppose the assumptions {\upshape\hyperref[A1]{(A1)}--\hyperref[A6]{(A6)}} are in force. The following properties are satisfied
		\begin{enumerate}[\upshape(i)]
			\item \phantomsection\label{Max_prin_1} Let $\beta<\beta^+$. Then for any subsolution $u\in\mathrm{USC}(\overline{\Omega})$ of the equation
			\begin{align*}
				\Phi(x,|\nabla u|)F(D^2u)-H(x,\nabla u)+(c(x)+\beta)u_{+}^{i(\Phi)+1}=0\text{ in }\Omega,
			\end{align*}
			with $u\le 0$ on $\partial\Omega$, we must have $u\le 0$ in $\Omega$.  
			\item \phantomsection\label{Max_prin_2} Let $\beta<\beta^-$. For any supersolution $u\in\mathrm{LSC}(\overline{\Omega})$ of the problem
			\begin{align*}
				\Phi(x,|\nabla u|)F(D^2u)-H(x,\nabla u)+(c(x)+\beta)|u|^{i(\Phi)}u=0\text{ in }\Omega,
			\end{align*}
			with $u\geq 0$ on $\partial\Omega$, then $u\geq 0$ in $\Omega$.  
		\end{enumerate}
	\end{theorem}
	\begin{proof}
		\textbf{}
		
		\textit{Step 1: Proof of \upshape\hyperref[Max_prin_1]{(i)}.} By definition we can find $\beta_1\in (\beta,\beta^+)$ and $v\in C(\overline\Omega)$, $v>0$ in $\overline\Omega$,
		\begin{align}\label{33.1'}
			\Phi(x,|\nabla v|)F(D^2v)-H(x,\nabla v)+(c(x)+\beta_1)v^{i(\Phi)+1}\le 0\text{ in }\Omega.
		\end{align}
		Suppose $u^+\neq 0$ in $\Omega$ and set $\kappa=\max\limits_{\Omega}{\dfrac{u^+}{v}}>0$. 
		Since $v > 0$ in $\overline{\Omega}$ and $u \leq 0$ on $\partial \Omega$, we have $u \leq \kappa v$ in $\overline{\Omega}$, with equality at some point $z \in \Omega$ such that $u(z) = \kappa v(z) > 0$. Moreover, thanks to the assumption {\upshape\hyperref[A6]{(A6)}}, it follows that
		\begin{align}\label{33.2}
			\begin{split}
				&\Phi(x,\kappa|\nabla v|)F(\kappa D^2v)-H(x,\kappa\nabla v)+(c(x)+\beta_1)\kappa^{i(\Phi)+1}v^{i(\Phi)+1}\\
				= &~\kappa^{i(\Phi)+1}[\Phi(x,|\nabla v|)F(D^2v)-H(x,\nabla v)+(c(x)+\beta_1)v^{i(\Phi)+1}].
			\end{split}
		\end{align}
		Combining \eqref{33.1'} and \eqref{33.2}, we obtain that
		\begin{align*}
			\Phi(x,|\nabla (\kappa v)|)F(D^2(\kappa v))-H(x,\nabla (\kappa v))+(c(x)+\beta_1)(\kappa v)^{i(\Phi)+1}\le 0\text{ in }\Omega.
		\end{align*}
		Applying the doubling-variable and zero-gradient cutting argument (Lemma \ref{Lem: cutting_zero_grad}) from the proof of Lemma \ref{Comp.Prin} to the subsolution $u$ and the supersolution $\kappa v$, we obtain, after passing to the limit, for some $z_0\in\Omega$ satisfying $u(z_0) = \kappa v(z_0) >0$
		\begin{align*}
			(c(z_0)+\beta)u_{+}^{i(\Phi)+1}(z_0)\geq (c(z_0)+\beta_1)(\kappa v(z_0))^{i(\Phi)+1},
		\end{align*}
		leading to $\beta_1\le \beta$. This is a contradiction, thus the result follows.
		
		\textit{Step 2: Proof of \upshape\hyperref[Max_prin_2]{(ii)}.} The proof follows directly from that of \hyperref[Max_prin_1]{\upshape (i)} and will therefore not be repeated.
	\end{proof}
	Now, we show the quantity $\beta^+$ is finite. A similar result holds for the principal eigenvalue $\beta^-$. By the fact $\Omega$ contains a ball $B_R$ and the domain monotonicity of the generalized principal eigenvalue, it is sufficient to prove $\beta^+$ is finite over balls.
	\begin{theorem}
		Let $\beta_R^+$ be the principal eigenvalue in $B_R$, which is defined in Definition \ref{pre}. Then it holds that $\beta_R^+<\infty$.
	\end{theorem}
	\begin{proof}
		Let us consider the function $v(x)=K(|x|^{\alpha}-R^{\alpha})^2, x\in B_R$. We demonstrate that for some $K>0$ small enough and $\alpha>\max\left\{2,\frac{i(\Phi)+2}{i(\Phi)+1}\right\}$, $\beta>0$ large enough, we obtain
		\begin{align*}
			\Phi(x,|\nabla v|)F(D^2v)-H(x,\nabla v)+(c(x)+\beta)v^{i(\Phi)+1}\geq 0\text{ in }B_R.
		\end{align*}
		A straightforward calculation shows that
		\begin{align*}
			&\nabla v(x)=2K\alpha(|x|^{\alpha}-R^{\alpha})|x|^{\alpha-2}x,\\
			&D^2v(x)=2K\alpha\left[((2\alpha-1)|x|^{2\alpha-2}-(\alpha-1)R^{\alpha}|x|^{\alpha-2})\dfrac{x\otimes x}{|x|^2}\right.\\
			&\hspace*{4cm}\left.+(|x|^{2\alpha-2}-R^{\alpha}|x|^{\alpha-2})\left(I-\dfrac{x\otimes x}{|x|^2}\right)\right].
		\end{align*}
		Let $\rho=\left(\dfrac{\alpha-1}{2\alpha-1}\right)^{\frac{1}{\alpha}}R$. It is easy to check that $\rho<R$ for $\alpha>1$. Let us distinguish two cases:
		
		\textbf{Case 1:} 
		Now if $|x|< \rho$, a simple calculation reveals that
		\begin{align*}
			\mathcal{P}_{\lambda,\Lambda}^+(D^2v(x))=2\lambda K\alpha|x|^{\alpha-2}D, \mbox{ and }\mathcal{P}_{\lambda,\Lambda}^-(D^2v(x))=2\Lambda K\alpha|x|^{\alpha-2}D,
		\end{align*}
		where $D = (2\alpha+n-2)|x|^{\alpha}+(2-\alpha-n)R^{\alpha}$. Moreover, from \eqref{Pucci_ope.}, we have 
		\begin{align*}
			2\Lambda K\alpha|x|^{\alpha-2}D \leq F(D^2v(x)) \leq 2\lambda K\alpha|x|^{\alpha-2}D.
		\end{align*} 
		Consequently, we deduce that $F(D^2v(x))<0$. Additionally, for small $K$ and applying the assumptions {\upshape\hyperref[A2]{(A2)}} and {\upshape\hyperref[A6]{(A6)}}, we arrive at 
		\begin{align*}
			&|\nabla v(x)|=2K\alpha(R^{\alpha}-|x|^{\alpha})|x|^{\alpha-1}\le 2K\alpha R^{\alpha}\rho^{\alpha-1}\le 1,\\ &\Phi(x,|\nabla v|)\le \nu_1|\nabla v|^{i(\Phi)}.
		\end{align*}
		It yields that 
		\begin{align*}
			&\Phi(x,|\nabla v|)F(D^2v)-H(x,\nabla v)+(c(x)+\beta)v^{i(\Phi)+1}\\
			&\geq (R^{\alpha}-|x|^{\alpha})^{i(\Phi)}[\nu_1\Lambda(2K\alpha)^{i(\Phi)+1}R^{(\alpha-1)i(\Phi)+\alpha-2}(2-\alpha-n)R^{\alpha}\\
			&\quad-\mathcal{C}(2K\alpha)^{i(\Phi)+1}R^{\alpha}R^{(\alpha-1)(i(\Phi)+1)}\\ &\quad+(\beta-\|c\|_{L^{\infty}(B_R)})K^{i(\Phi)+1}(R^{\alpha}-\rho^{\alpha})^{i(\Phi)+2}].
		\end{align*}
		Choosing $\alpha>\max\left\{2,\frac{i(\Phi)+2}{i(\Phi)+1}\right\}$ and $\beta>\|c\|_{\infty}$ large enough, we conclude
		\begin{align}\label{5.1}
			\Phi(x,|\nabla v|)F(D^2v)-H(x,\nabla v)+(c(x)+\beta)v^{i(\Phi)+1}\geq 0\text{ in }B_\rho.
		\end{align}
		\textbf{Case 2:} $\rho\leq |x|<R$, we calculate the value of Pucci extremal operators as follows
		\begin{align*}
			&\mathcal{P}_{\lambda,\Lambda}^+(D^2v(x))=\mathcal{D}[(\Lambda(2\alpha-1)+\lambda(n-1))|x|^{\alpha}-(\Lambda(\alpha-1)+\lambda(n-1))R^{\alpha}],\\
			&\mathcal{P}_{\lambda,\Lambda}^-(D^2v(x))=\mathcal{D}[(\lambda(2\alpha-1)+\Lambda(n-1))|x|^{\alpha}-(\lambda(\alpha-1)+\Lambda(n-1))R^{\alpha}],
		\end{align*}
		where $\mathcal{D} = 2K\alpha|x|^{\alpha-2}$. We can easily check that $\mathcal{P}_{\lambda,\Lambda}^+(D^2v(x))\le 0$ for every $x$ with $|x| = \rho$. Therefore, we can follow similar arguments as in \textbf{Case 1} to choose $K$ small enough and $\beta>0$ large enough such that
		\begin{align*}
			\Phi(x,|\nabla v|)F(D^2v)-H(x,\nabla v)+(c(x)+\beta)v^{i(\Phi)+1}>0\text{ for }|x|=\rho.
		\end{align*}
		Due to continuity, there exists $\rho_0\in (\rho,R)$ such that
		\begin{align}\label{5.2}
			\Phi(x,|\nabla v|)F(D^2v)-H(x,\nabla v)+(c(x)+\beta)v^{i(\Phi)+1}>0\text{ in }\rho\le |x|<\rho_0.
		\end{align}
		Finally, we assert that we can choose $\beta>0$ large enough for $\rho_0\le |x|<R$, we have
		\begin{align}\label{5.3}
			\Phi(x,|\nabla v|)F(D^2v)-H(x,\nabla v)+(c(x)+\beta)v^{i(\Phi)+1}>0.
		\end{align}
		If $\mathcal{P}_{\lambda,\Lambda}^-(D^2v(x))\le 0$ for every $x$ with $|x| = \rho_0$, we can apply similar arguments as in \textbf{Case 1} to obtain the conclusion. Now, we consider the case  $\mathcal{P}_{\lambda,\Lambda}^-(D^2v(x))>0$ for $|x|=\rho_0$. Thanks to the assumption {\upshape\hyperref[A6]{(A6)}}, we have $\Phi(x,|\nabla v|)\geq \nu_0|\nabla v|^{i(\Phi)}$, thus this case is similar to \textbf{Case 1}.
		
		From \eqref{5.1}, \eqref{5.2}, and \eqref{5.3}, it follows that
		\begin{align*}
			\Phi(x,|\nabla v|)F(D^2v)-H(x,\nabla v)+(c(x)+\beta)v^{i(\Phi)+1}\geq 0\text{ in } B_R.
		\end{align*}
		By Theorem \ref{max.prin}, we must have $\beta_R^+\le \beta$, which allows us to obtain the conclusion.
	\end{proof}
	Next, we prove strong maximum principle and a Hopf estimate.
	\begin{theorem}[\bf Strong maximum principle]\label{strong-com}
		Under the assumptions {\upshape\hyperref[A1]{(A1)}--\hyperref[A6]{(A6)}}, if $u\in\mathrm{LSC}(\overline{\Omega})$ is a non-negative viscosity supersolution of
		\begin{align*}
			\Phi(x,|\nabla u|)F(D^2u)-H(x,\nabla u)+c(x)u^{i(\Phi)+1}=0\mbox{ in }\Omega,
		\end{align*}
		then either $u\equiv 0$ or $u>0$ in $\Omega$. Additionally, suppose that $u(x)>u(x_0)=0$ for all $x\in\Omega$ and some $x_0\in\partial\Omega$. Then there exists a positive constant $a$ such that for all $x\in B_{\delta}(y_0)$ we have
		\begin{align}\label{Hopf}
			u(x)\geq a(\delta-|x-y_0|),
		\end{align}
		where $B_{\delta}(y_0)\subset\Omega$ is a ball touching the point $x_0$.
	\end{theorem}
	\begin{proof}
		Without loss of generality, we may suppose that $c<0$ in $\overline{\Omega}$. Now, we prove that $u>0$ in $\Omega$ in the case $u\gneq 0$ in $\Omega$. For the sake of contradiction, assume that there exist
		$x_0\in\Omega$ and $r>0$ such that $B_{2r}(x_0)\subset\Omega$, $u>0$ in $B_r(x_0)$ and $u(z)=0$ for some $z\in\partial B_r(x_0)$.  For the sake of simplicity, we also assume that $x_0=0$. Let $v(x)=e^{-\alpha|x|}-e^{-\alpha r}$. Then $v>0$ in $B_r$. Let us choose $\alpha>0$ large enough so that for $r/2\le |x|\le r$, we have  $\mathcal{P}_{\lambda,\Lambda}^-(D^2v(x))>0$.
		Therefore, we obtain that
		\begin{align}\label{33.1}
			\begin{split}
				&\Phi(x,|\nabla v|)F(D^2v)-H(x,\nabla v)+c(x)v^{i(\Phi)+1}\\
				&\geq e^{-\alpha|x|(i(\Phi)+1)}\left[\nu_0 \alpha^{i(\Phi)}\left(\lambda\alpha^2-\dfrac{\Lambda\alpha(n-1)}{|x|}\right)-C\alpha^{i(\Phi)+1}\right.\\
				&\hspace*{4cm}\left.-\|c\|_{L^{\infty}(B_{2r})}(1-e^{-\alpha(r-|x|)})^{i(\Phi)+1}\right]>0,
			\end{split}
		\end{align}
		provided that $\alpha>0$ is chosen sufficiently large. It is possible to choose $\kappa$ small enough satisfying $\kappa v\le u$ on $\partial B_{r/2}$ because $u>0$ on $\partial B_{r/2}$. By comparison principle II (Lemma \ref{Comp.Prin}), we obtain that $u\geq \kappa v$ for $r/2\le |x|\le r$.  On the other hand, $\kappa v\le 0\le u$ on $B_{2r}\setminus B_r$. As a result, we conclude that $\kappa v$ touches $u$ from below at $z$ and $u(z)=0$. Applying the definition of viscosity supersolution, we deduce that
		\begin{align*}
			\Phi(z,\kappa|\nabla v|)F(\kappa D^2v(z))-H(z,\kappa \nabla v(z))+c(z)(\kappa v(z))^{i(\Phi)+1}\le 0,
		\end{align*}
		which contradicts \eqref{33.1}. This proves the first part of this theorem.
		
		In order to verify the estimate \eqref{Hopf}, we  follow analogous techniques as the proof of the previous strong maximum principle and apply the elementary inequality $e^{-k\rho}-e^{-kr}\geq ke^{-kr}(r-\rho)$, for all $k>0$ and $0<\rho\le r$. This completes our proof.
	\end{proof}
	\begin{lemma}\label{normalize} 
		Suppose the assumptions {\upshape\hyperref[A1]{(A1)}--\hyperref[A6]{(A6)}} are satisfied. Then the following results hold
		\begin{enumerate}[\upshape(i)]
			\item \phantomsection\label{Normal_eigen_1} Let $\beta<\beta^+$. Then there exists a positive Lipschitz continuous viscosity solution $u$ of
			\begin{align}
				\label{eigen_normalize_1}
				\left\{\begin{array}{clll}
					\Phi(x,|\nabla u|)F(D^2u)-H(x,\nabla u)+(c(x)+\beta)u^{i(\Phi)+1}&=-1&\text{in }\Omega,\\
					u&=~~0 &\mbox{on }\partial\Omega.
				\end{array}\right.
			\end{align}
			\item \phantomsection\label{Normal_eigen_2} Let $\beta<\beta^{-}$. Then there exists a negative Lipschitz continuous viscosity solution $u$ of
			\begin{align*}
				\left\{\begin{array}{clcl}
					\Phi(x,|\nabla u|)F(D^2u)-H(x,\nabla u)+(c(x)+\beta)|u|^{i(\Phi)}u&=1&\text{in }\Omega,\\
					u&=0 &\mbox{on }\partial\Omega.
				\end{array}\right.
			\end{align*}
		\end{enumerate}			
	\end{lemma}
	\begin{proof}
		\textbf{}
		
		\textit{Step 1: Proof of \upshape\hyperref[Normal_eigen_1]{(i)}.} Remark that $0$ is a subsolution. Thanks to Definition \ref{pre}, for any $\beta_1\in (\beta,\beta^+)$, there exists $v\in C(\overline{\Omega})$, $v>0$ in $\overline\Omega$, and satisfies
		\begin{align*}
			\Phi(x,|\nabla v|)F(D^2v)-H(x,\nabla v)+(c(x)+\beta_1)v^{i(\Phi)+1}\le 0\text{ in }\Omega.
		\end{align*}
		Therefore, for large $\gamma>1$, we have
		\begin{align*}
			\Phi(x,|\gamma\nabla v|)F(D^2(\gamma v))-H(x,\gamma \nabla v)+c(x)(\gamma v)^{i(\Phi)+1}&\le -\beta_1(\gamma v)^{i(\Phi)+1}\\
			&\le -\beta (\gamma v)^{i(\Phi)+1}-1.
		\end{align*}
		It yields that $\gamma v$ is a supersolution of \eqref{eigen_normalize_1}. Next, choosing a constant $\mu > \max\{\|c\|_{L^\infty(\Omega)}, -\beta\}$, we construct a sequence of functions $(u_k)_{k \in \mathbb{N}}$ inductively as follows: set $u_0 = 0$, and for each $k \in \mathbb{N}$, let $u_{k+1}$ be the viscosity solution of the problem
		\begin{align*}
			\left\{\begin{array}{clll}
				\Phi(x,|\nabla u_{k+1}|)F(D^2u_{k+1})-H(x,\nabla u_{k+1})+(c(x)-\mu)u_{k+1}^{i(\Phi)+1}&\hspace*{-0.29cm}=\mathcal{S}_k&\hspace*{-0.22cm}\text{in }\Omega,\\
				u_{k+1}&\hspace*{-0.29cm}=0 &\hspace*{-0.22cm}\mbox{on }\partial\Omega.
			\end{array}\right.
		\end{align*}
		where $\mathcal{S}_k = -1 - (\beta + \mu)u_k^{i(\Phi)+1}$. The sequence $(u_k)_{k\in\mathbb{N}}$ constructed in this manner is well-defined. Indeed, we observe that
		\begin{align*}
			&\Phi(x,|\gamma\nabla v|)F(D^2(\gamma v))-H(x,\gamma \nabla v)+(c(x)-\mu)(\gamma v)^{i(\Phi)+1}\\
			&\leq -(\beta+\mu) (\gamma v)^{i(\Phi)+1}-1.
		\end{align*}
		Consequently, the function $\gamma v$ is a viscosity supersolution of the problem
		\begin{align}
			\label{eigen_normalize_2}
			\left\{\begin{array}{clll}
				\Phi(x,|\nabla z|)F(D^2z)-H(x,\nabla z)+(c(x)-\mu)z^{i(\Phi)+1}&=\mathcal{S}_0 &\text{in }\Omega,\\
				z&=0 &\mbox{on }\partial\Omega.
			\end{array}\right.
		\end{align}
		Furthermore, by Theorem \ref{Theo: Exist}, problem \eqref{eigen_normalize_2} admits a viscosity solution $u_1$, satisfying $0 \leq u_1 \leq \gamma v$ in $\overline{\Omega}$. Proceeding inductively and using the comparison principle II (Lemma \ref{Comp.Prin}), together with the fact that $\beta + \mu >0$, we can establish the existence of a function $u_{k+1}$, given $u_k$, such that
		\begin{align*}
			0 \le u_k \le u_{k+1} \le \gamma v \mbox{ in }\overline{\Omega},
		\end{align*}
		for every $k \in \mathbb{N}$. Thus, the sequence $(u_k)$ is monotone increasing and uniformly bounded by $\|\gamma v\|_{L^\infty(\overline{\Omega})}$, which implies its pointwise convergence on $\overline{\Omega}$. Furthermore, Corollary \ref{local_estimate_(1.1)} and Lemma \ref{dist} ensure that the sequence $(u_k)$ is equicontinuous. By the Arzelà-Ascoli theorem, $(u_k)$ admits a subsequence that converges uniformly to some function $u$. It is straightforward to verify that $u$ is a viscosity solution of problem \eqref{eigen_normalize_1}. Finally, the strong maximum principle (Theorem \ref{strong-com}) yields $u > 0$ in $\Omega$.
		
		\textit{Step 2: Proof of \upshape\hyperref[Normal_eigen_2]{(ii)}.} The proof is strictly analogous to that of part \hyperref[Normal_eigen_1]{\upshape(i)}. 
		Additionally, we notice a slight difference in proving the strict negativity of the solution $u$, which is the limit of the sequence of monotone iteration. Let us define $\tilde{F}(M)=-F(-M)$ and $\tilde{H}(x,p)=-H(x,-p)$. We can easily check that $w=-u$ is a nonnegative viscosity supersolution of the equation
		\begin{align*}
			\Phi(x,|\nabla w|)\tilde{F}(D^2w)-\tilde{H}(x,\nabla w)+(c(x)+\beta)w^{i(\Phi)+1}=0\text{ in }\Omega,
		\end{align*}
		where $\tilde{F}$ and $\tilde{H}$ satisfy {\upshape\hyperref[A1]{(A1)}}, {\upshape\hyperref[A3]{(A3)}}, {\upshape\hyperref[A5]{(A5)}}, and {\upshape\hyperref[A6]{(A6)}}. Therefore, we can apply strong maximum principle (Theorem \ref{strong-com}) for $w$ and then obtain our desired result. 
	\end{proof}
	\begin{theorem}\label{eigenfuction}
		Suppose the assumptions {\upshape\hyperref[A1]{(A1)}--\hyperref[A6]{(A6)}} hold. Then, we have the following results
		\begin{enumerate}[\upshape(i)]
			\item \phantomsection\label{eigenfunction_1} There exists a positive solution $\varphi^+$ of 
			\begin{align*}
				\left\{\begin{array}{cll}
					\Phi(x,|\nabla \varphi^+|)F(D^2\varphi^+)-H(x,\nabla \varphi^+)+(c(x)+\beta^+)(\varphi^+)^{i(\Phi)+1}&\hspace*{-0.2cm}=0\text{ in }\Omega,\\
					\varphi^+&\hspace*{-0.2cm}=0 \mbox{ on }\partial\Omega.
				\end{array}\right.
			\end{align*}
			Moreover, $\varphi^+$ is Lipschitz continuous in $\overline{\Omega}$.
			\item \phantomsection\label{eigenfunction_2} There exists a negative solution $\varphi^-$ of 
			\begin{align*}
				\left\{\begin{array}{clr}
					\Phi(x,|\nabla \varphi^-|)F(D^2\varphi^-)-H(x,\nabla \varphi^-)+(c(x)+\beta^-)|\varphi^-|^{i(\Phi)}\varphi^-&\hspace*{-0.2cm}=0\text{ in }\Omega,\\
					\varphi^-&\hspace*{-0.2cm}=0 \mbox{ on }\partial\Omega.
				\end{array}\right.
			\end{align*}
			Moreover, $\varphi^-$ is Lipschitz continuous in $\overline{\Omega}$.
		\end{enumerate}
	\end{theorem}
	\begin{proof}
		\textbf{}
		
		\textit{Step 1: Proof of \upshape\hyperref[eigenfunction_1]{(i)}.} Let $(\psi_n,\beta_n)$ be a sequence of solutions from Lemma \ref{normalize}\hyperref[Normal_eigen_1]{\upshape(i)} and $\beta_n\to\beta^+$.  We verify that $\{\psi_n\}_{n\in\mathbb{N}}$ is unbounded in $L^\infty(\Omega)$. If not, thanks to Theorem \ref{Lipschitz_estimate} and Lemma \ref{dist}, there exists a subsequence of $\{\psi_n\}$, converging to $\varphi$ with $\|\varphi\|_{L^{\infty}(\Omega)}>0$, and
		\begin{align*}
			\left\{\begin{array}{clr}
				\Phi(x,|\nabla \varphi|)F(D^2\varphi)-H(x,\nabla \varphi)+(c(x)+\beta^+)\varphi^{i(\Phi)+1}&\hspace*{-0.2cm}=-1\text{ in }\Omega,\\
				\varphi&\hspace*{0.1cm}=0 \mbox{ on }\partial\Omega.
			\end{array}\right.
		\end{align*}
		By the strong maximum principle (Theorem \ref{strong-com}), we obtain $\varphi>0$ in $\Omega$. For operator $\mathcal{G}$ defined in \eqref{mathcal_G} and $\varphi_{\varepsilon}=\varphi+\varepsilon$, we derive the following estimate
		\begin{align*}
			&\mathcal{G}(x,\varphi_{\varepsilon}, \nabla \varphi_{\varepsilon}, D^2\varphi_{\varepsilon})\le -1-\beta^+\varphi^{i(\Phi)+1}+c(x)(\varphi_{\varepsilon}^{i(\Phi)+1}-\varphi^{i(\Phi)+1}).
		\end{align*}
		Choose $\delta>0$ such that $\delta(1+\|\varphi\|_{L^{\infty}(\Omega)})^{i(\Phi)+1}\le\frac{1}{4}$ and $\mu=\beta^++\delta$, we obtain
		\begin{align*}
			&\mathcal{G}(x,\varphi_{\varepsilon}, \nabla \varphi_{\varepsilon}, D^2\varphi_{\varepsilon})+\mu(\varphi_{\varepsilon})^{i(\Phi)+1}\\
			&\le -1+\delta(\varphi_{\varepsilon})^{i(\Phi)+1}+(c(x)+\beta^+)(\varphi_{\varepsilon}^{i(\Phi)+1}-\varphi^{i(\Phi)+1})\\
			&\le -1+\delta(1+\|\varphi\|_{L^{\infty}(\Omega)})^{i(\Phi)+1}+\|c+\beta^+\|_{L^{\infty}(\Omega)}o(1)\\
			&<\dfrac{-1}{2}<0,
		\end{align*}
		provided that $\varepsilon>0$ is small enough. This contradicts the definition of $\beta^+$, which confirms the claim. Now define $\varphi_n=[\|\psi_n\|_{L^{\infty}(\Omega)}]^{-1}\psi_n$. Then we obtain a principal eigenfunction by 
		using Theorem \ref{Lipschitz_estimate} and Lemma \ref{dist} and passing to the limit. Global Lipschitz regularity of the principal eigenfunction is obtained by Theorem \ref{Lipschitz_estimate} and Corollary \ref{local_estimate_(1.1)}.
		
		\textit{Step 2: Proof of \upshape\hyperref[eigenfunction_2]{(ii)}.} The proof is analogous to that of part \hyperref[eigenfunction_1]{\upshape(i)}, with Lemma \ref{normalize}\hyperref[Normal_eigen_2]{\upshape(ii)} replacing Lemma \ref{normalize}\hyperref[Normal_eigen_1]{\upshape(i)}.
	\end{proof}		
	\section{\bf Liouville properties}\label{section: 6}
	In the last section, we prove Liouville-type results as direct consequences of Harnack’s inequality - Lemma \ref{cor:scaled-Harnack}. 
	\begin{theorem}[\bf Liouville-type result I] 
		Under the assumptions {\upshape\hyperref[A1]{(A1)}--\hyperref[A3]{(A3)}} and {\upshape\hyperref[A6]{(A6)}}, let $u\in C(\mathbb{R}^n)$ be a viscosity solution of the equation
		\begin{align}\label{38}
			\Phi(x,|\nabla u|)F(D^2u)-H(x,\nabla u)=0\text{ in } \mathbb{R}^n.
		\end{align}
		If $\displaystyle\inf_{\mathbb{R}^n} u>-\infty$ and there exists $x_0\in\mathbb{R}^n$ such that $u(x_0)=\displaystyle\inf_{\mathbb{R}^n}u$, then $u$ is necessarily constant.
	\end{theorem}
	\begin{proof}
		Let $m=\displaystyle\inf_{\mathbb{R}^n}u\in\mathbb{R}$ and $w=u-m$. We have $w(x_0)=0$ and $w\geq 0$ in $\mathbb{R}^n$.
		Since equation \eqref{38} is invariant under addition of constants, the function $w$ is still a viscosity solution of
		\begin{align*}
			\Phi(x,|\nabla w|)F(D^2w)-H(x,\nabla w)=0 \text{ in }\mathbb{R}^n.
		\end{align*}
		By scaled Harnack inequality - Lemma \ref{cor:scaled-Harnack}, applied to $w$, it follows that
		\begin{align*}
			\sup_{B_{1/2}(x_0)}w
			\leq C_H\inf_{B_{1/2}(x_0)}w.
		\end{align*}
		Since $\inf\limits_{B_{1/2}(x_0)}w=0$ then $w=0$ in $B_{1/2}(x_0)$.
		Next, we consider the following set
		\begin{align*}
			Z=\{x\in\mathbb{R}^n:w(x)=0\}.
		\end{align*}
		We readily obtain that $x_0\in Z$. Since $w$ is a continuous function then $Z$ is a closed set. Analogous arguments enable us to follow that
		\begin{align*}
			B_{1/2}(z)\subset Z,\text{ for all }z\in Z.
		\end{align*}
		Therefore, $Z$ is an open set. As a result, we conclude that $Z=\mathbb{R}^n$, that means $w=0$ in $\mathbb{R}^n$ and $u$ is constant in $\mathbb{R}^n$.
	\end{proof}
	\begin{remark}
		We cannot replace the assumption that the viscosity solution $u$ of \eqref{38} is only bounded below. For example, take functions $\Phi,F,$ and $H$ defined as follows
		\begin{align*}
			\Phi(x,t)=t^{\gamma},~ F(M)=\mathrm{tr}(M), \mbox{ and }H(x,p)= \dfrac{-2x_1}{1+x_1^2}|p|^{\gamma}p_1,
		\end{align*}  
		for all $x\in\mathbb{R}^n,~t\geq 0,~\gamma\geq 0,~M\in\mathbb S^{n\times n},~p\in\mathbb{R}^n$. It is straightforward to check that $\Phi,F,$ and $H$ satisfy the assumptions {\upshape\hyperref[A1]{(A1)}--\hyperref[A3]{(A3)}} and {\upshape\hyperref[A6]{(A6)}}. Let $u:\mathbb{R}^n\to\mathbb{R}$ defined by $u(x)=\mathrm{arctan}x_1$. A direct calculation reveals that 
		\begin{align*}
			&\Phi(x,|\nabla u|)F(D^2u)-H(x,\nabla u)\\
			&=|\nabla u|^{\gamma}\Delta u+\dfrac{2x_1}{1+x_1^2}|\nabla u|^{\gamma}\langle\nabla u,e_1\rangle\\
			&=\left(\dfrac{1}{1+x_1^2}\right)^{\gamma}\cdot\left[-\dfrac{2x_1}{\left(1+x_1^2\right)^2}\right]+\dfrac{2x_1}{1+x_1^2}\cdot \left(\dfrac{1}{1+x_1^2}\right)^{\gamma}\dfrac{1}{1+x_1^2}=0.
		\end{align*}
		Therefore, $u$ is a viscosity solution of \eqref{38}. However, $\displaystyle\inf_{\mathbb{R}^n}u=-\dfrac{\pi}{2}$ and there does not exist $x_0\in\mathbb{R}^n$ such that $u(x_0)=\dfrac{-\pi}{2}$. 
	\end{remark}
	\begin{theorem}[\bf Liouville-type result II]
		Given $c\leq0$ and $c\in C(\mathbb{R}^n) \cap L^\infty(\mathbb{R}^n)$. Suppose the assumptions {\upshape\hyperref[A1]{(A1)}--\hyperref[A3]{(A3)}} and {\upshape\hyperref[A6]{(A6)}} hold in $\mathbb{R}^n$. Let $u\in C(\mathbb{R}^n)$ be a non-negative viscosity solution of
		\begin{align*}
			\Phi(x,|\nabla u|)F(D^2u)-H(x,\nabla u) +c(x)u^{i(\Phi)+1}=0\text{ in }\mathbb{R}^n.
		\end{align*}
		Then there exists a constant $\kappa_0> 0$, depending on $n,\lambda,\Lambda,i(\Phi),s(\Phi),L,\nu_0,\mathcal{C}$, $\|c\|_{L^{\infty}(\mathbb{R}^n)}$, such that if the following condition holds
		\begin{align}\label{growth-con}
			\limsup_{|x|\to\infty} {\dfrac{u(x)}{e^{-\alpha |x|^\beta}}}<\infty,
		\end{align}
		with $\alpha,\beta$ satisfying one of the following assumptions:
		\begin{enumerate}[\upshape(i)]
			\item $\beta>1$ and $\alpha>0$;
			\item $\beta=1$ and $\alpha>\kappa_0$,
		\end{enumerate}
		then $u\equiv0$ in $\mathbb{R}^n$.
	\end{theorem}
	\begin{proof}
		Suppose by contradiction that $u\not\equiv0$. Since $u\geq0$, there exists $x_0\in\mathbb{R}^n$ such that $u(x_0)>0$. Let $y\in\mathbb{R}^n$ be arbitrary. We connect $x_0$ to $y$ by a finite Harnack chain. More precisely, let us define $N=[2|x_0-y|]+1$ and 
		\begin{align*}
			x_k=x_0+\dfrac{k}{N}(y-x_0), \mbox{ for } k=\overline{1,N}.
		\end{align*}
		It is simple to verify that
		\begin{align}\label{1}
			|x_k-x_{k-1}|\leq \frac{1}{2}, \mbox{ and } N\leq 2(1+|x_0-y|).
		\end{align}
		For each $k$, we choose a unit ball whose half ball contains both
		$x_{k-1}$ and $x_k$. Applying the scaled Harnack inequality - Lemma \ref{cor:scaled-Harnack} successively, we obtain
		\begin{align*}
			u(x_{k-1})\leq C_H u(x_k)\text{ for all }k\in\mathbb{N},
		\end{align*}
		where $C_H=C_H\left(i(\Phi),s(\Phi),n,\nu_0,L,\lambda,\Lambda,\mathcal{C},\|c\|_{L^{\infty}(\mathbb{R}^n)}\right)>1$. As a consequence, we arrive at
		\begin{align}\label{2}
			u(y)\geq C_H^{-N}u(x_0).
		\end{align}
		Combining \eqref{1} and \eqref{2}, we deduce
		\begin{align*}
			u(y)\geq A e^{-\kappa_0 |y|}\text{ for all }y\in\mathbb{R}^n,
		\end{align*}
		for $A=u(x_0)C_H^{-2(1+|x_0|)}>0$ and $\kappa_0=2\log C_H$. On the other hand, by assumption \eqref{growth-con}, there exists $R>0$ large enough such that for all $y\in\mathbb{R}^n$ with $|y|>R$, we have
		\begin{align*}
			u(y)\leq C e^{-\alpha |y|^\beta},
		\end{align*}
		for some $C>0$. Combining the two estimates gives
		\begin{align}\label{L2}
			A e^{-\kappa_0 |y|} \leq C e^{-\alpha |y|^\beta} \iff A\leq C e^{\kappa_0 |y|-\alpha |y|^\beta}.
		\end{align}
		We now distinguish between the following two cases based on the range of parameters $\beta,
		\alpha$.
		\begin{enumerate}[\upshape(i)]
			\item If $\beta>1$ and $\alpha > 0$, then
			\begin{align*}
				\kappa_0 |y|-\alpha |y|^\beta\to-\infty\text{ as } |y|\to\infty,
			\end{align*}
			which implies that the right-hand side of \eqref{L2} converges to $0$ when $|y|\to\infty$, yielding a contradiction.
			\item If $\beta=1$ and $\alpha>\kappa_0$, then
			\begin{align*}
				\kappa_0 |y|-\alpha |y|=-(\alpha-\kappa_0)|y|\to-\infty\text{ as } |y|\to\infty,
			\end{align*}
			again enabling us to obtain a contradiction.
		\end{enumerate}
		Thus $u\equiv0$ in $\mathbb{R}^n$.
	\end{proof}
	\renewcommand{\refname}{
\end{document}